\documentclass[a4paper,twoside,final]{report}
\pdfoutput=1

\usepackage[T1]{fontenc}				
\usepackage[utf8]{inputenc}				
\usepackage[francais]{layout}
\usepackage[francais,english]{babel}

\usepackage{amsmath,amssymb,amsfonts,mathabx,stmaryrd}
\usepackage{amsthm}
\usepackage{xargs}
\usepackage{mdwlist}
\usepackage[poly, all]{xy}
\usepackage{ifthen}

\usepackage{soul}
\setuldepth{a}

\usepackage[pdftex]{graphicx}
\usepackage{enumitem}

\usepackage{ltugcomn}

\usepackage{ifdraft}
\usepackage{comment} 

\usepackage[obeyDraft,textwidth=2cm]{todonotes}
\input{hyperrefsetup.tex}

\usepackage{aliascnt}
\def\NewTheorem#1{%
  \newaliascnt{#1}{thm}
  \newtheorem{#1}[#1]{\csname #1Name\endcsname}
  \aliascntresetthe{#1}
  \expandafter\def\csname #1autorefname\endcsname{\csname #1name\endcsname}
  \expandafter\def\csname #1Autorefname\endcsname{\csname #1Name\endcsname}
}

\newcommand{\theoremName}{\iflanguage{francais}{Th\'eor\`eme}{Theorem}}

\newcommand{\pbName}{\iflanguage{francais}{Probl\`eme}{Problem}}

\newtheorem{thm}{\theoremName}[subsection]

\newtheorem{pbnonumber}{\pbName}

\newtheorem{thmintro}{\theoremName}

\NewTheorem{lem}
\NewTheorem{prop}
\NewTheorem{cor}
\NewTheorem{conj}
\NewTheorem{pb}

\theoremstyle{definition}
\NewTheorem{df}
\NewTheorem{exo}

\theoremstyle{remark}
\NewTheorem{ex}
\NewTheorem{rmq}

\makeatletter
\renewenvironment{proof}[1][]{\par
  \pushQED{\qed}%
  \normalfont \topsep6\p@\@plus6\p@\relax
  \trivlist
  \item[\hskip\labelsep
        \bfseries
    \proofname\ifthenelse{\equal{#1}{}}{}{\textmd{ (#1)}}\@addpunct{.}]\ignorespaces
}{%
  \popQED\endtrivlist\@endpefalse
}
\makeatother

\makeatletter
\@addtoreset{thm}{section}
\makeatother

 \input{mathbinop.tex}

\newcommand{\N}{\mathbb{N}}
\newcommand{\Z}{\mathbb{Z}}

\newcommand{\A}{\mathbb{A}}

\newcommand{\Lcot}{\mathbb{L}}
\newcommand{\T}{\mathbb{T}}

\newcommand{\pt}{{*}}
\newcommand{\id}{\operatorname{id}}
\newcommand{\ev}{\operatorname{ev}}
\newcommand{\pr}{\operatorname{pr}}
\newcommand{\dual}[1]{{#1}^{\vee}}
\newcommand{\op}{^{\mathrm{op}}}

\newcommand{\noloc}{\,:}

\newcommand{\loccit}{\emph{loc. cit.} }

 % ...
\newcommand{\Ff}{\mathcal{F}} 
\newcommand{\Gg}{\mathcal{G}} 
\newcommand{\Cc}{\mathcal{C}} 
\newcommand{\Dd}{\mathcal{D}} 
 % ...
\newcommand{\Oo}{\mathcal{O}} 
 % ...

\newcommand{\inftyCat}{\mathbf{Cat}_\infty}

\newcommand{\sSets}{\mathbf{sSets}}

\newcommand{\PresLeft}{\mathbf{Pr}^{\mathrm{L}}_\infty}
\newcommand{\presh}{\operatorname{\mathcal{P}}}
\newcommand{\sifted}{\operatorname{\presh_\Sigma}}

\newcommand{\dAff}{\mathbf{dAff}}
\newcommand{\cdgaunbounded}{\mathbf{cdga}}
\newcommand{\cdga}{\cdgaunbounded^{\leq 0}}

\newcommand{\CAlg}{\mathrm{CAlg}}

\newcommand{\dgArt}{\mathbf{dgArt}}
\newcommand{\dgLie}{\mathbf{dgLie}}
\newcommand{\dgAlg}{\mathbf{dgAlg}}
\newcommand{\dgMod}{\mathbf{dgMod}}

\newcommand{\Perf}{\mathbf{Perf}}

\newcommand{\dSt}{\mathbf{dSt}}

\let\originalleft\left
\let\originalright\right
\renewcommand{\left}{\mathopen{}\mathclose\bgroup\originalleft}
\renewcommand{\right}{\aftergroup\egroup\originalright}

\newcommand{\Map}{\operatorname{Map}}
\newcommand{\Mapstack}{\operatorname{\underline{Ma}p}}

\newcommand{\Homint}{\operatorname{\underline{Hom}}}
\newcommand{\RHomint}{\operatorname{\mathbb{R}\underline{Hom}}}

\newcommand{\Fct}{\operatorname{Fct}}

\DeclareMathOperator*{\colim}{colim}

\newcommand{\B}{\operatorname{B}}

\newcommand{\K}{\operatorname{K}}

\let\OLDtimes\times
\DeclareMathBinOp*{\times}{\OLDtimes}
\let\OLDamalg\amalg
\DeclareMathBinOp*{\amalg}{\OLDamalg}
\let\OLDotimes\otimes
\DeclareMathBinOp*{\otimes}{\OLDotimes}
\let\OLDwedge\wedge
\DeclareMathBinOp*{\wedge}{\OLDwedge}

\newcommandx*{\timesunder}[5][1={},2={},3=-2pt,4=0pt,5=0mm,usedefault]{\times_{\makebox[#5]{\raisebox{#3}{\ensuremath{\scriptstyle #1}}}}^{\makebox[#5]{\raisebox{#4}{\ensuremath{\scriptstyle #2}}}}}

\newcommand{\End}{\operatorname{End}}

\newcommand{\Spec}{\operatorname{Spec}}

\newcommand{\Der}{\operatorname{Der}}

\newcommand{\Gm}{\mathbb{G}_m}

\newcommand{\eldebutpardefaut}{1}
\newcommandx*{\el}[4][2=\eldebutpardefaut,4={,}]{#1_{#2}#4\dots#4#1_{#3}}
\newcommandx*{\ael}[7][3=\eldebutpardefaut,5={,},6=(,7=),usedefault]{#1\left#6#2_{#3}\right#7#5\dots#5#1\left#6#2_{#4}\right#7}
\newcommandx*{\iel}[5][2=i,3=\eldebutpardefaut,5={,}]{#1_{#2_{#3}}#5\dots#5#1_{#2_{#4}}}

\let\OLDto\to
\DeclareArrow{\to}{\OLDto}
\DeclareArrow{\from}{\leftarrow}

\newcommand{\app}[4]{\begin{array}{c@{\hskip 2pt}c@{\hskip 2pt}c} #1 & \to & #2 \\ #3 & \mapsto & #4 \end{array}}

\newcommand{\quot}[2]{\ensuremath \mathchoice
{
\displaystyle #1 
\raisebox{-2pt}{$\displaystyle \hspace{-1pt}{/} $}
\raisebox{-4pt}{$\displaystyle \hspace{-1pt}{#2}$}
}{
\textstyle #1 
\raisebox{-1pt}{$\textstyle \hspace{-1pt}{/} $}
\raisebox{-2pt}{$\textstyle \hspace{-1pt}{#2}$}
}{
\scriptstyle #1 
\raisebox{-1pt}{$\scriptstyle \hspace{-1pt}{/} $}
\raisebox{-2pt}{$\scriptstyle \hspace{-1pt}{#2}$}
}{
\scriptscriptstyle #1 
\raisebox{-1pt}{$\scriptscriptstyle \hspace{-1pt}{/} $}
\raisebox{-2pt}{$\scriptscriptstyle \hspace{-1pt}{#2}$}
}
}

\newcommand{\comma}[2]{\ensuremath \mathchoice
{
\raisebox{4pt}{$\displaystyle #1 $}
\raisebox{2pt}{$\displaystyle / $}
\displaystyle \hspace{-1pt}{#2}
}{
\raisebox{2pt}{$\textstyle #1 $}
\raisebox{1pt}{$\textstyle / $}
\textstyle \hspace{-1pt}{#2}
}{
\raisebox{2pt}{$\scriptstyle #1 $}
\raisebox{1pt}{$\scriptstyle / $}
\scriptstyle \hspace{-1pt}{#2}
}{
\raisebox{2pt}{$\scriptscriptstyle #1 $}
\raisebox{1pt}{$\scriptscriptstyle / $}
\scriptscriptstyle \hspace{-1pt}{#2}
}
}

\newcommand{\mymatrix}{\shorthandoff{;:!?} \xymatrix}

\newcommandx*{\dcell}[6][1,2,3,4,5={=>},6={1pc},usedefault]{\ar@/^#6/[#1]^{#2}_{}="UP" \ar@/_#6/[#1]_{#3}^{}="DOWN" \ar @{#5} "UP";"DOWN" ^{#4} }
\newcommand{\commute}{\ar@{}|-{\circlearrowleft}}

\newcommandx*{\hfibre}[3][3={}]{\xymatrix{ #1 \times^h \ifthenelse{\equal{#3}{}}{}{_{#3}} #2 \ar[d] \ar[r] & #1 \ar[d] \\ #2 \ar[r] & \ifthenelse{\equal{#3}{}}{\bullet}{#3}}}

\shorthandoff{:;!?}
\newcommandx*{\cart}[3][1=1,2=5,3=10,usedefault]{\ar@{-}[]+D+<#3pt,#1pt>+<#2pt,0pt>;[]+D+<#3pt,-#3pt>+<#2pt,#1pt> \ar@{-}[]+D+<0pt,-#3pt>+<#2pt,#1pt>;[]+D+<#3pt,-#3pt>+<#2pt,#1pt>}
\newcommandx*{\cocart}[3][1=-1,2=8,3=10,usedefault]{\ar@{-}[]+U+<-#2pt,-#1pt>;[]+U+<-#2pt,-#1pt>+<0pt,#3pt> \ar@{-}[]+U+<-#2pt,-#1pt>;[]+U+<-#2pt,-#1pt>+<-#3pt,0pt>}

\usepackage{geometry}
\setlist[enumerate]{label=\emph{(\roman*)},ref=\emph{(\roman*)}}

\entrymodifiers={!!<0pt,0.7ex>+}

\title{Formal loops and tangent Lie algebras}
\author{Benjamin Hennion\\PhD thesis\\ \vspace{5mm}Université de Montpellier -- France}
\date{\today}

\newcommand{\ptfin}{\mathrm{Fin}^\pt}

\newcommand{\Ind}{\operatorname{\mathbf{Ind}}}
\newcommand{\Pro}{\operatorname{\mathbf{Pro}}}
\newcommand{\Tate}{\operatorname{\mathbf{Tate}}}

\newcommand{\Indu}[1]{\Ind^{\mathbb{#1}}}
\newcommand{\Prou}[1]{\Pro^{\mathbb{#1}}}
\newcommand{\Tateu}[1]{\Tate^{\mathbb{#1}}}
\newcommand{\inftyCatu}[1]{\inftyCat^{\mathbb{#1}}}
\newcommand{\PresLeftu}[1]{\mathbf{Pr}^{\mathrm{L,}\mathbb{#1}}_\infty}
\newcommand{\PresRightu}[1]{\mathbf{Pr}^{\mathrm{R,}\mathbb{#1}}_\infty}
\newcommand{\monoidalinftyCatu}[1]{\inftyCat^{\otimes, \mathbb{#1}}}

\renewcommand{\j}{\operatorname{j}}

\newcommand{\IPP}{\operatorname{\mathbf{IPP}}}

\newcommand{\Qcoh}{\operatorname{\mathbf{Qcoh}}}
\newcommand{\IP}{\mathbf{IP}}
\newcommand{\PI}{\mathbf{PI}}
\newcommand{\PIQ}{\mathbf{PIQ}}
\newcommand{\PIPerf}{\operatorname{\mathbf{PIPerf}}}
\newcommand{\IPPerf}{\operatorname{\mathbf{IPPerf}}}
\newcommand{\IPerf}{\operatorname{\mathbf{IPerf}}}
\newcommand{\PPerf}{\operatorname{\mathbf{PPerf}}}
\newcommand{\PIQcoh}{\operatorname{\mathbf{PIQcoh}}}

\newcommand{\cotangent}{\lambda}

\newcommand{\Indext}{\operatorname{\underline{\mathbf{Ind}}}^\mathbb U}
\newcommand{\Proext}{\operatorname{\underline{\mathbf{Pro}}}^\mathbb U}
\newcommand{\Indextu}[1]{\operatorname{\underline{\mathbf{Ind}}}^\mathbb{#1}}
\newcommand{\Tateextu}[1]{\operatorname{\underline{\mathbf{Tate}}}^\mathbb{#1}}
\newcommand{\Proextu}[1]{\operatorname{\underline{\mathbf{Pro}}}^\mathbb{#1}}

\newcommand{\undercat}{\operatorname{U}}
\newcommand{\overcat}{\operatorname{O}}  
\newcommand{\btw}{\operatorname{\mathrm{B}}}

\newcommand{\bubblestack}{\operatorname{\underline{\mathfrak{B}}}}
\newcommand{\bubblespace}{\operatorname{\mathfrak{B}}}
\newcommand{\formalsphere}{{\hat{\mathrm{S}}}}

\newcommand{\kaploop}{\mathcal L}
\newcommand{\kaplooppre}{\tilde{\mathcal L}}

\newcommand{\dStArt}{\dSt^{\mathrm{Art}}}
\newcommand{\dStArtlfp}{\dSt^{\mathrm{Art,lfp}}}
\newcommand{\dStptArt}{\dSt^{\pt,\operatorname{Art}}}

\newcommand{\Tatestack}{\dSt^{\mathrm{Tate}}}

\newcommand{\IQcoh}{\operatorname{\mathbf{IQcoh}}}

\newcommand{\shybounded}{\mathbf{IP} \dSt^{\mathrm{shy,b}}}

\newcommand{\closedforms}[1]{\mathbf A^{#1\mathrm{,cl}}}
\newcommand{\forms}[1]{\mathbf A^{#1}}
\newcommand{\IPclosedforms}[1]{\closedforms{#1}_{\IP}}
\newcommand{\IPforms}[1]{\forms{#1}_{\IP}}

\newcommand{\dStF}{\dSt^{\operatorname{f}}}
\newcommand{\Tens}{\operatorname{T}}
\newcommand{\homol}{\operatorname H}
\newcommand{\coho}{\operatorname C}

\newcommand{\dgLieLib}{\dgLie^{\operatorname{f,ft,}\geq1}}
\newcommand{\dgLieGood}{\dgLie^{\operatorname{good}}}
\newcommand{\dgModLib}{\dgMod^{\operatorname{f,ft,}\geq1}}
\newcommand{\dgExt}{\operatorname{\mathbf{dgExt}}}
\newcommand{\dgRep}{\operatorname{\mathbf{dgRep}}}

\newcommand{\gl}{\mathfrak{gl}}

\newcommand{\Envel}{\operatorname{\mathcal U \hspace{-1pt}}}
\newcommand{\Sym}{\operatorname{Sym}}

\newcommand{\libre}{\operatorname{Free}}
\newcommand{\oubli}{\operatorname{Forget}}

\newcommand{\unit}{1}

\newcommand{\for}[1]{{#1}^{\operatorname{f}}}

\newcommand{\lie}{\mathfrak{L}}
\newcommand{\formal}{\mathcal{F}}
\newcommand{\Lqcoh}{\operatorname{L}_{\operatorname{qcoh}}}

\newcommand{\tgtlie}{\ell}

\newcommand{\Ee}{\mathcal E}

\newcommand{\lierep}{\operatorname{Rep}}

\newcommand{\atiyah}{\operatorname{\mathbf{at}}}

\newcommand{\MCdgMod}{\mathrm{dgMod}}
\newcommand{\MCdgRep}{\mathrm{dgRep}}
\newcommand{\MCdgLie}{\mathrm{dgLie}}
\newcommand{\MCdgLieLib}{\MCdgLie^{\operatorname{f,ft,}\geq1}}
\newcommand{\MCcdgaunbounded}{\mathrm{cdga}}
\newcommand{\MCcdga}{\mathrm{cdga}^{\leq 0}}
\newcommand{\MCdgAlg}{\mathrm{dgAlg}}

\newcommand{\adjrep}{\mathrm{Ad}}

\usepackage[toc,nomain,nostyles]{glossaries}

\usepackage{longtablex}

\newglossarystyle{notationlist}{%
{  }
{  }%
\renewcommand*{\glossaryheader}{%
}%
\renewcommand*{\glsgroupheading}[1]{}%
\renewcommand*{\glossentry}[2]{%
\glsentryitem{##1}
\glossentrysymbol{##1}
& \glossentrydesc{##1}
& p.##2
\vspace{0.5pt}%
\tabularnewline
}%
\renewcommand*{\subglossentry}[3]{%
\glstarget{##2}{\glossentryname{##2}}
& \glossentrydesc{##2}
& \glossentrysymbol{##2}
& \glsentryuseri{##2}
& \glsentryuserii{##2}
& ##3
\tabularnewline 
}%
\renewcommand*{\glsgroupskip}{}%
}

\newglossary{not}{gls}{glo}{Notations glossary}
\noist
\makeglossaries

\newglossaryentry{ind}{name=ind,symbol=$\Indu U(\Cc)$,
  description={Category of $\mathbb U$-small ind-objects in $\Cc$}}

\newglossaryentry{pro}{name=pro,symbol=$\Prou U(\Cc)$,
  description={Category of $\mathbb U$-small pro-objects in $\Cc$}}

\newglossaryentry{indext}{name=indext,symbol=$\Indext_\Cc$,
  description={Ind-extension functor to $\Indu U(\Cc)$}}

\newglossaryentry{proext}{name=proext,symbol=$\Proext_\Cc$,
  description={Pro-extension functor to $\Prou U(\Cc)$}}

\newglossaryentry{catstable}{name=catstable,symbol=$\inftyCat^{\mathbb V\mathrm{,st}}$,
  description={Category of $\mathbb V$-small stable $(\infty,1)$-categories}}

\newglossaryentry{catsstableid}{name=catstableid,symbol=$\inftyCat^{\mathbb V\mathrm{,st,id}}$,
  description={Category of $\mathbb V$-small stable and idempotent complete $(\infty,1)$-categories}}

\newglossaryentry{tate0}{name=tate0,symbol=$\Tateu U_0(\Cc)$,
  description={Category of pure Tate objects in $\Cc$}}

\newglossaryentry{tateel}{name=tateel,symbol=$\Tateu U_\mathrm{el}(\Cc)$,
  description={Category of elementary Tate objects in $\Cc$}}

\newglossaryentry{tate}{name=tate,symbol=$\Tateu U(\Cc)$,
  description={Category of Tate objects in $\Cc$}}

\newglossaryentry{K}{name=K,symbol=$\mathrm K(\Cc){,}\mathbb K(\Cc)$,
  description={Connective (resp. non-connective) $K$-theory spectrum of $\Cc$}}

\newglossaryentry{cone}{name=Kcone,symbol=$K^\triangleright$,
  description={Simplicial set obtained from $K$ by adding a final object}}

\newglossaryentry{overcat}{name=overcat,symbol=$\overcat_\Cc$,
  description={Functor $\Cc \to \inftyCatu V$ mapping $c \in \Cc$ to the overcategory $\quot{\Cc}{c}$}}

\newglossaryentry{undercat}{name=undercat,symbol=$\undercat_\Cc$,
  description={Functor $\Cc\op \to \inftyCatu V$ mapping $c \in \Cc$ to the undercategory $\comma{c}{\Cc}$}}

\newglossaryentry{btwamalg}{name=btwamalg,symbol=$\btw_\Cc^{\amalg}$,
  description={Functor $\Cc^{\Delta^1} \to \inftyCatu V$ mapping $c \to d \in \Cc$ to the category $\comma{c}{\quot{\Cc}{d}}$ -- with pushforward functors}}

\newglossaryentry{btwtimes}{name=btwtimes,symbol=$\btw_\Cc^{\times}$,
  description={Functor $(\Cc^{\Delta^1})\op \to \inftyCatu V$ mapping $c \to d \in \Cc$ to the category $\comma{c}{\quot{\Cc}{d}}$  -- with pullback functors}}

\newglossaryentry{overprod}{name=overprod,symbol=$\overcat_\Cc^{\times}$,
  description={Functor $\Cc\op \to \inftyCatu V$, $c \mapsto \quot{\Cc}{c}$ -- with pullback functors}}

\newglossaryentry{underamalg}{name=underzamalg,symbol=$\undercat_\Cc^{\amalg}$,
  description={Functor $\Cc \to \inftyCatu V$, $c \mapsto \comma{c}{\Cc}$ -- with pushforward functors}}

\newglossaryentry{bubble}{name=bubble,symbol=$\bubblestack^d(X)$,
  description={Ind-pro-stack of bubbles of dimension $d$ in $X$}}

\newglossaryentry{catinfty}{name=catinfty,symbol=$\inftyCatu U$,
  description={Category of $\mathbb U$-small $(\infty,1)$-categories}}

\newglossaryentry{presleft}{name=presleft,symbol=$\PresLeftu U$,
  description={Category of $\mathbb U$-presentable $(\infty,1)$-categories}}

\newglossaryentry{fct}{name=fct,symbol=$\Fct(\Cc{,}\Dd)$,
  description={Category of functors $\Cc \to \Dd$}}

\newglossaryentry{ssets}{name=ssets,symbol=$\sSets$,
  description={Category of spaces}}

\newglossaryentry{presh}{name=presh,symbol=$\presh(\Cc)$,
  description={Category of presheaves over $\Cc$}}

\newglossaryentry{cdga}{name=cdga,symbol=$\cdga_A$,
  description={Category of cdga's over $A$}}

\newglossaryentry{cdgaunbounded}{name=cdgaunbounded,symbol=$\cdgaunbounded_A$,
  description={Category of unbounded cdga's over $A$}}

\newglossaryentry{daff}{name=daff,symbol=$\dAff_k$,
  description={Category of derived affine schemes over $k$}}

\newglossaryentry{dst}{name=dst,symbol=$\dSt_k$,
  description={Category of derived stacks over $k$}}

\newglossaryentry{dstart}{name=dstart,symbol=$\dStArt_k$,
  description={Category of derived Artin stacks over $k$}}

\newglossaryentry{qcoh}{name=qcoh,symbol=$\Qcoh(X)$,
  description={Derived category of quasi-coherent sheaves on $X$}}

\newglossaryentry{perf}{name=perf,symbol=$\Perf(X)$,
  description={Derived category of perfect complexes on $X$}}

\newglossaryentry{map}{name=map,symbol=$\Map_\Cc(c{,}d)$,
  description={Mapping space from $c$ to $d$ in $\Cc$}}
  
\newglossaryentry{mapstack}{name=mapstack,symbol=$\Mapstack(X{,}Y)$,
  description={Mapping stack from $X$ to $Y$}}

\newglossaryentry{iperf}{name=iperf,symbol=$\IPerf(X)$,
  description={Derived category of ind-perfect complexes on $X$}}

\newglossaryentry{pperf}{name=pperf,symbol=$\PPerf(X)$,
  description={Derived category of pro-perfect complexes on $X$}}

\newglossaryentry{piperf}{name=piperf,symbol=$\PIPerf(X)$,
  description={Derived category of pro-ind-perfect complexes on $X$}}

\newglossaryentry{ipperf}{name=ipperf,symbol=$\IPPerf(X)$,
  description={Derived category of ind-pro-perfect complexes on $X$}}

\newglossaryentry{cotangent}{name=Lcotangent,symbol=$\Lcot_{X/S}$,
  description={Cotangent (pro-)(ind-)complex of $X$ over $S$}}

\newglossaryentry{iqcoh}{name=iqcoh,symbol=$\IQcoh(X)$,
  description={Derived category of ind-quasi-coherent sheaves on $X$}}

\newglossaryentry{piqcoh}{name=piqcoh,symbol=$\PIQcoh(X)$,
  description={Derived category of pro-ind-quasi-coherent sheaves on $X$}}

\newglossaryentry{trivext}{name=xtrivext,symbol=$X[M]$,
  description={Trivial square zero extension of $X$ by $M$}}

\newglossaryentry{shybounded}{name=ipdstshybounded,symbol=$\shybounded_S$,
  description={Category of shy and bounded ind-pro-stacks over $S$}}

\newglossaryentry{forms}{name=aforms,symbol=$\forms p(X){,} \closedforms p(X)$,
  description={Complex of $k$-modules of (closed) $p$-forms on $X$}}

\newglossaryentry{ipforms}{name=aformsip,symbol=$\IPforms p(X){,} \IPclosedforms p(X)$,
  description={Pro-ind-complex of $k$-modules of (closed) $p$-forms on $X$}}

\newglossaryentry{tatestack}{name=dsttate,symbol=$\Tatestack_S$,
  description={Category of Tate stacks over $S$}}

\newglossaryentry{derham}{name=xderham,symbol=$X_\mathrm{dR}$,
  description={De Rham prestack associated to $X$}}

\newglossaryentry{loopstack}{name=loopstack,symbol=$\kaploop^d(X)$,
  description={Stack of dimension $d$ loops in $X$}}
  
\newglossaryentry{loopipstack}{name=loopzipstack,symbol=$\underline \kaploop^d(X)$,
  description={Ind-pro-stack of dimension $d$ loops in $X$}}

\newglossaryentry{sifted}{name=psifted,symbol=$\sifted(\Cc)$,
  description={Category of sifted objects in $\Cc$}}

\newglossaryentry{dgmod}{name=dgmod,symbol=$\dgMod_A$,
  description={Category of $A$-dg-modules}}

\newglossaryentry{dglie}{name=dglie,symbol=$\dgLie_A$,
  description={Category of $A$-dg-Lie algebras}}

\newglossaryentry{dgext}{name=dgext,symbol=$\dgExt_A$,
  description={Category of trivial square zero extensions of $A$}}

\newglossaryentry{dgrep}{name=dgrep,symbol=$\dgRep_A(L)$,
  description={Category of representations of the $A$-dg-Lie algrebra $L$}}

\newglossaryentry{dstf}{name=dstf,symbol=$\dStF_A$,
  description={Category of formal stacks over $A$}}

\newglossaryentry{ptfin}{name=finpt,symbol=$\ptfin$,
  description={Category of pointed finite sets}}

\newglossaryentry{monoidalcats}{name=catmonoidal,symbol=$\monoidalinftyCatu V$,
  description={Category of $\mathbb U$-small symmetric monoidal $(\infty,1)$-categories}}

\newglossaryentry{lqcoh}{name=lqcoh,symbol=$\Lqcoh^X(Y)$,
  description={Derived category of the $X$-formal stack $Y$}}

\newglossaryentry{tgtlie}{name=lietgt,symbol=$\tgtlie_X$,
  description={Tangent Lie algebra of the stack $X$}}

\newglossaryentry{lie}{name=lie,symbol=$\lie_X$,
  description={Functor mapping a formal stack over $X$ to its tangent Lie algebra}}

\newglossaryentry{formal}{name=formal,symbol=$\formal_X$,
  description={Functor mapping a Lie algebra over $X$ to a formal stack}}

\newglossaryentry{chevalley}{name=chevalley,symbol=$\coho_A L$,
  description={Chevalley-Eilenberg cohomology complex of the $A$-dg-Lie algebra $L$}}

\setglossarystyle{notationlist}

\newlist{assertions}{enumerate}{1}
\setlist[assertions]{label={(\alph*)}, ref={assertion (\alph*)}}
\newlist{disjunction}{enumerate}{1}
\setlist[disjunction]{label={(\arabic*)}, ref={case (\arabic*)}}

\begin{document}

\selectlanguage{english}
\maketitle

\includecomment{chap-intro}
\includecomment{chap-prelim}
\includecomment{chap-cats}
\includecomment{chap-indpro}
\includecomment{chap-loops}
\includecomment{chap-tgtlie}
\includecomment{chap-perspect}

\begin{abstract}
If $M$ is a symplectic manifold then the space of smooth loops $\mathrm C^{\infty}(\mathrm S^1,M)$ inherits of a quasi-symplectic form. We will focus in this thesis on an algebraic analogue of that result.
In their article \cite{kapranovvasserot:loop1}, Kapranov and Vasserot introduced and studied the formal loop space of a scheme $X$. It is an algebraic version of the space of smooth loops in a differentiable manifold.

We generalize their construction to higher dimensional loops. To any scheme $X$ -- not necessarily smooth -- we associate $\kaploop^d(X)$, the space of loops of dimension $d$. We prove it has a structure of (derived) Tate scheme -- ie its tangent is a Tate module: it is infinite dimensional but behaves nicely enough regarding duality.
We also define the bubble space $\bubblespace^d(X)$, a variation of the loop space.
We prove that $\bubblespace^d(X)$ is endowed with a natural symplectic form as soon as $X$ has one (in the sense of \cite{ptvv:dersymp}).

To prove our results, we develop a theory of Tate objects in a stable $(\infty,1)$-category $\Cc$. We also prove that the non-connective K-theory of $\Tate(\Cc)$ is the suspension of that of $\Cc$, giving an $\infty$-categorical version of a result of \cite{saito:deloop}.

The last chapter is aimed at a different problem: we prove there the existence of a Lie structure on the tangent of a derived Artin stack $X$. Moreover, any quasi-coherent module $E$ on $X$ is endowed with an action of this tangent Lie algebra through the Atiyah class of $E$. This in particular applies to not necessarily smooth schemes $X$.

Throughout this thesis, we will use the tools of $(\infty,1)$-categories and symplectic derived algebraic geometry.
\end{abstract}

\pagenumbering{roman}
\tableofcontents
\listoftodos

\chapter*{Introduction}
\addcontentsline{toc}{chapter}{Introduction}%
\begin{chap-intro}
\section*{Context}
Considering a differential manifold $M$, one can build the space of smooth loops $\operatorname{L}(M)$ in $M$. It is a central object of string theory. Moreover, if $M$ is symplectic then so is $\operatorname{L}(M)$ -- more precisely quasi-symplectic since it is not of finite dimension -- see for instance \cite{munozpresas:symp}.
We will be interested here in an algebraic analogue of that result.

The first question is then the following: what is an algebraic analogue of the space of smooth loops? 
An answer appeared in 1994 in Carlos Contou-Carrère's work (see \cite{contoucarrere:jacobienne}). He studies there $\Gm(\mathbb{C}(\!(t)\!))$, some sort of holomorphic functions in the multiplicative group scheme, and uses it to define the famous Contou-Carrère symbol.
This is the first occurrence of a \emph{formal loop space} known to the author.
This idea was then generalised to algebraic groups as the affine grassmanian $\mathfrak{Gr}_G = \quot{G(\mathbb C (\!(t)\!)}{G(\mathbb C[\![t]\!])}$ show up and got involved in the geometric Langlands program.
In their paper \cite{kapranovvasserot:loop1}, Mikhail Kapranov and Éric Vasserot introduced and studied the formal loop space of a smooth scheme $X$. It is an ind-scheme $\kaploop(X)$ which we can think of as the space of functions $\Spec \mathbb C(\!(t)\!) \to X$. This construction strongly inspired the one presented in this thesis.
Kapranov and Vasserot moreover endow the formal loop space with a factorisation structure and link it with vertex algebras and, in \cite{kapranovvasserot:loop4}, with chiral differential operators.

We will focus on higher dimensional formal loops.
There are at least two ways of considering them. The first one consists in replacing $\mathbb C (\!( t )\!)$ with some $d$-dimensional function field $\mathbb C(\!(t_1)\!) \dots (\!(t_d)\!)$, as was done for Beilinson's adèles as well as higher Contou-Carrère symbols -- see \cite{osipovzhu:contoucarrere} for dimension $2$ and \cite{bgw:contoucarrere} for higher dimensions.
This very fruitful idea requires a choice in ordering the local coordinates $\el{t}{d}$.
The author likes to think of adèles as maps out of some formal torus of dimension $d$ -- with a given flag, corresponding to the ordering of the variables $\el{t}{n}$ -- but not out of a formal sphere.
The second way is the one we will use in this thesis. It consists in considering the formal loop space as a derived geometrical object. The author learned after writing down this thesis that Mikhail Kapranov had an unpublished document in which this derived approach was developed. He his grateful to Kapranov for letting him read those notes, both inspired and inspiring.
Derived algebraic geometry appeared in several steps. First Maxim Kontsevich somehow predicted their existence in \cite{kontsevich:torusaction}. Ionut Ciocan-Fontanine and Mikhail Kapranov then introduced the first formal definition of a derived geometrical object: dg-schemes and dg-manifolds -- see \cite{ciocanfontaninekapranov:derivedquot}.
Bertrand Toën and Gabriele Vezzosi then developed a fully equipped theory of derived (higher) stacks using model categories in \cite{toen:hagii}.
The subject then expanded tremendously with the work of Jacob Lurie and his series of DAG papers,  of which we will only use \cite{lurie:dagx}.
He also modernized the theory by developing and using $(\infty,1)$-category theory -- see \cite{lurie:htt} and \cite{lurie:halg}.

Once the (higher dimensional) formal loop space defined, the second question finally comes up: the formal loop space is an infinite dimensional derived geometric object, how can we then talk about symplectic forms?
The formalism for (shifted) symplectic forms over finite derived stacks was developed by Tony Pantev, Bertrand Toën, Michel Vaquié and Gabriele Vezzosi in \cite{ptvv:dersymp}.
One core feature of derived algebraic geometry is the cotangent complex, derived version of the usual cotangent.
For a $2$-form to be non-degenerated, one should introduce a shift so that the amplitudes of the cotangent complex and its shifted dual coincide.
In our situation, the formal loop space is not finitely presented. To deal with its infinite dimensional (co)tangent, we will identify some more structure: it is a Tate module.

Tate vector spaces appeared in John Tate's thesis, and were then used and generalised by several authors. Let us cite here Alexander Beilinson or Vladimir Drinfeld (see \cite{drinfeld:tate}).
The core idea is to consider an infinite dimensional vector space $V$ as a discrete topological space. Its dual $V^*$ is then considered as a topological vector space. Dualising once more on gets $V$ back.
Tate vector spaces are then the extensions of a discrete topological vector space by the dual of one. One gets a family of dualisable -- although infinite dimensional -- vector spaces.
More recently, Oliver Bräunling, Michael Gröchenig and Jesse Wolfson (see \cite{bgw:tate}) developed a general theory for Tate objects in exact categories.

\section*{In this thesis}
In this thesis, we generalize the definition of Kapranov and Vasserot to higher dimensional loops -- which do not have an equivalent in differentiable geometry -- and approach the symplectic question. For $X$ a scheme of finite presentation, not necessarily smooth, we define $\kaploop^d(X)$, the space of formal loops of dimension $d$ in $X$.
We define $\kaploop^d_V(X)$ the space of maps from the formal neighbourhood of $0$ in $\A^d$ to $X$. This is a higher dimensional version of the space of germs of arcs as studied by Jan Denef and François Loeser in \cite{denefloeser:germs}.
Let also $\kaploop_U^d(X)$ denote the space of maps from a \emph{punctured} formal neighbourhood of $0$ in $\A^d$ to $X$. The formal loop space $\kaploop^d(X)$ is the formal completion of $\kaploop_V^d(X)$ in $\kaploop_U^d(X)$.
Understanding those three items is the main goal of this work.
The problem is mainly to give a meaningful definition of the punctured formal neighbourhood of dimension $d$. We can describe what its cohomology should be: 
\[
\homol^n(\hat \A^d\smallsetminus \{0\}) = \begin{cases}
k[\![\el{X}{d}]\!] & \text{ if } n = 0 \\ (\el{X}{d}[])^{-1} k[\el{X^{-1}}{d}] & \text{ if } n=d-1 \\ 0 & \text{ otherwise}
\end{cases}
\]
but defining this punctured formal neighbourhood with all its structure is actually not an easy task. Nevertheless, we can describe what maps out of it are, hence the definition of $\kaploop^d_U(X)$ and the formal loop space.
This geometric object is of infinite dimension, and part of this study is aimed at identifying some structure. Here comes the first result in that direction.
\begin{thmintro}[see \autoref{L-indpro}]\label{intro-kaploop}
The formal loop space of dimension $d$ in a scheme $X$ is represented by a derived ind-pro-scheme.
Moreover, the functor $X \mapsto \kaploop^d(X)$ satisfies the étale descent condition.
\end{thmintro}
We use here methods from derived algebraic geometry as developed by Bertrand Toën and Gabriele Vezzosi in \cite{toen:hagii}. The author would like to emphasize here that the derived structure is necessary since, when $X$ is a scheme, the underlying schemes of $\kaploop^d(X)$, $\kaploop^d_U(X)$ and $\kaploop^d_V(X)$ are isomorphic as soon as $d \geq 2$.
Let us also note that derived algebraic geometry allowed us to define $\kaploop^d(X)$ for more general $X$'s, namely any derived stack. It for instance work for $X$ a classifying stack $\B G$ of an algebraic group. In those cases, the formal loop space $\kaploop^d(X)$ is no longer a derived ind-pro-scheme but an ind-pro-stack.
The case $d = 1$ and $X$ is a smooth scheme gives a derived enhancement of Kapranov and Vasserot's definition. This derived enhancement is conjectured to be trivial when $X$ is a smooth affine scheme in  \cite[9.2.10]{gaitsgoryrozenblyum:dgindschemes}. Gaitsgory and Rozenblyum also prove in \loccit their conjecture holds when $X$ is an algebraic group.

The proof of \autoref{intro-kaploop} is based on an important lemma. We identify a full sub-category $\Cc$ of the category of ind-pro-stacks such that the realisation functor $\Cc \to \dSt_k$ is fully faithful. We then prove that whenever $X$ is a derived affine scheme, the stack $\kaploop^d(X)$ is in the essential image of $\Cc$ and is thus endowed with an \emph{essentially unique} ind-pro-structure satisfying some properties. 
The generalisation to any $X$ is made using a descent argument.
Note that for general $X$'s, the ind-pro-structure is not known to satisfy nice properties one could want to have, for instance on the transition maps of the diagrams.

We then focus on the following problem: can we build a symplectic form on $\kaploop^d(X)$ when $X$ is symplectic?
Again, this question requires the tools of derived algebraic geometry and \emph{shifted symplectic structures} as in \cite{ptvv:dersymp}.

Because $\kaploop^d(X)$ is not finite, linking its tangent to its dual -- through an alleged symplectic form -- requires to identify once more some structure. We already know that it is an ind-pro-scheme but the proper context seems to be what we call Tate stacks.

Before saying what a Tate stack is, let us talk about Tate modules. They define a convenient context for infinite dimensional vector spaces. They where studied by Drinfeld, Beilinson and Lefschetz, among others, and more recently by Previdi \cite{previdi:thesis} and Bräunling, Gröchenig and Wolfson \cite{bgw:tate}.

As discussed above, a Tate module is an extension of a so-called discrete module with the dual of one. We can think of discrete modules as ind-objects in the category of finite dimensional modules. Their duals then naturally live in the category of pro-objects.

We will here introduce the notion of Tate objects in the context of stable $(\infty,1)$-categories as studied by Lurie in \cite{lurie:halg}. If $\Cc$ is a stable $(\infty,1)$-category -- playing the role of the category of finite dimensional vector spaces, we define $\Tate(\Cc)$ as the full subcategory of the $(\infty,1)$-category of pro-ind-objects $\Pro \Ind(\Cc)$ in $\Cc$ containing both $\Ind(\Cc)$ and $\Pro(\Cc)$ and stable by extensions and retracts.
The category $\Ind(\Cc)$ plays the role of the category of discrete modules while $\Pro(\Cc)$ contains their duals. Note that in a general setting, one must ask $\Tate(\Cc)$ to be stable by retracts as well as extensions.
We then have the following universal property
\begin{thmintro}[see \autoref{univpropT0} and \autoref{ktheorysusp}]\label{intro-tate}
For any $(\infty,1)$-category $\Dd$ both stable and idempotent complete, and any commutative diagram
\[
\mymatrix{
\Cc \ar[r]^-i \ar[d]_-j & \Ind(\Cc) \ar[d]^f \\ \Pro(\Cc) \ar[r]_-g & \Dd
}
\]
such that $i$ and $j$ are the canonical inclusions, $f$ and $g$ are exact, $f$ preserves filtered colimits and $g$ preserves cofiltered limits,
there exist an essentially unique functor $\phi \colon \Tate(\Cc) \to \Dd$ such that both $f$ and $g$ factor through $\phi$.
As a consequence, the category $\Tate(\Cc)$ is equivalent to the smallest stable and idempotent complete subcategory of $\Ind \Pro(\Cc)$ generated by both ind- and pro-objects.
Moreover, the non-connective K-theory spectrum of $\Tate(\Cc)$ is the suspension of that of $\Cc$
\[
\mathbb K(\Tate(\Cc)) \simeq \Sigma \mathbb K(\Cc)
\]
\end{thmintro}
Note that the last part of this theorem is an $\infty$-categorical version of a result of Saito in \cite{saito:deloop}.
To prove \autoref{intro-tate} we use tools from \cite{lurie:halg}. Once defined Tate objects in a stable $(\infty,1)$-category, the arising complication is inherent to $(\infty,1)$-categories and the infinitely many coherences one has to check. We have to find a way of expressing them concisely.
The last part of \autoref{intro-tate} is proven by computing the quotients $\quot{\Ind(\Cc)}{\Cc}$ and $\quot{\Tate(\Cc)}{\Pro(\Cc)}$ in stable and idempotent complete $(\infty,1)$-categories and by showing those two categories are equivalent. This proof strategy is strongly inspired by Saito's proof in \cite{saito:deloop} and uses tools developed by Blumberg, Gepner and Tabuada in \cite{bgt:characterisationk}.

We also define the derived category of Tate modules on a scheme -- and more generally on a derived ind-pro-stack.
An Artin ind-pro-stack $X$ -- meaning an ind-pro-object in derived Artin stacks -- is then gifted with a cotangent complex $\Lcot_X$. This cotangent complex inherits a natural structure of pro-ind-module on $X$. This allows us to define a Tate stack as an Artin ind-pro-stack whose cotangent complex is a Tate module.
The formal loop space $\kaploop^d(X)$ is then a Tate stack as soon as $X$ is a finitely presented derived affine scheme. For a more general $X$, what precedes makes $\kaploop^d(X)$ some kind of \emph{locally} Tate stack. This structure suffices to define a determinantal anomaly
\[
\left[\mathrm{Det}_{\kaploop^d(X)}\right] \in \homol^2\left(\kaploop^d(X), \Oo_{\kaploop^d(X)}^{\times}\right)
\]
for any quasi-compact quasi-separated (derived) scheme $X$ -- this construction also works for slightly more general $X$'s, namely Deligne-Mumford stacks with algebraisable diagonal, see \autoref{determinantalanomaly}.
Kapranov and Vasserot proved in \cite{kapranovvasserot:loop4} that in dimension $1$, the determinantal anomaly governs the existence of sheaves of chiral differential operators on $X$.
One could expect to have a similar result in higher dimensions, with higher dimensional analogues of chiral operators and vertex algebras. This author plans on studying this in a future work.

Another feature of Tate modules is duality. It makes perfect sense and behaves properly. Using the theory of symplectic derived stacks developed by Pantev, Toën, Vaquié and Vezzosi in \cite{ptvv:dersymp}, we are then able to build a notion of symplectic Tate stack: a Tate stack $Z$ equipped with a ($n$-shifted) closed $2$-form which induces an equivalence
\[
\T_Z \to^\sim \Lcot_Z[n]
\]
of Tate modules over $Z$ between the tangent and (shifted) cotangent complexes of $Z$.

To make a step toward proving that $\kaploop^d(X)$ is a symplectic Tate stack, we actually study the bubble space $\bubblespace^d(X)$ -- see \autoref{dfbubble}. When $X$ is affine, we get an equivalence
\[
\bubblespace^d(X) \simeq \kaploop_V^d(X) \timesunder[\kaploop_U^d(X)] \kaploop^d_V(X)
\]
We then prove the following result
\begin{thmintro}[see \autoref{B-symplectic}]\label{intro-bubble}
If $X$ is an $n$-shifted symplectic stack then the bubble space $\bubblespace^d(X)$ is endowed with a structure of $(n-d)$-shifted symplectic Tate stack.
\end{thmintro}

The proof of this result is based on a classical method. The bubble space is in fact, as an ind-pro-stack, the mapping stack from what we call the formal sphere $\hat S^d$ of dimension $d$ to $X$.
There are therefore two maps
\[
\mymatrix{
\bubblespace^d(X) & \bubblespace^d(X) \times \hat S^d \ar[r]^-{\ev} \ar[l]_-{\pr} & X
}
\]
The symplectic form on $\bubblespace^d(X)$ is then $\int_{\hat S^d} \ev^* \omega_X$, where $\omega_X$ is the symplectic form on $X$.
The key argument is the construction of this integration on the formal sphere, ie on an oriented pro-ind-stack of dimension $d$.

This method would not work on $\kaploop^d(X)$, since the punctured formal neighbourhood does not have as much structure as the formal sphere: it is not known to be a pro-ind-scheme. Nevertheless, \autoref{intro-bubble} is a first step toward proving that $\kaploop^d(X)$ is symplectic. We can consider the nerve $Z_\bullet$ of the map $\kaploop^d_V(X) \to \kaploop^d_U(X)$. It is a groupoid object in ind-pro-stacks whose space of maps is $\bubblespace^d(X)$. The author expects that this groupoid is compatible in some sense with the symplectic structure so that $\kaploop^d_U(X)$ would inherit a symplectic form from realising this groupoid.
One the other hand, if $\kaploop^d_U(X)$ was proven to be symplectic, then the fibre product defining $\bubblespace^d(X)$ should be a Lagrangian intersection. The bubble space would then inherit a symplectic structure from that on $\kaploop^d(X)$.

In the last chapter, we will focus on a different problem. Whenever $X$ is a variety or a scheme, its shifted tangent complex $\T_X[-1]$ has long been suspected of having a (weak) Lie bracket, given by the Atiyah class.
In \cite{kapranov:atiyah}, Kapranov proved that statement for $X$ a smooth variety. He moreover showed that any quasi-coherent module over $X$ is endowed with a natural Lie action of $\T_X[-1]$.
We will prove in \autoref{chapterlie} the following
\begin{thmintro}\label{intro-tgtlie}
Let $X$ be an algebraic derived stack locally of finite presentation over a field $k$ of characteristic zero.
The shifted tangent complex $\T_X[-1] \in \Qcoh(X)$ admits a Lie structure (\autoref{tangent-lie}).
Moreover, the forgetful functor $\dgRep_X(\T_X[-1]) \to \Qcoh(X)$ -- from the $(\infty,1)$-category of Lie representations of $\T_X[-1]$ to the derived category of quasi-coherent modules over $X$ -- admits a section (\autoref{derived-global}). The action on a sheaf $E$ is given by its Atiyah class $\T_X[-1] \otimes E \to E$.
\end{thmintro}

This in particular applies to non-smooth varieties, schemes or classifying stacks.
To prove \autoref{intro-tgtlie}, we will introduce the category $\dStF_X$ of so called formal stacks over a stack $X$ and build an adjunction
\[
\formal_X \colon \dgLie_X \rightleftarrows \dStF_X \noloc \lie_X
\]
The Lie structure on $\T_X[-1]$ is then the image through $\lie_X$ of the formal completion of the diagonal embedding $X \to X \times X$.
Let us emphasize here that the adjunction above is usually not an equivalence. It is though, when $X$ is both affine and noetherian.

The second part of \autoref{intro-tgtlie} is proved by constructing an action of $\T_X[-1]$ on the pullback functor $\Delta^* \colon \Qcoh(X \times X) \to \Qcoh(X)$, where $\Delta$ is the diagonal of $X$. This action is given by the so-called universal Atiyah class of $X$. It then follows that $\T_X[-1]$ acts on the functor $\id_{\Qcoh(X)} \simeq \Delta^* q^*$, where $q$ is one of the projections $X \times X \to X$.

\section*{Perspectives:}
To go further down the road of higher dimensional formal loop spaces, we give here a few research paths the author hopes to follow in the future.

It is known since \cite{kapranovvasserot:loop1} that in dimension $1$, the formal loop space can be defined over a curve $C$ and has a structure of factorisation monoid over $C$. It allowed Kapranov and Vasserot to use some local-to-global principle to prove their results.
This technique is also used on formal loops with values in a classifying stack $\B G$ in a very recent paper of Gaitsgory and Lurie -- see \cite{gaitsgorylurie:weil}.
In the case of  higher dimension, the formal loop space over a variety $V$ is still to be defined. It should also have a natural factorisation structure.

In their series of articles about formal loops, Kapranov and Vasserot link the $1$-dimensional formal loop space in a scheme $X$ with sheaves of chiral differential operators on $X$.
With the definition of higher dimensional formal loop spaces we provide, finding an analogue of their result in higher dimension becomes a very natural question.
The first step would be to find a suitable definition of higher dimensional chiral and vertex algebras. The author wishes to studying this point, in collaboration with Giovanni Faonte and Mikhail Kapranov.

Another direction the author is interested in is the development of a so-called local geometry -- named after local fields. The main motivating example of an object in the alleged geometry is the algebra of Laurent series $k(\!(t)\!)$.
We will see in this thesis that one of the main problem is to define the punctured formal neighbourhood with all its structure. For instance $k(\!(t)\!)$ is naturally an algebra in ind-pro-vector spaces of finite dimension but it is not an ind-pro-algebra in vector spaces.
The idea of local geometry is precisely to replace the site of algebras (or cdga's) by a site of commutative algebras in the symmetric monoidal category of ind-pro-vector spaces.

\section*{Techniques and conventions:}
Throughout this thesis, we will use the techniques of $(\infty,1)$-category theory. We will once in a while use explicitly the model of quasi-categories developed by Joyal and Lurie (see \cite{lurie:htt}). That being said, the results should be true with any equivalent model.
Let us fix now two universes $\mathbb U \in \mathbb V$ to deal with size issues. Every algebra, module or so will implicitly be $\mathbb U$-small.
The first part of this thesis will consist of reminders about $(\infty,1)$-categories and derived algebraic geometry. We will fix then some notations. The reader will find a notations glossary at the end of this document.

We will also use derived algebraic geometry, as introduced in \cite{toen:hagii}. We refer to \cite{toen:dagems} for a recent survey of this theory.
We will denote by $k$ a base field and by $\dSt_k$ the $(\infty,1)$-category of ($\mathbb U$-small) derived stacks over $k$.

\section*{Outline:}
This thesis begins with a few paragraphs, recalling some notions we will use. Among them are (monoidal) $(\infty,1)$-categories and derived algebraic geometry.
In \autoref{chaptercats}, we develop some more $(\infty,1)$-categorical tools. We also prove \autoref{intro-tate} (see \autoref{univpropT0} and \autoref{ktheorysusp}).
In \autoref{chapterIPandTate}, we set up a theory of geometric ind-pro-stacks.
We also define symplectic Tate stacks and give a few properties, including the construction of the determinantal anomaly (see \autoref{determinantalanomaly}).
Comes \autoref{chapterloops} where we finally define higher dimensional loop spaces and prove \autoref{intro-kaploop} (see \autoref{L-indpro}). We introduce the bubble space and prove \autoref{intro-bubble} (see \autoref{B-symplectic}).
In \autoref{chapterlie}, we will prove \autoref{intro-tgtlie} (see \autoref{tangent-lie} and \autoref{derived-global}).
Finally, in the last chapter, we will describe a few perspectives the author will, hopefully, someday follow.

\section*{Aknowledgements:}
I would like to thank my advisor Bertrand Toën for this opportunity, and for teaching my most of what I know. I also thank Damien Calaque and Marco Robalo for the many discussions we had about the content of this work.
I am grateful to Mikhail Kapranov, James Wallbridge and Giovanni Faonte for inviting me at the IPMU. My stay there was very fruitful and the discussions we had were very interesting.
I finally thank Michael Gröchenig for pinpointing a mistake in the previous version, and for the hole lot of comments and pieces of advise he offered me.%
\end{chap-intro}
\cleardoublepage
\chapter*{Preliminaries}%
\pagenumbering{arabic}
\addcontentsline{toc}{chapter}{Preliminaries}%
\begin{chap-prelim}
In this part, we recall some results and definitions from $(\infty,1)$-category theory and derived algebraic geometry.

\section{A few tools from higher category theory}

In the last decades, theory of $(\infty,1)$-categories has tremendously grown.
The core idea is to consider categories enriched over spaces, so that every object or morphism is considered up to higher homotopy.
The typical example of such a category is the category of topological spaces itself: for any topological spaces $X$ and $Y$, the set of maps $X \to Y$ inherits a topology.
It is often useful to talk about topological spaces up to homotopy equivalences. Doing so, one must also consider maps up to homotopy.
To do so, one can of course formally invert every homotopy equivalence and get a set of morphisms $[X,Y]$. This process loses information and mathematicians tried to keep trace of the space of morphisms.

The first fully equipped theory handy enough to work with such examples, called model categories, was introduced by Quillen.
A model category is a category with three collections of maps -- weak equivalences (typically homotopy equivalences), fibrations and cofibrations -- satisfying a bunch of conditions.
The datum of such collections allows us to compute limits and colimits up to homotopy. We refer to \cite{hovey:modcats} for a comprehensive review of the subject.

Using model categories, several mathematicians developed theories of $(\infty,1)$-categories. Let us name here Joyal's quasi-categories, complete Segal spaces or simplicial categories.
Each one of those theories is actually a model category and they are all equivalent one to another -- see \cite{bergner:oneinfty} for a review.

In \cite{lurie:htt}, Lurie developed the theory of quasi-categories. In this book, he builds everything necessary so that we can think of $(\infty,1)$-categories as we do usual categories. To prove something in this context still requires extra care though. We will use throughout this thesis the language as developed by Lurie, but we will try to keep in mind the $1$-categorical intuition.

In this section, we will fix a few notations and recall some results to which we will often refer.

\paragraph*{Notations:}
Let us first fix a few notations, borrowed from \cite{lurie:htt}.
\begin{itemize}
\item We will denote by $\inftyCatu U$ the $(\infty,1)$-category of $\mathbb U$-small $(\infty,1)$-categories -- see \cite[3.0.0.1]{lurie:htt};\glsadd{catinfty}
\item Let $\PresLeftu U$ denote the $(\infty,1)$-category of $\mathbb U$-presentable (and thus $\mathbb V$-small) $(\infty,1)$-categories with left adjoint functors -- see \cite[5.5.3.1]{lurie:htt};\glsadd{presleft}
\item The symbol $\sSets$ will denote the $(\infty,1)$-category of $\mathbb U$-small spaces;\glsadd{ssets}
\item For any $(\infty,1)$-categories $\Cc$ and $\Dd$ we will write $\Fct(\Cc,\Dd)$ for the $(\infty,1)$-category of functors from $\Cc$ to $\Dd$ (see \cite[1.2.7.3]{lurie:htt}). The category of presheaves will be denoted $\presh(\Cc) = \Fct(\Cc\op, \sSets)$.\glsadd{fct}\glsadd{presh}
\item For any $(\infty,1)$-category $\Cc$ and any objects $c$ and $d$ in $\Cc$, we will denote by $\Map_{\Cc}(c,d)$ the space of maps from $c$ to $d$.\glsadd{map}
\end{itemize}
The following theorem is a concatenation of results from Lurie.
\begin{thm}[Lurie]\label{indu-thm}
Let $\Cc$ be a $\mathbb V$-small $(\infty,1)$-category.\glsadd{ind}
There is an $(\infty,1)$-category $\Indu U(\Cc)$ and a functor $j \colon \Cc \to \Indu U(\Cc)$ such that
\begin{enumerate}
\item The $(\infty,1)$-category $\Indu U(\Cc)$ is $\mathbb V$-small;
\item The $(\infty,1)$-category $\Indu U(\Cc)$ admits $\mathbb U$-small filtered colimits and is generated by $\mathbb U$-small filtered colimits of objects in $j(\Cc)$;
\item The functor $j$ is fully faithful and preserves finite limits and finite colimits which exist in $\Cc$;
\item For any $c \in \Cc$, its image $j(c)$ is $\mathbb U$-small compact in $\Indu U (\Cc)$;
\item For every $(\infty,1)$-category $\Dd$ with every $\mathbb U$-small filtered colimits, the functor $j$ induces an equivalence
\[
 \Fct^{\mathbb U\mathrm{-c}}(\Indu U(\Cc), \Dd) \to^\sim \Fct(\Cc,\Dd)
\]
where $\Fct^{\mathbb U \mathrm{-c}}(\Indu U(\Cc), \Dd)$ denote the full subcategory of $\Fct(\Indu U(\Cc),\Dd)$ spanned by functors preserving $\mathbb U$-small filtered colimits.
\item If $\Cc$ is $\mathbb U$-small and admits all finite colimits then $\Indu U(\Cc)$ is $\mathbb U$-presentable;
\item If $\Cc$ is endowed with a symmetric monoidal structure then there exists such a structure on $\Indu U(\Cc)$ such that the monoidal product preserves $\mathbb U$-small filtered colimits in each variable.\label{indmonoidal}
\end{enumerate}
\end{thm}

\begin{proof}
Let us use the notations of \cite[5.3.6.2]{lurie:htt}. Let $\mathcal K$ denote the collection of $\mathbb U$-small filtered simplicial sets. We then set $\Indu U (\Cc) = \presh^\mathcal K _\emptyset(\Cc)$.
It satisfies the required properties because of \loccit 5.3.6.2 and 5.5.1.1. We also need tiny modifications of the proofs of \loccit 5.3.5.14 and 5.3.5.5. The last item is proved in \cite[6.3.1.10]{lurie:halg}.
\end{proof}

Lurie proved in \cite[5.3.5.15]{lurie:htt} that any map $c \to d \in \Indu U (\Cc)$ is a colimit of a $\mathbb U$-small filtered diagram $K \to \Fct(\Delta^1, \Cc)$. We will need afterwards the following small refinement of this statement, inspired by \cite[3.9]{bgw:tate}
\begin{prop}[Strictifying maps]\label{strictification}
Let $\Cc$ be a $\mathbb V$-small $(\infty,1)$-category and let $f \colon c \to d$ be a morphism in $\Indu U(\Cc)$. Let $\bar c \colon K \to \Cc$ and $\bar d \colon L \to \Cc$ be $\mathbb U$-small filtered diagrams of whom respectively $c$ and $d$ are colimits in $\Indu U(\Cc)$.
There exists a $\mathbb U$-small filtered diagram $\bar f \colon J \to \Fct(\Delta^1, \Cc)$ and a commutative diagram
\[
\mymatrix{
K \ar[d]^{\bar c} & J \ar[r]^{p_L} \ar[l]_{p_K} \ar[d]^{\bar f} & L \ar[d]^{\bar d} \\
\Cc & \Fct(\Delta^1, \Cc) \ar[r]^-{\ev_1} \ar[l]_-{\ev_0} & \Cc
}
\]
such that both maps $p_K$ and $p_L$ are cofinal, and such that $f$ is the colimit of $\bar f$.
\end{prop}

\begin{proof}
Using \cite[4.3.2.14]{lurie:htt} we can assume that both $K$ and $L$ are filtered partially ordered sets.
Let us denote by $J'$ the fibre product
\[
\mymatrix{
J' \cart \ar[d] \ar[rr] && \quot{\Fct(\Delta^1, \Indu U(\Cc))}{f} \ar[d] \\
\displaystyle K \times_\Cc \Fct(\Delta^1, \Cc) \times_\Cc L \ar[r] & \Fct(\Delta^1, \Cc) \ar[r] & \Fct(\Delta^1, \Indu U(\Cc))
}
\]
Let us first prove that $J'$ is filtered.
Let $P$ be a partially ordered finite set and $P^\triangleright$ denote the partially ordered set $P \cup \{\infty\}$, where $\infty$ is a maximal element.
A morphism $P \to J'$ is the datum of a commutative diagram
\[
\mymatrix@!0@R=6mm@C=10mm{
&& P \times \{0 \} \ar[rrr]^\kappa \ar[ddr] &&& K \ar[rdd]^{\bar c} & \\
&P \times \{1\} \ar[rrr]|!{[ur];[drr]}\hole^(0.7)\lambda \ar[rrd] &&& L \ar[drr]^(0.4){\bar d} && \\
&&& P \times \Delta^1 \ar[rrr]_\psi \ar[dd] &&& \Cc \ar[dd] \\
\\
\{\infty\} \times \Delta^1 \ar[rrr] &&& P^\triangleright \times \Delta^1 \ar[rrr] &&& \Indu U(\Cc)
}
\]
Let us denote by $P_+$ the partially ordered set $P \cup \{ + \}$ where $+$ is a maximal element. Because $K$ is filtered, the map $\kappa$ extends to a morphism $\kappa' \colon P_+ \times \{0\} \to K$.
There exists $l \in L$ such that the induced map $\bar c(\kappa'(+)) \to c \to d$ factors through $\bar d(l) \to d$. Since $L$ is filtered, there is a map $\lambda' \colon P_+ \times \{1\} \to L$ extending $\lambda$.
We can moreover chose $\lambda'(+)$ greater than $l$ (ie with a map $l \to \lambda'(+)$ in $L$).
Using the map $\bar c(\kappa'(+)) \to \bar d(l) \to \bar d(\lambda'(+))$, we get a morphism $\psi' \colon P_+ \times \Delta^1 \to \Cc$ extending $\psi$, which by construction extends to $P_+^\triangleright \times \Delta^1$ -- where we set $\infty \geq +$.
This defines a morphism $P_+ \to J'$, proving that $J'$ is filtered.
Using \cite[4.3.2.14]{lurie:htt} we define $J$ to be a filtered partially ordered set with a cofinal map $J \to J'$.
Proving that the maps $J \to K$ and $J \to L$ are cofinal is now straightforward.
This also implies that the induced diagram $\bar f \colon J \to \Fct(\Delta^1, \Cc)$ has colimit $f$ in $\Indu U(\Cc)$.
\end{proof}

\begin{rmq}
Note that when $\Cc$ admits finite colimits then the category $\Indu U(\Cc)$ embeds in the $\mathbb V$-presentable category $\Indu V(\Cc)$.
\end{rmq}

\begin{df}
Let $\Cc$ be a $\mathbb V$-small $\infty$-category. We define $\Prou U(\Cc)$ as the $(\infty,1)$-category\glsadd{pro}
\[
\Prou U(\Cc) = \left( \Indu U(\Cc\op) \right) \op
\]
It satisfies properties dual to those of $\Indu U(\Cc)$.
\end{df}

\begin{df}[{see \cite[1.2.8.4]{lurie:htt}}]\glsadd{cone}
Let $K$ be a simplicial set. We will denote by $K^{\triangleright}$ the simplicial set obtained from $K$ by formally adding a final object. This final object will be called the cone point of $K^\triangleright$.
\end{df}

The following lemma is a direct consequence of results from Lurie's \cite{lurie:htt}.
\begin{lem}[Stabilisation]\label{tatification}
Let $\Cc$ be a $\mathbb V$-small pointed category with all suspensions. Let us assume that the suspension functor $\Cc \to \Cc$ is an equivalence.
There exists an $(\infty,1)$-category $\Cc^\mathrm{st}$ with a map $j \colon \Cc \to \Cc^\mathrm{st}$ such that
\begin{enumerate}
\item The category $\Cc^\mathrm{st}$ is $\mathbb V$-small and stable.
\item The functor $j$ is fully faithful and preserves all limits and finite colimits which exist in $\Cc$.
\item For any stable $(\infty,1)$-category $\Dd$ the induced map
\[
\Fct^\mathrm{ex}(\Cc^\mathrm{st},\Dd) \to \Fct^\mathrm{lex}(\Cc,\Dd)
\]
between exact functors and left exact functors is an equivalence.
\item For any stable category $\Dd$ with a fully faithful functor $\Cc \to \Dd$ preserving finite colimits and limits which exist in $\Cc$, the smallest stable subcategory of $\Dd$ containing the image of $\Cc$ is equivalent to $\Cc^\mathrm{st}$.
\end{enumerate}
\end{lem}

\begin{proof}
Let us denote by $K$ the simplicial set corresponding to a diagram $\bullet \from \bullet \to \bullet$. Let $\mathcal R$ denote the collection of all cocartesian diagrams $K^\triangleright \to \Cc$ and the zero $\emptyset^\triangleright = \bullet \to \Cc$ in $\Cc$.
We then set $\Cc^\mathrm{st} = \presh^{\{K,\emptyset\}}_{\mathcal R}(\Cc)$ using the notation of \cite[5.3.6.2]{lurie:htt}. Note that \emph{(iii)} is proven in \emph{loc. cit.}.
The category $\Cc^\mathrm{st}$ is pointed and it comes with two natural fully faithful maps
\[
\mymatrix{
\Cc \ar[r]^-j & \Cc^\mathrm{st} \ar[r] & \presh(\Cc)
}
\]
whose composite is the Yoneda functor and therefore preserves limits which exist in $\Cc$. It follows that $j$ also preserves those limits.
By definition, the functor $j$ preserves finite colimits which exist in $\Cc$.

Any object of $\Cc^\mathrm{st}$ is a finite colimit of objects in $\Cc$. Its suspension is therefore the colimit of the suspensions of those objects.
The suspension functor $\Cc^\mathrm{st} \to \Cc^\mathrm{st}$ is thus an equivalence.
Corollary \cite[1.4.2.27]{lurie:halg} implies that $\Cc^\mathrm{st}$ is stable.

We now focus on the assertion \emph{(iv)}. Let $f \colon \Cc \to \Dd$ be as required. Because of the third point, there is an essentially unique functor $g \colon \Cc^\mathrm{st} \to \Dd$ lifting $f$. Every object in $\Cc^\mathrm{st}$ can be written as both a colimit and a limit of objects of $\Cc$. It follows that $g$ is fully faithful and then that $\Cc^\mathrm{st}$ contains the smallest full and stable subcategory $\Dd'$ of $\Dd$ extending $\Cc$. There is also a universal map $\Cc^\mathrm{st} \to \Dd'$ which is easily seen to be an inverse to the inclusion.
\end{proof}

\begin{lem}
Let $\Cc$ be an idempotent complete $\mathbb V$-small $(\infty,1)$-category. We consider the natural embeddings $i \colon \Prou U (\Cc) \to \Prou U \Indu U(\Cc)$ and $j \colon \Indu U(\Cc) \to \Prou U \Indu U(\Cc)$. We will also denote by $k$ the embedding $\Cc \to \Prou U \Indu U(\Cc)$.
If an object of $\Prou U \Indu U(\Cc)$ is in both the essential images of $i$ and $j$, then it is in the essential image of $k$.
\end{lem}

\begin{proof}
Let $x \in \Indu U(\Cc)$.
Let us assume there exists a pro-object $y \in \Prou U(\Cc)$ and an equivalence $f \colon x \to y$. Let $\bar y \colon K \op \to \Cc$ be a cofiltered diagram of whom $y$ is a limit in $\Prou U(\Cc)$. The equivalence $f$ induces a morphism from the constant diagram $x \colon K\op \to \Indu U(\Cc)$ to $\bar y \colon K \op \to \Indu U(\Cc)$.
An inverse $g \colon y \to x$ of $f$ then induces a map $y_k = \bar y(k) \to x$ for some $k \in K$ such that the composite morphism $x \to y_k \to x$ is homotopic to the identity. Idempotent completeness and \cite[5.4.2.4]{lurie:htt} finish the proof.
\end{proof}

\begin{df} \label{indext}
Let $\Cc$ be a $\mathbb V$-small $(\infty,1)$-category.
Let 
\[
i \colon \Fct(\Cc, \inftyCatu V) \to \Fct(\Indu U(\Cc), \inftyCatu V)
\]
denote the left Kan extension functor.
We will denote by $\Indext_\Cc$ the composite functor\glsadd{indext}
\[
\mymatrix{
\Fct(\Cc, \inftyCatu V) \ar[r]^-i & \Fct(\Indu U(\Cc), \inftyCatu V) \ar[rr]^-{\Indu U \circ -} && \Fct(\Indu U(\Cc), \inftyCatu V)
}
\]
We will denote by $\Proext_\Cc$ the composite functor\glsadd{proext}
\[
\mymatrix{
\Fct(\Cc, \inftyCatu V) \ar[rr]^{\Prou U \circ -} && \Fct(\Cc, \inftyCatu V) \ar[r] & \Fct(\Prou U(\Cc), \inftyCatu V)
}
\]
We define the same way
\begin{align*}
&\Indextu V_C \colon \Fct(\Cc, \inftyCatu V) \to \Fct(\Indu V(\Cc), \inftyCat) \\
&\Proextu V_C \colon \Fct(\Cc, \inftyCatu V) \to \Fct(\Prou V(\Cc), \inftyCat) \\
\end{align*}
\end{df}

\begin{rmq}
The \autoref{indext} can be expanded as follows. To any functor $f \colon \Cc \to \inftyCatu V$ and any ind-object $c$ colimit of a diagram
\[
\mymatrix{
K \ar[r]^-{\bar c} & \Cc \ar[r] & \Indu U(\Cc)
}
\]
we construct an $(\infty,1)$-category
\[
\Indext_\Cc(f)(c) \simeq \Indu U\left(\colim f(\bar c)\right)
\]
To any pro-object $d$ limit of a diagram
\[
\mymatrix{
K\op \ar[r]^-{\bar d} & \Cc \ar[r] & \Prou U(\Cc)
}
\]
we associate an $(\infty,1)$-category
\[
\Proext_\Cc(f)(d) \simeq \lim \Prou U(f(\bar d))
\]
\end{rmq}

\begin{df}
Let $\inftyCat^{\mathbb V\mathrm{,st}}$ denote the subcategory of $\inftyCatu V$ spanned by stable categories with exact functors between them -- see \cite[1.1.4]{lurie:halg}.\glsadd{catstable}
Let $\inftyCat^{\mathbb V\mathrm{,st,id}}$ denote the full subcategory of $\inftyCat^{\mathbb V\mathrm{,st}}$ spanned by idempotent complete stable categories.\glsadd{catsstableid}
\end{df}

\begin{rmq}
It follows from \cite[1.1.4.6, 1.1.3.6, 1.1.1.13 and 1.1.4.4]{lurie:halg} that the functors $\Indext_\Cc$ and $\Proext_\Cc$ restricts to functors
\begin{align*}
& \Indext_\Cc \colon \Fct(\Cc, \inftyCat^{\mathbb V\mathrm{,st}}) \to \Fct(\Indu U (\Cc), \inftyCat^{\mathbb V\mathrm{,st}}) \\
& \Proext_\Cc \colon \Fct(\Cc, \inftyCat^{\mathbb V\mathrm{,st}}) \to \Fct(\Prou U (\Cc), \inftyCat^{\mathbb V\mathrm{,st}})
\end{align*}
\end{rmq}

\begin{df}\label{sifted}\glsadd{sifted}
Given an $(\infty,1)$-category $\Cc$ with all finite coproducts, we will denote by $\sifted(\Cc)$ its completion under sifted colimits. Recall (see \cite{lurie:htt}, part 5.5.8) that this is the $(\infty,1)$-category $\Fct^{\times}(\Cc\op,\sSets)$ of $(\infty,1)$-functors preserving finite products.
\end{df}

\paragraph{Symmetric monoidal $(\infty,1)$-categories:}
We will make use in the last chapter of the theory of symmetric monoidal $(\infty,1)$-categories as developed in \cite{lurie:halg}. Let us give a (very) quick review of those objects.

\begin{df}\label{ptfin}\glsadd{ptfin}
Let $\ptfin$ denote the category of pointed finite sets. For any $n \in \N$, we will denote by $\langle n \rangle$ the set $\{\pt, 1, \dots ,n\}$ pointed at $\pt$.
For any $n$ and $i \leq n$, the Segal map $\delta^i \colon \langle n \rangle \to \langle 1 \rangle$ is defined by $\delta^i(j) = 1$ if $j = i$ and $\delta^i(j) = \pt$ otherwise.
\end{df}
\begin{df}(see \cite[2.0.0.7]{lurie:halg})\label{monoidalcats}
Let $\Cc$ be an $(\infty,1)$-category. A symmetric monoidal structure on $\Cc$ is the datum of a coCartesian fibration $p \colon \Cc^{\otimes} \to \ptfin$ such that
\begin{itemize}
\item The fibre category $\Cc^{\otimes}_{\langle 1 \rangle}$ is equivalent to $\Cc$ and
\item For any $n$, the Segal maps induce an equivalence $\Cc^{\otimes}_{\langle n \rangle} \to (\Cc^{\otimes}_{\langle 1 \rangle})^n \simeq \Cc^n$.
\end{itemize}
where $\Cc^{\otimes}_{\langle n \rangle}$ denote the fibre of $p$ at $\langle n \rangle$.
We will denote by $\monoidalinftyCatu V$ the $(\infty,1)$-category of $\mathbb V$-small symmetric monoidal $(\infty,1)$-categories -- see \cite[2.1.4.13]{lurie:halg}.\glsadd{monoidalcats}
\end{df}
Such a coCartesian fibration is classified by a functor $\phi \colon \ptfin \to \inftyCatu V$ -- see \cite[3.3.2.2]{lurie:htt} -- such that $\phi(\langle n \rangle) \simeq \Cc^n$.
The tensor product on $\Cc$ is induced by the map of pointed finite sets $\mu \colon \langle 2 \rangle \to \langle 1 \rangle$ mapping both $1$ and $2$ to $1$
\[
\otimes = \phi(\mu) \colon \Cc^2 \to \Cc
\]

\begin{rmq}\label{indext-monoidal}
The forgetful functor $\monoidalinftyCatu V \to \inftyCatu V$ preserves all limits as well as filtered colimits -- see \cite[3.2.2.4 and 3.2.3.2]{lurie:halg}.
Moreover, it follows from \autoref{indu-thm} - \ref{indmonoidal} that the functor $\Indu U$ induces a functor
\[
\Indu U \colon \monoidalinftyCatu V \to \monoidalinftyCatu V
\]
The same holds for $\Prou U$.
The constructions $\Indext$ and $\Proext$ therefore restrict to
\begin{align*}
&\Indext_\Cc \colon \Fct(\Cc,\monoidalinftyCatu V) \to \Fct(\Indu U(\Cc), \monoidalinftyCatu V) \\
&\Proext_\Cc \colon \Fct(\Cc,\monoidalinftyCatu V) \to \Fct(\Prou U(\Cc), \monoidalinftyCatu V)
\end{align*}
\end{rmq}

\section{Derived algebraic geometry}

We present here some background results about derived algebraic geometry.
Let us assume $k$ is a field of characteristic $0$.
First introduced by Toën and Vezzosi in \cite{toen:hagii}, derived algebraic geometry is a generalisation of algebraic geometry in which we replace commutative algebras over $k$ by commutative differential graded algebras (or cdga's).
We refer to \cite{toen:dagems} for a recent survey of this theory.

\paragraph{Generalities on derived stacks :}We will denote by $\cdga_k$\glsadd{cdga} the $(\infty,1)$-category of cdga's over $k$ concentrated in non-positive cohomological degree. It is the $(\infty,1)$-localisation of a model category along weak equivalences. 
Let us denote $\dAff_k$\glsadd{daff} the opposite $(\infty,1)$-category of $\cdga_k$. It is the category of derived affine schemes over $k$.
In this thesis, we will adopt a cohomological convention for cdga's.

A derived prestack is a presheaf $\dAff_k\op \simeq \cdga_k \to \sSets$. We will thus write $\presh(\dAff_k)$ for the $(\infty,1)$-category of derived prestacks.
A derived stack is a prestack with a descent condition. We will denote by $\dSt_k$\glsadd{dst} the $(\infty,1)$-category of derived stacks.
It comes with an adjunction
\[
(-)^+ \colon \presh(\dAff_k) \rightleftarrows \dSt_k
\]
where the left adjoint $(-)^+$ is called the stackification functor.
\begin{rmq}
The categories of varieties, schemes or (non derived) stacks embed into $\dSt_k$.
\end{rmq}

\begin{df}
The $(\infty,1)$-category of derived stacks admits an internal hom $\Mapstack(X,Y)$ between two stacks $X$ and $Y$. It is the functor $\cdga_k \to \sSets$ defined by
\[
A \mapsto \Map_{\dSt_k}(X \times \Spec A, Y)
\]
We will call it the mapping stack from $X$ to $Y$.\glsadd{mapstack}
\end{df}

There is a derived version of Artin stacks of which we first give a recursive definition.
\begin{df}(see for instance \cite[5.2.2]{toenmoerdijk:crm})\label{artinstacks}
Let $X$ be a derived stack.
\begin{itemize}
\item We say that $X$ is a derived $0$-Artin stack if it is a derived affine scheme ;
\item We say that $X$ is a derived $n$-Artin stack if there is a family $(T_\alpha)$ of derived affine schemes and a smooth atlas
\[
u \colon \coprod T_\alpha \to X
\]
such that the nerve of $u$ has values in derived $(n-1)$-Artin stacks ;
\item We say that $X$ is a derived Artin stack if it is an $n$-Artin stack for some $n$.
\end{itemize}
We will denote by $\dStArt_k$\glsadd{dstart} the full subcategory of $\dSt_k$ spanned by derived Artin stacks.
\end{df}

To any cdga $A$ we associate the category $\dgMod_A$ of dg-modules over $A$.\glsadd{dgmod}
Similarly, to any derived stack $X$ we can associate a derived category $\Qcoh(X)$ of quasicoherent sheaves. It is a $\mathbb U$-presentable $(\infty,1)$-category given by the formula
\[
\Qcoh(X) \simeq \lim_{\Spec A \to X} \dgMod_A
\]
Moreover, for any map $f \colon X \to Y$, there is a natural pull back functor $f^* \colon \Qcoh(Y) \to \Qcoh(X)$. This functor admits a right adjoint, which we will denote by $f_*$.
This construction is actually a functor of $(\infty,1)$-categories.
\begin{df}
Let us denote by $\Qcoh$ the functor\glsadd{qcoh}
\[
\Qcoh \colon \dSt_k\op \to \PresLeftu U
\]
For any $X$ we can identify a full subcategory $\Perf(X) \subset \Qcoh(X)$\glsadd{perf} of perfect complexes. This defines a functor
\[
\Perf \colon \dSt_k\op \to \inftyCatu U
\]
\end{df}

\begin{rmq}
For any derived stack $X$ the categories $\Qcoh(X)$ and $\Perf(X)$ are actually stable and idempotent complete $(\infty,1)$-categories. The inclusion $\Perf(X) \to \Qcoh(X)$ is exact. Moreover, for any map $f \colon X \to Y$ the pull back functor $f^*$ preserves perfect modules and is also exact. 
\end{rmq}

Any derived Artin stack $X$ over a basis $S$ admits a cotangent complex $\Lcot_{X/S} \in \Qcoh(X)$. If $X$ if locally of finite presentation, then the its cotangent complex is perfect\glsadd{cotangent}
\[
\Lcot_{X/S} \in \Perf(X)
\]

\paragraph{Obstruction theory :}
Let $A \in \cdga_k$ and let $M \in \dgMod_A^{\leq -1}$ be an $A$-module concentrated in negative cohomological degrees. Let $d$ be a derivation $A \to A \oplus M$ and $s \colon A \to A \oplus M$ be the trivial derivation.
The square zero extension of $A$ by $M[-1]$ twisted by $d$ is the fibre product
\[
\mymatrix{
A \oplus_d M[-1] \cart \ar[r]^-p \ar[d] & A \ar[d]^d \\ A \ar[r]^-s & A \oplus M
}
\]
Let now $X$ be a derived stack and $M \in \Qcoh(X)^{\leq -1}$. We will denote by $X[M]$ the trivial square zero extension of $X$ by $M$.\glsadd{trivext}
Let also $d \colon X[M] \to X$ be a derivation -- ie a retract of the natural map $X \to X[M]$. We define the square zero extension of $X$ by $M[-1]$ twisted by $d$ as the colimit
\[
X_d[M[-1]] = \colim_{f \colon \Spec A \to X} \Spec(A \oplus_{f^*d} f^* M[-1])
\]
It is endowed with a natural morphism $X \to X_d[M[-1]]$ induced by the projections $p$ as above.
\begin{prop}[Obstruction theory on stacks]\label{obstruction}
Let $F \to G$ be an algebraic morphism of derived stacks. Let $X$ be a derived stack and let $M \in \Qcoh(X)^{\leq -1}$. Let $d$ be a derivation
\[
d \in \Map_{X/-}(X[M], X)
\]
We consider the map of simplicial sets
\[
\psi \colon \Map(X_d[M[-1]],F) \to \Map(X,F) \times_{\Map(X,G)} \Map(X_d[M[-1]],G)
\]
Let $y \in \Map(X,F) \times_{\Map(X,G)} \Map(X_d[M[-1]],G)$ and let $x \in \Map(X,F)$ be the induced map.
There exists a point $\alpha(y) \in \Map(x^* \Lcot_{F/G}, M)$ such that the fibre $\psi_y$ of $\psi$ at $y$ is equivalent to the space of paths from $0$ to $\alpha(y)$ in $\Map(x^* \Lcot_{F/G}, M)$
\[
\psi_y \simeq \Omega_{0,\alpha(y)}\Map(x^* \Lcot_{F/G}, M)
\]
\end{prop}
\begin{proof}
This is a simple generalisation of \cite[1.4.2.6]{toen:hagii}. The proof is very similar.
We have a natural commutative square
\[
\mymatrix{
X[M] \ar[r]^-d \ar[d] & X \ar[d] \\ X \ar[r] & X_d[M[-1]]
}
\]
It induces a map
\[
\alpha \colon \Map(X,F) \times_{\Map(X,G)} \Map(X_d[M[-1]],G) \to \Map_{X/-/G}(X[M],F) \simeq \Map(x^* \Lcot_{F/G},M)
\]
Let $\Omega_{0,\alpha(y)}\Map_{X/-/G}(X[M],F)$ denote the space of paths from $0$ to $\alpha(y)$. It is the fibre product
\[
\mymatrix{
\Omega_{0,\alpha(y)}\Map_{X/-/G}(X[M],F) \cart[][20] \ar[r] \ar[d] & \pt \ar[d]^{\alpha(y)} \\
\pt \ar[r]^-0 & \Map_{X/-/G}(X[M],F)
}
\]
The composite map $\alpha \psi$ is by definition homotopic to the $0$ map. This defines a morphism
\[
f \colon \Omega_{0,\alpha(y)}\Map_{X/-/G}(X[M],F) \to \psi_y
\]
It now suffices to see that the category of $X$'s for which $f$ is an equivalence contains every derived affine scheme and is stable by colimits.
The first assertion is exactly \cite[1.4.2.6]{toen:hagii} and the second one is trivial.
\end{proof}

\paragraph{Algebraisable stacks :}
Let $X$ be a derived stack and $A$ be a cdga. Let $a = (\el{a}{p})$ be a sequence of elements of $A^0$ forming a regular sequence in $\homol^0(A)$.
Let $\quot{A}{\el{a^n}{p}}$ denote the Kozsul complex associated with the regular sequence $(\el{a^n}{p})$. It is endowed with a cdga structure.
There is a canonical map 
\[
\psi(A)_a \colon \colim_n X\left( \textstyle \quot{A}{\el{a^n}{p}} \right) \to X\left(\lim_n \textstyle \quot{A}{\el{a^n}{p}} \right)
\]
This map is usually not an equivalence.
\begin{df}\label{alg-diag}
A derived stack $X$ is called algebraisable if for any $A$ and any regular sequence $a$ the map $\psi(A)_a$ is an equivalence.

A map $f \colon X \to Y$ is called algebraisable if for any derived affine scheme $T$ and any map $T \to Y$, the fibre product $X \times_Y T$ is algebraisable.

We will say that a derived stack $X$ has algebraisable diagonal if the diagonal morphism $X \to X \times X$ is algebraisable.
\end{df}

\begin{rmq}
A derived stack $X$ has algebraisable diagonal if for any $A$ and $a$ the map $\psi(A)_a$ is fully faithful.
One could also rephrase the definition of being algebraisable as follows. A stack is algebraisable if it does not detect the difference between
\[
\colim_n \Spec\left( \textstyle \quot{A}{\el{a^n}{p}} \right) \text{~~~and~~~} \Spec\left( \lim_n \textstyle \quot{A}{\el{a^n}{p}} \right)
\]
\end{rmq}

\begin{ex}
Any derived affine scheme is algebraisable. Another important example of algebraisable stack is the stack of perfect complexes. 
In \cite{bhatt:algebraisable}, Bhargav Bhatt gives some more examples of algebraisable (non-derived) stacks -- although our definition slightly differs from his.
He proves that any quasi-compact quasi-separated algebraic space is algebraisable and also provides with examples of non-algebraisable stacks. Let us name $\mathrm K(\Gm,2)$ -- the Eilenberg-Maclane classifying stack of $\Gm$ -- as an example of non-algebraisable stack.
Algebraisability of Deligne-Mumford stacks is also look at in \cite{lurie:dagxii}.
\end{ex}
\end{chap-prelim}

\chapter{Categorical results}\label{chaptercats}%
\begin{chap-cats}
\section{Tate objects}
In this subsection we define the category of Tate objects in a stable $(\infty,1)$-category. We prove \autoref{intro-tate} claimed in the introduction.
\begin{df}
Let $\Cc$ be a $\mathbb V$-small stable $(\infty,1)$-category. We define the category $\Tateu U_0(\Cc)$ of pure Tate objects in $\Cc$ as the full sub-category of $\Prou U\Indu U(\Cc)$ spanned by the images of $\Indu U(\Cc)$ and $\Prou U(\Cc)$ through the canonical functors.
\glsadd{tate0}%
The category $\Tateu U_0(\Cc)$ obviously satisfies the conditions of \autoref{tatification} and we define the category $\Tateu U_\mathrm{el}(\Cc)$ of elementary Tate objects in $\Cc$ as the completion
\[
\Tateu U_\mathrm{el}(\Cc) = \left( \Tateu U_0(\Cc) \right) ^{\mathrm{st}}
\]
\glsadd{tateel}%
We also define the category $\Tateu U(\Cc)$ of Tate objects in $\Cc$ as the idempotent completion of $\Tateu U_\mathrm{el}(\Cc)$.
\glsadd{tate}%
We have fully faithful exact functors between stable $(\infty,1)$-categories
\[
\Tateu U_\mathrm{el}(\Cc) \to \Tateu U(\Cc) \to \Prou U\Indu U(\Cc)
\]
\end{df}

\begin{rmq}
Note that $\Tateu U (\Cc)$ is $\mathbb V$-small.
The construction $\Tateu U(-)$ defines a functor 
\[
\inftyCat^{\mathbb V\mathrm{,st}} \to \inftyCat^{\mathbb V\mathrm{,st,id}}
\]
It comes with a fully faithful --- ie pointwise fully faithful --- natural transformation
\[
\mymatrix{
\inftyCat^{\mathbb V\mathrm{,st}} \ar[rr]^-{\Tateu U} \ar[d] && \inftyCat^{\mathbb V\mathrm{,st,id}} \ar[d] \ar@{=>}[dll] \\ \inftyCatu V \ar[rr]_{\Prou U \Indu U} && \inftyCatu V
}
\]
\end{rmq}

\begin{rmq}
We can immediately see that the functor $\Tateu U$ map any fully faithful and exact functor $\Cc \to \Dd$ between stable categories to a fully faithful (and exact) functor $\Tateu U(\Cc) \to \Tateu U(\Dd)$.
\end{rmq}

Let us now give a universal property for the category of pure Tate objects. The next theorem states that for any $(\infty,1)$-category $\Dd$ and any commutative diagram
\[
\mymatrix{
\Cc \ar[r] \ar[d] & \Indu U(\Cc) \ar[d]^-f \\ \Prou U(\Cc) \ar[r]_-g & \Dd
}
\]
such that $f$ preserves $\mathbb U$-small filtered colimits and $g$ preserves $\mathbb U$-small cofiltered limits there exists an essentially unique functor $\Tateu U_0(\Cc) \to \Dd$ such that $f$ and $g$ are respectively equivalent to the composite functors
\begin{align*}
&\Indu U(\Cc) \to \Tateu U_0(\Cc) \to \Dd \\
&\Prou U(\Cc) \to \Tateu U_0(\Cc) \to \Dd
\end{align*}
This universal property was discovered during a discussion with Michael Gröchenig, whom the author thanks greatly.
To state formally this property, let us fix some notations. Let $i \colon \Indu U(\Cc) \to \Tateu U_0(\Cc)$ and $p \colon \Prou U(\Cc) \to \Tateu U_0(\Cc)$ denote the canonical inclusions.
Let also $\Fct_t(\Tateu U_0(\Cc),\Dd)$ denote the full subcategory of $\Fct(\Tateu U_0(\Cc),\Dd)$ spanned by those functors $\xi$ such that
\begin{itemize}
\item The composite functor $\xi i$ maps filtered colimit diagrams to colimit diagrams.
\item The composite functor $\xi p$ maps cofiltered limit diagrams to limit diagrams.
\end{itemize}
Let also $\Fct_m(\Cc,\Dd)$ denote the category of functors $g \colon \Cc \to \Dd$ such that
\begin{itemize}
\item For any filtered diagram $K \to \Cc$, the composite diagram $K \to \Cc \to \Dd$ admits a colimit in $\Dd$.
\item For any cofiltered diagram $K\op \to \Cc$, the composite diagram $K\op \to \Cc \to \Dd$ admits a limit in $\Dd$.
\end{itemize}
\begin{thm}\label{univpropT0}
Let $\Cc$ be a $\mathbb V$-small stable $(\infty,1)$-category. For any $(\infty,1)$-category $\Dd$, the restriction functor induces an equivalence
\[
\mymatrix{
\Fct_t(\Tateu U_0(\Cc),\Dd) \ar[r] & \Fct_m(\Cc,\Dd)
}
\]
\end{thm}
\begin{proof}
Let us shorten the notations:
\[
\mathrm I \Cc = \Indu U(\Cc) \hspace{1cm} \mathrm P \Cc = \Prou U(\Cc) \hspace{1cm} \mathrm T = \Tateu U_0(\Cc) \hspace{1cm} \mathrm{PI} \Cc = \Prou U(\Indu U(\Cc))
\]
Recall that $\presh(\Dd)$ denotes the $(\infty,1)$-category of simplicial presheaves on $\Dd$.
The restriction functor $\Fct(\mathrm{PI} \Cc,\presh(\Dd)) \to \Fct(\mathrm T\Cc,\presh(\Dd))$ admits a left adjoint given by the left Kan extension. The restriction functor $\Fct(\mathrm{PI} \Cc,\presh(\Dd)) \to \Fct(\mathrm I\Cc,\presh(\Dd))$ admits a right adjoint, given by the right Kan extension. Let us fix their notation
\[
\mymatrix{
\Fct(\mathrm{T} \Cc,\presh(\Dd)) \ar@<2pt>[r]^\delta & \Fct(\mathrm{PI} \Cc,\presh(\Dd)) \ar@<2pt>[r]^\beta \ar@<2pt>[l]^\gamma & \Fct(\mathrm{I} \Cc,\presh(\Dd)) \ar@<2pt>[l]^\alpha
}
\]
the left adjoints being represented above their right adjoint. Note that both $\alpha$ and $\delta$ are fully faithful.
Let also $\tau$ denote the fully faithful functor
\[
\mymatrix{
\displaystyle \Fct_m(\Cc,\Dd) \simeq \Fct^{\mathrm c}(\mathrm I \Cc, \Dd) \timesunder[\Fct(\Cc,\Dd)] \Fct^{\mathrm l}(\mathrm P \Cc,\Dd) \ar[r]^-\tau &
\displaystyle \Fct(\mathrm I \Cc, \presh(\Dd)) \timesunder[\Fct(\Cc,\presh(\Dd))] \Fct^{\mathrm l}(\mathrm P \Cc,\presh(\Dd)) \simeq \Fct(\mathrm I \Cc, \presh(\Dd))
}
\]
where $\Fct^\mathrm c$ (resp. $\Fct^\mathrm l$) denotes the category of functors preserving filtered colimits (resp. cofiltered limits) which exist in the source. We use here that the Yoneda embedding $\Dd \to \presh(\Dd)$ preserves limits.
Let $\theta$ be the fully faithful functor
\[
\mymatrix{
\Fct_t(\mathrm T\Cc,\Dd) \ar[r]^-\theta &
\Fct(\mathrm T\Cc,\presh(\Dd))
}
\]
The composite functor $\beta \delta$ is nothing but the restriction along the canonical inclusion $\mathrm I \Cc \to \mathrm T \Cc$.
It follows that $\beta \delta \theta$ has image in the essential image of $\tau$. On the other hand, the functor $\gamma \alpha \tau$ has image in the essential image of $\theta$.
We hence get an adjunction
\[
f \colon \Fct_t(\mathrm T\Cc,\Dd) \rightleftarrows \Fct_m(\Cc,\Dd) \noloc g
\]
where $f$ is left adjoint to $g$. The functor $g$ is equivalent to the restriction functor and the unit transformation $fg \to \id_X$ is then an equivalence. Moreover, as objects of $\mathrm T\Cc$ are either pro-objects or ind-objects, the restriction functor $f$ is conservative. It follows that the above adjunction is an equivalence.
\end{proof}

\begin{cor}
The category of Tate objects is equivalent to the smallest stable and idempotent complete full subcategory of $\Indu U \Prou U(\Cc)$ generated by the images of $\Indu U(\Cc)$ and $\Prou U(\Cc)$.
\end{cor}
\begin{proof}
This follows from \autoref{univpropT0} and \autoref{tatification}.
\end{proof}

\begin{rmq}
The fully faithful functor $\j \colon \Tateu U_0(\Cc) \to \Prou U\Indu U(\Cc)$ preserves both the limits and colimits which exist in $\Tateu U_0(\Cc)$. Let indeed $\bar x \colon K \to \Tateu U_0(\Cc)$ be a diagram which admits a colimit $x \in \Tateu U_0(\Cc)$. Let us denote by $x'$ a colimit of $\j \bar x$ in $\Prou U\Indu U(\Cc)$. We have, for any cofiltered diagram $y \bar \colon L\op \to \Indu U(\Cc)$
\begin{align*}
\Map_{\Prou U\Indu U(\Cc)}(x', \lim \bar y) &\simeq \colim_l \colim_k \Map_{\Prou U\Indu U(\Cc)}(\j \bar x, \bar y) \simeq \colim_l \colim_k \Map_{\Tateu U_0(\Cc)}(\bar x, \bar y)
\\ &\simeq \colim_l \Map_{\Tateu U_0(\Cc)}(x,\bar y) \simeq \colim_l \Map_{\Prou U\Indu U(\Cc)}(x,\bar y)
\\ &\simeq \colim_l \Map_{\Prou U\Indu U(\Cc)}(x,\bar y) \simeq \Map_{\Prou U\Indu U(\Cc)}(x,\lim \bar y)
\end{align*}
We show symmetrically that the inclusion $\Tateu U_0(\Cc) \to \Indu U \Prou U(\Cc)$ preserves limits.
It follows that limits and colimits that exist in $\Tateu U_0(\Cc)$ are exactly those coming from diagram in either $\Indu U(\Cc)$ or $\Prou U(\Dd)$. We can hence reformulate the universal property from \autoref{univpropT0} as follows:
The datum of a commutative square
\[
\mymatrix{
\Cc \ar[r] \ar[d] & \Indu U(\Cc) \ar[d]^f \\ \Prou U(\Cc) \ar[r]_-g & \Dd
}
\]
such that $f$ preserves filtered colimits and $g$ preserves cofiltered limits is equivalent to that of a functor $\Tateu U_0(\Cc) \to \Dd$ preserving both filtered colimits and cofiltered limits which exist in $\Tateu U_0(\Cc)$.
\end{rmq}

\begin{prop}
Let $\Cc$ be a $\mathbb V$-small stable and idempotent complete $(\infty,1)$-category. There is a stable and idempotent complete $(\infty,1)$-category $\quot{\Indu U(\Cc)}{\Cc}$ who fits in a commutative diagram
\[
\mymatrix{
\Cc \ar[r] \ar[d] & \Indu U(\Cc) \ar[r] \ar[d] & \quot{\Indu U(\Cc)}{\Cc} \ar@{-}[d]^{=} \\
\Prou U(\Cc) \ar[r] & \Tateu U(\Cc) \ar[r] & \quot{\Indu U(\Cc)}{\Cc}
}
\]
such that the two horizontal lines are cofibre sequences in the category of stable and idempotent complete $(\infty,1)$-categories.
\end{prop}

\begin{proof}
Let us fix the following notations
\begin{align*}
&\mathrm I^\mathbb V \Cc = \Indu V(\Cc) \hspace{1cm} \mathrm I^\mathbb V \mathrm I \Cc = \Indu V \Indu U(\Cc) \\
&\mathrm I^\mathbb V \mathrm P \Cc = \Indu V \Prou U(\Cc) \hspace{1cm} \mathrm I^\mathbb V \mathrm T \Cc = \Indu V \Tateu U(\Cc)
\end{align*}
The commutative square
\[
\mymatrix{\Cc \ar[r] \ar[d] & \Indu U(\Cc) \ar[d] \\ \Prou U(\Cc) \ar[r] & \Tateu U(\Cc)}
\]
induces the square (1) of adjunctions between presentable stable $(\infty,1)$-categories
\[
\mymatrix{
\mathrm I^\mathbb V \Cc \ar@<2pt>[r] \ar@<-2pt>[d] &
\mathrm I^\mathbb V \mathrm I \Cc \ar@<2pt>[r] \ar@<-2pt>[d]_f \ar@<2pt>[l]^\varepsilon &
\Ee \ar@<-2pt>[d]_p \ar@<2pt>[l]^-\alpha \\
\mathrm I^\mathbb V \mathrm P \Cc \ar@<2pt>[r] \ar@<-2pt>[u]_g &
\mathrm I^\mathbb V \mathrm T \Cc \ar@<2pt>[l]^e \ar@<-2pt>[u] \ar@<2pt>[r] &
\Ee' \ar@<2pt>[l]^-a \ar@<-2pt>[u]_q
}
\]
which we complete on the right such that the two lines, from left to right, are exact sequences. We have represented here the left adjoints on top or on the left of their right adjoint.
From \cite[5.12 and 5.13]{bgt:characterisationk} we see that is suffices to prove that $p$ and $q$ are equivalences.
We will prove the sufficient assertions
\begin{assertions}
\item The functor $p$ is fully faithful.\label{pff}
\item The functor $q$ is conservative.\label{qcons}
\end{assertions}
Let us start with \ref{pff}. Using \cite[5.5]{bgt:characterisationk}, we deduce that both $a$ and $\alpha$ are fully faithful.
It thus suffices to prove the equivalence $ap \simeq f \alpha$. To do so, we will show that for any object $x \in \mathrm I^\mathbb V \mathrm I \Cc$, if $\varepsilon(x)$ vanishes, then so does $ef(x)$.
Let $\bar x \colon K \to \Indu U(\Cc)$ denote a $\mathbb V$-small filtered diagram whose colimit in $\mathrm I^\mathbb V \mathrm I \Cc$ is $x$. Let also $\bar y \colon L\op \to \Cc$ be a $\mathbb U$-small cofiltered diagram. We denote by $y$ its limit in $\Prou U(\Cc)$.
The image $ef(x)$ is the functor $\Prou U(\Cc) \to \sSets$ mapping $y$ to the simplicial set
\[
\colim_{k \in K} \colim_{l \in L} \Map_{\Indu U(\Cc)}(\bar y(l),\bar x(k)) \simeq
\colim_{l \in L} \colim_{k \in K} \Map_{\Indu U(\Cc)}(\bar y(l),\bar x(k))
\]
On the other hand, the assumption $\varepsilon(x) = 0$ implies that for any $c \in \Cc$, the space
\[
\colim_{k \in K} \Map_{\Indu U(\Cc)}(c,\bar x)
\]
is contractible. It follows that $ef(x)$ vanishes. We can now focus on \ref{qcons}. Since $q$ preserves exact sequences and $a$ is fully faithful, it suffices to prove that if $z \in \mathrm I^\mathbb V \mathrm T \Cc$ is such that both $f(z)$ and $e(z)$ vanish, then so does $z$.
We can see $z$ as a functor $\Tateu U(\Cc)\op \to \sSets$ preserving finite limits while $f(z)$ and $e(z)$ are its restriction respectively to $\Indu U(\Cc)\op$ and $\Prou U(\Cc)$. As $\Tateu U(\Cc)$ is generated by ind- and pro-objects under finite limits and retracts, we deduce that $z$ is equivalent to $0$.
\end{proof}

\begin{cor}\glsadd{K}\label{ktheorysusp}
Let $\Cc$ be a $\mathbb V$-small stable and idempotent complete $(\infty,1)$-category.
The spectrum of non-connective K-theory of $\Tateu U(\Cc)$ is the suspension of the non-connective $K$-theory of $\Cc$:
\[
\mathbb K(\Tateu U(\Cc)) \simeq \Sigma \mathbb K(\Cc)
\]
\end{cor}
\begin{rmq}
This corollary is an $\infty$-categorical version of a theorem of Sho Saito in exact $1$-categories in \cite{saito:deloop}.
\end{rmq}
\begin{proof}
Let us use the notations $\mathrm{I} \Cc = \Indu U(\Cc)$, $\mathrm P \Cc = \Prou U(\Cc)$ and $\mathrm T \Cc = \Tateu U(\Cc)$. 
Because the K-theory functor preserves cofibre sequences of stable categories (see \cite[sect. 9]{bgt:characterisationk}), we get two exact sequences in the $(\infty,1)$-category of spectra
\begin{align*}
\mathbb K (\Cc) \to \mathbb K (&\mathrm I \Cc) \to \mathbb K \left(\quot{\Indu U(\Cc)}{\Cc}\right) \\
\mathbb K (\mathrm P \Cc) \to \mathbb K (&\mathrm T \Cc) \to \mathbb K \left(\quot{\Indu U(\Cc)}{\Cc}\right)
\end{align*}
The vanishing of $\mathbb K(\mathrm P \Cc)$ and $\mathbb K (\mathrm I \Cc)$ -- since those categories contain countable sums -- concludes the proof.
\end{proof}

\begin{lem}\label{tate-dual}
Let $\Cc$ be a $\mathbb V$-small stable $(\infty,1)$-category with a functor $f \colon \Cc\op \to \Cc$.
The functor $f$ induces a functor
\[
\tilde f \colon \left(\Prou V\Indu U(\Cc)\right)\op \to \Prou V\Indu U(\Cc)
\]
which maps (elementary) $\mathbb U$-Tate objects to (elementary) $\mathbb U$-Tate objects.

If moreover the functor $f$ is an equivalence, then $\tilde f$ induces an equivalence 
\[
\left(\Tateu U(\Cc)\right)\op \simeq \Tateu U(\Cc)
\]
\end{lem}

\begin{proof}
The category $\Prou V\Indu U(\Cc)$ has all $\mathbb V$-small limits and colimits --- it is the opposite category of a $\mathbb V$-presentable category. We define the functor $\tilde f$ as the extension of the composition
\[
\Cc\op \to \Cc \to \Prou V\Indu U(\Cc)
\]
It maps objects of $\Indu U(\Cc)$ to objects of $\Prou U(\Cc) \subset \Prou V(\Cc)$ and vice-versa and therefore preserves pure Tate objects. The functor $\tilde f$ also preserves finite limits. It follows that it preserves Tate objects.
\end{proof}

\begin{rmq}
Assume that the category $\Cc$ in a closed symmetric monoidal category. We have seen that $\Prou V \Indu U(\Cc)$ is also symmetric monoidal. It is moreover closed. 
Let $f \colon \Cc\op \to \Cc$ be the functor defined through the internal hom 
\[
f(-) = \Homint_\Cc(-, 1)
\]
The functor $\tilde f$ of the previous lemma is then equivalent to
\[
\Homint_{\Prou V\Indu U(\Cc)}(-,1)
\]
\end{rmq}

\begin{df}
Let $\Cc$ be a $\mathbb V$-small $(\infty,1)$-category.
We define the functor
\[
\Tateextu U \colon \mymatrix@1{\Fct(\Cc,\inftyCat^{\mathbb V\mathrm{,st}}) \ar[r]^-i & \Fct(\Indu U(\Cc),\inftyCat^{\mathbb V\mathrm{,st}}) \ar[rr]^-{\Tateu U \circ -} && \Fct(\Indu U(\Cc),\inftyCat^{\mathbb V\mathrm{,st,id}})}
\]
\end{df}


\begin{prop}\label{tatediagexist}
Let $\Cc$ be a $\mathbb V$-small stable $(\infty,1)$-category. For any elementary Tate objects $X \in \Tateu U_\mathrm{el}(\Cc)$ there exists a $\mathbb U$-small cofiltered diagram $\bar X \colon K\op \to \Indu U(\Cc)$ such that
\begin{itemize}
\item The object $X$ is a limit of $\bar X$ in $\Prou U \Indu U(\Cc)$ and
\item For any $k \in K$ the diagram $\ker\left( \bar X \to \bar X(k) \right) \colon \comma{k}{K}\op \to \Indu U(\Cc)$ has values in the essential image of $\Cc$.
\end{itemize}
\end{prop}

\begin{rmq}
To state the above proposition in an informal way, any elementary Tate object $X$ can be represented by a diagram $\lim_\alpha \colim_\beta X_{\alpha\beta}$ such that for any $\alpha_0$, the kernel of canonical projection $X \to \colim_\beta X_{\alpha_0\beta}$ is actually a pro-object.
It implies that any elementary Tate object $X$ fits into an exact sequence
\[
X^p \to X \to X^i
\]
where $X^p \in \Prou U(\Cc)$ and $X^i \in \Indu U(\Cc)$. Such an exact sequence is called a lattice for $X$. The existence of lattices characterises elementary Tate objects.
We can moreover define an $\infty$-categorical version of Sato's Grassmannian: a functor mapping a Tate object to its category of lattices.
\end{rmq}

\begin{df}\label{tatediag}
Let $\Cc$ be a $\mathbb V$-small stable $(\infty,1)$-category. For any elementary Tate object $X \in \Tateu U_\mathrm{el}(\Cc)$, we will call a Tate diagram for $X$ any $\mathbb U$-small cofiltered diagram $\bar X \colon K\op \to \Indu U(\Cc)$ as in \autoref{tatediagexist}.
\end{df}

\begin{proof}[of \autoref{tatediagexist}]
Let $\Dd$ denote the full subcategory of $\Prou U \Indu U(\Cc)$ spanned by those objects $X$ satisfying the conclusion of the proposition.
The category $\Dd$ obviously contains both the essential images of $\Indu U(\Cc)$ and $\Prou U(\Cc)$.
It suffices to prove that $\Dd$ is stable by extension. We see that it is stable by shifts and we can thus consider an exact sequence $X \to X_0 \to X_1$ in $\Prou U \Indu U(\Cc)$ such that both $X_0$ and $X_1$ are in $\Dd$.
Let $\bar X_0 \colon K\op \to \Indu U(\Cc)$ and $\bar X_1 \colon L\op \to \Indu U(\Cc)$ be a $\mathbb U$-small cofiltered diagrams of whom $X_0$ and $X_1$ are limits in $\Prou U \Indu U(\Cc)$.
Using \autoref{strictification}, we can assume $K = L$ and that we have a diagram $K\op \to \Fct(\Delta^1, \Indu U(\Cc))$ of whom the map $X_0 \to X_1$ is a limit. Considering the pointwise kernel, we get a diagram $\bar X \colon K\op \to \Indu U(\Cc)$ of whom $X$ is a limit. It obviously satisfies the required property.
\end{proof}

\section{Adjunction and unit transformation}

We prove here a result about adjunction units between $(\infty,1)$-categories. A trustful reader could skip this part and refer to the results when needed.

Let $\Cc$ be a $\mathbb U$-small $(\infty,1)$-category. 
Let $s \colon \comma{\Cc}{\inftyCatu U} \to \inftyCatu V$ denote the constant functor $\Cc$ and $t$ the target functor $(\Cc \to \Dd) \mapsto \Dd$ --- composed with the inclusion $\inftyCatu U \to \inftyCatu V$.
The evaluation map
\[
(\inftyCatu U)^{\Delta^1} \times \Delta^1 \to \inftyCatu U \to \inftyCatu V
\]
define a natural transformation $e \colon s \to t$.
Let $\int t \to \comma{\Cc}{\inftyCatu U}$ denote the coCartesian fibration classfying $t$. The one classifying $s$ is the projection $\int s = \Cc \times \comma{\Cc}{\inftyCatu U} \to \comma{\Cc}{\inftyCatu U}$.
We can thus consider the map $E \colon \int s \to \int t$ induced by $e$.

\begin{df}
Let us denote by $F_\Cc$ the functor
\[
\mymatrix{
\Cc\op \times \int t \ar[r]^-\psi &
\Cc\op \times \left(\comma{\Cc}{\inftyCatu U}\right)\op \times \int t \ar[r]^-{E} &
\left(\int t\right)\op \times \int t \ar[r]^-{\Map_{\int t}} &
\sSets
}
\]
where $\psi$ is induced by the initial object of $\comma{\Cc}{\inftyCatu U}$.
\end{df}

\begin{lem}
Let $f$ be a functor $\Cc \to \Dd$ between $\mathbb U$-small $(\infty,1)$-categories.
It induces a map $\Dd \to \int t$. Moreover the functor
\[
\mymatrix{
\Cc\op \times \Dd \ar[r] & \Cc\op \times \int t \ar[r]^-{F_\Cc} & \sSets
}
\]
is equivalent to the functor 
\[
\mymatrix{
\Cc\op \times \Dd \ar[r]^-{f\op,\id} & \Dd\op \times \Dd \ar[r]^-{\Map_\Dd} & \sSets
}
\]
\end{lem}
\begin{proof}
There is by definition a natural transformation $\theta$ between the two functors at hand. 
To any pair $(c,d) \in \Cc\op \times \Dd$, it associates the natural map
\[
\Map_\Dd(f(c),d) \simeq \Map_{\int t_\Cc}((f,f(c)),(f,d)) \to \Map_{\int t_\Cc}((\id_\Cc,c),(f,d))
\]
which is an equivalence (see \cite[2.4.4.2]{lurie:htt}).
\end{proof}

\newcommand{\unitfct}{\operatorname{\varepsilon}}
\newcommand{\mapfct}{\operatorname{M}}
We will denote by $\inftyCat^{\mathbb U \mathrm{,L}}$ the sub-category of $\inftyCatu U$ of all categories but only left adjoint functors between them.
\begin{prop}\label{prop-unit-transformation}
Let $\Cc$ be a $\mathbb U$-small $(\infty,1)$-category. There exists a functor 
\[
\mapfct_\Cc \colon \comma{\Cc}{\inftyCatu U} \to \comma{\Map_\Cc(-,-)}{\Fct(\Cc \times \Cc\op, \sSets)}
\]
mapping a functor $f \colon \Cc \to \Dd$ to the functor $\Map_\Dd(f(-),f(-))$.
It restricts to a functor
\[
\unitfct_\Cc \colon \comma{\Cc}{\inftyCat^{\mathbb U \mathrm{,L}}} \to \comma{\id}{\Fct(\Cc, \Cc)}
\]
mapping a functor $f \colon \Cc \to \Dd$ with a right adjoint $g$ the unit transformation of the adjunction $\id \to gf$.
\end{prop}
\begin{proof}
We consider the composition
\[
\mymatrix{
\Cc\op \times \Cc \times \comma{\Cc}{\inftyCatu U} \ar[r]^-E & \Cc\op \times \int t \ar[r]^-{F_\Cc} & \sSets
}
\]
It induces a functor 
\[
\comma{\Cc}{\inftyCatu U} \to \Fct(\Cc \times \Cc\op, \sSets)
\]
The image of the initial object $\id_\Cc$ is the functor $\Map_\Cc(-,-)$. We get the required
\[
\mapfct_\Cc \colon \comma{\Cc}{\inftyCatu U} \to \comma{\Map_\Cc(-,-)}{\Fct(\Cc \times \Cc\op, \sSets)}
\]
Let $i$ denote the fully faithful functor
\[
\Fct(\Cc,\Cc) \to \Fct(\Cc, \presh(\Cc)) \simeq \Fct(\Cc \times \Cc\op, \sSets)
\]
The restriction of $\mapfct_\Cc$ to $\comma{\Cc}{\inftyCat^{\mathbb U\mathrm{,L}}}$ has image in the category of right representable functors $\Cc \times \Cc\op \to \sSets$. It therefore factors through $i$ and induces the functor
\[
\unitfct_\Cc \colon \comma{\Cc}{\inftyCat^{\mathbb U \mathrm{,L}}} \to \comma{\id}{\Fct(\Cc, \Cc)}
\]
\end{proof}

\newcommand{\counitfct}{\eta}
\begin{rmq}
There is a dual statement to \autoref{prop-unit-transformation}. Namely, there exists a functor
\[
\quot{\inftyCatu U}{\Cc} \to \quot{\Fct(\Cc \times \Cc\op, \sSets)}{\Map_\Cc(-,-)}
\]
which restricts to a functor
\[
\counitfct_\Cc \colon \quot{\inftyCat^{\mathbb U\mathrm{,L}}}{\Cc} \to \quot{\Fct(\Cc,\Cc)}{\id_\Cc}
\]
mapping a left adjoint $f$ to the counit transformation $fg \to \id_\Cc$ --- where $g$ is the right adjoint of $f$.
\end{rmq}

\begin{prop}\label{lim-adjoint}
Let $K$ be a $\mathbb U$-small filtered simplicial set. Let $\bar D \colon (K^{\triangleright})\op \to \inftyCat^{\mathbb V\mathrm{,L}}$ be a diagram. Let $\Dd$ be a limit of $K\op \to (K^\triangleright)\op \to \inftyCat^{\mathbb V\mathrm{,L}}$.
Let also $\Cc \in \inftyCat$ be the cone point of $\bar D$.
If the category $\Cc$ admits $K\op$-indexed limits then the natural functor $f \colon \Cc \to \Dd$ admits a right adjoint. This right adjoint is the limit in $\Fct(\Dd,\Cc)$ of a $K\op$-diagram induced by $\bar D$.
\end{prop}

\begin{proof}
The diagram $\bar D$ corresponds to a diagram $\tilde D \colon K\op \to \comma{\Cc}{\inftyCatu U}$.
Let us consider the pullback diagram
\[
\mymatrix{
\int (t_\Cc \circ \tilde D) \cart \ar[r] \ar[d] & \int t \ar[d] \\ K\op \ar[r]^-{\tilde D} & \comma{\Cc}{\inftyCatu U}
}
\]
The category $\Dd$ being a limit of $\tilde D$, there is a canonical natural transformation from the constant diagram $\Dd \colon K\op \to \inftyCatu U$ to $t \circ \tilde D$.
It induces a map $p \colon K\op \times \Dd \to \int t \circ \tilde D$.
Let us then consider the composite functor
\[
\mymatrix{
\Cc\op \times K\op \times \Dd \ar[r]^p & \Cc\op \times \int t \circ \tilde D \ar[r] & \Cc\op \times \int t \ar[r]^-{F_\Cc} & \sSets
}
\]
We get a functor $\psi \in \Fct(K\op, \Fct(\Dd \times \Cc\op, \sSets))$.
It maps a vertex $k \in K$ to the functor $\Map_{\Dd_k}(f_k(-), p_k(-))$ --- where $f_k \colon \Cc \to \Dd_k$ is $\tilde D(k)$ and $p_k \colon \Dd \to \Dd_k$ is the projection. For every $k$, the functor $f_k$ admits right adjoint.
It follows that $\psi$ has values in the full sub-category $\Fct(\Dd,\Cc)$ of $\Fct(\Dd \times \Cc\op, \sSets)$ spanned by right representable functors:
\[
\psi \colon K\op \to \Fct(\Dd,\Cc)
\]
Let $g$ be a limit of $\psi$. We will prove that $g$ is indeed a right adjoint of $f \colon \Cc \to \Dd$.
We can build, using the same process as for $\psi$, a diagram
\[
(K^\triangleright)\op \to \Fct(\Cc, \Cc)
\]
which corresponds to a diagram $\mu \colon K\op \to \comma{\id_\Cc}{\Fct(\Cc,\Cc)}$. The composition
\[
K\op \to \comma{\id_\Cc}{\Fct(\Cc,\Cc)} \to \Fct(\Cc,\Cc)
\]
is moreover equivalent to
\[
K \op \to^\psi \Fct(\Dd,\Cc) \to^{- \circ f} \Fct(\Cc,\Cc)
\]
The limit of $\mu$ therefore defines a natural transformation $\id_\Cc \to fg$. It exhibits $g$ as a right adjoint to $f$.
\end{proof}

\begin{lem}\label{colim-map}
Let $K$ be a filtered simplicial set and let $\bar \Cc \colon K \to \inftyCat$ be a diagram. For any $k \in K$ we will write $\Cc_k$ instead of $\bar \Cc(k)$. We will also write $\Cc$ for a colimit of $\bar \Cc$.
Every object of $\Cc$ is in the essential image of at least one of the canonical functors $f_k \colon \Cc_k \to \Cc$.
For any pair of objects in $\Cc$, we can assume they are the images of $x$ and $y$ in $\Cc_k$ for some $k$, and we have
\[
\Map_\Cc(f_k(x), f_k(y)) \simeq \colim_{\phi \colon k \to l} \Map_{\Cc_l}(\bar \Cc(\phi)(x),\bar \Cc(\phi)(y))
\]
\end{lem}

\begin{proof}
This is a simple computation, using that finite simplicial sets are compact in $\inftyCat$.
\begin{align*}
\Map_\Cc(f_k(x),f_k(y))
&\simeq \Map(\Delta^1, \Cc) \timesunder[\Map( \pt \amalg \pt, \Cc)] \{(f_k(x),f_k(y))\}\\
&\simeq \colim_{\phi \colon k \to l} \left(\Map(\Delta^1, \Cc_l) \timesunder[\Map( \pt \amalg \pt, \Cc_l)] \{(\bar \Cc(\phi)(x),\bar \Cc(\phi)(y))\}\right)\\
&\simeq \colim_{\phi \colon k \to l} \Map_{\Cc_l}(\bar \Cc(\phi)(x),\bar \Cc(\phi)(y))\\
\end{align*}
\end{proof}

\begin{lem} \label{colimcats}
Let $K$ be a $\mathbb V$-small filtered simplicial set and let $\bar \Cc \colon K \to \inftyCat^\mathbb V$ be a diagram of $\mathbb V$-small categories. Let us assume that for each vertex $k \in K$ the category $\Cc_k = \bar \Cc(k)$ admits finite colimits and that the transition maps in the diagram $\bar \Cc$ preserve finite colimits.
For any $k \to l \in K$, let us fix the following notations
\[
\mymatrix{
 &\colim \bar \Cc \ar[r]^-i & \Indu V(\colim \bar \Cc) & \\
\Cc_l & \Cc_k \ar[l]_{\phi_{kl}} \ar[u]^{u_k} \ar[r]_-{j_k} & \Indu V(\Cc_k) \ar[u]^{a_k} \ar@<2pt>[r]^{f_{kl}} & \Indu V(\Cc_l) \ar@<2pt>[l]^{g_{kl}}
}
\]
where the functor $a_k$ is $\Indu V(u_k)$. The functor $\phi_{kl}$ is the transition map $\bar \Cc(k \to l)$, the functor $f_{kl}$ is $\Indu V(\phi_{kl})$ and $g_{kl}$ is its right adjoint.
\begin{enumerate}
\item (Lurie) The category $\colim \bar \Cc$ admits finite colimits and for any $k$ the functor $u_k$ preserves such colimits. It follows that $a_k$ admits a right adjoint $b_k$.
\item (Lurie) The natural functor $\Indu V(\colim \bar \Cc) \to \colim \Indu V(\bar \Cc) \in \PresLeftu V$ is an equivalence. Those two categories are also equivalent to the limit of the diagram $\Indu V(\bar \Cc)^\mathrm R$ of right adjoints
\[
\Indu V(\bar \Cc)^\mathrm R \colon \mymatrix@1@!0@C=3cm{
K \ar[r]^-{\Indu V(\bar \Cc)} & \PresLeftu V \simeq \left(\PresRightu V\right) \op
}
\]
\item For any $k \in K$, the adjunction transformation $j_k \to b_k a_k j_k$ is a colimit of the diagram
\[
\mu_k \colon \mymatrix{
\displaystyle \comma{k}{K} \ar[rr]^-{\Indu 
V(\bar \Cc)} && \displaystyle \comma{\Indu V(\Cc_k)}{\inftyCat^{\mathbb V\mathrm{,L}}} \ar[r]^-{\unitfct \circ j_k} & \displaystyle \comma{j_k}{\Fct(\Cc_k, \Indu V(\Cc_k))}
}
\]
\end{enumerate}
If moreover $K$ is $\mathbb U$-small and if for any $k \to l \in K$, the map $g_{kl} \colon \Indu V(\Cc_l) \to \Indu V(\Cc_k)$ restricts to a map $\tilde g_{kl} \colon \Indu U(\Cc_l) \to \Indu U(\Cc_k)$ then
\begin{enumerate}[start = 4]
\item For any $k \in K$ the functor $b_k$ restricts to a functor $\tilde b_k \colon \Indu U(\Cc) \to \Indu U(\Cc_k)$, right adjoint to $\tilde a_k = \Indu U(u_k)$. Moreover for any $k \to l$ the map $\tilde g_{kl}$ is a right adjoint to $\tilde f_{kl} = \Indu U(\phi_{kl})$.
\item There exists a diagram $\Indu U(\bar \Cc)^\mathrm R \colon K\op \to \inftyCat$ mapping $k \to l$ to $\tilde g_{kl}$ whose limit satisfies
\[
\lim \Indu U(\bar \Cc)^\mathrm R \simeq \Indu U(\colim \bar \Cc)
\]
\item For any $k \in K$,  the adjunction transformation $\tilde \jmath_k \to \tilde b_k \tilde a_k \tilde \jmath_k$ is a colimit of the diagram
\[
\tilde \mu_k \colon \mymatrix{
\displaystyle \comma{k}{K} \ar[rr]^-{\Indu 
U(\bar \Cc)} && \displaystyle \comma{\Indu U(\Cc_k)}{\inftyCat^{\mathbb V\mathrm{,L}}} \ar[r]^-{\unitfct \circ \tilde \jmath_k} & \displaystyle \comma{\tilde \jmath_k}{\Fct(\Cc_k, \Indu U(\Cc_k))}
}
\]
where $\tilde \jmath_k$ is the canonical map $\Cc_k \to \Indu U(\Cc_k)$.
\end{enumerate}
\end{lem}

\begin{proof}
The first item is \cite[5.5.7.11]{lurie:htt}. The second is a combination of \cite[5.5.7.10, 5.5.3.4 and 5.5.3.18]{lurie:htt} and \cite[6.3.7.9]{lurie:halg}.
Concerning \emph{(iii)}, we consider the colimit of the diagram
\[
\mymatrix{
\displaystyle \comma{k}{K} \ar[r]^-{\mu_k} & \comma{j_k}{\Fct(\Cc_k, \Indu V(\Cc_k))} \ar[r] & \comma{\Map_{\Cc_k}(-,-)}{\Fct(\Cc_k \times \Cc_k\op,\sSets)}
}
\]
This diagram is equivalent to 
\[
\mymatrix{
\theta \colon \displaystyle \comma{k}{K} \ar[r]^-{\bar \Cc} & \comma{\Cc_k}{\inftyCatu U} \ar[r]^-{\mapfct_{\Cc_k}} & \comma{\Map_{\Cc_k}(-,-)}{\Fct(\Cc_k \times \Cc_k\op,\sSets)}
}
\]
From \autoref{colim-map}, the colimit of $\theta$ is the functor
\begin{align*}
\Map_{\Cc}(u_k(-),u_k(-)) \simeq& \Map_{\Indu V(\Cc)}(i u_k(-),i u_k(-)) \\ \simeq& \Map_{\Indu V(\Cc)}(a_k j_k(-),a_k j_k(-))
\end{align*}
where $\Cc$ denotes a colimit of $\bar \Cc$.
This concludes the proof of \emph{(iii)} and we now focus on \emph{(iv)}.

Let $k \to l \in K$ and let $\id \to g_{kl} f_{kl}$ denote a unit for the adjunction. It restricts to a natural transformation $\id \to \tilde g_{kl} \tilde f_{kl}$ which exhibits $\tilde g_{kl}$ as a right adjoint to $\tilde f_{kl}$.
Using the same mechanism, if the functor $b_k$ restricts to $\tilde b_k$ as promised then $\tilde b_k$ is indeed a right adjoint to $\tilde a_k$. It thus suffices to prove that the functor $b_k i$ factors through the canonical inclusion $t_k \colon \Indu U(\Cc_k) \to \Indu V(\Cc_k)$.
Every object of $\Cc$ is in the essential image of $u_l$ for some $k \to l \in K$.
It is therefore enough to see that for any $k \to l$, the functor $b_k i u_l$  factors through $t_k$. We compute
\[
b_k i u_l \simeq b_k a_l j_l \simeq g_{kl} b_l a_l j_l \simeq g_{kl} (\colim \mu_l)
\]
The diagram $\mu_l \colon \comma{l}{K} \to \comma{j_l}{\Fct(\Cc_l, \Indu V(\Cc_l))}$ factors into
\[
\mymatrix{
\displaystyle \comma{l}{K} \ar[r]^-{\tilde \mu_l} & \displaystyle \comma{\tilde \jmath_l}{\Fct(\Cc_l, \Indu U(\Cc_l))} \ar[r]^-{t_l} & \displaystyle \comma{j_l}{\Fct(\Cc_l, \Indu V(\Cc_l))}
}
\]
Because $g_{kl}$, $\tilde g_{kl}$ and $t_l$ preserve $\mathbb U$-small filtered colimits, the functor $b_k i u_l$ is the colimit of the diagram
\[
\mymatrix{
\displaystyle \comma{l}{K} \ar[r]^-{\tilde \mu_l} & \displaystyle \comma{\tilde \jmath_l}{\Fct(\Cc_l, \Indu U(\Cc_l))} \ar[r]^-{t_k \tilde g_{kl}} & \displaystyle \comma{t_k \tilde g_{kl} \tilde \jmath_l}{\Fct(\Cc_l, \Indu V(\Cc_k))}
}
\]
The functor $t_k$ also preserves $\mathbb U$-small filtered colimits and we have
\[
b_k i u_l \simeq t_k (\colim \tilde g_{kl} \circ \tilde \mu_l)
\]

To prove \emph{(v)}, we use \cite[6.2.3.18]{lurie:halg} to define the diagram $\Indu U(\bar \Cc)^\mathrm R$. It then follows that the equivalence of \emph{(ii)}
\[
\lim \Indu V(\bar \Cc)^\mathrm R \simeq \Indu V(\colim \bar \Cc)
\]
restricts to the required equivalence.
We finally deduce \emph{(vi)} from the \emph{(iii)}.
\end{proof}

\begin{cor}\label{ind-right-adjoint}
Let $\Cc$ be an $(\infty,1)$-category and let $F \colon \Cc \to \inftyCat^{\mathbb V\mathrm{,L}}$ be a functor.
For any $c \in \Cc$ and any $f \colon c \to d \in \Indu U(\Cc)$, the functor
\[
\Indextu U_\Cc(F)(f) \colon \Indu U(F(c)) \simeq \Indextu U_\Cc(F)(c) \to \Indextu U_\Cc(F)(d)
\]
admits a right adjoint.
\end{cor}

\section{Computation techniques}

We will now establish a few computational rules for the functors $\Indext$ and $\Proext$.
A trustful reader not interested in $(\infty,1)$-category theory could skip this subsection and come back for the results when needed.
Let us start with a $\mathbb V$-small $\infty$-category $\Cc$.
Let $s_\Cc \colon \Cc^{\Delta^1} \to \Cc$ denote the source functor while $t_\Cc \colon \Cc^{\Delta^1} \to \Cc$ denote the target functor.
Using \cite[2.4.7.11 and 2.4.7.5]{lurie:htt} we see that $s_\Cc$ is a Cartesian fibration and $t_\Cc$ is a coCartesian fibration.

\begin{df}
let $\Cc$ be a $\mathbb V$-small $(\infty,1)$-category.
Let us denote by $\undercat_\Cc \colon \Cc\op \to \inftyCatu V$ the functor classified by $s_\Cc$.\glsadd{undercat}
Let us denote by $\overcat_\Cc \colon \Cc \to \inftyCatu V$ the functor classified by $t_\Cc$.\glsadd{overcat}
\end{df}

\begin{rmq}
The functor $\undercat_\Cc$ map an object $c \in \Cc$ to the comma category $\comma{c}{\Cc}$ and an arrow $c \to d$ to the forgetful functor
\[
\comma{d}{\Cc} \to \comma{c}{\Cc}
\]
The functor $\overcat_\Cc$ map an object $c \in \Cc$ to the comma category $\quot{\Cc}{c}$ and an arrow $c \to d$ to the forgetful functor
\[
\quot{\Cc}{c} \to \quot{\Cc}{d}
\]
\end{rmq}

\begin{lem}\label{quot-ind}
Let $\Cc$ be a $\mathbb V$-small $(\infty,1)$-category.
There is a natural equivalence
\[
\Indextu V_\Cc(\overcat_\Cc) \simeq \overcat_{\Indu V(\Cc)}
\]
It induces an equivalence
\[
\Indext_\Cc(\overcat_\Cc) \simeq \overcat_{\Indu U(\Cc)}
\]
\end{lem}

\begin{rmq}\label{lim}
Because of \emph{(ii)} in \autoref{colimcats}, if the category $\Cc$ admits all finite colimits then we have
\[
\lim_k \quot{\Indu V(\Cc)}{\underline c_k} \simeq \Indu V\left( \colim_k \quot{\Cc}{\underline c_k} \right) \simeq \quot{\Indu V(\Cc)}{c}
\]
where the limit on the left hand side is computed using base change functors. If $K$ is $\mathbb U$-small and if $\Indu U(\Cc)$ admits pullbacks then it restricts to an equivalence
\[
\lim_k \quot{\Indu U(\Cc)}{\underline c_k} \simeq \quot{\Indu U(\Cc)}{c}
\]
Let us also note that there is a dual statement to \autoref{quot-ind} involving $\Proextu U$:
\[
\Proextu U_\Cc(\overcat_\Cc) \simeq \overcat_{\Prou U(\Cc)}
\]
\end{rmq}

\begin{proof}
Let us first consider the pullback category
\[
\mymatrix{
\comma{\Cc}{\Indu V(\Cc)} \ar[r] \ar[d]_q \cart & \Indu V(\Cc)^{\Delta^1} \ar[d]^{\mathrm{source,target}} \\ \Cc \times \Indu V(\Cc) \ar[r] & \Indu V(\Cc) \times \Indu V(\Cc)
}
\]
The functor $q \colon \comma{\Cc}{\Indu V(\Cc)} \to \Cc \times \Indu V(\Cc) \to \Indu V(\Cc)$ is a coCartesian fibration. Let $p$ denote the coCartesian fibration $p \colon \mathcal E \to \Indu V(\Cc)$ classified by the extension of $\overcat_\Cc$
\[
\tilde \overcat_\Cc \colon \Indu V(\Cc) \to \inftyCatu V
\]
There is a natural morphism functor $g \colon \mathcal E \to \comma{\Cc}{\Indu V(\Cc)}$ over $\Indu V(\Cc)$. It induces an equivalence fiberwise and therefore $g$ is an equivalence.
Let $\Dd \to \Indu V(\Cc)$ denote a coCartesian fibration classifying the functor
\[
\Indextu V_\Cc(\overcat_\Cc) \simeq \Indu V \circ \left( \tilde \overcat_\Cc \right) \colon  \Indu V(\Cc) \to \inftyCat
\]
We have a diagram of coCartesian fibration over $\Indu V(\Cc)$
\[
\Dd \from \mathcal E \simeq \comma{\Cc}{\Indu V(\Cc)} \to \Indu V(\Cc)^{\Delta^1}
\]
We consider the relative Kan extension $\Dd \to \Indu V(\Dd)^{\Delta^1}$ of $\comma{\Cc}{\Indu V(\Cc)} \to \Indu V(\Cc)^{\Delta^1}$.
We thus have the required natural transformation $T \colon \Indextu V_\Cc(\overcat_\Cc) \to \overcat_{\Indu V(\Cc)}$.

Let now $c \in \Indu V(\Cc)$. Let $\underline c \colon K \to \Cc$ be a $\mathbb V$-small filtered diagram whose colimit in $\Indu V(\Cc)$ is $c$.
The map
\[
T(c) \colon \Indu V\left( \colim_k \quot{\Cc}{\underline c_k} \right) \to \quot{\Indu V(\Cc)}{c}
\]
is equivalent to the ind-extension of the universal map
\[
f \colon \colim_k \quot{\Cc}{\underline c_k} \to \quot{\Indu V(\Cc)}{c}
\]
For every $k \in K$, let us denote by $f_k$ the natural functor
\[
f_k \colon \quot{\Cc}{\underline c_k} \to \quot{\Indu V(\Cc)}{c}
\]
Using \cite[5.3.5.11]{lurie:htt}, to prove $T(c)$ is an equivalence, it suffices to see that :
\begin{itemize}
\item the functors $f_k$ have values in compact objects,
\item the functor $f$ is fully faithful,
\item and the functor $T(c)$ is essentially surjective.
\end{itemize}
Those three items are straightforwardly proved. We will still expand the third one. Let thus $d \in \Indu V(\Cc)$ with a map $d \to c$ in $\Indu V(\Cc)$.
There exists a $\mathbb V$-small filtered diagram $\underline d \colon L \to \Cc$ whose colimit in $\Indu V(\Cc)$ is $d$.
For every $l \in L$ there exists an $k(l)$ such that the map $\underline d_l \to c$ factors through $\underline c_{k(l)} \to c$.
This implies that $d$ is in the essential image of $T(c)$.

The construction of a natural transformation $S \colon \Indext_\Cc(\overcat_\Cc) \to \overcat_{\Indu U(\Cc)}$ is similar to that of $T$.
If $\underline c \colon K \to \Cc$ is $\mathbb U$-small then the equivalence
\[
T(c) \colon \Indu V\left( \colim_k \quot{\Cc}{\underline c_k} \right) \to^\sim \quot{\Indu V(\Cc)}{c}
\]
restricts to the equivalence
\[
S(c) \colon \Indu U\left( \colim_k \quot{\Cc}{\underline c_k} \right) \to^\sim \quot{\Indu U(\Cc)}{c}
\]
\end{proof}

\begin{lem}
Let $\Cc$ be a $\mathbb V$-small $(\infty,1)$-category with all pushouts. The Cartesian fibration
\[
s_\Cc \colon \Cc^{\Delta^1} \to \Cc
\]
is then also a coCartesian fibration.
\end{lem}

\begin{proof}
This is a consequence of \cite[5.2.2.5]{lurie:htt}.
\end{proof}

\begin{rmq}
If $\Cc$ is an $(\infty,1)$-category with all pullbacks, then the target functor
\[
t_\Cc \colon \Cc^{\Delta^1} \to \Cc
\]
is also a Cartesian fibration.
\end{rmq}

\begin{df}
Let $\Cc$ be an $(\infty,1)$-category.
If $\Cc$ admits all pushouts, we will denote by $\undercat_\Cc^{\amalg}$ the functor classifying the coCartesian fibration $s_\Cc$ :\glsadd{underamalg}
\[
\undercat_\Cc^{\amalg} \colon \Cc \to \inftyCatu V
\]
If $\Cc$ admits all pullbacks, we will denote by $\overcat_\Cc^{\times}$ the functor classifying the Cartesian fibration $t_\Cc$ :\glsadd{overprod}
\[
\overcat_\Cc^{\times} \colon \Cc\op \to \inftyCatu V
\]
Note that those two constructions are of course linked : the functor $\overcat_\Cc^{\times}$ is the composition of $\undercat_{\Cc\op}^{\amalg}$ with the functor $(-)\op \colon \inftyCatu V \to \inftyCatu V$.
\end{df}
\begin{rmq}
The functor $\undercat_\Cc^{\amalg}$ map an object $c$ to the comma category $\comma{c}{\Cc}$ and a map $c \to d$ to the functor
\[
- \amalg_c d \colon \comma{c}{\Cc} \to \comma{d}{\Cc}
\]
The functor $\overcat_\Cc^{\times}$ maps a morphism $c \to d$ to the pullback functor
\[
- \times_d c \colon \quot{\Cc}{d} \to \quot{\Cc}{c}
\]
\end{rmq}

\begin{lem}\label{comma-ind}
Let $\Cc$ be a $\mathbb V$-small $\infty$-category with all pushouts.
There is a natural equivalence
\[
\Indextu V_\Cc(\undercat^{\amalg}_\Cc) \simeq \undercat^{\amalg}_{\Indu V(\Cc)}
\]
It induces an equivalence 
\[
\Indextu U_\Cc(\undercat^{\amalg}_\Cc) \simeq \undercat^{\amalg}_{\Indu U(\Cc)}
\]
\end{lem}
\begin{rmq}\label{comma-pro}
Unwinding the definition, we can stated the above lemma as follows. Let $\bar c \colon K \to \Cc$ be a filtered diagram. The canonical functor
\[
\Indu U\left(\colim_k \comma{\bar c_k}{\Cc} \right) \to \comma{c}{\Indu U(\Cc)}
\]
is an equivalence -- where $c$ is a colimit of $\bar c$ in $\Indu U(\Cc)$.
Using \autoref{lim}, we can shows the following similar statement. If $\Cc$ admits pullbacks then there is an equivalence
\[
\Proextu U_{\Cc\op}(\undercat_{\Cc\op}^{\amalg}) \simeq \undercat_{\Prou U (\Cc\op)}^{\amalg}
\]
\end{rmq}
\begin{proof}
This is very similar to the proof of \autoref{quot-ind}. Let us first form the pullback category
\[
\mymatrix{
\quot{\Indu V(\Cc)}{\Cc} \cart \ar[r]^\pi \ar[d] & \Indu V(\Cc)^{\Delta^1} \ar[d]^{\mathrm{source, target}} \\ \Indu V(\Cc) \times \Cc \ar[r] & \Indu V(\Cc) \times \Indu V(\Cc)
}
\]
The induced map $q \colon \quot{\Indu V(\Cc)}{\Cc} \to \Indu V(\Cc) \times \Cc \to \Indu V(\Cc)$ is a coCartesian fibration. We can \todo{Détailler ?}show the same way we did in the proof of \autoref{quot-ind} that it is classified by the extension of $\undercat_\Cc^{\amalg}$
\[
\tilde \undercat_\Cc^{\amalg} \colon \Indu V(\Cc) \to \inftyCatu V
\]
The functor $\pi$ preserves coCartesian morphisms and therefore induces a natural transformation $\tilde \undercat_\Cc^{\amalg} \to \undercat^{\amalg}_{\Indu V(\Cc)}$.
This transformation extends to a natural transformation
\[
T \colon \Indextu V_\Cc(\undercat_\Cc^{\amalg}) \to \undercat^{\amalg}_{\Indu V(\Cc)}
\]
To prove that $T$ is an equivalence, it suffices to prove that for every $c \in \Indu V(\Cc)$ and any $\mathbb V$-small filtered diagram $\underline c \colon K \to \Cc$ whose colimit is $c$, the induced functor
\[
T(c) \colon \Indu V\left(\colim_k \comma{\underline c_k}{\Cc}\right) \to^\sim \comma{c}{\Indu V(\Cc)}
\]
is an equivalence.

Let us first assume that $K$ is a point and thus that $c$ belong to $\Cc$. The canonical functor $\comma{c}{\Cc} \to \comma{c}{\Indu V(\Cc)}$ is fully faithful and its image is contained in the category of compact objects of $\comma{c}{\Indu V(\Cc)}$.
The induced functor
\[
T(c) \colon \Indu V \left(\comma{c}{\Cc} \right) \to \comma{c}{\Indu V(\Cc)}
\]
is therefore fully faithful (see \cite[5.3.5.11]{lurie:htt}). Let $d \in \Indu V(\Cc)$ with a map $c \to d$.
Let $\underline d \colon L \to \Cc$ be a $\mathbb V$-small filtered diagram whose colimit in $\Indu V(\Cc)$ is $d$.
There exist some $l_0 \in L$ such that the map $c \to d$ factors through $\underline d_{l_0} \to d$.
The diagram $\comma{l_0}{L} \to \Cc$ is in the image of $F$ and its colimit in $\Indu V(\Cc)$ is $d$. The functor $F$ is also essentially surjective and thus an equivalence.
It restricts to an equivalence
\[
\Indu U \left(\comma{c}{\Cc} \right) \to \comma{c}{\Indu U(\Cc)}
\]
Let us go back to the general case $c \in \Indu V(\Cc)$. The targeted equivalence is
\[
\Indu V\left( \colim_k \comma{\underline c_k}{\Cc} \right) \simeq \lim_k \comma{\underline c_k}{\Indu V(\Cc)} \simeq \comma{c}{\Indu V(\Cc)}
\]
where the limit is computed using the forgetful functors. The same argument works when replacing $\mathbb V$ by $\mathbb U$, using \autoref{colimcats}, item \emph{(iv)}.
\end{proof}

\begin{lem}\label{pro-quot-times}
Let $\Cc$ be an $(\infty,1)$-category with all pullbacks. Let us denote by $j$ the inclusion $\Indu U (\Cc) \to \IP(\Cc) = \Indu U \Prou U(\Cc)$.
There is a fully faithful natural transformation
\[
\Upsilon^\Cc \colon \Proext_{\Cc\op}(\overcat_\Cc^{\times}) \to \overcat_{\IP(\Cc)}^{\times} \circ (j\op)
\]
between functors $(\Indu U(\Cc))\op \to \inftyCatu U$
\end{lem}

\begin{rmq}\label{ind-quot-times}
To state this lemma more informally, for any filtered diagram $\bar c \colon K \to \Cc$, we have a fully faithful functor
\[
\lim_k \Prou U\left(\quot{\Cc}{\bar c_k}\right) \to \quot{\IP(\Cc)}{j(c)}
\]
where $c$ is a colimit of $\bar c$ in $\Indu U(\Cc)$.
This lemma has an ind-version, actually easier to prove. If $\bar d \colon K\op \to \Cc$ is now a cofiltered diagram, then there is a fully faithful functor
\[
\Indu U\left(\colim_k \quot{\Cc}{\bar d_k} \right) \to \quot{\IP(\Cc)}{i(d)}
\]
where $d$ is a limit of $\bar d$ in $\Prou U(\Cc)$.
To state that last fact formally, if $\Cc$ be an $(\infty,1)$-category with all pullbacks then there is a fully faithful natural transformation
\[
\Xi^\Cc \colon \Indext_{\Cc\op} (\overcat_\Cc^{\times}) \to \overcat_{\IP(\Cc)}^{\times} \circ (i\op)
\]
where $i$ is the canonical inclusion $\Prou U(\Cc) \to \IP(\Cc)$.
\end{rmq}

\begin{proof}
Let us first consider the functor $\Prou U \circ \overcat_\Cc^{\times} \colon \Cc \op \to \inftyCatu V$. It classifies the Cartesian fibration $F$ defined as the pullback
\[
\mymatrix{
\quot{\Prou U(\Cc)}{\Cc} \ar[d]_F \ar[r] \cart & \Prou U(\Cc)^{\Delta^1} \ar[d]^{t_{\Prou U(\Cc)}} \\ \Cc \ar[r] & \Prou U(\Cc)
}
\]
The canonical inclusion $\Prou U(\Cc) \to \IP(\Cc)$ defines a functor $f$ fitting in the commutative diagram
\[
\mymatrix{
\quot{\Prou U(\Cc)}{\Cc} \ar[r]^-f \ar[rd]_F & \quot{\IP(\Cc)}{\Cc} \ar[r] \ar[d] \cart & \IP(\Cc)^{\Delta^1} \ar[d]^{t_{\IP(\Cc)}} \\
& \Cc \ar[r] & \IP(\Cc)
}
\]
From \cite[2.4.7.12]{lurie:htt} we deduce that $f$ preserves Cartesian morphisms.
It therefore defines a natural transformation $u^\Cc$ from $\Prou U \circ \overcat_\Cc^{\times}$ to the restriction to $\Cc\op$ of $\overcat_{\IP(\Cc)}^{\times}$.
Since $\overcat_{\Prou U(\Cc)}^{\times} \circ (j\op)$ is the right Kan extension of its restriction to $\Cc\op$ (see \autoref{lim}), this defines the required natural transformation 
\[
\Upsilon^\Cc \colon \Proext_{\Cc\op} (\overcat_\Cc^{\times}) \to \overcat_{\IP(\Cc)}^{\times} \circ (j\op)
\]
To see that for any $c \in \Prou U(\Cc)$, the induced functor $\Upsilon^\Cc_c$ is fully faithful, it suffices to see that for any $c \in \Cc$ the functor $u^\Cc_c$ is fully faithful, which is obvious.
\end{proof}

\begin{lem}\label{bicomma}
Let $\Cc$ be a simplicial set.
If $\Cc$ is a quasi-category then the map $\Delta^1 \to \Delta^2$
\[
\mymatrix@!0@C=2pc@R=1pc{\bullet \ar[rr]^-f && \bullet} \hspace{7mm} \to \hspace{7mm} \mymatrix@!0@C=2pc@R=1pc{\bullet \ar[rr]^-f \ar[rd] &&  \bullet \\ & \bullet \ar[ru] & }
\]
induces an inner fibration $p \colon \Cc^{\Delta^2} \to \Cc^{\Delta^1}$.
If moreover $\Cc$ admits pullbacks then $p$ is a Cartesian fibration.
\end{lem}

\begin{df}\label{betweendef}
Let $\Cc$ be an $(\infty,1)$-category with pullbacks. Let us denote by\glsadd{btwtimes}
\[
\btw_\Cc^{\times} \colon \left(\Cc^{\Delta^1}\right)\op \to \inftyCat
\]
the functor classified by the Cartesian fibration $p$ of \autoref{bicomma}.
If $\Dd$ is an $(\infty,1)$-category with pushouts, we define similarly\glsadd{btwamalg}
\[
\btw_\Dd^{\amalg} \colon \Dd^{\Delta^1} \to \inftyCat
\]
\end{df}

\begin{rmq}
Let $\Cc$ be an $\infty$-category with pullbacks. The functor $\btw^{\times}_\Cc$ maps a morphism $f \colon x \to y$ to the category $\comma{f}{\left(\quot{\Cc}{y}\right)}$ of factorisations of $f$. It maps a commutative square
\[
\mymatrix{
x \ar[r]_f \ar[d] & y \ar[d] \\ z \ar[r]^g & t
}
\]
seen as a morphism $f \to g$ in $\Cc^{\Delta^1}$ to the base change functor
\[
(z \to a \to t) \mapsto (x \to a \times_t y \to y)
\]
\end{rmq}

\begin{proof}[of \autoref{bicomma}]
For every $0 < i < n$ and every commutative diagram
\[
\mymatrix{
\Lambda^n_i \ar[r] \ar[d] & \Cc^{\Delta^2} \ar[d]^p \\ \Delta^n \ar[r] & \Cc^{\Delta^1} 
}
\]
we must build a lift $\Delta^n \to \Cc^{\Delta^2}$. The datum of such a lift is equivalent to that of a lift $\phi$ in the induced commutative diagram
\[
\mymatrix{
\displaystyle \Delta^2 \times \Lambda^n_i \coprod_{\Delta^1 \times \Lambda^n_i} \Delta^1 \times \Delta^n \ar[r] \ar[d]
& \Cc \ar[d]
\\ \Delta^n \times \Delta^2 \ar[r] \ar@{-->}[ur]_\phi & \pt
}
\]
The existence of $\phi$ then follows from the fact that $\Cc$ is a quasi-category.

Let us now assume that $\Cc$ admits pullbacks. The functor $p$ is a Cartesian fibration if and only if every commutative diagram
\[
\mymatrix{
\displaystyle \Delta^2 \times \{1\} \coprod_{\Delta^1 \times \{1\}} \Delta^1 \times \Delta^1 \ar[r]^-f \ar[d]  & \Cc \ar[d] \\
\Delta^2 \times \Delta^1 \ar[r] & \pt
}
\]
admits a lift $\Delta^2 \times \Delta^1 \to \Cc$ which corresponds to a Cartesian morphism of $\Cc^{\Delta^2}$. Let us fix such a diagram. It corresponds to a diagram in $\Cc$
\[
\entrymodifiers={+[o]}
\mymatrix@!0@C=2pc@R=1.7pc{
& y \ar[dr] & \\
x \ar[rr] \ar[ru] && z \\ \\
a \ar[rr] \ar[rruu] \ar[uu] && c \ar[uu]
}
\]
Because $\Cc$ is a quasi-category, we can complete the diagram above with an arrow $a \to y$, faces and a tetrahedron $[a,x,y,z]$.
Let $g$ denote the map
\[
g \colon \Lambda^2_2 \hookrightarrow \Delta^2 \times \{ 1\} \coprod_{\Delta^1 \times \{ 1 \} } \Delta^1 \times \Delta^1 \to^f \Cc
\]
corresponding to the sub-diagram $y \to z \from c$.
By assumption, there exists a limit diagram $\bar b \colon \pt \star \Lambda^2_2 \to \Cc$ -- where $\star$ denotes the joint construction, see \cite[1.2.8 ]{lurie:htt}.
Note that the plain square 
\[
\entrymodifiers={+[o]}
\mymatrix@!0@C=2pc@R=1.7pc{
& y \ar[dr] &\\
x \ar@{.>}[rr] \ar@{.>}[ru] && z \\ \\
a  \ar[rr] \ar[rruu] \ar@{.>}[uu] \ar[uuur]|!{[uu];[uur]}\hole && c \ar[uu]
}
\]
forms a map $\bar a \colon \{0\} \star \Lambda^2_2 \to \Cc$. Because $\bar b$ is a limit diagram, there exists a map $\Delta^1 \star \Lambda^2_2 \to \Cc$ whose restriction to $\{0\} \star \Lambda^2_2$ is $\bar a$ and whose restriction to $\{1\} \star \Lambda^2_2$ is $\bar b$. This defines two tetrahedra $[a,b,c,z]$ and $[a,b,y,z]$ represented here
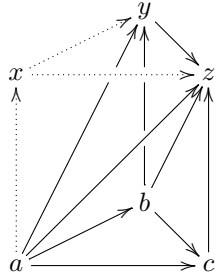
\[
\entrymodifiers={+[o]}
\mymatrix@!0@C=2pc@R=2pc{
& &y \ar[dr] & \\
x \ar@{.>}[rrr] \ar@{.>}[rru] &&  \ar[u] & z \\
&&  \ar@{-}[u] \\
&& b \ar[dr] \ar@{-}[u] \ar[ruu] & \\
a \ar[rrr] \ar[rrruuu] \ar@{.>}[uuu] \ar[uuuurr]|!{[uuu];[uuur]}\hole \ar[urr] &&& c \ar[uuu]
}
\]
Completing with the doted tetrahedron $[a,x,y,z]$ we built above, we at last get the required map $\phi \colon \Delta^2 \times \Delta^1 \to \Cc$.
To prove that the underlying morphism of $\Cc^{\Delta^2}$ is a Cartesian morphism, we have to see that for every commutative diagram
\[
\mymatrix{
\Delta^{\{n-1,n\}} \times \Delta^2 \ar[d] \ar[dr]^\phi \\
\displaystyle \Delta^n \times \Delta^1 \coprod_{\Lambda^n_n \times \Delta^1} \Lambda^n_n \times \Delta^2 \ar[r] \ar[d] & \Cc \\
\Delta^n \times \Delta^2
}
\]
there exists a lift $\Delta^n \times \Delta^2 \to \Cc$.
Let $A$ denote the sub-simplicial set of 
\[
\Delta^n \times \Delta^1 \coprod_{\Lambda^n_n \times \Delta^1} \Lambda^n_n \times \Delta^2
\]
defined by cutting out the vertex $x$. Let $B$ denote the sub-simplicial set of $\Delta^n \times \Delta^2$ defined by cutting out the vertex $x$. We get a commutative diagram
\[
\mymatrix{
\Delta^1 \star \Lambda^2_2 \ar[r] \ar[d] & \Delta^{\{n-1,n\}} \times \Delta^2 \ar[d] \ar[dr]^\phi \\
A \ar[r] \ar[d] & \displaystyle \Delta^n \times \Delta^1 \coprod_{\Lambda^n_n \times \Delta^1} \Lambda^n_n \times \Delta^2 \ar[r] \ar[d] & \Cc \\
B \ar[r] & \Delta^n \times \Delta^2
}
\]
Let also $E$ be the sub-simplicial set of $A$ defined by cutting out $\Lambda^2_2$ and $F$ the sub-simplicial set of $B$ obtained by cutting out $\Lambda^2_2$.
We now have $A \simeq E \star \Lambda^2_2$ and $B \simeq F \star \Lambda^2_2$ and a commutative diagram
\[
\mymatrix{
\Delta^1 \ar[dr] \ar[d] \\ E \ar[r] \ar[d] & \Cc_{/g} \\ F
}
\]
The map $E \to F$ is surjective on vertices. Adding cell after cell using the finality of $\bar b$, we build a lift $F \to \Cc_{/g}$.
We therefore have a lift $B \to \Cc$. Using now that fact that $\Cc$ is a quasi-category, this lifts again to a suitable map $\Delta^n \times \Delta^2 \to \Cc$
\end{proof}

Let $D$ be a filtered poset, which we see as a $1$-category. Let us define $\comma{D}{D^\triangleright}$ the category whose set of objects is the disjoint union of the set of objects and the set of morphisms of $D$ -- ie the set of pairs $x \leq y$. For any object $x \in D$, we will denote by $x \leq \infty$ the corresponding object of $\comma{D}{D^\triangleright}$.
A morphism $(a \colon x \leq y) \to (b \colon z \leq t)$ in $\comma{D}{D^\triangleright}$ is by definition a commutative square in $D$
\[
\mymatrix{
x \ar[r]_a \ar[d] & y \ar[d] \\ z \ar[r]^b & t
}
\]
which therefore corresponds to inequalities $x \leq z$ and $y \leq t$.
A morphism $(x \leq y) \to (z \leq \infty)$ or $(x \leq \infty) \to (z \leq \infty)$ is an inequality $x \leq z$ in $D$. There are no morphisms $(x \leq \infty) \to (z \leq t)$.
The functor
\[
\theta \colon \app{D}{\comma{D}{D^\triangleright}}{x}{(x \leq \infty)}
\]
is fully faithful.
Using Quillen's theorem A and the fact that $D$ is filtered (so that its nerve is contractible), we see $\theta$ is cofinal.
There is also a fully faithful functor
\[
D^{\Delta^1} \to \comma{D}{D^\triangleright}
\]
Let $L$ be the nerve of the category $\comma{D}{D^\triangleright}$ and $K$ the nerve of $D$.
For any object $x \in D$ we also define $K_x \subset K^{\Delta^1}$ to be the nerve of the full subcategory of $D^{\Delta^1}$ spanned by the objects $y \leq z$ where $y \leq x$.

\begin{lem}[Lurie]\label{colim-decomp}
Let $\Cc$ be an $\infty$-category. Let $\phi \colon K^{\Delta^1} \to \Cc$ be a diagram.
For any vertex $k \in K$, let $\phi_k$ denote a colimit diagram for the induced map
\[
K_k \to K^{\Delta^1} \to^\phi \Cc
\]
Then the diagram $\phi$ factors through some map $\kappa$
\[
K^{\Delta^1} \to L \to^\kappa \Cc
\]
such that 
\begin{enumerate}
\item The induced functor $\Cc_{\kappa/} \to \Cc_{\phi/}$ is a trivial fibration.
\item For any vertex $k \in K$, the induced map $(\comma{k}{K})^{\triangleright} \to L \to \Cc$ is a colimit diagram.
\end{enumerate}
\end{lem}

\begin{rmq}
The above lemma can be informally stated as an equivalence
\[
\colim_{k \to l} \phi(k \to l) \simeq \colim_{k \in K} \colim_{l \in \comma{k}{K}} \phi(k \to l)
\]
where for any $k \to k'$, the induced morphism $\colim_{l \in \comma{k}{K}} \phi(k \to l) \to \colim_{l \in \comma{k'}{K}} \phi(k' \to l)$ is given by
\[
\colim_{l \in \comma{k}{K}} \phi(k \to l) \from^\sim \colim_{l \in \comma{k'}{K}} \phi(k \to k' \to l) \to \colim_{l \in \comma{k'}{K}} \phi(k' \to l)
\]
\end{rmq}

\begin{proof}
The existence of the diagram and the first item follows from \cite[4.2.3.4]{lurie:htt} applied to the functor
\[
\app{D}{\quot{\sSets}{K}}{x}{K_x}
\]
For the second item, we simply observe that the inclusion $\comma{x}{K} \to K_x$ is cofinal.
\end{proof}

\begin{prop}\label{comma-quot-ind}
Let $\Cc$ be a $\mathbb V$-small $\infty$-category with finite colimits.
There is a natural equivalence
\[
\Indextu V_{\Cc^{\Delta^1}}(\btw_\Cc^{\amalg}) \simeq \btw^{\amalg}_{\Indu V (\Cc)^{\Delta^1}}
\]
It induces an equivalence
\[
\Indextu U_{\Cc^{\Delta^1}}(\btw_\Cc^{\amalg}) \simeq \btw^{\amalg}_{\Indu U (\Cc)^{\Delta^1}}
\]
\end{prop}
\begin{rmq}\label{between-pro}
There is a "pro" counterpart of \autoref{comma-quot-ind}. If $\Cc$ is an $\infty$-category which admits all pullbacks then 
\[
\Proextu U_{(\Cc\op)^{\Delta^1}}(\btw_{\Cc\op}^{\amalg}) \simeq \btw_{\Prou U(\Cc\op)}^{\amalg}
\]
\end{rmq}

\begin{rmq}
We can state informally \autoref{comma-quot-ind} as follows. For any morphism $f \colon x \to y$ in $\Indu V(\Cc)$ and any diagram $\bar f \colon K \times \Delta^1 \to \Cc$ whose colimit is $f$, the canonical functor
\[
\Indu V\left( \colim_k \quot{\comma{\bar x(k)}{\Cc}}{\bar y(k)} \right) \to \quot{\comma{x}{\Indu V(\Cc)}}{y}
\]
is an equivalence -- where $\bar x = \bar f(-,0)$ and $\bar y = \bar f(-,1)$. The proof is based on the following informal computation:
\begin{align*}
\Indu V\left( \colim_k \quot{\comma{\bar x(k)}{\Cc}}{\bar y(k)} \right) &\simeq \Indu V\left( \colim_{k \to l} \quot{\comma{\bar x(k)}{\Cc}}{\bar y(l)} \right) \simeq \Indu V\left( \colim_k \colim_{l \in \comma{k}{K}} \quot{\comma{\bar x(k)}{\Cc}}{\bar y(l)} \right) \\
&\simeq \lim_k \comma{\bar x(k)}{\Indu V\left(\colim_{l \in \comma{k}{K}} \quot{\Cc}{y(l)}\right)} \simeq \lim_k \comma{\bar x(k)}{\quot{\Indu V(\Cc)}{y}} \simeq \quot{\comma{x}{\Indu V(\Cc)}}{y}
\end{align*}
\end{rmq}

\begin{proof}
Let us deal with the case of $\Indu V$. The case of $\Indu U$ is very similar.
Let us consider the pullback category
\[
\mymatrix{
\comma{\Indu V(\Cc)}{\quot{\Cc}{\Indu V(\Cc)}}\ar[r]^-\psi \ar[d] \cart & \Indu V(\Cc)^{\Delta^2} \ar[d]^{p,q} \\
\Indu V(\Cc)^{\Delta^1} \times \Cc \ar[r] & \Indu V(\Cc)^{\Delta^1} \times \Indu V(\Cc)
}
\]
where $p$ is as in \autoref{bicomma} and $q$ is induces by the inclusion $\{1\} \to \Delta^2$.
The induced map 
\[
\comma{\Indu V(\Cc)}{\quot{\Cc}{\Indu V(\Cc)}} \to \Indu V(\Cc)^{\Delta^1}
\]
is a cocartesian fibration \todo{Détailler ?}classified by the extension of $\btw^{\amalg}_\Cc$
\[
\tilde \btw^{\amalg}_\Cc \colon \Indu V\left( \Cc^{\Delta^1} \right) \simeq \Indu V(\Cc)^{\Delta^1} \to \inftyCatu V
\]
The map $\psi$ therefore induces a natural transformation $\tilde \btw^{\amalg}_\Cc \to \btw^{\amalg}_{\Indu V (\Cc)^{\Delta^1}}$.
This naturally extends to the required transformation
\[
T \colon \Indextu V(\btw^{\amalg}_\Cc) \simeq \Indu V \circ \left(\tilde \btw^{\amalg}_\Cc\right) \to \btw^{\amalg}_{\Indu V (\Cc)^{\Delta^1}}
\]
Let now $f \colon c \to d$ be a morphism in $\Indu V(\Cc)$.
Let $K$ be a $\mathbb V$-small filtered simplicial set and let $\bar f \colon \K \to \Cc^{\Delta^1}$ such that $f$ is a colimit of $\bar f$ in
\[
\Indu V\left( \Cc^{\Delta^1} \right)
\]
Let $\phi \colon K \times \Delta^1 \to \Cc$ be induced by $\bar f$. Let us denote by $j \colon \Cc \to \Indu V(\Cc)$ the Yoneda embedding.
Let $\bar p$ denote a colimit diagram $(K \times \Delta^1)^{\triangleright} \to \Indu V(\Cc)$ extending $i \circ \phi$. The inclusion $K \simeq K \times \{1\} \to K \times \Delta^1$ is cofinal and the cone point of $\bar p$ is thus equivalent to $d$.
The restriction of $\bar p$ to $K^\triangleright \simeq (K \times \{0\})^\triangleright$ defines a diagram $\bar c \colon K \to \quot{\Indu V(\Cc)}{d}$ whose colimit is $f$.
Let us denote by $\tilde c$ the composite diagram
\[
\tilde c \colon \mymatrix@1{K \ar[r]^-{\bar c} & \quot{\Indu V(\Cc)}{d} \ar[r] & \Indu V(\Cc)^{\Delta^1} \simeq \Indu V \left( \Cc^{\Delta^1} \right)}
\]
It comes with a natural transformation $\alpha \colon \bar f  \to \tilde c$ induced by $\bar p$. Let us record for further use that the diagram $\bar c$ factors through
\[
\quot{\Cc}{d} = \Indu V(\Cc)^{\Delta^1} \timesunder[\Indu V(\Cc) \times \Indu V(\Cc)][][][][15pt] (\Cc \times \{d\})
\]
We now consider the map
\[
\gamma \colon \mymatrix@1{K^{\Delta^1} \times \Delta^1 \ar[r]^-{\ev,\pr} & K \times \Delta^1 \ar[r]^-\phi & \Cc}
\]
and denote $\bar g$ the induced map $K^{\Delta^1} \to \Cc^{\Delta^1}$. Note that the composition
\[
\mymatrix{ K \ar[r]^{\id_-} & K^{\Delta^1} \ar[r]^{\bar g} & \Cc^{\Delta^1}}
\]
equals $\bar f$.
We define the functor
\[
\bar h \colon \mymatrix@1{K^{\Delta^1} \ar[r]^-{\bar g} & \Cc^{\Delta^1} \ar[r]^-{\btw_\Cc^{\amalg}} & \inftyCat^{\mathbb V\mathrm{,fc}} \ar[r]^-{\Indu V} & \PresLeftu V}
\]
where $\inftyCat^{\mathbb V\mathrm{,fc}}$ is the category of $\mathbb V$-small $(\infty,1)$-categories with all finite colimits.
We can assume that $K$ is the nerve of a filtered $1$-category $D$.
Using \autoref{colim-decomp} (and its notations) we get a diagram $\kappa \colon L \to \PresLeftu V$ such that we have categorical equivalences
\[
\left( \PresLeftu V \right)_{\kappa \circ \theta/} \simeq \left( \PresLeftu V \right)_{\kappa/} \simeq \left( \PresLeftu V \right)_{\bar h/} \simeq \left( \PresLeftu V \right)_{\bar h \circ \id_-/}
\]
The natural transformation $\alpha$ defined above induces an object of
\[
\left( \PresLeftu V \right)_{\bar h \circ \id_-/} \simeq \left( \PresLeftu V \right)_{\kappa \circ \theta/}
\]
It defines a natural transformation of functors $K \to \PresLeftu V$
\[
A \colon \kappa \circ \theta \to \Indu V \circ \undercat^{\amalg}_{\quot{\Cc}{d}}(\bar c)
\]
Let $k$ be a vertex of $K$. Using \autoref{quot-ind}, we deduce that the functor $A_k$ is an equivalence and the natural transformation $A$ is thus an equivalence too. Now using \autoref{comma-ind}, we see that $T(f)$ is equivalent to the colimit of the diagram induced by $A$
\[
K \to \left(\PresLeftu V\right)^{\Delta^1}
\]
It follows that $T$ is an equivalence.
\end{proof}

We will finish this section with one more result. 
Let $\Cc$ be a $\mathbb V$-small $(\infty,1)$-category with finite colimits and $\Dd$ be any $(\infty,1)$-category. Let $g$ be a functor $\Dd \to \comma{\Cc}{\inftyCat^{\mathbb V\mathrm{,L}}}$ and let $\tilde g$ denote the composition of $g$ with the natural functor $\comma{\Cc}{\inftyCat^{\mathbb V\mathrm{,L}}} \to \inftyCatu V$.
We assume that for any object $x \in \Dd$, the category $\tilde g(x)$ admits finite colimits.
Let also $\alpha \colon \overcat_\Dd \to \tilde g$ be a natural transformation.
We consider the diagram
\[
\entrymodifiers={+[o]}
\mymatrix{
\Dd^{\Delta^1} \ar[r]^-{\int \alpha} \ar[d]_{t_\Dd} & \displaystyle \int \tilde g \ar[dl] & \Cc \times \Dd \ar[l]_-F \ar[dll]^\pr \\
\Dd
}
\]
where the map $F$ is induced by $g$.
The functor $F$ admits a relative right adjoint $G$ over $\Dd$ (see \cite[8.3.3]{lurie:halg}).
The source functor $t_\Dd$ admits a section $\id_-$ induced by the map $\Delta^1 \to \pt$.
It induces a functor $h \colon \Dd \to \Cc$
\[
h \colon \mymatrix{
\Dd \ar[r]^-{\id_-} & \Dd^{\Delta^1} \ar[r] & \displaystyle \int \tilde g \ar[r]^-G & \Cc \times \Dd \ar[r] & \Cc
}
\]
We define the same way $H \colon \Indu U(\Dd) \to \Indu U(\Cc)$, using \autoref{ind-right-adjoint}
\[
H \colon \mymatrix{
\Indu U(\Dd) \ar[r]^-{\id_-} & \Indu U(\Dd)^{\Delta^1} \ar[r] & \displaystyle \int \Indextu U_\Dd(\tilde g) \ar[r] & \Indu U(\Cc) \times \Indu U(\Dd) \ar[r] & \Indu U(\Cc)
}
\]
Let also $I \colon \Prou U(\Dd) \to \Prou U(\Cc)$ be defined similarly, but using \autoref{lim-adjoint}:
\[
I \colon \mymatrix@1{\Prou U(\Dd) \ar[r]^-{\id_-} & \Prou U(\Dd)^{\Delta^1} \ar[r] & \displaystyle \int \Proextu U_\Dd(\tilde g) \ar[r] & \Prou U(\Cc) \times \Prou U(\Dd) \ar[r] & \Prou U(\Cc)}
\]
\begin{lem}\label{two-ext}
The two functors $H$ and $\Indu U(h) \colon \Indu U(\Dd) \to \Indu U(\Cc)$ are equivalent.
The functors $I$ and $\Prou U(h) \colon \Prou U(\Dd) \to \Prou U(\Cc)$ are equivalent.
\end{lem}

\begin{rmq}
For an enlightening example of this construction, we invite the reader to look at \autoref{rmq-two-ext} or \autoref{IPcotangent-underlying}.
\end{rmq}

\begin{proof}
Let us deal with the case of $H$ and $\Indu U(h)$, the other one is similar.
We will prove the following sufficient conditions
\begin{enumerate}
\item The restrictions of both $\Indu U(h)$ and $H$ to $\Dd$ are equivalent ;
\item The functor $H$ preserves $\mathbb U$-small filtered colimits.
\end{enumerate}
To prove item \emph{(i)}, we consider the commutative diagram
\[
\mymatrix{
\Cc \times \Dd \ar[r] \ar[d] & \Indu U(\Cc) \times \Indu U(\Dd) \ar[d] \\
\int \tilde g \ar[r] \ar[d] & \int \Indextu U_\Dd(\tilde g) \ar[d] \\
\Dd \ar[r] & \Indu U(\Dd)
}
\]
The sections $\Dd \to \int \tilde g$ and $\Indu U(\Dd) \to \int \Indextu U_\Dd(\tilde g)$ are compatible: the induced diagram commutes
\[
\mymatrix{
\int \tilde g \ar[r] & \int \Indextu U_\Dd(\tilde g) \\
\Dd \ar[u] \ar[r] & \Indu U(\Dd) \ar[u]
}
\]
Moreover the right adjoints $\int \tilde g \to \Cc \times \Dd$ and $\int \Indextu U_\Dd(\tilde g) \to \Indu U(\Cc) \times \Indu U(\Dd)$ are weakly compatible: there is a natural transformation
\[
\mymatrix{
\Cc \times \Dd \ar[r] \ar@{=>}[rd] & \Indu U(\Cc) \times \Indu U(\Dd) \\
\int \tilde g \ar[r] \ar[u] & \int \Indextu U_\Dd(\tilde g) \ar[u]
}
\]
It follows that we have a natural transformation between the functors 
\[
\Dd \to^h \Cc \to \Indu U(\Cc) \text{ and } \Dd \to \Indu U(\Dd) \to^H \Indu V(\Cc)
\]
For any $x \in \Dd$, the induced map $h(x) \to H(x)$ in $\Indu U(\Cc)$ is an equivalence.
This concludes the proof of item \emph{(i)}.

Let us now prove the item \emph{(ii)}.
It suffices to look at a $\mathbb U$-small filtered diagram $\bar x \colon K \to \Dd$. Let $x$ denote a colimit of $\bar x$ in $\Indu U(\Dd)$.
Let us denote by $A$ the natural transformation
\[
A = \Indextu U_\Dd(\alpha) \colon \overcat_{\Indu U(\Dd)} \to \Indextu U_\Dd(\tilde g) = G
\]
between functors $\Indu U(\Dd) \to \inftyCat$.
Let us also denote by $\pi_*$ the right adjoint $G(x) \to \Indu U(\Cc)$.
By definition, we have $H(x) \simeq \pi_* A_x(\id_x)$. The functors $\pi_*$ and $A_x$ preserve $\mathbb U$-small filtered colimits and $H(x)$ is therefore the colimit of the diagram
\[
\bar A \colon K \to^{\bar x} \quot{\Indu U(\Dd)}{x} \to^{A_x} G(x) \to^{\pi_*} \Indu U(\Cc)
\]
We consider the functor
\[
\bar H_x \colon \mymatrix{
K^{\Delta^1} \ar[r]^-{\bar x^{\Delta^1}} & \Dd^{\Delta^1} \ar[r] & \displaystyle \int \tilde g \ar[r] & \Cc \times \Dd \ar[r] & \Cc \ar[r] & \Indu U(\Cc)
}
\]
We can assume that $K$ is the nerve of a filtered $1$-category. Using \autoref{colim-decomp} and its notations, we extend $\bar H_x$ to a map
\[
\zeta \colon L \to \Indu U(\Cc)
\]
and equivalences
\[
\Indu U(\Cc)_{\zeta \circ \theta/} \simeq \Indu U(\Cc)_{\zeta/} \simeq \Indu U(\Cc)_{\bar H_x/} \simeq \Indu U(\Cc)_{\bar H_x \circ \id_-/}
\]
Using the proof of \emph{(i)}, we have a natural transformation $\bar H_x \circ \id_- \to \bar A$. It induces a natural transformation $\zeta \circ \theta \to \bar A$. Using \autoref{colimcats} we see that it is an equivalence.
It follows that $H(x)$ is a colimit of the diagram
\[
K \to K^{\Delta^1} \to^{\bar H_x} \Indu U(\Cc)
\]
which equals $K \to^{\bar x} \Dd \to^h \Cc \to \Indu U(\Cc)$.
We now conclude using item \emph{(i)}.
\end{proof}
\end{chap-cats}

\chapter{Ind-pro-stacks and Tate stacks}\label{chapterIPandTate}%
\begin{chap-indpro}
In this chapter, we will develop the theory of Tate stacks. Those are infinite dimensional geometric objects whose tangent complex is a Tate module.
With the setting we are then able to define symplectic forms on infinite dimensional algebraic stacks.
\section{Ind-pro-stacks}%

Throughout this section, we will denote by $S$ a derived stack over some base field $k$ and by $\dSt_S$ the category of derived stack over the base $S$.

\subsection{Cotangent complex of a pro-stack}
\begin{df}
A pro-stack over $S$ an object of $\Prou U \dSt_S$.
\end{df}

\begin{rmq}
Note that the category $\Prou U \dSt_S$ is equivalent to the category of pro-stacks over $R$ with a morphism to $S$.
\end{rmq}

\begin{df}\label{iperf-df}
Let $\Perf \colon \dSt_S\op \to \inftyCatu U$ denote the functor mapping a stack to its category of perfect complexes.
We will denote by $\IPerf$ the functor\glsadd{iperf}
\[
\IPerf = \Indext_{\dSt_S\op}(\Perf) \colon (\Prou U\dSt_S)\op \to \PresLeft
\]
where $\Indext$ was defined in \autoref{indext}.
Whenever $X$ is a pro-stack, we will call $\IPerf(X)$ the derived category of ind-complexes on $X$. It is $\mathbb U$-presentable.
If $f \colon X \to Y$ is a map of pro-stacks, then the functor
\[
\IPerf(f) \colon \IPerf(Y) \to \IPerf(X)
\]
admits a right adjoint. We will denote $f_\mathbf{I}^* = \IPerf(f)$ and $f^\mathbf{I}_*$ its right adjoint.
\end{df}

\begin{rmq}
Let $X$ be a pro-stack and let $\bar X \colon K\op \to \dSt_S$ denote a $\mathbb U$-small cofiltered diagram of whom $X$ is a limit in $\Prou U\dSt_S$.
The derived category of ind-perfect complexes on $X$ is by definition the category
\[
\IPerf(X) = \Indu U(\colim \Perf(\bar X))
\]
It thus follows from \cite[1.1.4.6 and 1.1.3.6]{lurie:halg} that $\IPerf(X)$ is stable.
Note that it is also equivalent to the colimit
\[
\IPerf(X) = \colim \IPerf(\bar X) \in \PresLeftu V
\]
It is therefore equivalent to the limit of the diagram
\[
\IPerf_*(\bar X) \colon K \to \dSt_S\op \to \PresLeftu V \simeq (\PresRightu V)\op
\]
An object $E$ in $\IPerf(X)$ is therefore the datum of an object ${p_k}_* E$ of $\IPerf(X_k)$ for each $k \in K$ --- 
where $X_k = \bar X(k)$ and $p_k \colon X \to X_k$ is the natural projection --- and of some compatibilities between them.
\end{rmq}

\begin{df}
Let $X$ be a pro-stack. We define its derived category of pro-perfect complexes\glsadd{pperf}
\[
\PPerf(X) = \left( \IPerf(X) \right) \op
\]
The duality $\Perf(-) \to^\sim (\Perf(-))\op$ implies the equivalence
\[
\PPerf(X) \simeq \Prou U(\colim \Perf(\bar X))
\]
whenever $\bar X \colon \K\op \to \dSt_S$ is a cofiltered diagram of whom $X$ is a limit in $\Prou U\dSt_S$.
\end{df}

\begin{df}
Let us define the functor $\Tateu U_\mathbf P \colon (\Prou U \dSt_S)\op \to \inftyCat^{\mathbb V\mathrm{,st,id}}$
\[
\Tateu U_\mathbf P = \Tateextu U_{\dSt_S\op}(\Perf)
\]
\end{df}

\begin{rmq}
The functor $\Tateu U_\mathbf P$ maps a pro-stack $X$ given by a diagram $\bar X \colon K\op \to \dSt_S$ to the stable $(\infty,1)$-category
\[
\Tateu U_\mathbf P(X) = \Tateu U(\colim \Perf(\bar X))
\]
There is a canonical fully faithful natural transformation
\[
\Tateu U_\mathbf P \to \Prou U \circ \IPerf
\]
From \autoref{tate-dual} we also get a fully faithful
\[
\Tateu U_\mathbf P \to \Indu U \circ \PPerf
\]
\end{rmq}

\begin{df}
Let $\Qcoh \colon \dSt_S\op \to \inftyCatu V$ denote the functor mapping a derived stack to its derived category of quasi-coherent sheaves. It maps a morphism between stacks to the appropriate pullback functor.
We will denote by $\IQcoh$ the functor\glsadd{iqcoh}
\[
\IQcoh = \Indext_{\dSt_S\op}(\Qcoh) \colon (\Prou U\dSt_S)\op \to \inftyCatu V
\]
If $f \colon X \to Y$ is a map of pro-stacks, we will denote by $f^*_\mathbf{I}$ the functor $\IQcoh(f)$.
We also define
\[
\IQcoh^{\leq 0} = \Indext_{\dSt_S\op}(\Qcoh^{\leq 0})
\]
the functor of connective modules.
\end{df}

\begin{rmq}
There is a fully faithful natural transformation $\IPerf \to \IQcoh$ ; for any map $f \colon X \to Y$ of pro-stacks, there is therefore a commutative diagram
\[
\mymatrix{
\IPerf(Y) \ar[r] \ar[d]_{f_\mathbf I^*} & \IQcoh(Y) \ar[d]^{f_\mathbf I^*}\\
\IPerf(X) \ar[r] & \IQcoh(X)
}
\]
The two functors denoted by $f_\mathbf I^*$ are thus compatible.
Let us also say that the functor 
\[
f_\mathbf I^* \colon \IQcoh(Y) \to \IQcoh(X)
\]
does \emph{not} need to have a right adjoint. We next show that it sometimes has one.
\end{rmq}
\begin{prop}\label{prop-iqcoh}
Let $f \colon X \to Y$ be a map of pro-stacks. If $Y$ is actually a stack then the functor $f_\mathbf I^* \colon \IQcoh(Y) \to \IQcoh(X)$ admits a right adjoint.
\end{prop}
\begin{proof}
It follows from \autoref{ind-right-adjoint}.
\end{proof}

\newcommand{\IQ}{\mathbf{IQ}}
\begin{df}
Let $f \colon X \to Y$ be a map of pro-stacks. We will denote by $f_*^\IQ$ the right adjoint to $f_\mathbf I^* \colon \IQcoh(Y) \to \IQcoh(X)$ \emph{if it exists}.
\end{df}

\begin{rmq}
In the situation of \autoref{prop-iqcoh}, there is a natural transformation
\[
\mymatrix{
\IPerf(X) \ar[r] \ar@{=>}[dr] & \IQcoh(X) \\ \IPerf(Y) \ar[u]^{f_*^\mathbf I} \ar[r] & \IQcoh(Y) \ar[u]_{f_*^\IQ}
}
\]
It does not need to be an equivalence.
\end{rmq}

\begin{df}
Let $X$ be a pro-stack over $S$. The structural sheaf $\Oo_X$ of $X$ is the pull-back of $\Oo_S$ along the structural map $X \to S$.
\end{df}

\begin{ex}
Let $X$ be a pro-stack over $S$ and $\bar X \colon K\op \to \dSt_S$ be a $\mathbb U$-small cofiltered diagram of whom $X$ is a limit in $\Prou U\dSt_S$.
Let $k$ be a vertex of $K$, let $X_k$ denote $\bar X(k)$ and let $p_k$ denote the induced map of pro-stacks $X \to X_k$. If $f \colon k \to l$ is an arrow in $K$, we will also denote by $f$ the map of stacks $\bar X(f)$.
We have
\[
(p_k)_*^\IQ (\Oo_X) \simeq \colim_{f \colon k \to l} f_* \Oo_{X_l}
\]
One can see this using \autoref{colimcats}
\[
(p_k)_*^\IQ (\Oo_X) \simeq (p_k)_*^\IQ (p_k)_\mathbf I^* (\Oo_{X_k}) \simeq \colim_{f \colon k \to l} f_* f^* (\Oo_{X_k}) \simeq \colim_{f \colon k \to l} f_* \Oo_{X_l}
\]
\end{ex}

\newcommand{\sqzext}{\operatorname{Ex}}
\begin{df}\label{derivation-pdst}
Let $T$ be a stack over $S$. Let us consider the functor 
\[
\Qcoh(T)^{\leq 0} \to \btw_{\dSt_S\op}^{\amalg}(\id_T) \simeq \left(\comma{T}{\dSt_T}\right)\op
\]
mapping a quasi-coherent sheaf $E$ to the square zero extension $T \to T[E] \to T$. This construction is functorial in $T$ and actually comes from a \todo{Reference}natural transformation
\[
\sqzext \colon \Qcoh^{\leq 0} \to \btw_{\dSt_S\op}^{\amalg}(\id_-)
\]
of functors $\dSt_S\op \to \inftyCatu V$ -- recall notation $\btw^{\amalg}$ from \autoref{betweendef}.
We will denote by $\sqzext^{\Pro}$ the natural transformation
\[
\sqzext^{\Pro} = \Indext_{\dSt_S\op}(\sqzext) \colon \IQcoh^{\leq 0} \to \Indextu U_{\dSt_S\op}(\btw_{\dSt_S\op}^{\amalg}(\id_-)) \simeq \btw_{(\Prou 
U\dSt_S)\op}^{\amalg}(\id_-)
\]
between functors $(\Prou U\dSt_S)\op \to \inftyCat$. The equivalence on the right is the one from \autoref{comma-quot-ind}.
If $X$ is a pro-stack and $E \in \IQcoh(X)^{\leq 0}$ then we will denote by $X \to X[E] \to X$ the image of $E$ by the functor $\sqzext^{\Pro}(X)$.
\end{df}
\begin{rmq}
Let us give a description of this functor.
Let $X$ be a pro-stack and let $\bar X \colon K\op \to \dSt_S$ denote a $\mathbb U$-small cofiltered diagram of whom $X$ is a limit in $\Prou U\dSt_S$. For every $k \in K$ we can compose the functor mentioned above with the base change functor
\[
\mymatrix{
(\Qcoh(X_k))\op \ar[r]^-{X_k[-]} & \comma{X_k}{\dSt_{X_k}} \ar[r]^-{- \times_{X_k} X} & \comma{X}{\Prou U\dSt_X}
}
\]
This is functorial in $k$ and we get a functor $\left(\colim \Qcoh(\bar X) \right)\op \to \comma{X}{\Prou U\dSt_X}$ which we extend and obtain a more explicit description of the square zero extension functor
\[
X[-] \colon (\IQcoh(X))\op \to \comma{X}{\Prou U\dSt_X}
\]
\end{rmq}

\begin{df}
Let $X$ be a pro-stack. 
\begin{itemize}
\item We finally define the functor of derivations over $X$ :
\[
\Der(X,-) = \Map_{X/-/S}(X[-],X) \colon \IQcoh(X)^{\leq 0} \to \sSets
\]
\item We say that $X$ admits a cotangent complex if the functor $\Der(X,-)$ is corepresentable -- ie there exists a $\Lcot_{X/S} \in \IQcoh(X)$ such that\glsadd{cotangent}
for any $E \in \IQcoh(X)^{\leq 0}$
\[
\Der(X,E) \simeq \Map(\Lcot_{X/S},E)
\]
\end{itemize}
\end{df}

\begin{df}
Let $\dStArt_S$ denote the full sub-category of $\dSt_S$ spanned by derived Artin stacks over $S$. 
An Artin pro-stack is an object of $\Prou U \dStArt_S$.
Let $\dStArtlfp_S$ the full sub-category of $\dStArt_S$ spanned by derived Artin stacks locally of finite presentation over $S$.
An Artin pro-stack locally of finite presentation is an object of $\Prou U\dStArtlfp_S$
\end{df}

\begin{prop}\label{cotangent-pdst}
Any Artin pro-stack $X$ over $S$ admits a cotangent complex $\Lcot_{X/S}$.
Let us assume that $\bar X \colon K\op \to \dStArt_S$ is a $\mathbb U$-small cofiltered diagram of whom $X$ is a limit in $\Prou U\dStArt_S$. When $k$ is a vertex of $K$, let us denote by $X_k$ the derived Artin stack $\bar X(k)$. If $f \colon k \to l$ is an arrow in $K$, we will also denote by $f \colon X_l \to X_k$ the map of stacks $\bar X(f)$.
The cotangent complex is given by the formula
\[
\Lcot_{X/S} = \colim_k p_k^* \Lcot_{X_k/S} \in \Indu U\left(\colim \Qcoh(\bar X) \right) \simeq \IQcoh(X)
\]
where $p_k$ is the canonical map $X \to X_k$.
The following formula stands
\[
{p_k}_*^\IQ \Lcot_{X/S} \simeq \colim_{f \colon k \to l} f_* \Lcot_{X_l/S}
\]
If $X$ is moreover locally of finite presentation over $S$, then its cotangent complex belongs to $\IPerf(X)$.
\end{prop}
Before proving this proposition, let us fix the following notation
\begin{df}
Let $\Cc$ be a full sub-category of an $\infty$-category $\Dd$. There is a natural transformation from $\overcat_\Dd \colon d \mapsto \quot{\Dd}{d}$ to the constant functor $\Dd \colon \Dd \to \inftyCat$. We denote by $\overcat_\Dd^\Cc$ the fiber product
\[
\overcat_\Dd^\Cc = \overcat_\Dd \times_\Dd \Cc \colon \Dd \to \inftyCat
\]
\end{df}
\begin{rmq}
The functor $\overcat_\Dd^\Cc \colon \Dd \to \inftyCat$ maps an object $d \in \Dd$ to the comma category of objects in $\Cc$ over $d$
\[
\quot{\Cc}{d} = (\Cc \times \{d\} ) \timesunder[\Dd \times \Dd] \Dd^{\Delta^1}
\]
The \autoref{quot-ind} still holds when replacing $\overcat_\Cc$ by $\overcat_\Dd^\Cc$.
\end{rmq}
\begin{proof}[of the proposition]
The cotangent complex defines a natural transformation
\[
\cotangent \colon \overcat_{\dSt_S\op}^{(\dStArt_S)\op}  \to \Qcoh(-)
\]
To any stack $T$ and any Artin stack $U$ over $S$ with a map $f \colon T \to U$, it associates the quasi-coherent complex $f^* \Lcot_{U/S}$ on $T$.
Applying the functor $\Indext_{\dSt_S\op}$ we get a natural transformation $\cotangent^{\Pro}$
\[
\cotangent^{\Pro} = \Indext_{\dSt_S\op}(\cotangent) \colon \overcat^{(\Prou U\dStArt_S)\op}_{(\Prou U\dSt_S)\op} \to \IQcoh(-)
\]
Specifying it to $X$ we get a functor
\[
\cotangent^{\Pro}_X \colon \left( \comma{X}{\Prou U\dStArt_S} \right)\op \to \IQcoh(X)
\]
Let us set $\Lcot_{X/S} =\lambda^{\Pro}_X(X) \in \IQcoh(X)$. We have by definition the equivalence
\[
\Lcot_{X/S} \simeq \colim_k p_k^* \Lcot_{X_k/S}
\]
Let us now check that it satisfies the required universal property.
The functor $\Der(X,-)$ is the limit of the diagram $K\op \to \Fct(\IQcoh(X)^{\leq 0}, \sSets)$
\[
\Map_{X/-/S}(X[-], \bar X)
\]
This diagram factors by definition through a diagram
\[
\delta \colon K\op \to \Fct\left(\colim \Qcoh(\bar X)^{\leq 0}, \sSets\right) \simeq \lim \Fct(\Qcoh(\bar X)^{\leq 0}, \sSets)
\]
On the other hand, the functor $\Map(\Lcot_{X/S},-)$ is the limit of a diagram
\[
\mymatrix{K\op \ar[r]^-\mu & \lim \Fct(\Qcoh(\bar X)^{\leq 0}, \sSets) \ar[r] & \Fct(\IQcoh(X)^{\leq 0}, \sSets)}
\]
The universal property of the natural transformation $\lambda$ defines an equivalence between $\delta$ and $\mu$.
The formula for ${p_k}_*^\IQ \Lcot_{X/S}$ is a direct consequence of \autoref{two-ext} and the last statement is obvious.
\end{proof}

\begin{rmq}[about \autoref{two-ext}]\label{rmq-two-ext}
There are two ways of constructing the underlying complex of the cotangent complex of a pro-stack. One could first consider the functor
\[
\Lcot^1 \colon {\dStArt_S}\op \to \Qcoh(S)
\]
mapping a derived Artin stack $\pi \colon Y \to S$ to the quasi-coherent module $\pi_* \Lcot_{Y/S}$ and extend it
\[
\Indu U(\Lcot^1) \colon {\Prou U\dStArt_S}\op \to \Indu U\Qcoh(S) = \IQcoh(S)
\]
The second method consists in building the cotangent complex of a pro-stack $\varpi \colon X \to S$ as above 
\[
\Lcot_{X/S} \in \IQcoh(X)
\]
and considering $\varpi_*^\IQ \Lcot_{X/S} \in \IQcoh(S)$. This defines a functor
\[
\Lcot^2 \colon \app{{\Prou U\dStArt_S}\op}{\IQcoh(S)}{(X \to^\varpi S)}{\varpi_*^\IQ \Lcot_{X/S}}
\]
Comparing those two approaches is precisely the role of \autoref{two-ext}. It shows indeed that the functors $\Indu U(\Lcot^1)$ and $\Lcot^2$ are equivalent.
\end{rmq}

\begin{rmq}
The definition of the derived category of ind-quasi-coherent modules on a pro-stack is build for the above proposition and remark to hold.
\end{rmq}

\begin{rmq}\label{cotangent-prostacks}
We have actually proven that for any pro-stack $X$, the two functors 
\[
\IQcoh(X)^{\leq 0} \times \comma{X}{\dStArt_S} \to \sSets
\]
defined by
\begin{align*}
(E,Y) &\mapsto \Map_{X/-/S}(X[E], Y) \\
(E,Y) &\mapsto \Map_{\IQcoh(X)}(\lambda_X^{\Pro}(Y), E)
\end{align*}
are equivalent.
\end{rmq}

\subsection{Cotangent complex of an ind-pro-stack}

\begin{df}
An ind-pro-stack is an object of the category
\[
\IP\dSt_S = \Indu U \Prou U\dSt_S
\]
\end{df}

\begin{df}
Let us define the functor $\PIPerf \colon (\IP\dSt_S)\op \to \inftyCatu V$ as\glsadd{piperf}
\[
\PIPerf = \Proext_{(\Prou U\dSt_S)\op}(\IPerf)
\]
where $\Proext$ was defined in \autoref{indext}.
Whenever we have a morphism $f \colon X \to Y$ of ind-pro-stacks, we will denote by $f^*_\PI$ the functor
\[f^*_\PI = \PIPerf(f) \colon \PIPerf(Y) \to \PIPerf(X)
\]
\end{df}

\begin{rmq}
Let $X$ be an ind-pro-stack. Let $\bar X \colon K \to \Prou U\dSt_S$ denote a $\mathbb U$-small filtered diagram of whom $X$ is a colimit in $\IP\dSt_S$.
We have by definition
\[
\PIPerf(X) \simeq \lim \Prou U(\IPerf(\bar X))
\]
 admits a right adjoint $f_*^{\PI}$. It is the pro-extension of the right adjoint $f_*^{\mathbf I}$ to $f^*_{\mathbf I}$.
This result extends to any map $f$ of ind-pro-stacks since the limit of adjunctions is still an adjunction.
\end{rmq}

\begin{prop}\label{prop-piperf-right-adjoint}
Let $f \colon X \to Y$ be a map of ind-pro-stacks. If $Y$ is a pro-stack then the functor $f^*_\PI \colon \PIPerf(Y) \to \PIPerf(X)$ admits a right adjoint.
\end{prop}

\begin{df}
Let $f \colon X \to Y$ be a map of ind-pro-stacks. If the functor
\[
f^*_\PI \colon \PIPerf(Y) \to \PIPerf(X)
\]
admits a right adjoint, we will denote it by $f^\PI_*$.
\end{df}

\begin{proof}[of the proposition]
If both $X$ and $Y$ are pro-stacks, then $f^\PI_* = \Prou U(f^\mathbf I_*)$ is right adjoint to $f_\PI^* = \Prou U(f_\mathbf I^*)$.
Let now $X$ be an ind-pro-stack and let $\bar X \colon K \to \Prou U\dSt_S$ denote a $\mathbb U$-small filtered diagram of whom $X$ is a colimit in $\IP\dSt_S$. We then have
\[
f_\PI^* \colon \PIPerf(Y) \to \PIPerf(X) \simeq \lim \PIPerf(\bar X) 
\]
The existence of a right adjoint $f^\PI_*$ then follows from \autoref{lim-adjoint}.
\end{proof}

\begin{df}
Let $X \in \IP\dSt_S$. We define $\IPPerf(X) = (\PIPerf(X))\op$.
If $X$ is the colimit in $\IP\dSt_S$ of a filtered diagram $K \to \Prou U\dSt_S$ then we have
\[
\IPPerf(X) \simeq \lim (\Indu U \circ \PPerf \circ \bar X)
\]
There is therefore a fully faithful functor $\Tateu U_\IP(X) \to \IPPerf(X)$.
We will denote by
\[
\dual{(-)} \colon \IPPerf(X) \to (\PIPerf(X))\op
\]
the duality functor.
\end{df}

\begin{df}
Let us define the functor $\Tateu U_\IP \colon (\IP\dSt_S)\op \to \inftyCat^{\mathbb V\mathrm{,st,id}}$ as the right Kan extension of $\Tateu U_\mathbf P$ along the inclusion $(\Prou U \dSt_S)\op \to (\IP\dSt_S)\op$.
It is by definition endowed with a canonical fully faithful natural transformation
\[
\Tateu U_\IP \to \PIPerf
\]
For any $X \in \IP\dSt_S$, an object of $\Tateu U_\IP(X)$ will be called a Tate module on $X$.
\end{df}

\begin{rmq}
We can characterise Tate objects: 
a module $E \in \PIPerf(X)$ is a Tate module if and only if for any pro-stack $U$ and any morphism $f \colon U \to X \in \IP\dSt_S$, the pullback $f^*_\IP(E)$ is in $\Tateu U_\mathbf P(U)$.

Let us also remark here that 
\end{rmq}

\begin{lem}\label{ipdst-tate-in-ipp}
Let $X$ be an ind-pro-stack over $S$. The fully faithful functors
\[
\mymatrix{
\Tateu U_\IP(X) \ar[r] & \PIPerf(X) \ar@{=}[r]^-{\dual{(-)}}& (\IPPerf(X))\op & \left(\Tateu U_\IP(X)\right)\op \ar[l]
}
\]
have the same essential image. We thus have an equivalence
\[
\dual{(-)} \colon \Tateu U_\IP(X) \simeq \left(\Tateu U_\IP(X)\right)\op
\]
\end{lem}
\begin{proof}
This is a corollary of \autoref{tate-dual}.
\end{proof}

\begin{df}
Let us define $\PIQcoh \colon (\IP\dSt_S)\op \to \inftyCatu V$ to be the functor\glsadd{piqcoh}
\[
\PIQcoh = \Proext_{(\Prou U\dSt_S)\op}(\IQcoh)
\]
From \autoref{indext-monoidal}, for any ind-pro-stack $X$, the category $\PIQcoh(X)$ admits a natural monoidal structure.
We also define the subfunctor
\[
\PIQcoh^{\leq 0} = \Proext_{(\Prou U\dSt_S)\op}(\IQcoh^{\leq 0})
\]
\end{df}

\begin{rmq}
Let us give an informal description of the above definition. To an ind-pro-stack $X = \colim_\alpha \lim_\beta X_{\alpha\beta}$ we associate the category
\[
\PIQcoh(X) = \lim_\alpha \Prou U \Indu U\left(\colim_\beta \Perf(X_{\alpha\beta})\right)
\]
\end{rmq}

\begin{df}
Let $f \colon X \to Y$ be a map of ind-pro-stacks. We will denote by $f_\PI^*$ the functor $\PIQcoh(f)$. Whenever it exists, we will denote by $f^\PIQ_*$ the right adjoint to $f_\PI^*$.
\end{df}

\begin{prop}\label{prop-piqcoh-right-adjoint}
Let $f \colon X \to Y$ be a map of ind-pro-stacks. If $Y$ is actually a stack, then the induced functor $f^*_\PI$ admits a right adjoint.
\end{prop}
\begin{proof}
This is very similar to the proof of \autoref{prop-piperf-right-adjoint} but using \autoref{prop-iqcoh}.
\end{proof}

\begin{rmq}
There is a fully faithful natural transformation $\PIPerf \to \PIQcoh$. Using the same notation $f_\PI^*$ for the images of a map $f \colon X \to Y$ is therefore only a small abuse.
Moreover, for any such map $f \colon X \to Y$, for which the right adjoints drawn below exist, there is a natural tranformation
\[
\mymatrix{
\PIPerf(Y) \ar[r] \ar[d]_{f^\PI_*} & \PIQcoh(Y) \ar[d]^{f^\PIQ_*} \\ \PIPerf(X) \ar[r] \ar@{=>}[ur] & \PIQcoh(X)
}
\]
It is generally not an equivalence.
\end{rmq}

\begin{df}
Let $\sqzext^{\IP}$ denote the natural transformation $\Proext_{(\Prou U\dSt_S)\op}(\sqzext^{\Pro})$
\[
\sqzext^{\IP} \colon \PIQcoh^{\leq 0} \to \Proext_{(\Prou U\dSt_S)\op} \left( \btw_{(\Prou U\dSt_S)\op}^{\amalg}(\id_-) \right) \simeq \btw_{(\IP\dSt_S)\op}^{\amalg}(\id_-)
\]
of functors $(\IP\dSt_S)\op \to \inftyCat$. The equivalence on the right hand side is the one of \autoref{between-pro}.
If $X$ is an ind-pro-stack and $E \in \PIQcoh(X)^{\leq 0}$ then we will denote by $X \to X[E] \to X$ the image of $E$ by the functor 
\[
\sqzext^{\IP}(X) \colon \PIQcoh(X)^{\leq 0} \to \left(\comma{X}{\IP\dSt_X}\right)\op
\]
\end{df}

\begin{rmq}
Let us decipher the above definition. Let $X = \colim_\alpha \lim_\beta X_{\alpha\beta}$ be an ind-pro-stack and let $E$ be a pro-ind-module over it. By definition $E$ is the datum, for every $\alpha$, of a pro-ind-object $E^\alpha$ in the category $\colim_\beta \Qcoh^{\leq 0}(X_{\alpha\beta})$.
Let us denote $E^\alpha = \lim_\gamma \colim_\delta E^\alpha_{\gamma\delta}$.
For any $\gamma$ and $\delta$, there is a $\beta_0(\gamma,\delta)$ such that $E^\alpha_{\gamma\delta}$ is in the essential image of $\Qcoh^{\leq 0}ccccccvc                                                                                                                   (X_{\alpha\beta_0(\gamma,\delta)})$.
We then have
\[
X[E] = \colim_{\alpha,\gamma} \lim_{\delta} \lim_{\beta \geq \beta_0(\gamma,\delta)} X_{\alpha\beta}[E_{\gamma\delta}] \in \IP\dSt_S
\]
\end{rmq}

\begin{df}\label{derivation-ipdst}
Let $X$ be an ind-pro-stack.
\begin{itemize}
\item We define the functor of derivations on $X$ 
\[
\Der(X,-) = \Map_{X/-/S}(X[-],X)
\]
\item We say that $X$ admits a cotangent complex if there exists $\Lcot_{X/S} \in \PIQcoh(X)$ such that for any $E \in \PIQcoh(X)^{\leq 0}$\glsadd{cotangent}
\[
\Der(X,E) \simeq \Map(\Lcot_{X/S},E)
\]
\item Let us assume that $f \colon X \to Y$ is a map of ind-pro-stacks and that $Y$ admits a cotangent complex. We say that $f$ is formally étale if $X$ admits a cotangent complex and the natural map $f^* \Lcot_{Y/S} \to \Lcot_{X/S}$ is an equivalence.
\end{itemize}
\end{df}

\begin{df}
An Artin ind-pro-stack over $S$ is an object in the category
\[
\IP\dStArt_S = \Indu U \Prou U\dStArt_S
\]
An Artin ind-pro-stack locally of finite presentation is an object of
\[
\IP\dStArtlfp_S = \Indu U \Prou U\dStArtlfp_S
\]
\end{df}

\begin{prop}\label{ipcotangent}
Any Artin ind-pro-stack $X$ admits a cotangent complex
\[
\Lcot_{X/S} \in \PIQcoh(X)
\]
Let us assume that $\bar X \colon K \to \Pro\dStArt_S$ is a $\mathbb U$-small filtered diagram of whom $X$ is a colimit in $\IP\dStArt_S$.
For any vertex $k \in K$ we will denote by $X_k$ the pro-stack $\bar X(k)$ and by $i_k$ the structural map $X_k \to X$.
For any $f \colon k \to l$ in $K$, let us also denote by $f$ the induced map $X_k \to X_l$.
We have for all $k \in K$
\[
i_{k,\PI}^* \Lcot_{X/S} \simeq \lim_{f \colon k \to l} f^*_\mathbf I \Lcot_{X_l/S} \in \PIQcoh(X_k)
\]
If moreover $X$ is locally of finite presentation then $\Lcot_{X/S}$ belongs to $\PIPerf(X)$.
\end{prop}

\begin{proof}
Let us recall the natural transformation $\cotangent^{\Pro}$ from the proof of \autoref{cotangent-pdst}
\[
\cotangent^{\Pro} = \Indext_{\dSt_S\op}(\cotangent) \colon \overcat^{(\Prou U\dStArt_S)\op}_{(\Prou U\dSt_S)\op} \to \IQcoh(-)
\]
of functors $(\Prou U\dSt_S)\op \to \inftyCat$. Applying the functor $\Proext_{(\Prou U\dSt_S)\op}$ we define the natural transformation $\cotangent^{\IP}$
\[
\cotangent^{\IP} = \Proext_{(\Prou U\dSt_S)\op} \left( \cotangent^{\Pro} \right) \colon \overcat_{(\IP\dSt_S)\op}^{(\IP\dStArt_S)\op} \to \PIQcoh(-)
\]
between functors $(\IP\dSt_S)\op \to \inftyCat$.
Specifying to $X$ we get a functor
\[
\cotangent^{\IP}_X \colon \left(\comma{X}{\IP\dStArt_S}\right) \op \to \PIQcoh(X)
\]
We now define $\Lcot_{X/S} = \cotangent^{\IP}_X(X)$.
By definition we have
\[
i_{k,\PI}^* \Lcot_{X/S} \simeq \lim \cotangent^{\Pro}_{X_k}(\bar X) \simeq \lim_{f \colon k \to l} f^*_\mathbf I \Lcot_{X_l/S}
\]
for every $k \in K$.
Let us now prove that it satisfies the expected universal property. It suffices to compare for every $k \in K$ the functors
\[
\Map_{X_k/-/S}(X_k[-], X) \hspace{1cm} \text{and} \hspace{1cm} \Map_{\PIQcoh(X_k)}(i_{k,\PI}^*\Lcot_{X/S}, -)
\]
defined on $\PIQcoh(X_k)^{\leq 0}$. They are both pro-extensions to $\PIQcoh(X_k)^{\leq 0}$ of their restrictions $\IQcoh(X_k)^{\leq 0} \to \sSets$.
The restricted functor $\Map_{X_k/-/S}(X_k[-], X)$ is a colimit of the diagram
\[
\Map_{X_k/-/S}(X_k[-], \bar X) \colon \left(\comma{k}{K}\right)\op \to \Fct(\IQcoh(X_k)^{\leq 0}, \sSets)
\]
while $\Map_{\PIQcoh(X_k)}(i_{k,\PI}^*\Lcot_{X/S}, -)$ is a colimit to the diagram
\[
\Map_{\IQcoh(X_k)}(\cotangent^{\Pro}_{X_k}(\bar X), -) \colon \left(\comma{k}{K}\right)\op \to \Fct(\IQcoh(X_k)^{\leq 0},\sSets)
\]
We finish the proof with \autoref{cotangent-prostacks}.
\end{proof}

\begin{prop}\label{IPcotangent-underlying}
Let $X \in \IP\dStArt_S$.
Let us denote by $\pi \colon X \to S$ the structural map.
Let also $\tilde \Lcot^\IP$ denote the functor
\[
\left(\IP\dStArt_S\right)\op \to \Prou U \Indu U \Qcoh(S)
\]
obtained by extending the functor $(\dStArt_S)\op \to \Qcoh(S)$ mapping $f \colon T \to S$ to $f_* \Lcot_{T/S}$.
Then we have $\pi_*^\PIQ \Lcot_{X/S} \simeq \tilde \Lcot^{\IP} (X)$
\end{prop}

\begin{proof}
The existence of $\pi_*^\PIQ$ is deduced from \autoref{prop-piqcoh-right-adjoint}. The result then follows by applying \autoref{two-ext} twice.
\end{proof}

\begin{df}
Let $X$ by an Artin ind-pro-stack locally of finite presentation over $S$. We will call the tangent complex of $X$ the ind-pro-perfect complex on $X$
\[
\T_{X/S} = \dual \Lcot_{X/S} \in \IPPerf(X)
\]
\end{df}

\subsection{Uniqueness of pro-structure}

\begin{lem}\label{coconnective}
Let $Y$ and $Z$ be derived Artin stacks. The following is true
\begin{enumerate}
\item The canonical map 
\[
\Map(Z,Y) \to \lim_n \Map(\tau_{\leq n} Z,Y)
\]
is an equivalence;
\item If $Y$ is $q$-Artin and $Z$ is $m$-truncated then the mapping space $\Map(Z,Y)$ is $(m + q)$-truncated.
\end{enumerate}
\end{lem}

\begin{proof}
We prove both items recursively on the Artin degree of $Z$. 
The case of $Z$ affine is proved in \cite[C.0.10 and 2.2.4.6]{toen:hagii}. We assume that the result is true for $n$-Artin stacks. Let $Z$ be $(n+1)$-Artin. There is an atlas $u \colon U \to Z$.
Let us remark that for $k \in\N$ the truncation $\tau_{\leq k} u \colon \tau_{\leq k} U \to \tau_{\leq k} Z$ is also a smooth atlas --- indeed we have $\tau_{\leq k} U \simeq U \times_Z \tau_{\leq k} Z$. Let us denote by $U_\bullet$ the nerve of $u$ and by $\tau_{\leq k}U_\bullet$ the nerve of $\tau_{\leq k} u$. Because $k$-truncated stacks are stable by flat pullbacks, the groupoid $\tau_{\leq k}U_\bullet$ is equivalent to $\tau_{\leq k}(U_\bullet)$. We have
\[
\Map(Z,Y) \simeq \lim_{[p] \in \Delta} \Map(U_p,Y) \simeq \lim_{[p] \in \Delta} \lim_k \Map(\tau_{\leq k}U_p,Y) \simeq \lim_k \Map(\tau_{\leq k} Z,Y)
\]
That proves item \emph{(i)}.
If moreover $Z$ is $m$-truncated, then we can replace $U$ by $\tau_{\leq m} U$. If follows that $\Map(Z,Y)$ is a limit of $(m + q)$-truncated spaces. This finishes the proof of \emph{(ii)}.
\end{proof}

We will use this well known lemma:
\begin{lem}\label{finite-groupoids}
\todo{Reference}Let $S \colon \Delta \to \sSets$ be a cosimplicial object in simplicial sets. Let us assume that for any $[p] \in \Delta$ the simplicial set $S_p$ is $n$-coconnective. Then the natural morphism
\[
\lim_{[p] \in \Delta} S_p \to \underset{p \leq n+1}{\lim_{[p] \in \Delta}} S_p
\]
is an equivalence.
\end{lem}

\begin{lem}\label{truncatedcompact}
Let $\bar X \colon \N\op \to \dSt_S$ be a diagram such that
\begin{enumerate}
\item There exists $m \in \N$ and $n \in \N$ such that for any $k \in K$ the stack $\bar X(k)$ is $n$-Artin, $m$-truncated and of finite presentation;
\item There exists a diagram $\bar u \colon \N \times \Delta^1 \to \dSt_S$ such that the restriction of $\bar u$ to $\N \times \{1\}$ is equivalent to $\bar X$, every map $\bar u(k) \colon \bar u(k)(0) \to \bar u(k)(1) \simeq \bar X(k)$ is a smooth atlas and the limit $\lim_k \bar u(k)$ is an epimorphism.
\label{truncatedcompact:atlas}
\end{enumerate}
If $Y$ is an algebraic derived stack of finite presentation then the canonical morphism
\[
\colim \Map\left( \bar X,Y\right) \to \Map\left(\lim \bar X, Y \right)
\]
is an equivalence.
\end{lem}

\begin{proof}
Let us prove the statement recursively on the Artin degree $n$. If $n$ equals $0$, this is a simple reformulation of the finite presentation of $Y$.
Let us assume that the statement at hand is true for some $n$ and let $\bar X(0)$ be $(n+1)$-Artin.
Considering the nerves of the epimorphisms $\bar u(k)$, we get a diagram
\[
\bar Z \colon \N\op \times \Delta\op \to \dSt_S
\]
Note that $\bar Z$ has values in $n$-Artin stacks.
The limit $\lim_k \bar u(k)$ is also an atlas and the natural map
\[
\colim_{[p] \in \Delta} \lim_{k \in \N} \bar Z(k)_p \to \lim_{k \in \N} \colim_{[p] \in \Delta} \bar Z(k)_p \simeq \lim \bar X
\]
is therefore an equivalence.
We now write
\begin{align*}
\Map\left(\lim \bar X,Y \right)
&\simeq \Map\left(\colim_{[p] \in \Delta} \lim_{k \in \N} \bar Z(k)_p,Y \right) \\
&\simeq \lim_{[p] \in \Delta} \Map\left(\lim_{k \in \N} \bar Z(k)_p,Y \right) \\
&\simeq \lim_{[p] \in \Delta} \colim_{k \in \N} \Map\left(\bar Z(k)_p,Y \right)
\end{align*}
We also have
\[
\colim \Map\left( \bar X,Y \right) \simeq \colim_{k \in \N} \lim_{[p] \in \Delta} \Map\left( \bar Z(k)_p, Y \right)
\]
It thus suffices to prove that the canonical morphism of simplicial sets
\[
\colim_{k \in \N} \lim_{[p] \in \Delta} \Map\left( \bar Z(k)_p, Y \right) \to \lim_{[p] \in \Delta} \colim_{k \in \N} \Map\left(\bar Z(k)_p,Y \right)
\]
is an equivalence. Let us notice that each $\bar Z(k)_p$ is $m$-truncated. It is indeed a fibre product of $m$-truncated derived stacks along flat maps.
Let $q$ be an integer such that $Y$ is $q$-Artin. The simplicial set $\Map(\bar Z(k)_p,Y)$ is then $(m + q)$-coconnective (\autoref{coconnective}).
It follows from \autoref{finite-groupoids} that the limit at hand is in fact finite and we have the required equivalence.
\end{proof}

\begin{lem}\label{exact-seq}
Let $\bar M \colon \N\op \to \sSets$ be a diagram. For any $i \in \N$ and any point $x = (x_n) \in \lim \bar M$, we have the following exact sequence
\[
\mymatrix{
0 \ar[r] & \displaystyle {\lim_n}^1 \pi_{i+1}(\bar M(n),x_n) \ar[r] & \displaystyle \pi_i\left(\lim_n \bar M(n), x\right) \ar[r] & \displaystyle \lim_n \pi_i(\bar M(n), x_n) \ar[r] & 0
}
\]
\end{lem}
A proof of that lemma can be found for instance in \cite{hirschhorn:limfibrations}.

\begin{lem}\label{colim-lim-swap}
Let $M \colon \N\op \times K \to \sSets$ denote a diagram, where $K$ is a filtered simplicial set. If for any $i \in \N$ there exists $N_i$ such that for any $n \geq N_i$ and any $k \in K$ the induced morphism $M(n,k) \to M(n-1,k)$ is an $i$-equivalence then the canonical map
\[
\phi \colon \colim_{k \in K} \lim_{n \in \N} M(n,k) \to \lim_{n \in \N} \colim_{k \in K} M(n,k)
\]
is an equivalence. We recall that an $i$-equivalence of simplicial sets is a morphism which induces isomorphisms on the homotopy groups of dimension lower or equal to $i$.
\end{lem}

\begin{proof}
We can assume that $K$ admits an initial object $k_0$.
Let us write $M_{nk}$ instead of $M(n,k)$. Let us fix $i \in \N$. If $i \geq 1$, we also fix a base point $x \in \lim_n M_{nk_0}$. Every homotopy group below is computed at $x$ or at the natural point induced by $x$. We will omit the reference to the base point. We have a morphism of short exact sequences
\[
\mymatrix{
0 \ar[r] &
\displaystyle \colim_k {\lim_{n}}^1 \pi_{i+1}(M_{nk}) \ar[r] \ar[d] &
\displaystyle \colim_k \pi_{i}\left(\lim_{n} M_{nk}\right) \ar[r] \ar[d] &
\displaystyle \colim_k \lim_{n} \pi_i(M_{nk}) \ar[d] \ar[r] & 0 \\
0 \ar[r] &
\displaystyle {\lim_{n}}^1 \colim_k \pi_{i+1}(M_{nk}) \ar[r] &
\displaystyle \pi_{i}\left(\lim_{n} \colim_k M_{nk}\right) \ar[r] &
\displaystyle \lim_{n} \colim_k \pi_i(M_{nk}) \ar[r] & 0 \\
}
\]
We can restrict every limit to $n \geq N_{i+1}$. Using the assumption we see that the limits on the right hand side are then constant and so are the $1$-limits on the left.  If follows that the vertical maps on the sides are isomorphisms, and so is the middle map.
This begin true for any $i$, we conclude that $\phi$ is an equivalence.
\end{proof}

\begin{df}\label{shy}
Let $\bar X \colon \N\op \to \dSt_S$ be a diagram. We say that $\bar X$ is a shy diagram if
\begin{enumerate}
\item For any $k \in \N$ the stack $\bar X(k)$ is algebraic and of finite presentation;
\item For any $k \in \N$ the map $\bar X(k \to k+1) \colon \bar X(k+1) \to \bar X(k)$ is affine;
\item The stack $\bar X(0)$ is of finite cohomological dimension.
\end{enumerate}
If $X$ is the limit of $\bar X$ in the category of prostacks, we will also say that $\bar X$ is a shy diagram for $X$.
\end{df}

\begin{prop}\label{prop-cocompact}
Let $\bar X \colon \N\op \to \dSt_S$ be a shy diagram.
If $Y$ is an algebraic derived stack of finite presentation then the canonical morphism
\[
\colim \Map\left( \bar X,Y\right) \to \Map\left(\lim \bar X, Y \right)
\]
is an equivalence.
\end{prop}

\begin{proof}
Since for any $n$, the truncation functor $\tau_{\leq n}$ preserves shy diagrams, let us use \autoref{coconnective} and \autoref{truncatedcompact}
\begin{align*}
\Map(\lim \bar X,Y) \simeq & \lim_n \Map(\tau_{\leq n}(\lim \bar X),Y)\\ & \simeq \lim_n \Map(\lim \tau_{\leq n}\bar X,Y) \simeq \lim_n \colim \Map(\tau_{\leq n} \bar X,Y)
\end{align*}
On the other hand we have 
\[
\colim \Map(\bar X,Y) \simeq \colim \lim_n \Map(\tau_{\leq n} \bar X,Y)
\]
and we are to study the canonical map 
\[
\phi \colon \colim \lim_n \Map(\tau_{\leq n} \bar X,Y) \to \lim_n \colim \Map(\tau_{\leq n} \bar X,Y)
\]
Because of \autoref{colim-lim-swap}, it suffices to prove the assertion
\begin{enumerate}[label=(\arabic*)]
\item For any $i \in \N$ there exists $N_i \in \N$ such that for any $n \geq N_i$ and any $k \in \N$ the map
\[
p_{n,k} \colon \Map\left( \tau_{\leq n} \bar X(k), Y \right) \to \Map\left( \tau_{\leq n -1} \bar X(k), Y \right)
\]
induces an equivalence on the $\pi_j$'s for any $j \leq i$.
\end{enumerate}
For any map $f \colon \tau_{\leq n-1} \bar X(k) \to Y$ we will denote by $F_{n,k}(f)$ the fibre of $p_{n,k}$ at $f$. We have to prove that for any such $f$ the simplicial set $F_{n,k}(f)$ is $i$-connective.
Let thus $f$ be one of those maps. The derived stack $\tau_{\leq n} \bar X(k)$ is a square zero extension of $\tau_{\leq n-1} \bar X(k)$ by a module $M[n]$, where
\[
M = \ker \left(\Oo_{\tau_{\leq n} \bar X(k)} \to \Oo_{\tau_{\leq n-1} \bar X(k)} \right) [-n]
\]
Note that $M$ is concentrated in degree $0$. It follows from the obstruction theory of $Y$ --see \autoref{obstruction} -- that $F_{n,k}(f)$ is not empty if and only if the obstruction class 

\[
\alpha(f) \in G_{n,k}(f) = \Map_{\Oo_{\tau_{\leq n-1} \bar X(k)}}(f^* \Lcot_Y, M[n+1])
\]
of $f$ vanishes.
Moreover, if $\alpha(f)$ vanishes, then we have an equivalence
\[
F_{n,k}(f) \simeq \Map_{\Oo_{\tau_{\leq n-1} \bar X(k)}}(f^* \Lcot_Y, M[n])
\]
Using assumptions \emph{(iii)} and \emph{(ii)} we have that $\bar X(k)$ --- and therefore its truncation too --- is of finite cohomological dimension $d$.
Let us denote by $[a,b]$ the Tor-amplitude of $\Lcot_Y$.
We get that $G_{n,k}(f)$ is $(s+1)$-connective for $s = a + n - d$ and that $F_{n,k}(f)$ is $s$-connective if $\alpha(f)$ vanishes.
%
Let us remark here that $d$ and $a$ do not depend on either $k$ or $f$ and thus neither does $N_i = i + d - a$ (we set $N_i = 0$ if this quantity is negative). For any $n \geq N_i$ and any $f$ as above, the simplicial set $G_{n,k}(f)$ is at least $1$-connective. The obstruction class $\alpha(f)$ therefore vanishes and $F_{n,k}(f)$ is indeed $i$-connective. This proves (1) and concludes this proof.
\end{proof}

\newcommand{\shy}{\mathbf P\dSt^{\mathrm{shy}}}
\begin{df}\label{prochamps}
Let $\shy_S$ denote the full subcategory of $\Prou U\dSt_S$ spanned by the prostacks which admit shy diagrams.
Every object $X$ in $\shy_S$ is thus the limit of a shy diagram $\bar X \colon \N\op \to \dSt_S$.

We will say that $X$ is of cotangent tor-amplitude in $[a,b]$ if there exists a shy diagram $\bar X \colon \N\op \to \dSt_S$ for $X$ such that every cotangent $\Lcot_{\bar X(n)}$ is of tor-amplitude in $[a,b]$. We will also say that $X$ is of cohomological dimension at most $d$ if there is a shy diagram $\bar X$ with values in derived stacks of cohomological dimension at most $d$.
The pro-stack $X$ will be called $q$-Artin if there is a shy diagram for it, with values in $q$-Artin derived stacks.
Let us denote by $\Cc^{[a,b]}_{d,q}$ the full subcategory of $\shy_S$ spanned by objects of cotangent tor-amplitude in $[a,b]$, of cohomological dimension at most $d$ and $q$-Artin.
\end{df}

\begin{thm}
The limit functor $i_\mathrm{shy} \colon \shy_S \to \dSt_S$ is fully faithful and has values in Artin stacks.
\end{thm}

\begin{proof}
This follows directly from \autoref{prop-cocompact}.
\end{proof}

\begin{df}
A map of pro-stacks $f \colon X \to Y$ if an open immersion if there exists a diagram
\[
\bar f \colon \N\op \times \Delta^1 \to \dSt_k
\]
such that 
\begin{itemize}
\item The limit of $\bar f$ in maps of pro-stacks is $f$;
\item The restriction $\N\op \times \{0\} \to \dSt_k$ of $\bar f$ is a shy diagram for $X$ and the restriction $\N\op \times \{1\} \to \dSt_k$ is a shy diagram for $Y$;
\item For any $n$, the induced map of stacks $\{n\} \times \Delta^1 \to \dSt_k$ is an open immersion.
\end{itemize}
\end{df}

\subsection{Uniqueness of ind-pro-structures}
\label{unique-ipstructure}

\begin{df}\label{indprochamps}
Let $\shybounded_S$ denote the full subcategory of $\Indu U(\shy_S)$ spanned by colimits of $\mathbb U$-small filtered diagrams $K \to \shy_S$  which factors through $\Cc^{[a,b]}_{d,q}$ for some 4-uplet $a,b,d,q$.\glsadd{shybounded}
For any $X \in \shybounded_S$ we will say that $X$ is of cotangent tor-amplitude in $[a,b]$ and of cohomological dimension at most $d$ if it is the colimit (in $\Indu U(\shy_S)$) of a diagram $K \to \Cc^{[a,b]}_{d,q}$.
\end{df}

\begin{thm}\label{ff-realisation}
The colimit functor $\Indu U(\shy_S) \to \dSt_S$ restricts to a full faithful embedding $\shybounded_S \to \dSt_S$.
\end{thm}

\begin{lem}\label{iequi}
Let $a,b,d,q$ be integers with $a \leq b$. Let $T \in \shy_S$ and $\bar X \colon K \to \Cc^{[a,b]}_{d,q}$ be a $\mathbb U$-small filtered diagram. For any $i \in \N$ there exists $N_i$ such that for any $n \geq N_i$ and any $k \in K$, the induced map
\[
\Map(\tau_{\leq n}T, \bar X(k)) \to \Map(\tau_{\leq n-1} T, \bar X(k))
\]
is an $i$-equivalence.
\end{lem}

\begin{rmq}
For the proof of this lemma, we actually do not need the integer $q$.
\end{rmq}

\begin{proof}
Let us fix $i \in \N$. Let $k \in K$ and $\bar T \colon \N \to \dSt_S$ be a shy diagram for $T$.
We observe here that $\tau_{\leq n} \bar T$ is a shy diagram whose limit is $\tau_{\leq n} T$. Let also $\bar Y_k \colon \N \to \dSt_S$ be a shy diagram for $\bar X(k)$.
The map at hand 
\[
\psi_{nk} \colon \Map(\tau_{\leq n}T, \bar X(k)) \to \Map(\tau_{\leq n-1} T, \bar X(k))
\]
is then the limit of the colimits 
\[
\lim_{p \in \N} \colim_{q \in \N} \Map(\tau_{\leq n} \bar T(q), \bar Y_k(p)) \to 
\lim_{p \in \N} \colim_{q \in \N} \Map(\tau_{\leq n-1} \bar T(q), \bar Y_k(p))
\]
Let now $f$ be a map $\tau_{\leq n-1} T \to \bar X(k)$. It corresponds to a family of morphisms
\[
f_p \colon \pt \to \colim_{q \in \N} \Map(\tau_{\leq n-1} \bar T(q), \bar Y_k(p))
\]
Moreover, the fibre $F_{nk}(f)$ of $\psi_{nk}$ over $f$ is the limit of the fibres $F_{nk}^p(f)$ of the maps
\[
\psi_{nk}^p \colon \colim_{q \in \N} \Map(\tau_{\leq n} \bar T(q), \bar Y_k(p)) \to 
\colim_{q \in \N} \Map(\tau_{\leq n-1} \bar T(q), \bar Y_k(p))
\]
over the points $f_p$.
Using the exact sequence of \autoref{exact-seq}, it suffices to prove that $F_{nk}^p(f)$ is $(i+1)$-connective for any $f$ and any $p$.
For such an $f$ and such a $p$, there exists $q_0 \in \N$ such that the map $f_p$ factors through the canonical map
\[
\Map(\tau_{\leq n-1} \bar T(q_0), \bar Y_k(p)) \to \colim_{q \in \N} \Map(\tau_{\leq n-1} \bar T(q), \bar Y_k(p))
\]
We deduce that $F_{nk}^p(f)$ is equivalent to the colimit
\[
F_{nk}^p(f) \simeq \colim_{q \geq q_0} G_{nk}^{pq}(f)
\]
where $G_{nk}^{pq}(f)$ is the fibre at the point induced by $f_p$ of the map
\[
\Map(\tau_{\leq n} \bar T(q), \bar Y_k(p)) \to \Map(\tau_{\leq n-1} \bar T(q), \bar Y_k(p))
\]
The interval $[a,b]$ contains the tor-amplitude of $\Lcot_{\bar Y_k(p)}$ and $d$ is an integer greater than the cohomological dimension of $\bar T(q)$. We saw in the proof of \autoref{prop-cocompact} that $G_{nk}^{pq}(f)$ is then $(a + n - d)$-connective.
We set $N_i = i + d - a +1$.
\end{proof}

\begin{proof}[of \autoref{ff-realisation}]
We will prove the sufficient following assertions
\begin{enumerate}[label=(\arabic*)]
\item The colimit functor $\Indu U(\shy_S) \to \presh(\dAff_S)$ restricts to a fully faithful functor
\[
\eta \colon \shybounded_S \to \presh(\dAff_S)
\]
\item The functor $\eta$ has values in the full subcategory of stacks.
\end{enumerate}
Let us focus on assertion (1) first. We consider two $\mathbb U$-small filtered diagrams $\bar X \colon K \to \shy_S$ and $\bar Y \colon L \to \shy_S$.
We have
\[
\Map_{\Indu U(\shy_S)}\left(\colim \bar X, \colim \bar Y\right) \simeq \lim_k \Map_{\Indu U(\shy_S)}(\bar X(k), \colim \bar Y)
\]
and
\[
\Map_{\presh(\dAff)}\left( \colim i_\mathrm{shy} \bar X, \colim i_\mathrm{shy} \bar Y \right) \simeq \lim_k \Map_{\presh(\dAff)}\left( i_\mathrm{shy} \bar X(k), \colim i_\mathrm{shy} \bar Y \right)
\]
We can thus replace the diagram $\bar X$ in $\shy_S$ by a simple object $X \in \shy_S$. We now assume that $\bar Y$ factors through $\Cc^{[a,b]}_{d,q}$ for some $a,b,d,q$. We have to prove that the following canonical morphism is an equivalence
\[
\phi \colon \colim_{l \in L} \Map(i_\mathrm{shy} X,i_\mathrm{shy} \bar Y(l)) \to \Map\left(i_\mathrm{shy} X, \colim i_\mathrm{shy} \bar Y\right)
\]
where the mapping spaces are computed in prestacks.
If $i_\mathrm{shy} X$ is affine then $\phi$ is an equivalence because colimits in $\presh(\dAff_S)$ are computed pointwise. Let us assume that $\phi$ is an equivalence whenever $i_\mathrm{shy} X$ is $(q-1)$-Artin and let us assume that $i_\mathrm{shy} X$ is $q$-Artin.
Let $u \colon U \to i_\mathrm{shy} X$ be an atlas of $i_\mathrm{shy} X$ and let $Z_\bullet$ be the nerve of $u$ in $\dSt_S$. We saw in the proof of \autoref{truncatedcompact} that $Z_\bullet$ factors through $\shy_S$.
The map $\phi$ is now equivalent to the natural map
\begin{align*}
\colim_{l \in L} \Map(i_\mathrm{shy} X,i_\mathrm{shy} \bar Y(l)) \to \lim_{[p]\in \Delta} & \colim_{l \in L} \Map(Z_p,i_\mathrm{shy} \bar Y(l)) \\ & \simeq \lim_{[p]\in \Delta} \Map\left(Z_p,\colim i_\mathrm{shy} \bar Y\right) \simeq \Map(i_\mathrm{shy} X,\colim i_\mathrm{shy} \bar Y)
\end{align*}
Remembering \autoref{coconnective}, it suffices to study the map
\[
\colim_{l \in L} \lim_n \Map(\tau_{\leq n} i_\mathrm{shy} X,i_\mathrm{shy} \bar Y(l)) \to \lim_{[p]\in \Delta} \colim_{l \in L} \lim_n \Map(\tau_{\leq n}Z_p,i_\mathrm{shy} \bar Y(l))
\]
Applying \autoref{iequi} and then \autoref{colim-lim-swap}, we see that $\phi$ is an equivalence if the natural morphism
\[
\lim_n \colim_{l \in L} \lim_{[p] \in \Delta} \Map(\tau_{\leq n}Z_p,i_\mathrm{shy} \bar Y(l)) \to \lim_n \lim_{[p] \in \Delta} \colim_{l \in L} \Map(\tau_{\leq n}Z_p, i_\mathrm{shy} \bar Y(l))
\]
is an equivalence. The stack $i_\mathrm{shy} \bar Y(l)$ is by assumption $q$-Artin, where $q$ does not depend on $l$. Now using \autoref{coconnective} and \autoref{finite-groupoids}, we conclude that $\phi$ is an equivalence. This proves (1). We now focus on assertion (2).
If suffices to see that the colimit in $\presh(\dAff_S)$ of the diagram $i_\mathrm{shy} \bar Y$ as above is actually a stack. Let $H_\bullet \colon \Delta\op \cup \{-1\} \to \dAff_S$ be an hypercovering of an affine $\Spec(A) = H_{-1}$. We have to prove the following equivalence
\[
\colim_l \lim_{[p] \in \Delta} \Map(H_p,i_\mathrm{shy} \bar Y(l)) \to \lim_{[p] \in \Delta} \colim_l \Map(H_p,i_\mathrm{shy} \bar Y(l))
\]
Using the same arguments as for the proof of (1), we have
\begin{align*}
\colim_l \lim_{[p] \in \Delta} \Map(H_p,i_\mathrm{shy} \bar Y(l))
&\simeq \colim_l \lim_{[p] \in \Delta} \lim_n \Map(\tau_{\leq n}H_p,i_\mathrm{shy} \bar Y(l)) \\
&\simeq \lim_n \colim_l \lim_{[p] \in \Delta} \Map(\tau_{\leq n}H_p,i_\mathrm{shy} \bar Y(l)) \\
&\simeq \lim_n \lim_{[p] \in \Delta} \colim_l \Map(\tau_{\leq n}H_p,i_\mathrm{shy} \bar Y(l)) \\
&\simeq \lim_{[p] \in \Delta} \colim_l \lim_n \Map(\tau_{\leq n}H_p,i_\mathrm{shy} \bar Y(l)) \\
&\simeq \lim_{[p] \in \Delta} \colim_l \Map(H_p,i_\mathrm{shy} \bar Y(l))
\end{align*}
\end{proof}

We will need one last lemma about that category $\shybounded_S$.
\begin{lem}\label{shydaff-limits}
The fully faithful functor $\shybounded_S \cap \IP\dAff_S \to \IP\dSt_S \to \dSt_S$ preserves finite limits.
\end{lem}

\begin{proof}
The case of an empty limit is obvious. Let then $X \to Y \from Z$ be a diagram in $\shybounded_S \cap \IP\dAff_S$. There exist $a$ and $b$ and a diagram 
\[
\sigma \colon K \to \Fct\left( \Lambda^2_1, \Cc^{[a,b]}_{0,0} \right)
\]
such that $K$ is a $\mathbb U$-small filtered simplicial set and the colimit in $\IP\dSt_S$ is $X \to Y \from Z$. We can moreover assume that $\sigma$ has values in $\Fct(\Lambda^2_1, \Prou U(\dAff_S)) \simeq \Prou U(\Fct(\Lambda^2_1, \dAff_S))$.
We deduce that the fibre product $X \times_Y Z$ is the realisation of the ind-pro-diagram in derived affine stacks with cotangent complex of tor amplitude in $[a-1,b+1]$. It follows that $X \times_Y Z$ is again in $\shybounded_S \cap \IP\dAff_S$. 
\end{proof}

\section{Symplectic Tate stacks}%
\subsection{Tate stacks: definition and first properties}
\label{tatestacks}We can now define what a Tate stack is.
\begin{df}
A Tate stack is a derived Artin ind-pro-stack locally of finite presentation whose cotangent complex -- see \autoref{ipcotangent} -- is a Tate module.
Equivalently, an Artin ind-pro-stack locally of finite presentation is Tate if its tangent complex is a Tate module.
 We will denote by $\Tatestack_k$ the full subcategory of $\IP\dSt_k$ spanned by Tate stacks.\glsadd{tatestack}
\end{df}

This notion has several good properties. For instance, using \autoref{ipdst-tate-in-ipp}, if a $X$ is a Tate stack then comparing its tangent $\T_X$ and its cotangent $\Lcot_X$ makes sense, in the category of Tate modules over $X$.
We will explore that path below, defining symplectic Tate stacks.

Another consequence of Tatity\footnote{or Tateness or Tatitude} is the existence of a determinantal anomaly as defined in \cite{kapranovvasserot:loop2}. If $X$ is a Tate stack, then using the determinant map $\mathrm K \to \B\mathbb G_m$ -- see \cite{toen:derivedK3} for its formal definition -- we get a class $[\mathrm{det}_X] \in \homol^2(\Oo_X^{\times})$ : the determinantal anomaly -- see \autoref{determinantalanomaly}.

Let us consider the natural morphism of prestacks
\[
\theta \colon \Tateu U \to \mathrm K^{\Tate}
\]
where $\Tateu U$ denote the prestack $A \mapsto \Tateu U(\Perf(A))$ and $\mathrm K^{\Tate} \colon A \mapsto \K(\Tateu U(\Perf(A)))$ -- $\mathrm K$ denoting the connective $K$-theory functor.
From \autoref{ktheorysusp} we get an exact sequence
\[
\mymatrix{
\B \mathrm K \ar[r] & \mathrm K^{\Tate} \ar[r] & \mathrm K_0^{\Tate}
}
\]
\begin{lem}
The prestack $\mathrm K_0^{\Tate}$ vanishes Nisnevich-locally. It follows that the map $\tilde \theta$, obtained from $\theta$ by stackifying both ends, factors through the stack associated to $\B \mathrm K$.
\end{lem}

\begin{proof}
Is suffices to prove that for any Henselian cdga $A$, we have
\[
\mathrm K_{-1}(\Perf(A) \simeq \mathrm K_0(\Tateu U(\Perf(A))) \simeq 0
\]
A cdga $A$ is Henselian if and only if $\homol^0(A)$ is. Using the Bass exact sequences, we get
\[
\mymatrix{
\mathrm K_0(A[t]) \oplus K_0(A[t^{-1}]) \ar[r] \ar[d]^f & \mathrm K_0(A[t,t^{-1}]) \ar[r] \ar[d]^g & \mathrm K_{-1}(A) \ar[r] \ar[d]^h & 0 \\
\mathrm K_0(\homol^0(A)[t]) \oplus K_0(\homol^0(A)[t^{-1}]) \ar[r] & \mathrm K_0(\homol^0(A)[t,t^{-1}]) \ar[r] & \mathrm K_{-1}(\homol^0(A)) \ar[r] & 0
}
\]
Since $\mathrm K_0$ only depends on the non-derived part of an affine scheme (see \cite[2.3.2]{waldhausen:ktheory}), both $f$ and $g$ are isomorphisms and hence so is $h$. We can thus restrict to the non-derived case -- which can be found in \cite[theorem 3.7]{drinfeld:tate}.
\end{proof}

\begin{df}
We define the Tate determinantal map as the composite map
\[
\Tateu U \to \mathcal T \to \B \mathrm K \to \mathrm K(\Gm,2)
\]
where $\mathcal T$ is the stack associated to $\Tateu U$ and $\mathrm K(\Gm,2)$ is the Eilenberg-Maclane classifying stack.
To any derived stack $X$ with a Tate module $E$, we associate the determinantal anomaly $[\det_E] \in \homol^2(X,\Oo_X^{\times})$, image of $E$ by the morphism
\[
\Map(X,\Tateu U) \to \Map(X,\mathrm K(\Gm,2))
\]
\end{df}

\begin{rmq}
We built here some determinant map for Tate objects. Those kind of determinant maps already appeared in \cite{osipovzhu:categorical}.
\end{rmq}

Let now $X$ be an ind-pro-stack. Let also $R$ denote the realisation functor $\Prou U\dSt_k \to \dSt_k$. Let finally $\bar X \colon K \to \Prou U\dSt_k$ denote a $\mathbb U$-small filtered diagram whose colimit in $\IP\dSt_k$ is $X$. We have a canonical functor
\[
F_X \colon \lim \Tateu U_\mathbf P(\bar X) \simeq \Tateu U_\IP(X) \to \lim \Tateu U(R \bar X)
\]
\begin{df}\label{determinantalanomaly}
Let $X$ be an ind-pro-stack and $E$ be a Tate module on $X$. Let $X'$ be the realisation of $X$ in $\Indu U\dSt_k$ and $X''$ be its image in $\dSt_k$.
We define the determinantal anomaly of $E$ the image of $F_X(E)$ by the map
\[
\mymatrix{
\Map_{\Indu U\dSt_k}(X',\mathcal T) \to \Map_{\Indu U\dSt_k}(X',\mathrm K(\Gm,2)) \simeq \Map_{\dSt_k}(X'',\mathrm K(\Gm,2))
}
\]
In particular if $X$ is a Tate stack, we will denote by $[\det_X] \in \homol^2(X'',\Oo_{X''}^{\times})$ the determinantal anomaly associated to its tangent $\T_X \in \Tateu U_\IP(X)$.
\end{df}

Let us conclude this section with following
\begin{lem}\label{tate-limits}
The inclusion $\Tatestack_k \to \IP\dSt_k$ preserves finite limits.
\end{lem}

\begin{proof}
Let us first notice that a finite limit of Artin ind-pro-stacks is again an Artin ind-pro-stack. Let now $X \to Y \from Z$ be a diagram of Tate stacks.
The fibre product
\[
\mymatrix{
X \times_Y Z \cart \ar[d]_{p_Z} \ar[r]^-{p_X} & X \ar[d]^g \\ Z \ar[r] & Y
}
\]
is an Artin ind-pro-stack. It thus suffices to test if its tangent $\T_{X \times_Y Z}$ is a Tate module. The following cartesian square concludes
\[
\mymatrix{
\T_{X \times_Y Z} \cart \ar[r] \ar[d] & p_X^* \T_X \ar[d] \\ p_Z^* \T_Z \ar[r] & p_X^* g^* \T_Y
}
\]
\end{proof}

\subsection{Shifted symplectic Tate stacks}

\newcommand{\deRham}{\operatorname{\mathbf{DR}}}
\newcommand{\NCw}{\operatorname{\mathrm{NC^w}}}
\newcommand{\globalcomplex}{\operatorname{C}}

We assume now that the basis $S$ is the spectrum of a ring $k$ of characteristic zero.
Recall from \cite{ptvv:dersymp} the stack in graded complexes $\deRham$ mapping a cdga over $k$ to its graded complex of forms.
It actually comes with a mixed structure induced by the de Rham differential.
The authors also defined there the stack in graded complexes $\NCw$ mapping a cdga to its graded complex of closed forms.
Those two stacks are linked by a morphism $\NCw \to \deRham$ forgetting the closure.

We will denote by $\forms p, \closedforms p \colon \cdga_k \to \dgMod_k$ the complexes of weight $p$ in $\deRham[-p]$ and $\NCw[-p]$ respectively.
The stack $\forms p$ will therefore map a cdga to its complexes of $p$-forms while $\closedforms p$ will map it to its closed $p$-forms.
For any cdga $A$, a cocycle of degree $n$ of $\forms p (A)$ is an $n$-shifted $p$-forms on $\Spec A$. 
The functors $\closedforms p$ and $\forms p$ extend to functors
\[
\closedforms p,~ \forms p \colon \dSt_k\op \to \dgMod_k
\]\glsadd{forms}
\begin{df}
Let us denote by $\IPforms p$ and $\IPclosedforms p$ the extensions
\[
(\IP\dSt_k)\op \to \Prou U\Indu U\dgMod_k
\]
of $\forms p$ and $\closedforms p$, respectively.
They come with a natural projection $\IPclosedforms p \to \IPforms p$.

Let $X \in \IP\dSt_k$. An $n$-shifted (closed) $p$-form on $X$ is a morphism $k[-n] \to \IPforms p(X)$ (resp. $\IPclosedforms p(X)$). \glsadd{ipforms}
For any closed form $\omega \colon k[-n] \to \IPclosedforms p(X)$, the induced map $k[-n] \to \IPclosedforms p(X) \to \IPforms p(X)$ is called the underlying form of $\omega$.
\end{df}

\begin{rmq}
In the above definition, we associate to any ind-pro-stack $X = \colim_\alpha \lim_\beta X_{\alpha\beta}$ its complex of forms
\[
\IPforms p(X) = \lim_\alpha \colim_\beta \forms p(X_{\alpha\beta}) \in \Prou U \Indu U \dgMod_k
\]
\end{rmq}

For any ind-pro-stack $X$, the derived category $\PIQcoh(X)$ is endowed with a canonical monoidal structure. In particular, one defines a symmetric product $E \mapsto \Sym^2_\PI(E)$ as well as an antisymmetric product
\[
E \wedge_\PI E = \Sym_\PI^2(E[-1])[2]
\]

\begin{thm}\label{prop-formsareforms}
Let $X$ be an Artin ind-pro-stack over $k$.
The push-forward functor
\[
\pi^\PIQ_* \colon \PIQcoh(X) \to \Prou V \Indu V(\dgMod_k)
\]
exists (see \autoref{prop-piqcoh-right-adjoint}) and maps $\Lcot_X \wedge_\PI \Lcot_X$ to $\IPforms 2(X)$.
In particular, any $2$-form $k[-n] \to \IPforms 2(X)$ corresponds to a morphism $\Oo_X[-n] \to \Lcot_X \wedge_\PI \Lcot_X$ in $\PIQcoh(X)$.
\end{thm}
\begin{proof}
This follows from \cite[1.14]{ptvv:dersymp}, from \autoref{IPcotangent-underlying} and from the equivalence
\[
\cotangent^\IP \wedge_\PI \cotangent^\IP = \Proextu U \Indextu U (\lambda) \wedge_\PI \Proextu U \Indextu U(\lambda) \simeq \Proextu U \Indextu U (\cotangent \wedge \cotangent)
\]
where $\cotangent^\IP$ is defined in the proof of \autoref{ipcotangent}.
\end{proof}

\begin{df}
Let $X$ be a Tate stack.
Let $\omega \colon k[-n] \to \IPforms 2(X)$ be an $n$-shifted $2$-form on $X$.
It induces a map in the category of Tate modules on $X$
\[
\underline \omega \colon \T_X \to \Lcot_X[n]
\]
We say that $\omega$ is non-degenerate if the map $\underline \omega$ is an equivalence.
A closed $2$-form is non-degenerate if the underlying form is.
\end{df}

\begin{df}\index{Symplectic Tate stack}
A symplectic form on a Tate stack is a non-degenerate closed $2$-form. A symplectic Tate stack is a Tate stack equipped with a symplectic form.
\end{df}

\subsection{Mapping stacks admit closed forms}

\newcommand{\EV}{\operatorname{EV}}

In this section, we will extend the proof from \cite{ptvv:dersymp} to ind-pro-stacks. Note that if $X$ is a pro-ind-stack and $Y$ is a stack, then $\Mapstack(X,Y)$ is an ind-pro-stack.
We will then need an evaluation functor $\Mapstack(X,Y) \times X \to Y$.
It appears that this evaluation map only lives in the category of ind-pro-ind-pro-stacks
\[
\colim_\alpha \lim_{\beta} \colim_{\xi} \lim_{\zeta} \Mapstack(X_{\alpha\zeta},Y) \times X_{\beta\xi} \to Y
\]

To build this map properly, we will need the following remark.
\begin{df}
Let $\Cc$ be a category.
There is one natural fully faithful functor
\[
\phi \colon \PI(\Cc) \to (\IP)^2(\Cc)
\]
but three $\IP(\Cc) \to (\IP)^2(\Cc)$. We will only consider the functor
\[
\psi \colon \IP(\Cc) \to (\IP)^2(\Cc)
\]
induced by the Yoneda embedding $\Pro(\Cc) \to \PI (\Pro (\Cc))$.
Let us also denote by $\xi$ the natural fully faithful functor $\Cc \to (\IP)^2 (\Cc)$.
\end{df}

We can now construct the required evaluation map. We will work for now on a more general basis. Let therefore $X$ be a pro-ind-stack over a stack $S$. Let also $Y$ be a stack.
Whenever $T$ is a stack over $S$, the symbol $\Map_S(T,Y)$ will denote the internal hom from $X$ to $Y \times S$ in $\dSt_S$. It comes with an evaluation map $\ev \colon \Mapstack_S(T,Y) \times_S T \to Y \times S \in \dSt_S$.

Let $y \colon \dSt_S \to \dSt_S$ denote the functor $T \mapsto Y \times T$
There exists a natural transformation
\[
\EV \colon \overcat_{\dSt_S\op} \to \overcat_{\dSt_S}^{\times} \circ y\op
\]
between functors $\dSt_S\op \to \inftyCat$. For a stack $X$ over $S$, the functor
\[
\EV_X \colon \left(\comma{X}{\dSt_S} \right) \op \to \dSt_{Y \times X}
\]
maps a morphism $X \to T$ to the map
\[
\mymatrix@1{
\displaystyle \Mapstack_S(T,Y) \times_S X \ar[r] & \displaystyle \Mapstack_S(X,Y) \times_S X \ar[r]^-{\ev \times \pr} &  Y \times X
}
\]
Let us consider the natural transformation
\[
\Proext_{\dSt_S\op}(\EV) \colon \overcat_{(\Indu U \dSt_S)\op} \to \Proext_{\dSt_S\op} \left( \overcat_{\dSt_S}^{\times} \circ y\op \right)
\]
of functors $(\Indu U\dSt_S)\op \to \inftyCat$.
We define $\EV^{\Ind}$ to be the natural transformation
\[
\EV^{\Ind} = \Upsilon^{\dSt_S\op} \circ \Proext_{\dSt_S\op}(\EV)
\]
where $\Upsilon^{\dSt_S\op}$ is defined as in \autoref{pro-quot-times}.
To any $X \in \Indu U\dSt_S$ it associates a functor
\[
\EV^{\Ind}_X \colon \left(\comma{X}{\Indu U\dSt_S} \right) \op \to \IP\dSt_{Y \times X}
\]

\begin{df}
Let $Y$ be a stack.
We define the natural transformation $\EV^{\PI}$
\[
\EV^{\PI} = \Xi^{\Indu U\dSt_S\op} \circ \Indext_{(\Ind\dSt_S)\op}(\EV^{\Ind}) \colon \overcat_{\PI\dSt_S} \to \overcat^{\times}_{\IP^2\dSt_S} \circ y\op
\]
where $\Xi^{\Indu U\dSt_S\op}$ is defined in \autoref{ind-quot-times}.
To any $X \in \PI\dSt_S$ it associates a functor
\[
\EV^{\PI}_X \colon \left( \comma{X}{\PI\dSt_S} \right) \op \to \quot{\IP^2\dSt_S}{Y \times X}
\]
We then define the evaluation map in $\IP^2\dSt_S$
\[
\ev^{X,Y} \colon \mymatrix@1{\displaystyle \psi \Mapstack_S(X,Y) \times_S \phi X \ar[rr]^-{\EV^{\PI}_X(X)} && \xi Y \times \phi X \ar[r] & \xi Y}
\]
\end{df}

We assume now that $S = \Spec k$.
Let us recall the following definition from \cite[2.1]{ptvv:dersymp}
\begin{df}
A derived stack $X$ is $\Oo$-compact if for any derived affine scheme $T$ the following conditions hold
\begin{itemize}
\item The quasi-coherent sheaf $\Oo_{X \times T}$ is compact in $\Qcoh(X \times T)$ ;
\item Pushing forward along the projection $X \times T \to T$ preserves perfect complexes.
\end{itemize}
Let us denote by $\dSt_k^{\Oo}$ the full subcategory of $\dSt_k$ spanned by $\Oo$-compact derived stacks.
\end{df}

\begin{df}
An $\Oo$-compact pro-ind-stack is a pro-ind-object in the category of $\Oo$-compact derived stacks.
We will denote by $\PI\dSt^{\Oo}_k$ their category.
\end{df}

\begin{lem}
There is a functor
\[
\PI \dSt_k^\Oo \to \Fct\left( \IP\dSt_k \times \Delta^1 \times \Delta^1, (\IP)^2 (\dgMod_k)\op \right)
\]
defining for any $\Oo$-compact pro-ind-stack $X$ and any ind-pro-stack $F$ a commutative square
\[
\mymatrix{
\closedforms p_{\IP^2} (\psi F \times \phi X) \ar[r] \ar[d] & \IPclosedforms p(\psi F) \otimes_k \phi \Oo_X \ar[d] \\
\forms p_{\IP^2} (\psi F \times \phi X) \ar[r] & \IPforms p(\psi F) \otimes_k \phi \Oo_X
}
\]
where $\closedforms p_{\IP^2}$ and $\forms p_{\IP^2}$ are the extensions of $\IPclosedforms p$ and $\IPforms p$ to
\[
(\IP)^2\dSt_k \to (\IP)^2 (\dgMod_k\op)
\]
\end{lem}

\begin{proof}
Recall in \cite[part 2.1]{ptvv:dersymp} the construction for any $\Oo$-compact stack $X$ and any stack $F$ of a commutative diagram:
\[
\mymatrix{
\NCw(F \times X) \ar[r] \ar[d] & \NCw(F) \otimes_k \Oo_X \ar[d] \\
\deRham(F \times X) \ar[r] & \deRham(F) \otimes_k \Oo_X
}
\]
Taking the part of weight $p$ and shifting, we get
\[
\mymatrix{
\closedforms p(F \times X) \ar[r] \ar[d] & \closedforms p(F) \otimes_k \Oo_X \ar[d] \\
\forms p(F \times X) \ar[r] & \forms p(F) \otimes_k \Oo_X
}
\]
This construction is functorial in both $F$ and $X$ so it corresponds to a functor
\[
\dSt^\Oo_k \to \Fct(\dSt_k \times \Delta^1 \times \Delta^1, \dgMod_k\op)
\]
We can now form the functor
\begin{align*}
\PI\dSt^\Oo_k
&\to \PI \Fct\left(\Pro \dSt_k \times \Delta^1 \times \Delta^1, \Pro (\dgMod_k\op) \right) \\
&\to \Fct\left(\Pro \dSt_k \times \Delta^1 \times \Delta^1, \PI \Pro (\dgMod_k\op) \right) \\
&\to \Fct\left(\IP \dSt_k \times \Delta^1 \times \Delta^1, (\IP)^2 (\dgMod_k\op) \right) \\
\end{align*}
By construction, for any ind-pro-stack $F$ and any $\Oo$-compact pro-ind-stack, it induces the commutative diagram
\[
\mymatrix{
\closedforms p_{\IP^2}(\psi F \times \phi X) \ar[r] \ar[d] & \psi\IPclosedforms p(F) \otimes_k \phi\Oo_X \ar[d] \\
\forms p_{\IP^2}(\psi F \times \phi X) \ar[r] & \psi\IPforms p(F) \otimes_k \phi\Oo_X
}
\]
\end{proof}
\begin{rmq}
Let us remark that we can informally describe the horizontal maps using the maps from \cite{ptvv:dersymp}:
\begin{align*}
\Theta_{\IP^2}(\psi F \times \phi X) = \lim_\alpha & \colim_\beta \lim_\gamma \colim_\delta \Theta(F_{\alpha\delta} \times X_{\beta\gamma})\\
& \to \lim_\alpha \colim_\beta \lim_\gamma \colim_\delta \Theta(F_{\alpha\delta}) \otimes (\Oo_{X_{\beta\gamma}}) = \psi \Theta_{\IP}(F) \otimes \phi \Oo_X \\
\end{align*}
where $\Theta$ is either $\closedforms p$ or $\forms p$.
\end{rmq}

\begin{df}
Let $F$ be an ind-pro-stack and let $X$ be an $\Oo$-compact pro-ind-stack. Let $\eta \colon \Oo_X \to k[-d]$ be a map of ind-pro-$k$-modules. Let finally $\Theta$ be either $\closedforms p$ or $\forms p$. We define the integration map
\[
\int_\eta \colon \mymatrix@1{\Theta_{\IP^2}(\psi F \times \phi X) \ar[r] & \psi \Theta_{\IP}(F) \otimes \phi \Oo_X \ar[r]^-{\id \otimes \phi \eta} & \psi\Theta_{\IP}(F)[-d]}
\]
\end{df}

\begin{thm}\label{ipdst-form}
Let $Y$ be a derived stack and $\omega_Y$ be an $n$-shifted closed $2$-form on $Y$. Let $X$ be an $\Oo$-compact pro-ind-stack and let also $\eta \colon \Oo_X \to k[-d]$ be a map. The mapping ind-pro-stack $\Mapstack(X,Y)$ admits an $(n-d)$-shifted closed $2$-form.
\end{thm}
\begin{proof}
Let us denote by $Z$ the mapping ind-pro-stack $\Map(X,Y)$.
We consider the diagram
\[
\mymatrix@1{
\chi k[-n] \ar[r]^-{\omega_Y} & \chi \closedforms 2 (Y) \ar[r]^-{\ev^*} & \closedforms 2_{\IP^2} (X \times Z) \ar[r]^-{\int_\eta} & \psi\closedforms 2_{\IP} (Z)[-d]
}
\]
where $\chi \colon \dgMod_k \to^\xi \IP(\dgMod_k\op) \to^\psi (\IP)^2(\dgMod_k\op)$ is the canonical inclusion.
Note that since the functor $\psi$ is fully faithful, this induces a map in $\IP(\dgMod_k\op)$
\[
\mymatrix@1{
\xi k \ar[r] & \IPclosedforms 2 (Z)[n-d]}
\]
and therefore a an $(n-d)$-shifted closed $2$-form on $Z = \Map(X,Y)$. The underlying form is given by the composition
\[
\mymatrix@1{
\chi k[-n] \ar[r]^-{\omega_Y} & \chi \forms 2 (Y) \ar[r]^-{\ev^*} & \forms 2_{\IP^2} (X \times Z) \ar[r]^-{\int_\eta} & \psi\forms 2_{\IP} (Z)[-d]
}
\]
\end{proof}
\begin{rmq} \label{describe-form}
\todo{Formalise}Let us describe the form issued by \autoref{ipdst-form}.
We set the notations $X = \lim_\alpha \colim_\beta X_{\alpha\beta}$ and $Z_{\alpha\beta} = \Map(X_{\alpha\beta},Y)$. By assumption, we have a map
\[
\eta \colon \colim_\alpha \lim_\beta \Oo_{X_{\alpha\beta}} \to k[-d]
\]
For any $\alpha$, there exists therefore $\beta(\alpha)$ and a map $\eta_{\alpha\beta(\alpha)} \colon \Oo_{X_{\alpha\beta(\alpha)}} \to k[-d]$ in $\dgMod(k)$.
Unwinding the definitions, we see that the induced form $\int_\eta \omega_Y$
\[
\mymatrix@1{
\xi k \ar[r] & \IPforms 2 (\Mapstack(X,Y))[n-d]  \simeq \lim_\alpha \colim_\beta \forms 2 (Z_{\alpha\beta})[n-d]
}
\]
is the universal map obtained from the maps
\[
\mymatrix@1{
k \ar[r]^-{\omega_{\alpha\beta(\alpha)}} & \forms 2 (Z_{\alpha\beta(\alpha)})[n-d] \ar[r] & \colim_\beta \forms 2 (Z_{\alpha\beta})[n-d]
}
\]
where $\omega_{\alpha\beta(\alpha)}$ is built using $\eta_{\alpha\beta(\alpha)}$ and the procedure of \cite{ptvv:dersymp}. Note that $\omega_{\alpha\beta(\alpha)}$ can be seen as a map $\T_{X_{\alpha\beta(\alpha)}} \otimes \T_{X_{\alpha\beta(\alpha)}} \to \Oo_{X_{\alpha\beta(\alpha)}}$.
We also know from \autoref{prop-formsareforms} that the form $\int_\eta \omega_Y$ induces a map
\[
\T_Z \otimes \T_Z \to \Oo_Z[n-d]
\]
in $\IPP(Z)$. Let us fix $\alpha_0$ and pull back the map above to $Z_{\alpha_0}$. We get
\[
\colim_{\alpha \geq \alpha_0} \lim_\beta g_{\alpha_0\alpha}^* p_{\alpha\beta}^* ( \T_{Z_{\alpha\beta}} \otimes \T_{Z_{\alpha\beta}}) \simeq i_{\alpha_0}^* (\T_Z \otimes \T_Z) \to \Oo_{Z_{\alpha_0}}[n-d]
\]
This map is the universal map obtained from the maps
\begin{align*}
\lim_\beta g_{\alpha_0\alpha}^* p_{\alpha\beta}^* ( \T_{Z_{\alpha\beta}} \otimes \T_{Z_{\alpha\beta}}) \to{} & g_{\alpha_0\alpha}^* p_{\alpha\beta(\alpha)}^* ( \T_{Z_{\alpha\beta(\alpha)}} \otimes \T_{Z_{\alpha\beta(\alpha)}})
\\ &\to g_{\alpha_0\alpha}^* p_{\alpha\beta(\alpha)}^* (\Oo_{X_{\alpha\beta(\alpha)}})[n-d] \simeq \Oo_{X_{\alpha_0}}[n-d]
\end{align*}
where $g_{\alpha_0\alpha}$ is the structural map $Z_{\alpha_0} \to Z_\alpha$ and $p_{\alpha\beta}$ is the projection $Z_\alpha = \lim_\beta Z_{\alpha\beta} \to Z_{\alpha\beta}$.
\end{rmq}

\subsection{Mapping stacks have a Tate structure}
\newcommand{\coker}{\operatorname{coker}}
\begin{df}\label{map-cotate}
Let $S$ be an $\Oo$-compact pro-ind-stack. We say that $S$ is an $\Oo$-Tate stack if
there exist a poset $K$ and a diagram $\bar S \colon K\op \to \Indu U \dSt_k$ such that
\begin{enumerate}
\item The limit of $\bar S$ in $\PI\dSt_k$ is equivalent to $S$ ;\label{map-diagproj}
\item For any $i \leq j \in K$ the pro-module over $\bar S(i)$ \label{map-structuralring}
\[
\coker\left(\Oo_{\bar S(i)} \to \bar S(i \leq j)_* \Oo_{\bar S(j)} \right)
\]
is trivial in the pro-direction -- ie belong to $\Qcoh(\bar S(i))$.
\item For any $i \leq j \in K$ the induced map $\bar S(i \leq j)$ is represented by a diagram
\[
\bar f \colon L \times \Delta^1 \to \dSt_k
\]
such that
\begin{itemize}
\item For any $l \in L$ the projections $\bar f(l,0) \to \pt$ and $\bar f(l,1) \to \pt$ satisfy the base change formula ;
\item For any $l \in L$ the map $\bar f(l)$ satisfies the base change and projection formulae ;
\item For any $m \leq l \in L$ the induced map $\bar f(m \leq l, 0)$ satisfies the base change and projection formulae.
\end{itemize}
\end{enumerate}
\end{df}

\begin{rmq}
We will usually work with pro-ind-stacks $S$ given by an explicit diagram already satisfying those assumptions.
\end{rmq}

\begin{prop}\label{map-tate}
Let us assume that $Y$ is a derived Artin stack locally of finite presentation.
Let $S$ be an $\mathcal O$-compact pro-ind-stack.
If $S$ is an $\Oo$-Tate stack then the ind-pro-stack $\Mapstack(S,Y)$ is a Tate stack.
\end{prop}
\begin{proof}
Let $Z = \Map(S,Y)$ as an ind-pro-stack.
Let $\bar S \colon K\op \to \Indu U \dSt_k$ be as in \autoref{map-cotate}.
We will denote by $\bar Z \colon K \to \Prou U\dSt_k$ the induced diagram and for any $i \in K$ by $s_i \colon \bar Z(i) \to \bar Z$ the induced map.

Let us first remark that $Z$ is an Artin ind-pro-stack locally of finite presentation.
It suffices to prove that $s_i^* \Lcot_{Z}$ is a Tate module on $\bar Z(i)$, for any $i \in K$. Let us fix such an $i$ and denote by $Z_i$ the pro-stack $\bar Z(i)$.

We consider the differential map
\[
s_i^* \Lcot_{Z} \to \Lcot_{Z_i}
\]
It is by definition equivalent to the natural map
\[
\lim \cotangent^{\Pro}_{Z_i}(\bar Z|_{K^{\geq i}}) \to^f \cotangent^{\Pro}_{Z_i}(Z_i)
\]
where $K^{\geq i}$ is the comma category $\comma{i}{K}$ and $\bar Z|_{K^{\geq i}}$ is the induced diagram
\[
K^{\geq i} \to \comma{Z_i}{\Prou U\dSt_S}
\]
Let $\phi_i$ denote the diagram
\[
\phi_i \colon \left(K^{\geq i}\right)\op \to \IPerf(Z_i)
\]
obtained as the kernel of $f$. It is now enough to prove that $\phi_i$ factors through $\Perf(Z_i)$.

Let $j \geq i$ in $K$ and let us denote by $g_{ij}$ the induced map $Z_i \to Z_j$ of pro-stacks.
Let $\bar f \colon L \times \Delta^1 \to \dSt_k$ represents the map $\bar S(i \leq j) \colon \bar S(j) \to \bar S(i) \in \Indu U \dSt_k$ as in assumption \ref{map-diagproj} in \autoref{map-cotate}.
Up to a change of $L$ through a cofinal map, we can assume that the induced diagram
\[
\coker\left(\Oo_{\bar S(i)} \to \bar S(i \leq j)_* \Oo_{\bar S(j)}\right)
\]
is essentially constant -- see assumption \ref{map-structuralring}.
We denote by $\bar h \colon L\op \times \Delta^1 \to \dSt_k$ the induced diagram, so that $g_{ij}$ is the limit of $\bar h$ in $\Prou U \dSt_k$.
For any $l \in L$ we will denote by $h_l \colon Z_{il} \to Z_{jl}$ the map $\bar h(l)$.
Let us denote by $\bar Z_i$ the induced diagram $l \mapsto Z_{il}$ and by $\bar Z_j$ the diagram $l \mapsto Z_{jl}$.
Let also $p_l$ denote the projection $Z_i \to Z_{il}$

We have an exact sequence
\[
\phi_i(j) \to \colim_l p_l^* h_l^* \Lcot_{Z_{jl}} \to \colim_l p_l^* \Lcot_{Z_{il}}
\]
Let us denote by $\psi_{ij}$ the diagram obtained as the kernel
\[
\psi_{ij} \to \cotangent^{\Pro}_{Z_i}(\bar Z_j) \to \cotangent^{\Pro}_{Z_i}(\bar Z_i)
\]
so that $\phi_i(j)$ is the colimit $\colim \psi_{ij}$ in $\IPerf(Z_i)$.
It suffices to prove that the diagram $\psi_{ij} \colon L \to \Perf(Z_i)$ is essentially constant (up to a cofinal change of posets).
By definition, we have
\[
\psi_{ij}(l) \simeq p_l^* \Lcot_{Z_{il}/Z_{jl}} [-1]
\]
Let $m \to l$ be a map in $L$ and $t$ the induced map $Z_{il} \to Z_{im}$. The map $\psi_{ij}(m \to l)$ is equivalent to the map $p_l^* \xi$ where $\xi$ fits in the fibre sequence in $\Perf(Z_{il})$
\[
\mymatrix{
t^* \Lcot_{Z_{im}/Z_{jm}} [-1] \ar[r] \ar[d]_\xi & t^* h_m^* \Lcot_{Z_{jm}} \ar[d] \ar[r] & t^* \Lcot_{Z_{im}} \ar[d] \\
\Lcot_{Z_{il}/Z_{jl}} [-1] \ar[r] & h_l^* \Lcot_{Z_{jl}} \ar[r] & \Lcot_{Z_{il}}
}
\]
We consider the dual diagram
\[
\mymatrix{
t^* \T_{Z_{im}/Z_{jm}} [1] \ar@{<-}[r] \ar@{<-}[d] & t^* h_m^* \T_{Z_{jm}} \ar@{<-}[d] \ar@{<-}[r] & t^* \T_{Z_{im}} \ar@{<-}[d] \\
\T_{Z_{il}/Z_{jl}} [1] \ar@{<-}[r] & h_l^* \T_{Z_{jl}} \ar@{<-}[r] & \T_{Z_{il}} \ar@{}[ul]|{(\sigma)}
}
\]
Using base change along the maps from $S_{im}$, $S_{jm}$ and $S_{jl}$ to the point, we get that the square $(\sigma)$ is equivalent to
\[
\mymatrix{
\pi_* (\id \times s f_m)_* (\id \times s f_m)^* E  & \ar[l] \pi_* (\id \times s)_* (\id \times s)^* E \\
\pi_* (\id \times f_l)_* (\id \times f_l)^* E \ar[u] & \pi_* E \ar[l] \ar[u]
}
\]
where $\pi \colon Z_{il} \times S_{il} \to Z_{il}$ is the projection, where $s \colon S_{im} \to S_{il}$ is the map induced by $m \to l$ and where $E \simeq \ev^* \T_Y$ with $\ev \colon Z_{il} \times S_{il} \to Y$ the evaluation map.
Note that we use here the well known fact $\T_{\Map(X,Y)} \simeq \pr_* \ev^* \T_Y$ where
\[
\mymatrix{
\Map(X,Y) & \Map(X,Y) \times X \ar[r]^-\ev \ar[l]_\pr & Y
}
\]
are the canonical maps.

Now using the projection and base change formulae along the morphisms $s$, $f_l$ and $f_m$ we get that $(\sigma)$ is equivalent to the image by $\pi_*$ of the square
\[
\mymatrix{
E \otimes p^* s_* {f_m}_* \Oo_{S_{jm}} & E \otimes p^* s_* \Oo_{S_{im}} \ar[l] \\
E \otimes p^* {f_l}_* \Oo_{S_{jl}} \ar[u] & E \otimes p^* \Oo_{S_{il}} \ar[u] \ar[l]  
}
\]
We therefore focus on the diagram
\[
\mymatrix{
s_* {f_{m}}_* \Oo_{S_{jm}} & s_* \Oo_{S_{im}} \ar[l] \\
{f_l}_* \Oo_{S_{jl}} \ar[u] & \Oo_{S_{il}} \ar[l] \ar[u]
}
\]
The map induced between the cofibres is an equivalence, using assumption \ref{map-structuralring}.
It follows that the diagram $\psi_{ij}$ is essentially constant, and thus that $Z$ is a Tate stack.
\end{proof}
\end{chap-indpro}

\chapter{Loop and bubble spaces}\label{chapterloops}%
\begin{chap-loops}
In this chapter, we will at last define and study the higher dimensional formal loop spaces. We will prove it admits a Tate structure and then study the bubble space, an object defined using the formal loop space. We will then prove the bubble space to admit a symplectic structure.
\section{Formal loops}%

\subsection{Dehydrated algebras and de Rham stacks}
In this part, we define a refinement of the reduced algebra associated to a cdga. This allows us to define a well behaved de Rham stack associated to an infinite stack. Indeed, without any noetherian assumption, the nilradical of a ring -- the ideal of nilpotent elements -- is a priori not nilpotent itself.
The construction below gives an alternative definition of the reduced algebra -- which we call the dehydrated algebra -- associated to any cdga $A$, so that $A$ is, in some sense, a nilpotent extension of its dehydrated algebra.
Whenever $A$ is finitely presented, this construction coincides with the usual reduced algebra. 
\begin{df}
Let $A \in \cdga_k$.
We define its dehydrated algebra as the ind-algebra $A_\mathrm{deh} = \colim_{I} \quot{\homol^0(A)}{I}$ where the colimit is taken over the filtered poset of nilpotent ideals of $\homol^0(A)$. The case $I = 0$ gives a canonical map $A \to A_\mathrm{deh}$ in ind-cdga's.
This construction is functorial in $A$.
\end{df}

\begin{rmq}
Whenever $A$ is of finite presentation, then $A_\mathrm{deh}$ is equivalent to the reduced algebra associated to $A$. In that case, the nilradical $\sqrt{A}$ of $A$ is nilpotent.
Moreover, if $A$ is any cdga, it is a filtered colimits of cdga's $A_\alpha$ of finite presentation. We then have $A_\mathrm{deh} \simeq \colim (A_\alpha)_\mathrm{red}$ in ind-algebras.
\end{rmq}

\begin{lem}
The realisation $B$ of $A_\mathrm{deh}$ in the category of algebras is equivalent to the reduced algebra $A_\mathrm{red}$.
\end{lem}

\begin{proof}
Let us first remark that $B$ is reduced. Indeed any nilpotent element $x$ of $B$ comes from a nilpotent element of $A$. It therefore belongs to a nilpotent ideal $(x)$.
This define a natural map of algebras $A_\mathrm{red} \to B$. To see that it is an isomorphism, it suffices to say that $\sqrt{A}$ is the union of all nilpotent ideals.
\end{proof}

\begin{df}
Let $X$ be a prestack. We define its de Rham prestack $X_\mathrm{dR}$ as the composition\glsadd{derham}
\[
\mymatrix{
\cdga_k \ar[r]^-{(-)_\mathrm{deh}} & \Indu U(\cdga_k) \ar[r]^-{\Indu U(X)} & \Indu U(\sSets) \ar[r]^-{\colim} & \sSets
}
\]
This defines an endofunctor of $(\infty,1)$-category $\presh(\dAff_k)$.
We have by definition
\[
X_\mathrm{dR}(A) = \colim_{I} X\left( \quot{\homol^0(A)}{I} \right)
\]
\end{df}

\begin{rmq}
If $X$ is a stack of finite presentation, then it is determined by the images of the cdga's of finite presentation. The prestack $X_\mathrm{dR}$ is then the left Kan extension of the functor
\[
\app{\cdgaunbounded_k^{\leq 0\mathrm{,fp}}}{\sSets}{A}{X(A_\mathrm{red})}
\]
\end{rmq}

\begin{df}
Let $f \colon X \to Y$ be a functor of prestacks. We define the formal completion $\hat X_Y$ of $X$ in $Y$ as the fibre product
\[
\mymatrix{
\hat X_Y \cart \ar[r] \ar[d] & X_\mathrm{dR} \ar[d] \\ Y \ar[r] & Y_\mathrm{dR}
}
\]
This construction obviously defines a functor $\mathrm{FC} \colon \presh(\dAff_k)^{\Delta^1} \to \presh(\dAff_k)$.
\end{df}

\begin{rmq}
The natural map $\hat X_Y \to Y$ is formally étale, in the sense that for any $A \in \cdga_k$ and any nilpotent ideal $I \subset \homol^0(A)$ the morphism
\[
\hat X_Y (A) \to \hat X_Y \left(\textstyle \quot{\homol^0(A)}{I} \right) \timesunder[Y\left(\quot{\homol^0(A)}{I} \right)][][-4pt] Y(A)
\]
is an equivalence.
\end{rmq}

\subsection{Higher dimensional formal loop spaces}
Here we finally define the higher dimensional formal loop spaces.
To any cdga $A$ we associate the formal completion $V_A^d$ of $0$ in $\A^d_A$. We see it as a derived affine scheme whose ring of functions $A[\![X_{1\dots d}]\!]$ is the algebra of formal series in $d$ variables $\el{X}{d}$. Let us denote by $U_A^d$ the open subscheme of $V_A^d$ complementary of the point $0$.
We then consider the functors $\dSt_k \times \cdga_k \to \sSets$
\begin{align*}
& \kaplooppre^d_V \colon (X,A) \mapsto \Map_{\dSt_k}(V_A^d, X) \\
& \kaplooppre^d_U \colon (X,A) \mapsto \Map_{\dSt_k}(U_A^d, X)
\end{align*}

\begin{df}
Let us consider the functors $\kaplooppre_U^d$ and $\kaplooppre_V^d$ as functors $\dSt_k \to \presh(\dAff)$. They come with a natural morphism $\kaplooppre_V^d \to \kaplooppre_U^d$.
We define $\kaplooppre^d$ to be the pointwise formal completion of $\kaplooppre_V^d$ into $\kaplooppre_U^d$ :
\[
\kaplooppre^d(X) = \mathrm{FC}\left(\kaplooppre^d_V(X) \to \kaplooppre^d_U(X)\right)
\]
We also define $\kaploop^d$, $\kaploop^d_U$ and $\kaploop^d_V$ as the stackified version of $\kaplooppre^d$, $\kaplooppre^d_U$ and $\kaplooppre^d_V$ respectively.
We will call $\kaploop^d(X)$ the formal loop stack in $X$.\glsadd{loopstack}
\end{df}

\begin{rmq}
The stack $\kaploop^d_V(X)$ is a higher dimensional analogue to the stack of germs in $X$, as studied for instance by Denef and Loeser in \cite{denefloeser:germs}.
\end{rmq}

\begin{rmq}
By definition, the derived scheme $U_A^d$ is the (finite) colimit in derived stacks
\[
U_A^d = \colim_q \colim_{\el{i}{q}} \Spec\left( A[\![X_{1\dots d}]\!][X^{-1}_{i_1\dots i_q}] \right)
\]
where $A[\![X_{1\dots d}]\!][X^{-1}_{i_1\dots i_q}]$ denote the algebra of formal series localized at the generators $\iel{X^{-1}}{q}$.
It follows that the space of $A$-points of $\kaploop^d(X)$ is equivalent to the simplicial set
\[
\kaploop^d(X)(A) \simeq \colim _{I\subset \homol^0(A)} \lim_q \lim_{\el{i}{q}} \Map\left( \Spec\left(A[\![X_{1\dots d}]\!][X^{-1}_{i_1\dots i_q}]^{\sqrt{I}}\right), X \right)
\]
where $A[\![X_{1 \dots d}]\!][X^{-1}_{i_1\dots i_q}]^{\sqrt{I}}$ is the sub-cdga of $A[\![X_{1 \dots d}]\!][X^{-1}_{i_1\dots i_q}]$ consisting of series
\[
\sum_{\el{n}{d}} a_{\el{n}{d}} X_1^{n_1} \dots X_d^{n_d}
\]
where $a_{\el{n}{d}}$ is in the kernel of the map $A \to \quot{\homol^0(A)}{I}$ as soon as at least one of the $n_i$'s is negative. Recall that in the colimit above, the symbol $I$ denotes a nilpotent ideal of $\homol^0(A)$.
\end{rmq}

\begin{lem}\label{LV-epi}
Let $X$ be a derived Artin stack of finite presentation with algebraisable diagonal (see \autoref{alg-diag}) and let $t \colon T = \Spec(A) \to X$ be a smooth atlas. The induced map $\kaploop_V^d(T) \to \kaploop_V^d(X)$ is an epimorphism of stacks.
\end{lem}

\begin{proof}
It suffices to study the map $\kaplooppre_V^d(T) \to \kaplooppre_V^d(X)$.
Let $B$ be a cdga.
Let us consider a $B$-point $x \colon \Spec B \to \kaplooppre_V^d(X)$. It induces a $B$-point of $X$
\[
\Spec B \to \Spec(B[\![X_{1\dots d}]\!]) \to^x X
\]
Because $t$ is an epimorphism, there exist an étale map $f \colon \Spec C \to \Spec B$ and a commutative diagram
\[
\mymatrix{
\Spec C \ar[r]^-c \ar[d]_f & T \ar[d]^t \\ \Spec B \ar[r] & X
}
\]
It corresponds to a $C$-point of $\Spec B \times_X T$.
For any $n \in \N$, let us denote by $S_n$ the spectrum $\Spec C_n$, by $X_n$ the spectrum $\Spec B_n$ and by $T_n$ the pullback $T \times_X X_n$.
We will also consider the natural fully faithful functor $\Delta^n \simeq \{0,\dots,n\} \to \N$.
We have a natural diagram
\[
\alpha_0 \colon \Lambda^{2,2} \times \N \amalg_{\Lambda^{2,2} \times \Delta^0} \Delta^2 \times \Delta^0 \to \dSt_k
\]
informally drown has a commutative diagram
\[
\mymatrix@R=0pt{
S_0 \ar[d] \ar@/_20pt/[dd] \ar[r] & \dots \ar[r] & S_n \ar[r] \ar[d] & \dots \\
X_0 \ar[r] & \dots \ar[r] & X_n \ar[r] & \dots \\
T_0 \ar[u] \ar[r] & \dots \ar[r] & T_n \ar[r] \ar[u] & \dots
}
\]
Let $n \in \N$ and let us assume we have built a diagram
\[
\alpha_n \colon (\Lambda^{2,2} \times \N) \amalg_{\Lambda^{2,2} \times \Delta^n}  \Delta^2 \times \Delta^n \to \dSt_k
\]
extending $\alpha_{n-1}$.
There is a sub-diagram of $\alpha_n$
\[
\mymatrix{
S_n \ar[r] \ar[d] & S_{n+1} \\ T_n \ar[r] & T_{n+1} \ar[d]^{t_{n+1}} \\ & X_{n+1}
}
\]
Since the map $t_{n+1}$ is smooth (it is a pullback of $t$), we can complete this diagram with a map $S_{n+1} \to T_{n+1}$ and a commutative square. Using the composition in $\dSt_k$, we get a diagram $\alpha_{n+1}$ extending $\alpha_n$.
We get recursively a diagram $\alpha \colon \Delta^2 \times \N \to \dSt_k$. Taking the colimit along $\N$, we get a commutative diagram
\[
\mymatrix{
\Spec C \ar[d]_f \ar[r] & \colim_n \Spec C_n \ar[d] \ar[rr] && T \ar[d]^t \\
\Spec B \ar[r] & \colim_n \Spec B_n \ar[r] & \Spec(B[\![X_{1\dots d}]\!]) \ar[r] & X
}
\]
This defines a map $\phi \colon \colim \Spec(C_n) \to \Spec(B[\![X_{1\dots d}]\!]) \times_X T$.
We have the cartesian diagram
\[
\mymatrix{
\Spec(B[\![X_{1\dots d}]\!]) \times_X T \ar[r] \ar[d] \cart & X \ar[d] \\ \Spec(B[\![X_{1 \dots d}]\!]) \times T \ar[r] & X \times X
}
\]
The diagonal of $X$ is algebraisable and thus so is the stack $\Spec(B[\![X_{1\dots d}]\!]) \times_X T$. The morphism $\phi$ therefore defines the required map
\[
\Spec(C[\![X_{1 \dots d}]\!]) \to \Spec(B[\![X_{1\dots d}]\!]) \times_X T
\]
\end{proof}

\begin{rmq}
Let us remark here that if $X$ is an algebraisable stack, then $\kaplooppre_V^d(X)$ is a stack, hence the natural map is an equivalence
\[
\kaplooppre_V^d(X) \simeq \kaploop_V^d(X)
\]
\end{rmq}

\begin{lem}\label{LU-fetale}
Let $f \colon X \to Y$ be an étale map of derived Artin stacks. 
For any cdga $A \in \cdga_k$ and any nilpotent ideal $I \subset \homol^0(A)$, the induced map
\[
\theta \colon \mymatrix{
\kaplooppre^d_U(X)(A) \ar[r] & \kaplooppre^d_U(X)\left(\quot{\homol^0(A)}{I}\right) \displaystyle \timesunder[\kaplooppre^d_U(Y)\left(\quot{\homol^0(A)}{I}\right)][][-4pt] \kaplooppre^d_U(Y)(A)
}
\]
is an equivalence.
\end{lem}

\begin{proof}
The map $\theta$ is a finite limit of maps
\[
\mu \colon \mymatrix{
X(\xi A) \ar[r] &  X\left(\xi \left(\quot{\homol^0(A)}{I}\right) \right) \displaystyle \timesunder[Y\left(\xi\left(\quot{\homol^0(A)}{I}\right) \right)][][-4pt] Y (\xi A )
}
\]
where $\xi A = A[\![X_{1 \dots d}]\!][X^{-1}_{i_1\dots i_p}]$ and $\xi (\quot{\homol^0(A)}{I})$ is defined similarly.
The natural map $\xi(\homol^0(A)) \to \xi (\quot{\homol^0(A)}{I})$ is also a nilpotent extension. We deduce from the étaleness of $f$ that the map
\[
\mymatrix{
X(\xi(\homol^0(A))) \ar[r] & X\left(\xi \left(\quot{\homol^0(A)}{I}\right) \right) \displaystyle \timesunder[Y\left(\xi\left(\quot{\homol^0(A)}{I}\right) \right)][][-4pt] Y (\xi( \homol^0(A)) )
}
\]
is an equivalence.
Let now $n \in \N$. We assume that the natural map
\[
\mymatrix{
X(\xi(A_{\leq n})) \ar[r] & X\left(\xi \left(\quot{\homol^0(A)}{I}\right) \right) \displaystyle \timesunder[Y\left(\xi\left(\quot{\homol^0(A)}{I}\right) \right)][][-4pt] Y (\xi(A_{\leq n}))
}
\]
is an equivalence. The cdga $\xi(A_{\leq n+1}) \simeq (\xi A)_{\leq n+1}$ is a square zero extension of $\xi(A_{\leq n})$ by $\homol^{-n-1}(\xi A)$. We thus have the equivalence
\[
\mymatrix{
X(\xi(A_{\leq n+1})) \ar[r]^-\sim & X(\xi(A_{\leq n})) \displaystyle \timesunder[Y(\xi(A_{\leq n}))][][-3pt] Y(\xi(A_{\leq n+1}))
}
\]
The natural map 
\[
\mymatrix{
X(\xi(A_{\leq n+1})) \ar[r] & X\left(\xi \left(\quot{\homol^0(A)}{I}\right) \right) \displaystyle \timesunder[Y\left(\xi\left(\quot{\homol^0(A)}{I}\right) \right)][][-4pt] Y (\xi(A_{\leq n+1}))
}
\]
is thus an equivalence too.
The stacks $X$ and $Y$ are nilcomplete, hence $\mu$ is also an equivalence -- recall that a derived stack $X$ is nilcomplete if for any cdga $B$ we have
\[
X(B) \simeq \lim_n X(B_{\leq n})
\]
It follows that $\theta$ is an equivalence.
\end{proof}

\begin{cor}\label{L-fetale}
Let $f \colon X \to Y$ be an étale map of derived Artin stacks.
For any cdga $A \in \cdga_k$ and any nilpotent ideal $I \subset \homol^0(A)$, the induced map
\[
\theta \colon
\mymatrix{
\kaplooppre^d(X)(A) \ar[r] & \kaplooppre^d(X)\left(\textstyle \quot{\homol^0(A)}{I}\right) \displaystyle \timesunder[\kaplooppre^d(Y)\left(\quot{\homol^0(A)}{I}\right)][][-4pt] \kaplooppre^d(Y)(A)
}
\]
is an equivalence.
\end{cor}

\begin{prop}\label{L-epi}
Let $X$ be a derived Deligne-Mumford stack of finite presentation with algebraisable diagonal. Let $t \colon T \to X$ be an étale atlas. The induced map $\kaploop^d(T) \to \kaploop^d(X)$ is an epimorphism of stacks.
\end{prop}

\begin{proof}
We can work on the map of prestacks $\kaplooppre^d(T) \to \kaplooppre^d(X)$.
Let $A \in \cdga_k$. Let $x$ be an $A$-point of $\kaplooppre^d(X)$. It corresponds to a vertex in the simplicial set
\[
\mymatrix{
\displaystyle \colim_I \textstyle {\kaplooppre_V^d(X)\left(\quot{\homol^0(A)}{I}\right)} \displaystyle \timesunder[\kaplooppre_U^d(X)\left(\quot{\homol^0(A)}{I}\right)][][-4pt] \kaplooppre_U^d(X)(A)
}
\]
There exists therefore a nilpotent ideal $I$ such that $x$ comes from a commutative diagram
\[
\mymatrix{
U_{\quot{\homol^0(A)}{I}}^d \ar[d] \ar[r] & U_A^d \ar[d] \\ V_{\quot{\homol^0(A)}{I}} \ar[r]_-v & X
}
\]
Using \autoref{LV-epi} we get an étale morphism $\psi \colon A \to B$ such that the map $v$ lifts to a map $u \colon V_{\quot{B}{J}} \to T$ where $J$ is the image of $I$ by $\psi$.
This defines a point in
\[
\textstyle \kaplooppre^d_U(T)\left(\quot{\homol^0(B)}{J} \right) \displaystyle \timesunder[\kaplooppre^d_U(X)\left(\quot{\homol^0(B)}{J}\right)][][-4pt] \kaplooppre^d_U(X)(B)
\]
Because of \autoref{LU-fetale}, we get a point of $\kaplooppre^d(T)(B)$. We now observe that this point is compatible with $x$.
\end{proof}

In the case of dimension $d=1$, \autoref{LU-fetale} can be modified in the following way. Let $f \colon X \to Y$ be a smooth map of derived Artin stacks.
For any cdga $A \in \cdga_k$ and any nilpotent ideal $I \subset \homol^0(A)$, the induced map
\[
\theta \colon \mymatrix{
\kaplooppre^1_U(X)(A) \ar[r] & \kaplooppre^1_U(X)\left(\quot{\homol^0(A)}{I}\right) \displaystyle \timesunder[\kaplooppre^1_U(Y)\left(\quot{\homol^0(A)}{I}\right)][][-4pt] \kaplooppre^1_U(Y)(A)
}
\]
is essentially surjective.
The following proposition follows.

\begin{prop}
Let $X$ be an Artin derived stack of finite presentation and with algebraisable diagonal. Let $t \colon T \to X$ be a smooth atlas. The induced map $\kaploop^1(T) \to \kaploop^1(X)$ is an epimorphism of stacks.
\end{prop}

\begin{ex}
The proposition above implies for instance that $\kaploop^1(\B G) \simeq \B \kaploop^1(G)$ for any algebraic group $G$ -- where $\B G$ is the classifying stack of $G$-bundles.
\end{ex}

\subsection{Tate structure and determinantal anomaly}\label{determinantalclass}
We saw in \autoref{tatestacks} that to any Tate stack $X$, we can associate a determinantal anomaly. It a class in $\homol^2(X,\Oo_X^{\times})$.
We will prove in this subsection that the stack $\kaploop^d(X)$ is endowed with a structure of Tate stack as soon as $X$ is affine. We will moreover build a determinantal anomaly on $\kaploop^d(X)$ for any quasi-compact and separated scheme $X$.

\begin{lem}\label{L-shy}
For any $B \in \cdga_k$ of finite presentation, the functors 
\[
\kaplooppre^d_U(\Spec B), \kaplooppre^d(\Spec B) \colon \cdga_k \to \sSets
\]
are in the essential image of the fully faithful functor 
\[
\shybounded_k \cap \IP\dAff_k \to \IP\dSt_k \to \dSt_k \to \presh(\dAff)
\]
(see \autoref{indprochamps}).
It follows that $\kaplooppre^d_U(\Spec B) \simeq \kaploop^d_U(\Spec B)$ and $\kaplooppre^d(\Spec B) \simeq \kaploop^d(\Spec B)$.
\end{lem}
\begin{proof}
Let us first remark that $\Spec B$ is a retract of a \emph{finite} limit of copies of the affine line $\A^1$.
It follows that the functor $\kaplooppre_U^d(\Spec B)$ is, up to a retract, a finite limit of functors
\[
Z_E^d \colon A \mapsto \Map\left(k[Y], A[\![X_{1 \dots d}]\!][X^{-1}_{i_1\dots i_q}]\right)
\]
where $E = \{ \el{i}{q} \} \subset F = \{1,\dots,d\}$.
The functor $Z^d_E$ is the realisation of an affine ind-pro-scheme
\[ 
Z^d_E \simeq \colim_n \lim_p \Spec\left( k[ a_{\el{\alpha}{d}}, -n \delta_i \leq \alpha_i \leq p] \right)
\]
where $\delta_i = 1$ if $i \in E$ and $\delta_i = 0$ otherwise. The variable $a_{\el{\alpha}{d}}$ corresponds to the coefficient of $X_1^{\alpha_1} \dots X_d^{\alpha_d}$.
The functor $Z^d_E$ is thus in the category $\shybounded \cap \IP\dAff_k$. The result about $\kaplooppre^d_U(\Spec B)$ then follows from \autoref{shydaff-limits}.
The case of $\kaplooppre^d(\Spec B)$ is similar: we decompose it into a finite limit of functors
\[
G_E^d \colon A \mapsto \colim_{I \subset \homol^0(A)} \Map\left(k[Y], A[\![X_{1 \dots d}]\!][X^{-1}_{i_1\dots i_q}]^{\sqrt{I}}\right)
\]
where $I$ is a nilpotent ideal of $\homol^0(A)$.
We then observe that $G_E^d$ is the realisation of the ind-pro-scheme
\[ 
G^d_E \simeq \colim_{n,m} \lim_p \Spec\left( \quot{k[ a_{\el{\alpha}{d}}, -n \delta_i \leq \alpha_i \leq p]}{J} \right)
\]
where $J$ is the ideal generated by the symbols $a_{\el{\alpha}{d}}^m$ with at least one of the $\alpha_i$'s negative.
\end{proof}

\begin{rmq}
Let $n$ and $p$ be integers and let $k(E,n,p)$ denote the number of families $(\el{\alpha}{d})$ such that $-n \delta_i \leq \alpha_i \leq p$ for all $i$. We have 
\[
Z^d_E \simeq \colim_n \lim_p (\A^1)^{k(E,n,p)}
\] 
\end{rmq}

\begin{df}
From \autoref{L-shy}, we get a functor $\underline \kaploop^d \colon \dAff_k^\mathrm{fp} \to \IP\dSt_k$. It follows from \autoref{L-epi} that $\underline \kaploop^d$ is a costack in ind-pro-stacks. We thus define\glsadd{loopipstack}
\[
\underline \kaploop^d \colon \dSt_k^{\mathrm{lfp}} \to \IP\dSt_k
\]
to be its left Kan extension along the inclusion $\dAff_k^\mathrm{fp} \to \dSt_k^\mathrm{lfp}$ -- where $\dSt_k^\mathrm{lfp}$ is $(\infty,1)$-category of derived stacks locally of finite presentation.
This new functor $\underline \kaploop^d$ preserves small colimits by definition.
\end{df}

\begin{prop}\label{L-indpro}
There is a natural transformation $\theta$ from the composite functor
\[
\mymatrix{
\dSt_k^\mathrm{lfp} \ar[r]^-{\underline \kaploop^d} & \IP\dSt_k \ar[r]^{|-|^\IP} & \dSt_k
}
\]
to the functor $\kaploop^d$.
Moreover, the restriction of $\theta$ to derived Deligne-Mumford stacks of finite presentation with algebraisable diagonal is an equivalence.
\end{prop}
\begin{proof}
There is by definition a natural transformation
\[
\theta \colon | \underline \kaploop^d(-) |^\IP \to \kaploop^d(-)
\]
Moreover, the restriction of $\theta$ to affine derived scheme of finite presentation is an equivalence -- see \autoref{L-shy}. 
The fact that $\theta_X$ is an equivalence for any Deligne-Mumford stack $X$ follows from \autoref{L-epi}.
\end{proof}

\begin{lem}\label{subsets-colim}
Let $F$ be a non-empty finite set.
For any family $(M_D)$ of complexes over $k$ indexed by subsets $D$ of $F$, we have
\[
\colim_{\emptyset \neq E \subset F} \bigoplus_{\emptyset \neq D \subset E} M_D \simeq M_F[d-1]
\]
where $d$ is the cardinal of $F$ (the maps in the colimit diagram are the canonical projections).
\end{lem}

\begin{proof}
We can and do assume that $F$ is the finite set $\{1, \dots, d\}$ and we proceed recursively on $d$.
The case $d = 1$ is obvious.
Let now $d \geq 2$ and let us assume the statement is true for $F \smallsetminus \{d\}$. Let $(M_D)$ be a family as above. We have a cocartesian diagram
\[
\mymatrix{
\displaystyle \colim_{\{d\} \subsetneq E \subset F} \bigoplus_{\emptyset \neq D \subset E} M_D \ar[r] \ar[d]
& \displaystyle \colim_{ \emptyset \neq E \subset F\smallsetminus\{d\}} \bigoplus_{\emptyset \neq D \subset E} M_D \ar[d] \\
\displaystyle M_{\{d\}} \ar[r]
& \displaystyle \colim_{\emptyset \neq E \subset F} \bigoplus_{\emptyset \neq D \subset E} M_D \cocart
}
\]
We have by assumption
\[
\colim_{ \emptyset \neq E \subset F\smallsetminus\{d\}} \bigoplus_{\emptyset \neq D \subset E} M_D \simeq M_{F \smallsetminus \{d\}} [d-2]
\]
and
\begin{align*}
\colim_{\{d\} \subsetneq E \subset F} \bigoplus_{\emptyset \neq D \subset E} M_D
&\simeq M_{\{d\}} \oplus \left(\colim_{\{d\} \subsetneq E \subset F} \bigoplus_{\{d\} \subsetneq D \subset E} M_D \right) \oplus \left(\colim_{\{d\} \subsetneq E \subset F} \bigoplus_{\emptyset \neq D \subset E \smallsetminus \{d\}} M_D\right) \\
&\simeq M_{\{d\}} \oplus M_F[d-2] \oplus M_{F \smallsetminus \{d\}}[d-2]
\end{align*}
The result follows.
\end{proof}

\begin{lem}\label{LU-affine-tate}
For any $B \in \cdga_k$ of \emph{finite presentation}, the ind-pro-stack $\underline \kaploop^d_U(\Spec B)$ is a Tate stack.
\end{lem}
\begin{proof}
Let us first focus on the case of the affine line $\A^1$.
We have to prove that the cotangent complex $\Lcot_{\underline \kaploop^d_U(\A^1)}$ is a Tate module.
For any subset $D \subset F$ we define $M_D^{p,n}$ to be the free $k$-complex generated by the symbols
\[
\{a_{\el{\alpha}{d}}, -n \leq \alpha_i < 0 \text{ if } i \in D, 0 \leq \alpha_i \leq p \text{ otherwise}\}
\]
in degree $0$.
From the proof of \autoref{L-shy}, we have
\[
Z^d_E \simeq \colim_n \lim_p \Spec \left( k\left[ \textstyle \bigoplus_{D \subset E} M_D^{p,n} \right] \right) \text{~~~ and ~~~} \underline \kaploop^d_U(\A^1) \simeq \lim_{\emptyset \neq E \subset F} Z^d_E
\]
where $F = \{1 , \dots , d\}$.
If we denote by $\pi$ the projection $\underline \kaploop_U^d(\A^1) \to \Spec k$, we get
\[
\Lcot_{\underline \kaploop^d_U(\A^1)} \simeq \pi^* \left(\colim_{\emptyset \neq E \subset F} \lim_n \colim_p \bigoplus_{D \subset E} M_D^{p,n}\right)
\simeq \pi^* \left(\lim_n \colim_p \colim_{\emptyset \neq E \subset F} \bigoplus_{D \subset E} M_D^{p,n}\right)
\]
Using \autoref{subsets-colim} we have
\[
\Lcot_{\underline \kaploop^d_U(\A^1)} \simeq \pi^* \left(\lim_n \colim_p M_\emptyset^{p,n} \oplus M_F^{p,n}[d-1] \right)
\]
Moreover, we have $M_\emptyset^{p,n} \simeq M_\emptyset^{p,0}$ and $M_F^{p,n} \simeq M_F^{0,n}$. It follows that $\Lcot_{\underline \kaploop^d_U(\A^1)}$ is a Tate module on the ind-pro-stack $\underline \kaploop^d_U(\A^1)$.
The case of $\underline \kaploop^d_U(\Spec B)$ then follows from \autoref{shydaff-limits} and from \autoref{tate-limits}.
\end{proof}

\begin{lem}\label{LU-ip-fetale}
Let $B \to C$ be an étale map between cdga's of finite presentation. The induced map $f \colon \underline \kaploop^d_U(\Spec C) \to \underline \kaploop^d_U(\Spec B)$ is formally étale -- see \autoref{derivation-ipdst}.
\end{lem}

\begin{proof}
Let us denote $X = \Spec B$ and $Y = \Spec C$.
We have to prove that the induced map
\[
j \colon \Map_{\underline \kaploop^d_U(Y)/-}\left(\underline \kaploop^d_U(Y)[-], \underline \kaploop^d_U(Y)\right) \to \Map_{\underline \kaploop^d_U(Y)/-}\left(\underline \kaploop^d_U(Y)[-], \underline \kaploop_U^d(X)\right) 
\]
is an equivalence of functors $\PIQcoh( \underline \kaploop^d(Y) )^{\leq 0} \to \sSets$.
Since $\underline \kaploop^d_U(Y)$ is ind-pro-affine, we can restrict to the study of the morphism
\[
j_Z \colon \Map_{Z/-}\left(Z[-], \underline \kaploop^d_U(Y)\right) \to \Map_{Z/-}\left(Z[-], \underline \kaploop_U^d(X)\right) 
\]
of functors $\IQcoh(Z)^{\leq 0} \to \sSets$, for any pro-affine scheme $Z$ and any map $Z \to \underline \kaploop^d_U(Y)$.
Let us fix $E \in \IQcoh(Z)^{\leq 0}$. The pro-stack $Z[E]$ is in fact an affine pro-scheme.
Recall that both $\underline \kaploop^d_U(Y)$ and $\underline \kaploop^d_U(X)$ belong to $\shybounded_k$.
It follows from the proof of \autoref{ff-realisation} that the morphism $j_Z(E)$ is equivalent to
\[
|j_Z(E)| \colon \Map_{|Z|/-}\left( |Z[E]|, \kaploop^d_U(Y) \right) \to \Map_{|Z|/-}\left( |Z[E]|, \kaploop^d_U(X) \right) 
\]
where $|-|$ is the realisation functor and the mapping spaces are computed in $\dSt_k$.
It now suffices to see that $|Z[E]|$ is a trivial square zero extension of the derived affine scheme $|Z|$ and to use \autoref{LU-fetale}.
\end{proof}

\begin{prop}\label{L-affine-tate}
Let $\Spec B$ be a derived affine scheme of finite presentation. The ind-pro-stack $\underline \kaploop^d(\Spec B)$ admits a cotangent complex. This cotangent complex is moreover a Tate module.
For any étale map $B \to C$ the induced map $f \colon \underline \kaploop^d(\Spec C) \to \underline \kaploop^d(\Spec B)$ is formally étale -- see \autoref{derivation-ipdst}.
\end{prop}

\begin{proof}
Let us write $Y = \Spec B$.
Let us denote by $i \colon \underline \kaploop^d(Y) \to \underline \kaploop^d_U(Y)$ the natural map. 
We will prove that the map $i$ is formally étale, the result will then follow from \autoref{LU-affine-tate} and \autoref{LU-ip-fetale}.
To do so, we consider the natural map
\[
j \colon \Map_{\underline \kaploop^d(Y)/-}\left(\underline \kaploop^d(Y)[-], \underline \kaploop^d(Y)\right) \to \Map_{\underline \kaploop^d(Y)/-}\left(\underline \kaploop^d(Y)[-], \underline \kaploop_U^d(Y)\right) 
\]
of functors $\PIQcoh( \underline \kaploop^d(Y) )^{\leq 0} \to \sSets$.
To prove that $j$ is an equivalence, we can consider for every affine pro-scheme $X \to \underline \kaploop^d(Y)$ the morphism of functors $\IQcoh(X)^{\leq 0} \to \sSets$
\[
j_X \colon \Map_{X/-}\left( X[-], \underline \kaploop^d(Y) \right) \to \Map_{X/-}\left( X[-], \underline \kaploop^d_U(Y) \right) 
\]
Let us fix $E \in \IQcoh(X)^{\leq 0}$. The morphism $j_X(E)$ is equivalent to
\[
|j_X(E)| \colon \Map_{|X|/-}\left( |X[E]|, \kaploop^d(Y) \right) \to \Map_{|X|/-}\left( |X[E]|, \kaploop^d_U(Y) \right) 
\]
where the mapping space are computed in $\dSt_k$. The map $|j_X(E)|$ is a pullback of the map
\[
f \colon \Map_{|X|/-}\left( |X[E]|, \kaploop^d_V(Y)_\mathrm{dR} \right) \to \Map_{|X|/-}\left( |X[E]|, \kaploop^d_U(Y)_\mathrm{dR} \right) 
\]
It now suffices to see that $|X[E]|$ is a trivial square zero extension of the derived affine scheme $|X|$ and thus $f$ is an equivalence (both of its ends are actually contractible).
\end{proof}

Let us recall from \autoref{determinantalanomaly} the determinantal anomaly
\[
[\mathrm{Det}_{\underline \kaploop^d(\Spec A)}] \in \homol^2\left(\kaploop^d(\Spec A), \Oo_{\kaploop^d(\Spec A)}^{\times}\right)
\]
It is associated to the tangent $\T_{\underline \kaploop^d(\Spec A)} \in \Tateu U_\IP(\underline \kaploop^d(\Spec A))$ through the determinant map.
Using \autoref{L-affine-tate}, we see that this construction is functorial in $A$, and from \autoref{L-epi} we get that it satisfies étale descent.
Thus, for any quasi-compact and quasi-separated (derived) scheme (or Deligne-Mumford stack with algebraisable diagonal), we have a well-defined determinantal anomaly
\[
[\mathrm{Det}_{\underline \kaploop^d(X)}] \in \homol^2\left(\kaploop^d(X), \Oo_{\kaploop^d(X)}^{\times}\right)
\]

\begin{rmq}
It is known since \cite{kapranovvasserot:loop4} that in dimension $d=1$, if $[\mathrm{Det}_{\kaploop^1(X)}]$ vanishes, then there are essentially no non-trivial automorphisms of sheaves of chiral differential operators on $X$.
\end{rmq}%

\section{Bubble spaces}%

\subsection{Local cohomology}

This subsection is inspired by a result from \cite[Éxposé 2]{SGA2}, giving a formula for local cohomology -- see \autoref{formula-coholoc}. We will first develop two duality results we will need afterwards, and then prove the formula.

Let $A \in \cdga_k$ be a cdga over a field $k$. Let $(\el{f}{p})$ be points of $A^0$ whose images in $\homol^0(A)$ form a regular sequence.

Let us denote by $A_{n,k}$ the Kozsul complex associated to the regular sequence $(\el{f^n}{k})$ for $k \leq p$. We set $A_{n,0} = A$ and $A_n = A_{n,p}$ for any $n$.
If $k<p$, the multiplication by $f^n_{k+1}$ induces an endomorphism $\varphi^n_{k+1}$ of $A_{n,k}$. Recall that $A_{n,k+1}$ is isomorphic to the cone of $\varphi^n_{k+1}$:
\[
\mymatrix{
A_{n,k} \ar[r]^{\varphi^n_{k+1}} \ar[d] & A_{n,k} \ar[d] \\ 0 \ar[r] & A_{n,k+1} \cocart
}
\]
Let us now remark that for any couple $(n,k)$, the $A$-module $A_{n,k}$ is perfect.
\begin{lem} \label{dual-over-A}
Let $k \leq p$.
The $A$-linear dual $A_{n,k}^{\quot{\vee}{A}} = \RHomint_A(A_{n,k},A)$ of $A_{n,k}$ is equivalent to $A_{n,k}[-k]$;
\end{lem}
\begin{proof}
We will prove the statement recursively on the number $k$.
When $k = 0$, the result is trivial.
Let $k \geq 0$ and let us assume that $A_{n,k}^{\quot{\vee}{A}}$ is equivalent to $A_{n,k}[-k]$. Let us also assume that for any $a \in A$, the diagram induced by multiplication by $a$ commutes
\[
\mymatrix{
A_{n,k}^{\quot{\vee}{A}} \ar@{-}[r]^-\sim \ar[d]_{\dual a} & A_{n,k}[-k] \ar[d]^a \\
A_{n,k}^{\quot{\vee}{A}} \ar@{-}[r]^-\sim & A_{n,k}[-k]
}
\]
We obtain the following equivalence of exact sequences
\[
\mymatrix{
A_{n,k+1}[-k-1] \ar[r] \ar@{-}[d]^\sim & A_{n,k}[-k] \ar@{-}[d]^\sim \ar[r]^{\varphi^n_{k+1}} & A_{n,k}[-k] \ar@{-}[d]^\sim \\
A_{n,k+1}^{\quot{\vee}{A}} \ar[r] & A_{n,k}^{\quot{\vee}{A}} \ar[r]^{\dual{\left(\varphi^n_{k+1}\right)}} & A_{n,k}^{\quot{\vee}{A}}
}
\]
The statement about multiplication is straightforward.
\end{proof}

\begin{lem} \label{dual-over-k}
Let us assume $A$ is a formal series ring over $A_1$:
\[
A = A_1[\![\el{f}{p}]\!]
\]
It follows that for any $n$, the $A_1$-module $A_n$ is free of finite type and that there is map $r_n \colon A_n \to A_1$ mapping $\el{f^n}{p}[]$ to $1$ and any other generator to zero.
We deduce an equivalence 
\[
A_n \to^\sim A_n^{\quot{\vee}{A_1}} = \RHomint_{A_1}(A_n, A_1)
\]
given by the pairing
\[
\mymatrix{
A_n \otimes_{A_1} A_n \ar[r]^-{\times} & A_n \ar[r]^{r_n} & A_1
}
\]
\end{lem}
\begin{rmq}
Note that we can express the inverse $A_n^{\quot{\vee}{A_1}} \to A_n$ of the equivalence above: it map a function $\alpha \colon A_n \to A_1$ to the serie
\[
\sum_{\underline i} \alpha(f^{\underline i}) f^{n-1-\underline i}
\]
where $\underline i$ varies through the uplets $(\el{i}{p})$ and where $f^{\underline i} = f_1^{i_1} \dots f_p^{i_p}$.
\end{rmq}

\todo{Hypothèse sur $X$}
We can now focus on the announced formula.
Let $X$ be a quasi-compact and quasi-separated derived scheme and let $i \colon Z \to X$ be a closed embedding defined be a finitely generated ideal $\mathcal I \subset \Oo_X$. Let $j \colon U \to X$ denote the complementary open subscheme.

Let us denote by $\bar Y$ the diagram $\N \to \dSt_X$ defined by 
\[
\bar Y(n) = Y_n = \Spec_X \left( \quot{\Oo_X}{\mathcal I^n} \right)
\]
For any $n \in \N$, we will denote by $i_n \colon Y_n \to X$ the inclusion.
Let us fix the notation
\[
\Qcoh_* \colon \mymatrix@1{ \dSt_k \ar[r]^-{\Qcoh\op} & \left(\PresLeftu V\right)\op \simeq \PresRightu V}
\]
It maps every morphism $\phi \colon S \to T$ to the forgetful functor $\phi_* \colon \Qcoh(S) \to \Qcoh(T)$. This functor also admit a right adjoint, denoted by $\phi^!$. We denote by
\[
\Qcoh^! \colon \dSt_k\op \to \PresRightu V
\]
the corresponding diagram. It will also be handy to denote $\Qcoh$ by $\Qcoh^*$.
We finally set the following notations
\[
\mymatrix{
\Qcoh(\hat X) = \lim \Qcoh^*(\bar Y) \ar@<-2pt>[r]_-{\hat\imath_*} & \Qcoh(X) \ar@<-2pt>[l]_-{\hat\imath^*} \ar@<2pt>[r]^{j^*} & \Qcoh(U)  \ar@<2pt>[l]^{j_*} \\
& \Qcoh_Z(X) \ar@<2pt>[u]^g \ar@{<-}@<-2pt>[u]_f&
}
\]
Gaitsgory has proven the functors $f \hat\imath_*$ and $\hat\imath^* g$ to be equivalences.
The functor $f$ then corresponds to $\hat\imath^*$ through this equivalence.
We can also form the adjunction
\[
\mymatrix{
\lim \Qcoh^!(\bar Y) \ar@<2pt>[r]^-{\tilde\imath_*} \ar@<-2pt>@{<-}[r]_-{\tilde\imath^!} & \Qcoh(X)
}
\]

\begin{lem}[Gaitsgory-Rozenblyum]\label{gaits-cofinal}
Let $A \in \cdga_k$ and let $p$ be a positive integer. The natural morphism induced by the multiplication $A_p \otimes_A A_p \to A_p$ is an equivalence
\[
\colim_n \RHomint_A\left(A_n \otimes_A A_p, -\right) \simeq \colim_{n \geq p} \RHomint_A\left(A_n \otimes_A A_p, -\right) \from^\sim \RHomint_A(A_p, -)
\]
\end{lem}
\begin{proof}
See \cite[7.1.5]{gaitsgoryrozenblyum:dgindschemes}.
\end{proof}
\begin{prop}
The functor $T = \tilde\imath_* \tilde\imath^!$ is the colimit of the diagram
\[
\mymatrix@1@C=15mm{\N \ar[r]^-{\bar Y} & \dSt_X \ar[r]^-{\Qcoh_*} & \displaystyle \quot{\PresLeftu V}{\Qcoh(X)} \ar[r]^-{\eta_{\Qcoh(X)}} & \displaystyle \quot{\Fct(\Qcoh(X),\Qcoh(X))}{\id} }
\]
It is moreover a right localisation equivalent to the local cohomology functor $g f$.
This induces an equivalence 
\[
\lim \Qcoh^!(\bar Y) \to \Qcoh_Z(X)
\]
commuting with the functors to $\Qcoh(X)$.
\end{prop}

\begin{rmq}\label{formula-coholoc}
Let us denote by $\Homint_{\Oo_X}(-,-)$ the internal hom of the category $\Qcoh(X)$. It corresponds to a functor $\Qcoh(X)\op \to \Fct(\Qcoh(X),\Qcoh(X))$.
There is moreover a functor $\Oo_* \colon \dSt_X \to \Qcoh(X)\op$ mapping a morphism $\phi \colon S \to X$ to $\phi_* \Oo_S$.
The composite functor
\[
\Homint_{\Oo_X}(\Oo_*(-),-) \colon \dSt_X \to \Fct(\Qcoh(X), \Qcoh(X))
\]
is then equivalent to $\eta_{\Qcoh(X)} \circ \Qcoh_*$, using the uniqueness of right adjoints.

It follows that for any quasi-coherent module $M \in \Qcoh(X)$, we have an exact sequence
\[
\colim_n \Homint_{\Oo_X}(\Oo_{Y_n}, M) \to M \to j_* j^* M
\]
and thus gives a (functorial) formula for local cohomology 
\[
\mathrm H_Z(M) \simeq \colim_n \Homint_{\Oo_X}(\Oo_{Y_n}, M)
\]
It is a generalisation to derived schemes of \cite[Exposé 2, Théorème 6]{SGA2}.
\end{rmq}

\begin{proof}[of the proposition]
The first statement follows from the proof of \autoref{lim-adjoint}, applied to the opposite adjunction.
Let us consider the adjunction morphism $\alpha \colon T = \tilde\imath_* \tilde\imath ^! \to \id$. We must prove that both the induced maps
\[
T^2 \to T
\]
are equivalences. We can restrict to the affine case which follows from \autoref{gaits-cofinal}. The functor $T$ is therefore a right localisation. We will denote by $\Qcoh^T(X)$ the category of $T$-local objects; it comes with functors:
\[
\mymatrix{
\Qcoh^T(X) \ar@<2pt>[r]^-{u} & \Qcoh(X) \ar@<2pt>[l]^-{v}
}
\]
such that $v u \simeq \id$ and $u v \simeq T$.
Using now the vanishing of $j^* \tilde\imath_*$, we get a canonical fully faithful functor $\psi \colon \Qcoh^T(X) \to \Qcoh_Y(X)$ such that $u = g \psi$. It follows that $\psi$ admits a right adjoint $\xi$ and that
\[
\psi = f u \hspace{1cm} \text{and} \hspace{1cm} \xi = v g
\]
We will now prove that the functor $\xi$ is conservative. Let therefore $E \in \Qcoh_Y(X)$ such that $\xi E = 0$. We need to prove that $E$ is equivalent to zero.
We have $T g E = 0$ and ${i_1}_* i_1^! T g E \simeq \RHomint_{\Oo_X}(\Oo_Z,g E)$. Because $\Oo_Z$ is a compact generator of $\Qcoh_Y(X)$ --- see \cite[3.7]{toen:dgazumaya} ---, this implies that $g E$ is supported on $U$. It therefore vanishes.

The vanishing of $j^* \tilde\imath_*$ implies the existence of a functor
\[
\mymatrix@1{\lim \Qcoh^!(\bar Y) \ar[r]^-{\gamma} & \Qcoh_Y(X)}
\]
such that $g \gamma \simeq \tilde \imath_*$. The functor $\varepsilon = \tilde \imath^! g$ is right adjoint to $\gamma$.
The computation
\[
g \gamma \varepsilon \simeq \tilde \imath_* \tilde \imath^! g = T g \simeq g
\]
proves that $\varepsilon$ is fully faithful. We now have to prove that $\gamma$ is conservative. Is it enough to prove that $\tilde\imath_*$ is conservative.
Let $(E_n) \in \lim \Qcoh^!(\bar Y)$.
The colimit
\[
\colim_n {i_n}_* E_n
\]
vanishes if and only if for any $n$, any $p \in \Z$ and any $e \colon \Oo_{Y_n}[p] \to E_n$, there exist $N \geq n$ such that the natural morphism $f \colon {h_{nN}}_* \Oo_{Y_n}[p] \to {h_{nN}}_* E_n \to E_N$ vanishes. The symbol $h_{nN}$ stands for the map $\bar Y(n \leq N)$.
We know that $e$ is the composite map
\[
\mymatrix{
\Oo_{Y_n}[p] \ar[r] & h_{nN}^! {h_{nN}}_* \Oo_{_n}[p] \ar[r]^-{h_{nN}^! f} & h_{nN}^! E_N = E_n
}
\]
The point $e$ is therefore zero and $E_n$ is contractible.
\end{proof}

\subsection{Definition and properties}

We define here the bubble space, obtained from the formal loop space. We will prove in the next sections it admits a structure of symplectic Tate stack.

\begin{df}
The formal sphere of dimension $d$ is the pro-ind-stack
\[
\formalsphere^d = \lim_n \colim_{p \geq n} \Spec(A_p \oplus \Homint_A(A_n,A)) \simeq \lim_n \colim_{p \geq n} \Spec(A_p \oplus A_n[-d])
\]
where $A = k[\el{x}{d}]$ and $A_n = \quot{A}{(\el{x^n}{d})}$.
\end{df}

\begin{rmq}
\todo{Spec pas affine}The notation $\Spec(A_p \oplus A_n[-d])$ is slightly abusive. The cdga $A_p \oplus A_n[-d]$ is not concentrated in non positive degrees. In particular, the derived stack $\Spec(A_p \oplus A_n[-d])$ is not a derived affine scheme.
It behaves like one though, regarding its derived category:
\[
\Qcoh(\Spec(A_p \oplus A_n[-d])) \simeq \dgMod_{A_p \oplus A_n[-d]}
\]
\end{rmq}

Let us define the ind-pro-algebra 
\[
\Oo_{\formalsphere^d} = \colim_n \lim_{p \geq n} A_p \oplus A_n[-d]
\]
where $A_p \oplus A_n[-d]$ is the trivial square zero extension of $A_p$ by the module $A_n[-d]$.
For any $m \in \N$, let us denote by $\formalsphere^d_m$ the ind-stack
\[
\formalsphere^d_m = \colim_{p \geq m} \Spec(A_p \oplus A_m[-d])
\]

\begin{df}\label{dfbubble}
\todo{écrire ça mieux}Let $T$ be a derived Artin stack. We define the $d$-bubble stack of $T$ as the mapping ind-pro-stack
\glsadd{bubble}
\[
\bubblestack(T) = \Mapstack(\formalsphere^d, T) \colon \Spec B \mapsto \colim_n \lim_{p \geq n} T \left(B \otimes (A_p \oplus A_n[-d]) \right)
\]
Again, the cdga $A_p \oplus A_n[-d]$ is not concentrated in non positive degree. This notation is thus slightly abusive and by $T(B \otimes (A_p \oplus A_n[-d]))$ we mean
\[
\Map(\Spec(A_p \oplus A_n[-d]) \times \Spec B,X)
\]
We will denote by $\bar \bubblestack(T)$ the diagram $\N \to \Prou U \dSt_k$ of whom $\bubblestack(T)$ is a colimit in $\IP\dSt_k$.
Let us also denote by $\bubblestack_m(T)$ the mapping pro-stack
\[
\bubblestack_m(T) = \Map(\formalsphere^d_m, T) \colon \Spec B \mapsto \lim_{p \geq m} T \left(B \otimes (A_p \oplus A_m[-d]) \right)
\]
and $\bar \bubblestack_m(T) \colon \{ p \in \N | p \geq m\}\op \to \dSt_S$ the corresponding diagram.
In particular
\[
\bubblestack_0(T) = \Map(\formalsphere^d_0, T) \colon \Spec B \mapsto \lim_{p} T \left(B \otimes A_p\right)
\]
Those stacks come with natural maps
\[
\mymatrix{
\bubblestack_0(T) \ar[r]^-{s_0} & \bubblestack(T) \ar[r]^-r & \bubblestack_0(T)
}
\]
\[
\mymatrix{
\bubblestack_m(T) \ar[r]^-{s_m} & \bubblestack(T)
}
\]
\end{df}

\begin{prop}\label{B-and-L-IP}
If $T$ is an affine scheme of finite type, the bubble stack $\bubblestack(T)$ is the product in ind-pro-stacks
\[
\mymatrix{
\bubblestack(T) \ar[r] \ar[d] \cart[3] & {} \underline{\kaploop}_V^d(T) \ar[d] \\ {} \underline \kaploop_V^d(T) \ar[r] & {} \underline \kaploop_U^d(T)
}
\]
\end{prop}

\begin{proof}
There is a natural map $V_k^d \to \formalsphere^d$ induced by the morphism
\[
\colim_n \lim_{p \geq n} A_p \oplus A_n[-d] \to \lim_p A_p
\]
Because $T$ is algebraisable, it induces a map $\bubblestack(T) \to \underline \kaploop_V^d(T)$ and thus a diagonal morphism
\[
\delta \colon \bubblestack(T) \to \underline \kaploop_V^d(T) \timesunder[\underline \kaploop_U^d(T)] \underline \kaploop_V^d(T)
\]
We will prove that $\delta$ is an equivalence.
Note that because $T$ is a finite limit of copies of $\A^1$, we can restrict to the case $T = \A^1$.
Let us first compute the fibre product $Z = \underline \kaploop^d_V(\A^1) \times_{\underline \kaploop_U^d(\A^1)} \underline \kaploop_V^d(\A^1)$. It is the pullback of ind-pro-stacks
\[
\mymatrix{
Z \ar[r] \ar[d] \cart &
\displaystyle \lim_p \Spec\left( k[ a_{\el{\alpha}{d}}, 0 \leq \alpha_i \leq p] \right) \ar[d] \\
\displaystyle \lim_p \Spec\left( k[ a_{\el{\alpha}{d}}, 0 \leq \alpha_i \leq p] \right) \ar[r] &
\displaystyle \colim_n \lim_p \lim_{I \subset J} \Spec\left( k[ a_{\el{\alpha}{d}}, -n \delta_{i \in I} \leq \alpha_i \leq p] \right)
}
\]
where $J = \{1, \dots , d\}$ and $\delta_{i \in I} = 1$ if $i \in I$ and $0$ otherwise.
For any subset $K \subset J$ we define $M_K^{p,n}$ to be the free complex generated by the  symbols
\[
\{a_{\el{\alpha}{d}}, -n \leq \alpha_i < 0 \text{ if } i \in K, 0 \leq \alpha_i \leq p \text{ otherwise}\}
\]
We then have the cartesian diagram
\[
\mymatrix{
Z \ar[r] \ar[d] \cart & \lim_p \Spec\left( k[ M^{p,0}_\emptyset ] \right) \ar[d] \\
\lim_p \Spec\left( k[ M^{p,0}_\emptyset ] \right) \ar[r] &
\colim_n \lim_p \lim_{I \subset J} \Spec\left(k\left[ \bigoplus_{K \subset I} M_K^{p,n} \right] \right)
}
\]
Using \autoref{subsets-colim} we get 
\[
Z \simeq \colim_n \lim_p \Spec\left( k\left[M^{p,0}_\emptyset \oplus M^{0,n}_J[d]\right] \right)
\]
\end{proof}

\begin{rmq}
Let us consider the map $\lim_p A_p \to A_0 \simeq k$ mapping a formal serie to its coefficient of degree $0$. The $(\lim A_p)$-ind-module $\colim A_n[-d]$ is endowed with a natural map to $k[-d]$. This induces a morphism $\Oo_{\formalsphere^d} \to k \oplus k[-d]$ and hence a map $\mathrm S^d \to \formalsphere^d$, where $\mathrm S^d$ is the topological sphere of dimension $d$. We then have a rather natural morphism
\[
\bubblestack^d(X) \to \Mapstack(\mathrm S^d,X)
\]
\end{rmq}

\subsection{Its tangent is a Tate module}

We already know from \autoref{map-tate} that the bubble stack is a Tate stack. We give here another decomposition of its tangent complex. We will need it when proving $\bubblestack^d(T)$ is symplectic.

\begin{prop}\label{prop-formal-tate}
Let us assume that the Artin stack $T$ is locally of finite presentation.
The ind-pro-stack $\bubblestack^d(T)$ is then a Tate stack.
Moreover for any $m \in \N$ we have an exact sequence
\[
\mymatrix{
s_m^*r^* \Lcot_{\bubblestack^d(T)_0} \ar[r] & s_m^*\Lcot_{\bubblestack^d(T)} \ar[r] & s_m^*\Lcot_{\bubblestack^d(T)/\bubblestack^d(T)_0}
}
\]
where the left hand side is an ind-perfect module and the right hand side is a pro-perfect module.
\end{prop}
\begin{proof}
\todo{Finir notations}Throughout this proof, we will write $\bubblestack$ instead of
$\bubblestack^d(T)$ and $\bubblestack_m$ instead of $\bubblestack^d(T)_m$ for any $m$.
Let us first remark that $\bubblestack$ is an Artin ind-pro-stack locally of finite presentation.
It suffices to prove that $s_m^* \Lcot_{\bubblestack}$ is a Tate module on $\bubblestack_m$, for any $m \in \N$. We will actually prove that it is an elementary Tate module.
We consider the map
\[
s_m^* r^* \Lcot_{\bubblestack_0} \to s_m^* \Lcot_{\bubblestack}
\]
It is by definition equivalent to the natural map
\[
\cotangent^{\Pro}_{\bubblestack_m}(\bubblestack_0) \to^f \lim \cotangent^{\Pro}_{\bubblestack_m}(\bar \bubblestack_{\geq m}(T))
\]
where $\bar \bubblestack_{\geq m}(T)$ is the restriction of $\bar \bubblestack(T)$ to $\{ n \geq m \} \subset \N$.
Let $\phi$ denote the diagram
\[
\phi \colon \{ n \in \N | n \geq m\}\op \to \IPerf(\bubblestack_m(T))
\]
obtained as the cokernel of $f$. It is now enough to prove that $\phi$ factors through $\Perf(\bubblestack_m(T))$. Let $n \geq m$ be an integer and let $g_{mn}$ denote the induced map $\bubblestack_m(T) \to \bubblestack_n(T)$. We have  an exact sequence
\[
s_m^* r^* \Lcot_{\bubblestack_0(T)} \simeq g_{mn}^* s_n^* r^* \Lcot_{\bubblestack_0(T)} \to g_{m,n}^* \Lcot_{\bubblestack_n(T)} \to \phi(n)
\]
Let us denote by $\psi(n)$ the cofiber
\[
s_n^* r^* \Lcot_{\bubblestack_0(T)} \to \Lcot_{\bubblestack_n(T)} \to \psi(n)
\]
so that $\phi(n) \simeq g_{mn}^* \psi(n)$.
This sequence is equivalent to the colimit (in $\IPerf(\bubblestack_n(T))$) of a cofiber sequence of diagrams $\{ p \in \N | p \geq n \}\op \to \Perf(\bubblestack_n(T))$
\[
\lambda^{\Pro}_{\bubblestack_n(T)}(\bar \bubblestack_0(T)) \to \lambda^{\Pro}_{\bubblestack_n(T)}(\bar \bubblestack_n(T) ) \to \bar \psi(n)
\]
It suffices to prove that the diagram $\bar \psi(n) \colon \{ p \in \N | p \geq n \}\op \to \Perf(\bubblestack_n(T))$ is (essentially) constant.
Let $p \in \N$, $p \geq n$. The perfect complex $\bar \psi(n)(p)$ fits in the exact sequence
\[
t_{np}^* \varepsilon_{np}^* \Lcot_{\bubblestack_{0,p}(T)} \to \pi_{n,p}^* \Lcot_{\bubblestack_{n,p}(T)} \to \bar \psi(n)(p)
\]
where $t_{np} \colon \bubblestack_n(T) \to \bubblestack_{n,p}(T)$ is the canonical projection and $\varepsilon_{np} \colon \bubblestack_{n,p}(T) \to \bubblestack_{0,p}(T)$ is induced by the augmentation $\Oo_{S_{n,p}} \to \Oo_{S_{0,p}}$.
It follows that $\bar \psi(n)(p)$ is equivalent to
\[
t_{np}^* \Lcot_{\bubblestack_{n,p}(T)/\bubblestack_{0,p}(T)}
\]
Moreover, for any $q \geq p \geq n$, the induced map $\bar \psi(n)(p) \to \bar \psi(n)(q)$ is obtained (through $t_{nq}^*$) from the cofiber, in $\Perf(\bubblestack_{n,q}(T))$
\[
\mymatrix@R=2mm{
\alpha_{npq}^* \varepsilon_{np}^* \Lcot_{\bubblestack_{0,p}(T)} \ar[r] \ar@{}[rdddd]|*{(\sigma)} & \alpha_{npq}^* \Lcot_{\bubblestack_{n,p}(T)} \ar[dddd] \ar[r] & \alpha_{npq}^* \Lcot_{\bubblestack_{n,p}(T)/\bubblestack_{0,p}(T)} \ar[dddd] \\
\varepsilon_{nq}^* \alpha_{0pq}^* \Lcot_{\bubblestack_{0,p}(T)} \ar@{=}[u] \ar[ddd] \\ \\ \\
\varepsilon_{nq}^* \Lcot_{\bubblestack_{0,q}(T)} \ar[r] & \Lcot_{\bubblestack_{n,q}(T)} \ar[r] & \Lcot_{\bubblestack_{n,q}(T)/\bubblestack_{0,q}(T)}
}
\]
where $\alpha_{npq}$ is the map $\bubblestack_{n,q}(T) \to \bubblestack_{n,p}(T)$.
Let us denote by $(\sigma)$ the square on the left hand side above.
Let us fix a few more notations
\[
\mymatrix@R=7mm@C=5mm{
&
\bubblestack_{n,p}(T) \times S_{0,p} \ar[dd]_{\varphi_{np}} \ar[dl]_(0.6){a_{0p}} & &
\bubblestack_{n,q}(T) \times S_{0,p} \ar[dd]_{\psi_{npq}} \ar[ll] \ar[rr] & &
\bubblestack_{n,q}(T) \times S_{0,q} \ar[dd]^{\varphi_{nq}} \\
S_{0,p} \ar[dd]_{\xi_{np}} \\ &
\bubblestack_{n,p}(T) \times S_{n,p} \ar[dl]^{a_{np}} \ar@{-}[d] & &
\bubblestack_{n,q}(T) \times S_{n,p} \ar[ll] \ar[rr]^{b_{npq}} \ar@{-}[d] \ar[ld] & &
\bubblestack_{n,q}(T) \times S_{n,q} \ar[dl]^{a_{nq}} \ar[dd]^{\varpi_{nq}} \ar[dr]^(0.7){\ev_{nq}} \\
S_{n,p} & \ar[d]^(0.35){\varpi_{np}} & S_{n,p} \ar@{-}[ll]_(0.4)= & \ar[d] & S_{n,q} \ar[ll]_(0.3){\beta_{npq}} & & T\\ &
\bubblestack_{n,p}(T) & & \bubblestack_{n,q}(T) \ar[ll]_{\alpha_{npq}} \ar@{-}[rr]_{=} & & \bubblestack_{n,q}(T)
}
\]
The diagram $(\sigma)$ is then dual to the diagram
\[
\mymatrix{
\alpha_{npq}^* \varepsilon_{np}^* {\varpi_{0p}}_* \ev_{0p}^* \T_T &
\alpha_{npq}^* {\varpi_{np}}_*  \ev_{np}^* \T_T \ar[l] \\
\varepsilon_{nq}^* {\varpi_{0q}}_* \ev_{0q}^* \T_T \ar[u] &
{\varpi_{nq}}_* \ev_{nq}^* \T_T \ar[l] \ar[u]
}
\]
\todo{En faire un lemme ?}Moreover, the functor $\varpi_{np}$ (for any $n$ and $p$) satisfies the base change formula. This square is thus equivalent to the image by ${\varpi_{nq}}_*$ of the square
\[
\mymatrix{
{\psi_{npq}}_* {b_{npq}}_* b_{npq}^* \psi_{npq}^* \ev_{nq}^* \T_T &
{b_{npq}}_* b_{npq}^* \ev_{nq}^* \T_T \ar[l] \\
{\varphi_{nq}}_* \varphi_{nq}^* \ev_{nq}^* \T_T \ar[u] &
\ev_{nq}^* \T_T \ar[l] \ar[u]
}
\]
Using now the \todo{En faire un lemme ?}projection and base change formulae along the morphisms $\varphi_{nq}$, $b_{npq}$ and $\psi_{npq}$, we see that this last square is again equivalent to
\[
\mymatrix{
(a_{nq}^* {\beta_{npq}}_* {\xi_{np}}_* \Oo_{S_{0,p}}) \otimes (\ev_{nq}^* \T_T) &
(a_{nq}^* {\beta_{npq}}_* \Oo_{S_{n,p}}) \otimes (\ev_{nq}^* \T_T) \ar[l]
\\
(a_{nq}^* {\xi_{nq}}_* \Oo_{S_{0,q}}) \otimes (\ev_{nq}^* \T_T) \ar[u] &
(a_{nq}^* \Oo_{S_{n,q}}) \otimes (\ev_{nq}^* \T_T) \ar[l] \ar[u]
}
\]
We therefore focus on the diagram
\[
\mymatrix{
\Oo_{S_{n,q}} \ar[r] \ar[d] & {\xi_{nq}}_* \Oo_{S_{0,q}} \ar[d] \\
{\beta_{npq}}_* \Oo_{S_{n,p}} \ar[r] & {\beta_{npq}}_* {\xi_{np}}_* \Oo_{S_{0,p}}
}
\]
By definition, the fibres of the horizontal maps are both equivalent to 
$A_n[-d]$
and the map induced by the diagram above is an equivalence.
We have proven that for any $q \geq p \geq n$ the induced map $\bar \psi(n)(p) \to \bar \psi(n)(q)$ is an equivalence.
It implies that $\Lcot_{\bubblestack(T)}$ is a Tate module.
\end{proof}

\subsection{A symplectic structure (shifted by $d$)}
In this subsection, we will prove the following
\begin{thm}\label{B-symplectic}
Assume $T$ is $q$-shifted symplectic.
The ind-pro-stack $\bubblestack^d(T)$ admits a symplectic Tate structure shifted by $q-d$.
Moreover, for any $m \in \N$ we have an exact sequence
\[
s_m^* r^* \Lcot_{\bubblestack^d(T)_0} \to s_m^* \Lcot_{\bubblestack^d(T)} \to s_m^* r^* \T_{\bubblestack^d(T)_0}[q-d]
\]
\end{thm}

\begin{proof}
Let us start with the following remark: the residue map $r_n \colon A_n \to k = A_1$ defined in \autoref{dual-over-k} defines a map $\Oo_{\formalsphere^d} \to k[-d]$.
From \autoref{ipdst-form}, we have a $(q-d)$-shifted closed $2$-form on $\bubblestack^d(T)$. We have a morphism from \autoref{prop-formsareforms}
\[
\Oo_{\bubblestack^d(T)}[q-d] \to \Lcot_{\bubblestack^d(T)} \otimes \Lcot_{\bubblestack^d(T)}
\]
in $\PIPerf(\bubblestack^d(T))$.
Let $m \in \N$. We get a map
\[
\Oo_{\bubblestack^d(T)_m}[q-d] \to s_m^* \Lcot_{\bubblestack^d(T)} \otimes s_m^* \Lcot_{\bubblestack^d(T)}
\]
and then
\[
s_m^* \T_{\bubblestack^d(T)} \otimes s_m^* \T_{\bubblestack^d(T)} \to \Oo_{\bubblestack^d(T)_m}[q-d]
\]
in $\IPPerf(\bubblestack^d(T)_m)$. We consider the composite map
\[
\theta \colon s_m^* \T_{\bubblestack^d(T)/\bubblestack^d(T)_0} \otimes s_m^* \T_{\bubblestack^d(T)/\bubblestack^d(T)_0} \to
s_m^* \T_{\bubblestack^d(T)} \otimes s_m^* \T_{\bubblestack^d(T)} \to \Oo_{\bubblestack^d(T)_m}[q-d]
\]
Using the \autoref{describe-form} and the proof of \autoref{prop-formal-tate} we see that $\theta$ is induced by the morphisms (varying $n$ and $p$)
\[
\mymatrix{
{\varpi_{np}}_* \left( E \otimes E \otimes \ev_{np}^* \left( \T_T \otimes \T_T \right) \right) \ar[r]^-A & {\varpi_{np}}_* \left( E \otimes E [q] \right) \ar[r]^-B & {\varpi_{np}}_* \left( \Oo_{\bubblestack^d(T)_{np} \times S_{n,p}} [q] \right)
}
\]
where $E = a_{np}^* {\xi_{np}}_* {h_{np}}_* \gamma_n^! \Oo_{\A^d}$ and the map $A$ is induced by the symplectic form on $T$. The map $B$ is induced by the multiplication in $\Oo_{S_{n,p}}$.
This sheaf of functions is a trivial square zero extension of augmentation ideal ${\xi_{np}}_* {h_{np}}_* \gamma_n^! \Oo_{\A^d}$ and $B$ therefore vanishes.
It follows that the morphism
\[
s_m^* \T_{\bubblestack^d(T)} \otimes s_m^* \T_{\bubblestack^d(T)/\bubblestack^d(T)_0} \to
s_m^* \T_{\bubblestack^d(T)} \otimes s_m^* \T_{\bubblestack^d(T)} \to \Oo_{\bubblestack^d(T)_m}[q-d]
\]
factors through $s_m^* \T_{\bubblestack^d(T)_0} \otimes s_m^* \T_{\bubblestack^d(T)/\bubblestack^d(T)_0}$.
Now using \autoref{prop-formal-tate} we get a map of exact sequences in the category of Tate modules over $\bubblestack^d(T)_m$
\[
\mymatrix{
s_m^* \T_{\bubblestack^d(T)/\bubblestack^d(T)_0} \ar[r] \ar[d]_{\tau_m} &
s_m^* \T_{\bubblestack^d(T)} \ar[r] \ar[d] & s_m^* r^* \T_{\bubblestack^d(T)_0} \ar[d] \\
s_m^* r^* \Lcot_{\bubblestack^d(T)_0}[d-q] \ar[r] & s_m^* \Lcot_{\bubblestack^d(T)}[d-q] \ar[r] & s_m^* \Lcot_{\bubblestack^d(T)/\bubblestack^d(T)_0}[d-q]
}
\]
where the maps on the sides are dual one to another.
It therefore suffices to see that the map $\tau_m \colon s_m^* \T_{\bubblestack^d(T)/\bubblestack^d(T)_0} \to s_m^* r^* \Lcot_{\bubblestack^d(T)_0}[d-q]$ is an equivalence.
We now observe that $\tau_m$ is a colimit indexed by $p \geq m$ of maps
\[
g_{pm}^* t_{pp}^* \left( \varepsilon_{pp}^* \Lcot_{\bubblestack^d(T)_{0p}} \to \T_{\bubblestack^d(T)_{pp}/\bubblestack^d(T)_{0p}} \right)
\]
Let us fix $p \geq m$ and $G = a_{pp}^* {\xi_{pp}}_* \Oo_{S_{0p}}$. The map $
F_p \colon \T_{\bubblestack^d(T)_{pp}/\bubblestack^d(T)_{0p}} \to \varepsilon_{pp}^* \Lcot_{\bubblestack^d(T)_{0p}}$ at hand is induced by the pairing
\[
\T_{\bubblestack^d(T)_{pp}/\bubblestack^d(T)_{0p}} \otimes \varepsilon_{pp}^* \T_{\bubblestack^d(T)_{0p}} \simeq
\mymatrix{
{\varpi_{pp}}_* \left(E \otimes \ev_{pp}^* \T_T \right) \otimes {\varpi_{pp}}_* \left( G \otimes \ev_{pp}^* \T_T \right) \ar[d] \\
{\varpi_{pp}}_* \left( E \otimes \ev_{pp}^* \T_T \otimes G \otimes \ev_{pp}^* \T_T \right) \ar[d] \\
{\varpi_{pp}}_* \left( E \otimes G \right)[q] \ar[d] \\
{\varpi_{pp}}_* \left( \Oo_{\bubblestack^d(T)_{pp} \times S_{pp}} \right) [q] \ar[d] \\
\Oo_{\bubblestack^d(T)_{pp}} [q-d]
}
\]
We can now conclude using \autoref{dual-over-k}.
\end{proof}
%

\end{chap-loops}

\chapter{Tangent Lie algebra}\label{chapterlie}%
\begin{chap-tgtlie}
We will study in this last chapter the tangent complex of a derived Artin stack. The main theorem builds a Lie structure on the shifted tangent complex $\T_X[-1]$ of a derived Artin stack $X$ locally of finite presentation.
We will also prove that for any perfect complex $E$ on $X$, its Atiyah class defines a Lie action of the tangent $\T_X[-1]$ on $E$.
In this chapter, $k$ will be a field of characteristic zero.
Given $A\in \cdga_k$, we will use the following notations
\begin{itemize}
\item The $(\infty,1)$-category $\dgMod_A$ of (unbounded) dg-modules over $A$ ;
\item The $(\infty,1)$-category $\cdgaunbounded_A$ of (unbounded) commutative dg-algebras over $A$ ;\glsadd{cdgaunbounded}
\item The $(\infty,1)$-category $\cdga_A$ of commutative dg-algebras over $A$ cohomologically concentrated in non positive degree ;
\item The $(\infty,1)$-category $\dgAlg_A$ of (neither bounded nor commutative) dg-algebras over $A$ ;
\item The $(\infty,1)$-category $\dgLie_A$ of (unbounded) dg-Lie algebras over $A$.\glsadd{dglie}
\end{itemize}
Each one of those $(\infty,1)$-categories appears as the underlying $(\infty,1)$-category of a model category. We will denote by $\MCdgMod_A$, $\MCcdgaunbounded_A$, $\MCcdga_A$, $\MCdgAlg_A$ and $\MCdgLie_A$ the model categories.
\todo{nom: formal group?}\section{Lie algebras and formal stacks over a cdga}
\newcommand{\adjoint}{\operatorname{D}}
In this part we will mimic a construction found in Lurie's \cite{lurie:dagx}

\begin{thm}[Lurie]\label{thm-lurie}
Let $k$ be a field of characteristic zero.
There is an adjunction of $(\infty,1)$-categories:
\[
\coho_k \colon \dgLie_k \rightleftarrows {\left(\quot{\cdgaunbounded_k}{k}\right)}\op \,: \adjoint_k
\]
Whenever $L$ is a dg-Lie algebra:
\begin{enumerate}
\item If $L$ is freely generated by a dg-module $V$ then the algebra $\coho_k(L)$ is equivalent to the trivial square zero extension $k \oplus \dual V[-1]$.
\item If $L$ is concentrated in positive degree and every vector space $L^n$ is finite dimensional, then the adjunction morphism $L \to \adjoint_k \coho_k L$ is an equivalence.
\end{enumerate}
\end{thm}
The goal is to extend this result to more general basis, namely a commutative dg-algebra over $k$ concentrated in non positive degree.
The existence of the adjunction and the point (i) will be proved over any basis, the analog of point (ii) will need the base dg-algebra to be noetherian.

Throughout this section, $A$ will be a commutative dg-algebra concentrated in non-positive degree over the base field $k$ (still of characteristic zero).
\subsection{Algebraic theory of dg-Lie algebras}
Let us consider the adjunction $\libre \colon \dgLie_A \rightleftarrows \dgMod_A \noloc \oubli$ of $(\infty,1)$-categories.
\begin{df}
Let $\dgModLib_A$ denote the full sub-category of $\dgMod_A$ spanned by the free dg-modules of finite type whose generators are in positive degree.
An object of $\dgModLib_A$ is thus (equivalent to) the dg-module
\[
\bigoplus_{i=1}^n A^{p_i}[-i]
\]
for some $n \geq 1$ and some family $(\el{p}{n})$ of non negative integers.

Let $\dgLieLib_A$ denote the essential image of $\dgModLib_A$ in $\dgLie_A$ by the functor $\libre$.
\end{df}
Let us recall that $\sifted(\Cc)$ stands for the sifted completion of a category $\Cc$ with finite coproducts.
\begin{prop} \label{dglie-algtheory}
The Yoneda functors
\begin{align*}
&\dgMod_A \to \sifted(\dgModLib_A) = \Fct^{\times} \left( {\left( \dgModLib_A \right)}\op,\sSets \right) \\
&\dgLie_A \to \sifted(\dgLieLib_A) = \Fct^{\times} \left( {\left( \dgLieLib_A \right)}\op,\sSets \right)
\end{align*}
are equivalences of $(\infty,1)$-categories.
\end{prop}
\begin{rmq}\label{dglie-sifted}
The above proposition implies that every dg-Lie algebra is colimit of a \emph{sifted} diagram of objects in $\dgLieLib_A$.
\end{rmq}
\begin{proof}
Every dg-module can be obtained as the colimit of a diagram in $\dgModLib_A$ and objects of $\dgModLib_A$ are compact projective in $\dgMod_A$ (an object is compact projective if the functor it corepresents preserves sifted colimits). The proposition \cite[5.5.8.25]{lurie:htt} makes the first functor an equivalence.

The forgetful functor $\oubli$ is monadic. Every dg-Lie algebra can thus be obtained as a colimit of a simplicial diagram with values in the $(\infty,1)$-category of free dg-Lie algebras (see \cite[6.2.2.12]{lurie:halg}). From those two facts we deduce that every dg-Lie algebra is a colimit of objects in $\dgLieLib_A$.
The forgetful functor $\oubli$ preserves sifted colimits and objects in $\dgLieLib_A$ are thus compact projective in $\dgLie_A$.
\end{proof}

\begin{rmq}
The above proposition implies that when $A \to B$ is a morphism in $\MCcdga_k$, the following square of $(\infty,1)$-categories commutes:
\[
\mymatrix{
\dgLie_A \ar@{-}[r]^-\sim \ar[d]_{B \otimes_A -} &
\sifted(\dgLieLib_A)
\ar[d]_{{(B \otimes_A -)}_!} \\
\dgLie_B \ar@{-}[r]^-\sim &
\sifted(\dgLieLib_B)
}
\]
The following proposition actually proves that this comes from a natural transformation between functors $\cdga_k \to \inftyCat$.
\end{rmq}

\begin{prop} \label{commute-dglie}
There are $(\infty,1)$-categories $\int \dgLie$ and $\int \sifted(\dgLieLib)$, each endowed with a coCartesian fibration to $\MCcdga_k$, respectively representing the functors $A \mapsto \dgLie_A$ and $A \mapsto \sifted(\dgLieLib_A)$. There is an equivalence over $\MCcdga_k$:
\[
\mymatrix{
\int \sifted(\dgLieLib) \ar[rr]^-\sim \ar[rd] && \int \dgLie \ar[dl] \\ & \MCcdga_k &
}
\]
This induces an equivalence of functors $\MCcdga_k \to \PresLeftu U$ which moreover descend to a natural transformation
\[
\mymatrix{\cdga_k \dcell[r][][][\sim][=>][12pt] & \PresLeftu U}
\]
\end{prop}

\begin{rmq}
This proposition establishes an equivalence of functors $\cdga_k \to \PresLeftu U$ between $A \mapsto \dgLie_A$ and $A \mapsto \sifted(\dgLieLib_A)$.
\end{rmq}

\begin{proof}
Let us define $\int \MCdgLie$ as the following category.
\begin{itemize}
\item An object is couple $(A,L)$ where $A \in \MCcdga_k$ and $L \in \MCdgLie_A$.
\item A morphism $(A,L) \to (B,L')$ is a morphism $A \to B$ and a morphism of $A$-dg-Lie algebras $L \to L'$
\end{itemize}
We define $\int \dgLie$ to be the localization of $\int \MCdgLie$ along quasi-isomorphisms of dg-Lie algebras.
Using \cite[2.4.19]{lurie:dagx}, there is a coCartesian fibration of $(\infty,1)$-categories $p \colon \int \dgLie \to \MCcdga_k$.

This coCartesian fibration defines a functor $\dgLie \colon \MCcdga_k \to \inftyCatu V$ mapping a cdga $A$ to $\dgLie_A$ and a morphism $A \to B$ to the base change functor.
It comes by construction with a subfunctor
\[
\dgLieLib \colon \MCcdga_k \to \inftyCatu U
\]
Let us denote by $\sifted(\dgLieLib)$ its composite functor with 
\[
\sifted \colon \inftyCatu U \to \inftyCatu V
\]
Let us denote by $\int \dgLieLib \to \MCcdga_k$ the coCartesian associated with the functor $\dgLieLib$ and by $\int \sifted(\dgLieLib)$ that associated to $\sifted(\dgLieLib)$.

We get a diagram
\[
\shorthandoff{:;!?}
\xy <6mm,0cm>:
(1,0)*+{\int \sifted(\dgLieLib)}="0",
(5,-2)*+{\int \dgLieLib}="1",
(-3,-2)*+{\int \dgLie}="2",
(3,-5)*+{\MCcdga_k}="3",
\ar "1";"0" _(0.4){G}
\ar@{-->} "0";"2" _(0.6){F}
\ar "1";"2" ^-{F_0}
\ar "0";"3" |!{"1";"2"}\hole
\ar "1";"3"
\ar "2";"3"
\endxy
\]
The functor $F_0$ has a relative left Kan extension $F$ along $G$ (see \cite[4.3.2.14]{lurie:htt}).
Because of \autoref{dglie-algtheory}, it now suffices to prove that $F$ preserves coCartesian morphisms. This is a consequence of \cite[4.3.1.12]{lurie:htt}.
We get the announced equivalence of functors
\[
\mymatrix{\MCcdga_k \dcell[r][][][\sim][=>][12pt] & \PresLeftu U}
\]
We now observe that both the involved functors map quasi-isomorphisms of $\MCcdga_k$ to equivalences of categories. It follows that this natural transformation factors through the localisation $\cdga_k$ of $\MCcdga_k$.
\end{proof}

\subsection{Poincaré-Birkhoff-Witt over a cdga in characteristic zero}
In this short parenthesis, we prove the PBW-theorem over a cdga of characteristic $0$. The prove is a simple generalisation of that of Paul M. Cohn over a algebra in characteristic $0$ -- see \cite[theorem 2]{cohn:pbw}.
\begin{thm}
Let $A$ be a commutative dg-algebra over a field $k$ of characteristic zero.
For any dg-Lie algebra $L$ over $A$, there is a natural isomorphism of $A$-dg-modules
\[
\Sym_A L \to \Envel_A L
\]
\end{thm}
\begin{proof}
Recall that $\Envel_A L$ can be endowed with a bialgebra structure such that an element of $L$ is primitive in $\Envel_A L$.
The morphism $L \to \Envel_A L$ therefore induces a morphism of dg-bialgebras
$\Tens_A L \to \Envel_A L$ which can be composed with the symmetrization map $\Sym_A L \to \Tens_A L$ given by
\[
\el{x}{n}[\otimes] \mapsto \frac{1}{n!} \sum_\sigma \varepsilon(\sigma,\bar x) x_{\sigma(1)} \otimes \dots \otimes x_{\sigma(n)}
\]
where $\sigma$ varies in the permutation group $\mathfrak{S}_n$ and where $\varepsilon(-,\bar x)$ is a group morphism $\mathfrak S_n \to \{-1,+1\}$ determined by the value on the permutations $(i~j)$
\[
\varepsilon((i~j),\bar x) = (-1)^{|x_i||x_j|}
\]
We finally get a morphism of $A$-dg-coalgebras $\phi \colon \Sym_A L \to \Envel_A L$.
Let us take $n \geq 1$ and let us assume that the image of $\phi$ contains $\Envel_A^{\leq n-1} L$.
The image of a symmetric tensor 
\[
\el{x}{n}[\otimes^s]
\]
by $\phi$ is the class
\[
\left[\frac{1}{n!}\sum_\sigma \varepsilon(\sigma, \bar x) \el{x}[\sigma(1)]{\sigma(n)}[\otimes] \right]
\]
which can be rewritten
\[
\left[
\el{x}{n}[\otimes] + \sum_\alpha \pm \frac{1}{n!} \el{y^\alpha}{n-1}[\otimes]
\right]
\]
where $y^\alpha_i$ is either some of the $x_j$'s or some bracket $[x_j,x_k]$.
This implies that $\Envel_A^{\leq n} L$ is in the image of $\phi$ and we therefore show recursively that $\phi$ is surjective (the filtration of $\Envel_A L$ is exhaustive).

There is moreover a section
\[
\Envel_A L \to \Sym_A L
\]
for which a formula is given in \cite{cohn:pbw} and which concludes the proof.
\end{proof}

\subsection{Almost finite cellular objects}
Let $A$ be a commutative dg-algebra over $k$.
\begin{df}
Let $M$ be an $A$-dg-module.
\begin{itemize}
\item We will denote by $\Mc(M)$ the mapping cone of the identity of $M$.
\item We will say that $M$ is an almost finite cellular object if there is a diagram
\[
0 \to A^{p_0} = M_0 \to M_1 \to \dots 
\]
whose colimit is $M$ and such that for any $n$, the morphism $M_n \to M_{n+1}$ fits into a cocartesian diagram
\[
\mymatrix{
A^{p_n}[n] \ar[r] \ar[d] & M_n \ar[d] \\ \Mc(A^{p_n}[n]) \ar[r] & M_{n+1} \cocart
}
\]
\end{itemize}
\end{df}

\begin{rmq}
The definition above states that a dg-module $M$ is an almost finite cellular object if it is obtained from $0$ by gluing a finite number of cells in each degree (although the total number of cells is not necessarily finite).
\end{rmq}

\begin{lem} Let $\phi \colon M \to N$ be a morphism of $A$-dg-modules.
\begin{itemize}
\item If $M$ is an almost finite cellular object then it is cofibrant.
\item Assume both $M$ and $N$ are almost finite cellular objects.
The morphism $\phi$ is a quasi-isomorphism if and only if for any field $l$ and any morphism $A \to l$ the induced map $\phi_k \colon M \otimes_A l \to N \otimes_A l$ is a quasi-isomorphism.
\end{itemize}
\end{lem}
\begin{proof}
Assume $M$ is an almost finite cellular object. Let us consider a diagram $M \to Q \from P$ where the map $P \to Q$ is a trivial fibration. Each morphism $M_n \to M_{n+1}$ is a cofibration and there thus is a compatible family of lifts $(M_n \to P)$. This gives us a lift $M \to P$. The $A$-dg-module $M$ is cofibrant.

Let now $\phi$ be a morphism $M \to N$ between almost finite cellular objects and that the morphism $\phi_l$ is a quasi-isomorphism for any field $l$ under $A$.
Replacing $M$ with the cone of $\phi$ (which is also an almost finite cellular object) we may assume that $N$ is trivial.
Notice first that an almost finite cellular object is concentrated in non positive degree.
Notice also that for any $n$ the truncated morphism $\phi^{\geq -n} \colon M_{n+1}^{\geq -n} \to M^{\geq -n}$ is a quasi-isomorphism.
We then have 
\[
0 \simeq \homol^j\left(M \otimes_A l\right) \simeq \homol^j\left(M_n \otimes_A l\right)
\]
whenever $-n < j \leq 0$ and for any $A \to l$. Since $\homol^j(M_n \otimes_A l) \simeq 0$ if $j \leq -n-2$ the $A$-dg-module $M_n$ is perfect and of amplitude $[-n-1,-n]$.
This implies the existence of two projective modules $P$ and $Q$ (ie retracts of some power of $A$) fitting in a cofier sequence (see \cite{toen:ttt})
\[
P[n] \to M_n \to Q[n+1]
\]
The dg-module $M_n$ is then cohomologically concentrated in degree $]-\infty,-n]$, and so is $M$.
This being true for any $n$ we deduce that $M$ is contractible.
\end{proof}

The next lemma requires the base $A$ to be noetherian.
Recall that $A$ is noetherian if $\homol^0(A)$ is noetherian and if $\homol^n(A)$ is trivial when $n$ is big enough and of finite type over $\homol^0(A)$ for any $n$.
\begin{lem}\label{afp-cotangent}
Assume $A$ is noetherian.
If $B$ is an object of $\quot{\MCcdga_A}{A}$ such that:
\begin{itemize}
\item The $\homol^0(A)$-algebra $\homol^0(B)$ is finitely presented,
\item For any $n \geq 1$ the $\homol^0(B)$-module $\homol^{-n}(B)$ is of finite type,
\end{itemize}
then the $A$-dg-module $\Lcot_{B/A} \otimes_B A$ is an almost finite cellular object.
\end{lem}

\begin{proof}
Because the functor $(A \to B \to A) \mapsto \Lcot_{B/A} \otimes_B A$ preserves colimits, it suffices to prove that $B$ is an almost finite cellular object in $\quot{\MCcdga_A}{A}$. This means we have to build a diagram
\[
B_0 \to B_1 \to \dots
\]
whose colimit is equivalent to $B$ and such that for any $n \geq 1$ the morphism $B_{n-1} \to B_n$ fits into a cocartesian diagram
\[
\mymatrix{
A[\el{R^{n-1}}{q}]^{dR_i^{n-1} = 0} \ar[r] \ar[d]_{R_i \mapsto dU_i} & B_{n-1} \ar[d] \\
A[\el{U^n}{q}, \el{X^n}{p}]^{dX^n_j = 0} \ar[r] & B_n \cocart
}
\]
where $R_i^{n-1}$ is a variable in degree $-(n-1)$ and $X_j^n$ and $U_i^n$ are variables in degree $-n$.

We build such a diagram recursively.
Let 
\[
\homol^0(B) \cong \quot{\homol^0(A)[\el{X^0}{p_0}]}{(\el{R^0}{q_0})}
\]
be a presentation of $\homol^0(B)$ as a $\homol^0(A)$-algebra.
Let $B_0$ be $A[\el{X^0}{p_0}]$ equipped with a morphism $\phi_0 \colon B_0 \to B$ given by a choice of coset representatives of $\el{X^0}{p_0}$ in $B$.
The induced morphism $\homol^0(B_0) \to \homol^0(B)$ is surjective and its kernel is of finite type (as a $\homol^0(A)$-module).
\\ Let $n \geq 1$. Assume $\phi_{n-1} \colon B_{n-1} \to B$ has been defined and satisfies the properties:
\begin{itemize}
\item If $n = 1$ then the induced morphism of $\homol^0(A)$-modules $\homol^0(B_0) \to \homol^{0}(B)$ is surjective and its kernel $K_0$ is a $\homol^0(A)$-module of finite type.
\item If $n \geq 2$, then the morphism $\phi_{n-1}$ induces isomorphisms $\homol^{-i}(B_{n-1}) \to \homol^{-i}(B)$ of $\homol^0(A)$-modules if $i = 0$ and of $\homol^0(B)$-modules for $1 \leq i \leq n-2$.
\item If $n \geq 2$ then the induced morphism of $\homol^0(B)$-modules $\homol^{-n+1}(B_{n-1}) \to \homol^{-n+1}(B)$ is surjective and its kernel $K_{n-1}$ is a $\homol^0(B)$-module of finite type.
\end{itemize}
Let $n \geq 1$. Let $\el{X^{n}}{p}$ be generators of $\homol^{-n}(B)$ as a $\homol^0(B)$-module and $\el{R^{n-1}}{q}$ be generators of $K_{n-1}$.
Let $B_n$ be the pushout:
\[
\mymatrix{
A[\el{R^{n-1}}{q}]^{dR_i^{n-1} = 0} \ar[r] \ar[d]_{R_i \mapsto dU_i} & B_{n-1} \ar[d] \\
A[\el{U^n}{q}, \el{X^n}{p}]^{dX^n_k = 0} \ar[r] & B_n \cocart
}
\]
Let $\el{r^{n-1}}{q}$ be the images of $\el{R^{n-1}}{q}$ (respectively) by the composite morphism
\[
A[\el{R^{n-1}}{q}]^{dR_i^{n-1} = 0} \to B_{n-1} \to B
\]
There exist $\el{u^n}{q} \in B$ such that $d u^n_i = r^{n-1}_i$ for all $i$.
Those $\el{u^n}{q}$ together with a choice of coset representatives of $\el{X^n}{p}$ in $B$ induce a morphism
\[
A[\el{U^n}{q}, \el{X^n}{p}]^{dX^n_k = 0} \to B
\]
which induces a morphism $\phi_n \colon B_n \to B$.
\begin{description}

\item If $n=1$ then a quick computation proves the isomorphism of $\homol^0(A)$-modules
\[
\homol^0(B_1) \cong \quot{\homol^0(B_0)}{(\el{R^0}{q})} \cong \homol^0(B)
\]
\item If $n \geq 2$ then the truncated morphism $B_n^{\geq 2-n} \to^\sim \B_{n-1}^{\geq 2-n}$ is a quasi-isomorphism and
the induced morphisms $\homol^{-i}(B_n) \to^\simeq \homol^{-i}(B)$ are thus isomorphisms of $\homol^0(B)$-modules for $i \leq n-2$.
We then get the isomorphism of $\homol^0(B)$-modules
\[
\homol^{-n+1}(B_n) \cong \quot{\homol^{-n+1}(B_{n-1})}{(\el{R^{n-1}}{q})} \cong \homol^{-n+1}(B)
\]
\end{description}
The natural morphism $\theta \colon \homol^{-n}(B_n) \to \homol^{-n}(B)$ is surjective.
The $\homol^0(B)$-module $\homol^{-n}(B_n)$ is of finite type and because $\homol^0(B)$ is noetherian, the kernel $K_n$ of $\theta$ is also of finite type.
The recursivity is proven and it now follows that the morphism $\colim_n B_n \to B$ is a quasi-isomorphism.
\end{proof}

\begin{df} Let $L$ be a dg-Lie algebra over $A$.
\begin{itemize}
\item We will say that $L$ is very good if there exists a finite sequence
\[
0 = L_0 \to \el{L}{n}[\to] = L
\]
such that each morphism $L_i \to L_{i+1}$ fits into a cocartesian square
\[
\mymatrix{
\libre(A[-p_i]) \ar[r] \ar[d] & L_i \ar[d] \\ \libre(\Mc(A[-p_i])) \ar[r] & L_{i+1} \cocart
}
\]
where $p_i \geq 2$. 
\item We will say that $L$ is good if it is quasi-isomorphic to a very good dg-Lie algebra.
\item We will say that $L$ is almost finite if it is cofibrant and if its underlying
\emph{graded} module is isomorphic to
\[
\bigoplus_{i\geq 1} A^{n_i} [-i]
\]
\end{itemize}
\end{df}

\begin{lem}
\begin{itemize}
The following assertions are true.
\item Any very good dg-Lie algebra is almost finite.
\item The underlying dg-module of a cofibrant dg-Lie algebra is cofibrant.
\end{itemize}
\end{lem}
\begin{proof}
Any free dg-Lie algebra generated by some $A[-p]$ with $p \geq 2$ is almost finite. Considering a pushout diagram
\[
\mymatrix{
\libre(A[-p]) \ar[r] \ar[d] & L \ar[d] \\ \libre(\Mc(A[-p])) \ar[r] & L' \cocart
}
\]
Whenever $L$ is almost finite, so is $L'$.
This proves the first item.

Let now $L$ be a dg-Lie algebra over $A$. There is a morphism of dg-modules $L \to \Envel_A L$.
The Poincaré-Birkhoff-Witt theorem states that the dg-module $\Envel_A L$ is isomorphic to $\Sym_A L$.
There is therefore a retract $\Envel_A L \to L$ of the universal morphism $L \to \Envel_A L$.
The functor $\Envel_A \colon \MCdgLie_A \to \MCdgAlg_A$ preserves cofibrant objects and using a result of \cite{schwedeshipley:monoidal}, so does the forgetful functor $\MCdgAlg_A \to \MCdgMod_A$. We therefore deduce that if $L$ is cofibrant in $\MCdgLie_A$ it is also cofibrant in $\MCdgMod_A$.
\end{proof}

\begin{df}
Let $\dgLieGood_A$ denote the sub-$(\infty,1)$-category of $\dgLie_A$ spanned by good dg-Lie algebras.
\end{df}

\begin{rmq}
We naturally have an inclusion $\dgLieLib_A \to \dgLieGood_A$.
\end{rmq}

\subsection{Homology and cohomology of dg-Lie algebras}\label{section-liecoho}
The content of this section can be found in \cite{lurie:dagx} when the base is a field.
Proofs are simple avatars of Lurie's on a more general base $A$.
Let then $A$ be a commutative dg-algebra concentrated in non-positive degree over a field $k$ of characteristic zero.
\begin{df}
Let $A[\eta]$ denote the (contractible) commutative $A$-dg-algebra generated by one element $\eta$ of degree -1 such that $\eta^2 = 0$ and $d\eta = 1$.
For any $A$-dg-Lie algebra $L$, the tensor product $A[\eta] \otimes_A L$ is still an $A$-dg-Lie algebra and we can thus define the homology of $L$:
\[
\homol_A (L) = \Envel_A\left(A[\eta] \otimes_A L\right) \otimes_{\Envel_A L} A
\]
where $\Envel_A \colon \MCdgLie_A \to \MCdgAlg_A$ is the functor sending a Lie algebra to its enveloping algebra.
This construction defines a strict functor:
\[
\homol_A \colon \MCdgLie_A \to \comma{A}{\MCdgMod_A}
\]
\end{df}
\begin{rmq}\label{rmq-coalg}
The homology $\homol_A(L)$ of $L$ is isomorphic \emph{as a graded module} to $\Sym_A(L[1])$, the symmetric algebra built on $L[1]$.
The differentials do not coincide though. The one on $\homol_A(L)$ is given on homogenous objects by the following formula:
\begin{align*}
d(\el{\eta.x}{n}[\otimes]) = \sum_{i<j} (-1)^{T_{ij}} \eta.[x_i,x_j] \otimes \eta&.x_1 \otimes \dots \otimes \widehat{\eta.x_i} \otimes \dots \otimes \widehat{\eta.x_j} \otimes \dots \otimes \eta.x_n &\\
-& \sum_i (-1)^{S_i} \eta.x_1 \otimes \dots \otimes \eta.d(x_i) \otimes \dots \otimes \eta.x_n
\end{align*}
where $\eta.x$ denotes the point in $L[1]$ corresponding to $x \in L$.
\begin{align*}
&S_i = i-1+|x_1|+ \dots + |x_{i-1}|\\
&T_{ij} = (|x_i|-1)S_i + (|x_j|-1)S_j + (|x_i|-1)(|x_j|-1)
\end{align*}\todo{$+ |x_i|$ ?}
The coalgebra structure on $\Sym_A(L[1])$ is compatible with this differential and the isomorphism above induces a coalgebra structure on $\homol_A(L)$ given for $x \in L$ homogenous by:
\[
\Delta(\eta.x) = \eta.x \otimes 1 + 1 \otimes \eta.x
\]
\end{rmq}

\begin{prop}\label{homol-inftyfunctor}
The functor $\homol_A$ preserves quasi-isomorphisms.
It induces a functor between the corresponding $(\infty,1)$-categories, which we will denote the same way:
\[
\homol_A \colon \dgLie_A \to \comma{A}{\dgMod_A}
\]
\end{prop}
\begin{proof}
Let $L \to L'$ be a quasi-isomorphism of $A$-dg-Lie algebras.
Both $\homol_A(L)$ and $\homol_A(L')$ are endowed with a natural filtration denoted $\homol_A^{\leq n}(L)$ (resp $L'$) induced by the canonical filtration of $\Sym_A(L[1])$.
Because quasi-isomorphisms are stable by filtered colimits, it is enough to prove that each morphism $\homol_A^{\leq n}(L) \to \homol_A^{\leq n}(L')$ is a quasi-isomorphism.
The case $n = 0$ is trivial. Let us assume $\homol_A^{\leq n-1}(L) \to \homol_A^{\leq n-1}(L')$ to be a quasi-isomorphism.
There are short exact sequences:
\[
\mymatrix{
0 \ar[r] & \homol_A^{\leq n-1}(L) \ar[r] \ar[d] & \homol_A^{\leq n}(L) \ar[r] \ar[d] & \Sym^n_A(L[1]) \ar[r] \ar[d]^\theta & 0 \\
0 \ar[r] & \homol_A^{\leq n-1}(L') \ar[r] & \homol_A^{\leq n}(L') \ar[r] & \Sym^n_A(L'[1]) \ar[r] & 0 \\
}
\]
The base dg-algebra $A$ is of characteristic zero and the morphism $\theta$ is thus a retract of the quasi-isomorphism $L[1]^{\otimes n} \to L'[1]^{\otimes n}$ (where the tensor product is taken over $A$).
\end{proof}

\begin{prop} \label{commute-homol}
Let $A \to B$ be a morphism in $\cdga_k$. The following square is commutative:
\[
\mymatrix{
\dgLie_A \ar[d]_{B \otimes_A -} \ar[r]^-{\homol_A} &  \comma{A}{\dgMod_A} \ar[d]^{B \otimes_A -} \\
\dgLie_B \ar[r]^-{\homol_B} &  \comma{B}{\dgMod_B}
}
\]
\end{prop}
\begin{proof}
This follows directly from the definition.
\end{proof}

\begin{cor}\label{ext-homol}
Let $L$ be in $\MCdgLie_A$ freely generated by some free dg-module $M$.
The homology of $L$ is quasi-isomorphic to the trivial square zero extension $A \to A \oplus M[1]$.
\end{cor}
\begin{proof}
This a consequence of the previous proposition and the corresponding result over a field in Lurie's \autoref{thm-lurie}.
\end{proof}

\begin{df}\label{liecohomology}
Let $L$ be an object of $\MCdgLie_A$. We define the Chevalley-Eilenberg cohomology of $L$ as the dual of its homology:\glsadd{chevalley}
\[
\coho_A(L) = \dual{\homol_A(L)} = \Homint_A\left(\homol_A(L), A\right)
\]
It is equipped with a commutative algebra structure (see \autoref{rmq-coalg}). This defines a functor:
\[
\coho_A \colon \dgLie_A \to {\left(\quot{\cdgaunbounded_A}{A}\right)}\op
\]
between $(\infty,1)$-categories.
\end{df}

\begin{rmq}
The cohomology of an object of $\dgLieLib_A$ is concentrated in non positive degree as it is, as a graded module, isomorphic to $\Sym_A(\dual L[-1])$.
The following proposition proves the cohomology of a good dg-Lie algebra is also concentrated in non-positive degree.
\end{rmq}

\begin{prop}\label{chevalleycolim}
The functor $\coho_A$ of $(\infty,1)$-categories maps colimit diagrams in $\dgLie_A$ to limit diagrams of $\quot{\cdgaunbounded_A}{A}$.
\end{prop}
\begin{proof}[sketch of a]
For a complete proof, the author refers to the proof of proposition 2.2.12 in \cite{lurie:dagx}.
We will only transcript here the main arguments.

A commutative $A$-dg-algebra $B$ is the limit of a diagram $B_\alpha$ if and only if the underlying dg-module is the limit of the underlying diagram of dg-modules.
It is thus enough to consider the composite $\infty$-functor $\dgLie_A \to {\left(\quot{\cdgaunbounded_A}{A}\right)}\op \to {\left(\quot{\dgMod_A}{A}\right)}\op$.
This functor is equivalent to $\dual{(\homol_A(-))}$.
It is then enough to prove $\homol_A \colon \dgLie_A \to \comma{A}{\dgMod_A}$ to preserve colimits.

To do so, we will first focus on the case of sifted colimits, which need only to be preserved by the composite functor $\dgLie_A \to \comma{A}{\dgMod_A} \to \dgMod_A$.
This last functor is the (filtered) colimits of the functors $\homol_A^{\leq n}$ as introduced in the proof of \autoref{homol-inftyfunctor}. We now have to prove that $\homol_A^{\leq n} \colon \dgLie_A \to \dgMod_A$ preserves sifted colimits, for any $n$.
There is a fiber sequence
\[
\homol_A^{\leq n-1} \to \homol_A^{\leq n} \to \Sym_A^n((-)[1])
\]
The functor $\Sym_A((-)[1])$ preserves sifted colimits in characteristic zero and an inductive process proves that $\homol_A$ preserves sifted colimits too.

We now have to treat the case of finite coproducts. The initial object is obviously preserved. Let $L = L' \amalg L''$ be a coproduct of dg-Lie algebras.
We proved in \autoref{dglie-sifted} that $L'$ an $L''$ can be written as sifted colimits of objects of $\dgLieLib_A$.
It is thus enough to prove that $\homol_A$ preserve the coproduct $L = L' \amalg L''$ when $L'$ and $L''$ (and thus $L$ too) are in $\dgLieLib_A$.
This corresponds to the following cocartesian diagram
\[
\mymatrix{
A \ar[r] \ar[d] & A \oplus M' [1] \ar[d] \\ A \oplus M''[1] \ar[r] & A \oplus (M' \oplus M'')[1] \cocart
}
\]
where $M'$ and $M''$ are objects of $\dgModLib_A$ generating $L'$ and $L''$ respectively.
\end{proof}

\begin{df}
The colimit-preserving functor $\coho_A$ between presentable $(\infty,1)$-categories admits a right adjoint which we will denote by $\adjoint_A$.
\end{df}

\begin{lem} \label{commute-coho-good}
Let $B \to A$ be a morphism of $\cdga_k$. The following diagram of $(\infty,1)$-categories commutes: 
\[
\mymatrix{
\dgLieGood_B \ar[r]^-{\coho_B} \ar[d]_{A \otimes_B -} & {\left(\quot{\cdga_B}{B}\right)}\op \ar[d]^{A \otimes_B -} \\
\dgLieGood_A \ar[r]^-{\coho_A} & {\left(\quot{\cdga_A}{A}\right)}\op
}
\]
\end{lem}

\begin{proof}
The \autoref{commute-homol} gives birth to a natural transformation $A \otimes_B \coho_B(-) \to \coho_A( A \otimes_B -)$. Let $L \in \dgLieGood_B$. The $B$-dg-module $\coho_A(A \otimes_B L)$ is equivalent to
\[
\coho_A\left(A \otimes_B L\right) \simeq \Homint_B(\homol_B(L),A)
\]
We thus study the natural morphism 
\[
\phi_L \colon A \otimes_B \coho_B(L) \to \Homint_B(\homol_B(L),A)
\]

Let us consider the case of the free dg-Lie algebra $L = \libre(B[-p])$ with $p \geq 1$.
If $B$ is the base field $k$ then $\homol_k(L)$ is perfect (\autoref{ext-homol}) and the morphism $\phi_L$ is an equivalence. If $B$ is any $k$-dg-algebra then $L$ is equivalent to $B \otimes_k \libre(k[-p])$ and we conclude using \autoref{commute-homol} that $\phi_L$ is an equivalence.

To prove the general case of any good dg-Lie algebra it is now enough to ensure that if $L_1 \from L_0 \to L_2$ is a diagram of good dg-Lie algebras such that $\phi_{L_1}$, $\phi_{L_0}$ and $\phi_{L_2}$ are equivalences then so is $\phi_L$, with $L = L_1 \amalg_{L_0} L_2$.
Using \autoref{chevalleycolim}, we see it can be tested in $\dgMod_A$ in which tensor product and fibre product commute.
\end{proof}

\begin{cor}\label{ext-coho}
The composite functor $\dgMod_A \to \dgLie_A \to {\left(\quot{\cdgaunbounded_A}{A}\right)}\op$ is equivalent to the functor $M \to A \oplus \dual M [-1]$.
\end{cor}
\begin{proof}
The $(\infty,1)$-category $\dgMod_A$ is generated under (sifted) colimits by $\dgModLib_A$. The functors at hand coincide on $\dgModLib_A$ and both preserve colimits.
\end{proof}

\begin{lem} \label{good-fullyff}
Assume $A$ is noetherian.
Let $L$ be a good dg-Lie algebra over $A$.
The adjunction morphism $L \to \adjoint_A \coho_A L$ is a quasi-isomorphism.
\end{lem}
\begin{proof}
Let us first proves that the morphism at hand is equivalent, as a morphism of dg-modules, to the natural morphism $L \to \dual{\dual L}$.
Because of \autoref{ext-coho}, the composite functor ${\left(\quot{\cdgaunbounded_A}{A} \right)}\op \to \dgLie_A \to \dgMod_A$ is equivalent to the functor :
\[
(A \to B \to A) \mapsto \dual{\left( A \otimes_B \Lcot_{B/A}[1] \right)}
\]
It now suffices to prove that the following morphism is an equivalence:
\[
f \colon \Lcot_{\coho_A L/A} \otimes_{\coho_A L} A \to \dual L [-1]
\]
As soon as $L$ is good and $A$ noetherian, the cohomology $\coho_A L$ of $L$ satisfies the finiteness conditions of \autoref{afp-cotangent}.
Because $L$ is almost finite (as a dg-Lie algebra), the dg-module $\dual L [-1]$ is an almost finite cell object.
Both domain and codomain of the morphism $f$ are thus almost finite cellular $A$-dg-modules. It is then enough to consider $f \otimes_A k$ for any field $k$ and any morphism $A \to k$
\[
f \otimes_A k \colon \left(\Lcot_{\coho_A(L)/A} \otimes A \right) \otimes_A k \to \dual{\left( L \otimes k \right)} [-1]
\]
The \autoref{commute-coho-good} gives us the equivalence $\coho_A(L) \otimes_A k \simeq \coho_k(L \otimes_A k)$ and the morphism $f \otimes_A k$ is thus equivalent to the morphism
\[
\Lcot_{\coho_k(L \otimes_A k)/k} \otimes k \to \dual{(L \otimes k)}[-1]
\]
This case is exactly Lurie's result \ref{thm-lurie} (ii).

We now prove that $L \to \dual{\dual L}$ is an equivalence.
We can assume that $L$ is very good.
The underlying $A$-dg-module of $L$ is cofibrant and the dual can be computed naively. Now $L$ is almost finite. There is therefore a family $(n_i)$ of integers and an isomorphism of \emph{graded} modules
\[
L = \bigoplus_{i \geq 1} A^{n_i}[-i]
\]
The dual of $L$ is isomorphic to $\prod_{i \geq 1} A^{n_i}[i]$ with an extra differential. Because $A$ in concentrated in non positive degree, the dual $\dual L$ is equivalent to $\bigoplus_{i \geq 1} A^{n_i}[i]$ (with the extra differential).
The natural morphism $L \to \dual{\dual L}$ therefore corresponds to the morphism
\[
\bigoplus_{i \geq 1} A^{n_i}[-i] \to \prod_{i \geq 1} A^{n_i}[-i]
\]
which is a quasi-isomorphism as soon as $A$ is cohomologically bounded.
\end{proof}

\begin{rmq}
The base dg-algebra $A$ needs to be cohomologically bounded for that lemma to be true.
Taking $L=\libre(A^2[-1])$ the adjunction morphism is equivalent to
\[
L \to \dual{\dual L}
\]
which is not a quasi-isomorphism.
\end{rmq}

\subsection{Formal stack over a dg-algebra}
Throughout this section $A$ will denote an object of $\cdga_k$.
\begin{df}
Let $\dgExt_A$ denote the full sub-category of $\quot{\cdga_A}{A}$ spanned be the trivial square zero extensions $A \oplus M$, where $M$ is a free $A$-dg-module of finite type concentrated in non positive degree.\glsadd{dgext}
\end{df}
\begin{df}
A formal stack over $A$ is a functor $\dgExt \to \sSets$ preserving finite products. We will denote by $\dStF_A$ the $(\infty,1)$-category of such formal stacks:
\[
\dStF_A = \sifted(\dgExt_A\op)
\]\glsadd{dstf}
\end{df}

\begin{rmq}
The $(\infty,1)$-category $\dStF_A$ is $\mathbb U$-presentable. \end{rmq}

\begin{prop}\label{adjoint-exist}
Let $A$ be in $\cdga_k$ and let $S=\Spec A$.
There is an adjunction 
\[
\formal_A \colon \dgLie_A \rightleftarrows \dStF_A \,: \lie_A
\]
such that
\begin{itemize}
\item The functor $\dStF_A \to \dgLie_A \to \dgMod_A$ is equivalent to the functor $X \mapsto \T_{X/S,x}[-1]$ where $\T_{X/S,x}$ is the tangent complex of $X$ over $S$ at the natural point $x$ of $X$: it is the dg-module representing the product-preserving functor
\[
\app{{\left(\dgModLib_A\right)}\op}{\sSets}{M}{X(A \oplus \dual M [-1])}
\]
\item The functor $\lie_A$ is conservative and preserves sifted colimits. Its restriction to $\dgExt_A \op$ is canonically equivalent to $\adjoint_A$.
\item If moreover $A$ is noetherian then the functors $\lie_A$ and $\formal_A$ are equivalences of $(\infty,1)$-categories.
\end{itemize}
\end{prop}

\begin{df}
Let $X$ be a formal stack over $A$. The Lie algebra $\lie_A X$ will be called the tangent Lie algebra of $X$ (over $A$).
\end{df}

\begin{proof}[of \autoref{adjoint-exist}]
Let us prove the first item. The functor $\coho_A$ restricts to a functor
\[
\coho_A \colon \dgLieLib_A \to \dgExt_A\op
\]
which composed with the Yoneda embedding defines a functor $\phi \colon \dgLieLib_A \to \dStF_A$.
This last functor extends by sifted colimits to 
\[
\formal_A \colon \dgLie_A \simeq \sifted(\dgLieLib_A) \to \dStF_A
\]
Because $\coho_A$ preserves coproducts, the functor $\formal_A$ admits a right adjoint $\lie_A$ given by right-composing by $\coho_A$.
The composite functor
\[
\mymatrix{
\dStF_A \ar[r]^-{\lie_A} & \sifted(\dgLieLib_A) \ar[r] & \sifted(\dgModLib_A) \simeq \dgMod_A
}
\]
then corresponds to the functor
\[
X \mapsto X(\coho_A(\libre(-))) \simeq X(A \oplus \dual{(-)}[-1])
\]
This proves the first item.
The functor 
\[
\coho_A \colon \left( \dgLieLib_A \right)\op \to \dgExt_A
\]
is essentially surjective. This implies that $\lie_A$ is conservative.
Let us consider the commutative diagram
\[
\mymatrix{
\sifted\left(\dgExt_A \op \right) \ar[d]^i \ar[r]^-{\lie_A} & \sifted\left( \dgLieLib_A \right) \ar[d]^j \\
\presh\left(\dgExt_A \op \right) \ar[r]^-{\coho_A^*} & \presh\left( \dgLieLib_A \right)
}
\]
The functors $i$, $j$ and $\coho_A^*$ preserve sifted colimits and therefore so does $\lie_A$.
The third item is a corollary of its fully faithfulness (\autoref{good-fullyff}) under the added assumption.
\end{proof}

Until the end of this section, we will focus on proving that the definition we give of a formal stack is equivalent to Lurie's definition of a formal moduli problem in \cite{lurie:dagx}, as soon as the base dg-algebra $A$ is noetherian.

\begin{df}
An augmented $A$-dg-algebra $B \in \quot{\cdga_A}{A}$ is called artinian if there is sequence 
\[
B = \el{B}[0]{n}[\to] = A
\]
and for $0 \leq i < n$ an integer $p_i \geq 1$ such that 
\[
B_i \simeq B_{i+1} \times_{A[\varepsilon_{p_i}]} A
\]
where $A[\varepsilon_{p_i}]$ denote the trivial square zero extension $A \oplus A[p_i]$.

We denote by $\dgArt_A$ the full subcategory of $\quot{\cdga_A}{A}$ spanned by the artinian dg-algebras.
\end{df}

\begin{df}
Assume $A$ is noetherian.
A formal moduli problem over $A$ is a functor
\[
X \colon \dgArt_A \to \sSets
\]
satisfying the conditions:
\begin{enumerate}[label=(F\arabic{*})]
\item For $n \geq 1$ and $B \in \quot{\dgArt_A}{A[\varepsilon_n]}$ the following natural morphism is an equivalence:
\[
X \left( B \timesunder[{A[\varepsilon_n]}] A \right) \to^\sim X(B) \timesunder[{X(A[\varepsilon_n])}] X(A) 
\]
\item The simplicial set $X(A)$ is contractible.
\end{enumerate}
Let $\widetilde{\dStF_A}$ denote the full subcategory of $\presh(\dgArt_A\op)$ spanned by the formal moduli problems.
This is an accessible localization of the presentable $(\infty,1)$-category $\presh(\dgArt_A)$ of simplicial presheaves over $\dgArt_A$.
\end{df}

\begin{prop}\label{formal-dgext}
Let $A \in \cdga_k$ be noetherian.
The left Kan extension of the inclusion functor $i \colon \dgExt_A \to \dgArt_A$ induces an equivalence of $(\infty,1)$-categories
\[
j \colon \dStF_A \to \widetilde{\dStF_A}
\]
\end{prop}

\begin{proof}
We will actually prove that the composed functor
\[
f \colon \dgLie_A \to \dStF_A \to \widetilde{\dStF_A}
\]
is an equivalence. The functor $f$ admits a right adjoint $g = \lie_A i^*$.

Given $n \geq 1$ and a diagram $B \to A[\varepsilon_n] \from A$ in $\dgArt_A$, \autoref{good-fullyff} implies that the natural morphism
\[
\adjoint_A(B) \amalg_{\adjoint_A(A[\varepsilon_n])} \adjoint_A(A)\to^\sim\adjoint_A \left( B \timesunder[{A[\varepsilon_n]}] A \right)
\] 
is an equivalence.
For any $B \in \dgArt_A$ the adjunction morphism $B \to \coho_A \adjoint_A B$ is then an equivalence.
Given $L \in \dgLie_A$ the functor $\adjoint^* L \colon \dgArt_A \to \sSets$ defined by $\adjoint^*(B) = \Map(\adjoint_A(B),L)$ is a formal moduli problem.
The natural morphism $\id \to \adjoint^* g$ of $\infty$-functors from $\dStF_A$ to itself is therefore an equivalence.
The same goes for the morphism $g \adjoint^*  \to \id$ of $\infty$-functors from $\dgLie_A$ to itself.
The functor $g$ is an equivalence, so is $f$ and hence so is $j$.
\end{proof}

\section{Tangent Lie algebra}

We now focus on gluing the functors built in the previous section, proving the following statement.
\begin{thm}\label{tangent-lie}
Let $X$ be a derived Artin stack locally of finite presentation.
Then there is a Lie algebra $\tgtlie_X$ over $X$ whose underlying module is equivalent to the shifted tangent complex $\T_X[-1]$ of $X$.

Moreover if $f \colon X \to Y$ is a morphism between algebraic stacks locally of finite presentation then there is a tangent Lie morphism $\tgtlie_X \to f^* \tgtlie_Y$.
More precisely, there is a functor
\[
\comma{X}{\dStArtlfp_k} \to \comma{\tgtlie_X}{\dgLie_X}
\]
sending a map $f \colon X \to Y$ to a morphism $\tgtlie_X \to f^* \tgtlie_Y$. The underlying map of quasi-coherent sheaves is indeed the tangent map (shifted by $-1$).
\end{thm}

\subsection{Formal stacks and Lie algebras over a derived Artin stack}

Let $A \to B$ be a morphism in $\cdga_k$. There is an \emph{exact} scalar extension functor $B \otimes_A -\, \colon \dgExt_A \to \dgExt_B$ and therefore an adjunction
\[
\left(B \otimes_A -\right)_! \colon 
\mymatrix{
\dStF_A \ar@<2pt>[r] & \dStF_B \ar@<2pt>[l]
}
: \left(B \otimes_A -\right)^*
\]

\begin{prop}
There is a natural functor $\dStF \colon \cdga_k \to \PresLeftu U$ mapping $A$ to $\dStF_A$.
There moreover exists a natural transformation $\Ff \colon \dgLie \to \dStF$
\end{prop}

\begin{proof}
Let us recall the category $\int \MCdgLie$ defined in the proof of \autoref{commute-dglie}. Its objects are pairs $(A,L)$ where $A \in \MCcdga_k$ and $L \in \MCdgLie_A$.

We define $\int \left(\quot{\MCcdgaunbounded}{-}\right)\op$ to be the following (1-)category.
\begin{itemize}
\item An object is a pair $(A,B)$ where $A \in \MCcdga_k$ and $B \in \quot{\MCcdgaunbounded_A}{A}$.
\item A morphism $(A,B) \to (A', B')$ is a map $A \to A'$ together with a map $B' \to B \otimes_A^{\Lcot} A'$ of $A'$-dg-algebras.
\end{itemize}
From \autoref{liecohomology}, we get a functor $\coho \colon \int \MCdgLie \to \int \left( \quot{\MCcdgaunbounded}{-}\right)\op$ preserving quasi-isomorphisms. This induces a diagram of $(\infty,1)$-categories
\[
\mymatrix{
\int \dgLie \ar[rr]^-{\coho} \ar[dr] && \int \left(\quot{\cdgaunbounded}{-}\right)\op \ar[dl] \\ & \MCcdga_k &
}
\]
Restricting to the full subcategory spanned by pairs $(A,L)$ where $L \in \dgLieLib_A$, we get a functor
\[
\mymatrix{\int \dgLieLib \ar[rr]^-\coho && \int \dgExt\op}
\]
Using \autoref{commute-coho-good}, we see that this last functor preserves coCartesian morphisms over $\MCcdga_k$.
It defines a natural transformation between functors $\MCcdga_k \to \inftyCatu U$.
Since both the functors at hand map quasi-isomorphisms to categorical equivalences, it factors through the localisation $\cdga_k$ of $\MCcdga_k$.
We now have a natural transformation
\[
\mymatrix{
\cdga_k \dcell[r][{\dgLieLib}][{\dgExt\op}][][=>][12pt] & \inftyCatu U
}
\]
Composing with the sifted extension functor $\sifted$, we get a natural transformation $\formal \colon \dgLie \simeq \sifted(\dgLieLib) \to \dStF$.
\end{proof}

\begin{rmq}\label{formal-nattrans}
The Grothendieck construction defines a functor
\[
\formal \colon \mymatrix{
\int \dgLie \ar[r] & \int \dStF
}
\]
over $\cdga_k$. Note that we also have a composite functor
\[
\mymatrix{
G \colon \int \dgLie \ar[r] & \int (\quot{\cdga}{-})\op \ar[r]^-h & \int \dStF
}
\]
where $h$ is deduced from the Yoneda functor. The functor $\formal$ is by definition the relative left Kan extension of the restriction of $G$ to $\int \dgLieLib$. It follows that we have a natural transformation $\formal \to G$. We will use that fact a few pages below.
\end{rmq}

\begin{prop}
The functor
\[
\dgLie \colon \cdga_k \to \PresLeftu U
\]
is a stack for the ffqc topology.
\end{prop}
\begin{proof}
The functor $\dgLie$ is endowed with a forgetful natural transformation to $\dgMod$, the stack of dg-modules. This forgetful transformation is pointwise conservative and preserves limits. This implies that $\dgLie$ is also a stack.\todo{$\dStF$ champ pour la topologie étale ?}
\end{proof}

\begin{df}\label{dstfoverstack}
Let $X$ be an algebraic derived stack.
The $(\infty,1)$-category of formal stacks over $X$ is
\[
\dStF_X = \lim_{\Spec A \to X} \dStF_A
\]
The $(\infty,1)$-category of dg-Lie algebras over $X$ is
\[
\dgLie_X = \lim_{\Spec A \to X} \dgLie_A
\]
where both limits are taken in $\PresLeftu U$.
There is a colimit preserving functor\glsadd{formal}
\[\formal_X \colon \dgLie_X \to \dStF_X\]
It admits a right adjoint denoted by $\lie_X$.\glsadd{lie}
\end{df}

\begin{rmq}
The functor $\lie_X$ may not commute with base change.
\end{rmq}

\subsection{Tangent Lie algebra of a derived Artin stack locally of finite presentation}

Let us consider $\Cc$ the following category.
\begin{itemize}
\item An object is a pair $(A, F \to G)$ where $A$ is in $\MCcdga_k$ and $F \to G$ is a morphism in the model category of simplicial presheaves over $\MCcdga_A$.
\item A morphism $(A, F \to G) \to (B, F' \to G')$ is the datum of a morphism $A \to B$ together with a commutative square
\[
\mymatrix{
F \ar[r] \ar[d] & F' \ar[d]\\
G \ar[r] & G'
}
\]
of presheaves over $\MCcdga_B$
\end{itemize}
We set $\int {\presh(\dAff)}^{\Delta^1}$ to be the $(\infty,1)$-localization of $\Cc$ along weak equivalences of presheaves.
The natural functor $\int {\presh(\dAff)}^{\Delta^1} \to \MCcdga_k$ is a coCartesian fibration classified by the functor $A \mapsto \presh(\dAff_A)^{\Delta^1}$.

Let $\Dd$ denote the following category.
\begin{itemize}
\item An object is a pair $ (A,F) $ where $F$ is a simplicial presheaf over the opposite category of morphisms in $\MCcdga_A$.
\item A morphism $(A,F) \to (B,G)$ is a morphism $A \to B$ and a map $F \to G$ as simplicial presheaves over ${(\MCcdga_B)\op}^{\Delta^1}$.
\end{itemize}
We will denote by $\int \presh\left(\dAff^{\Delta^1}\right)$ the $(\infty,1)$-category obtained from $\Dd$ by localizing along weak equivalences. The natural functor $\int \presh\left(\dAff^{\Delta^1}\right) \to \MCcdga_k$ is a coCartesian fibration.

\begin{lem}\label{maps-stacks}
There is a relative adjunction
\[
f \colon \int \presh\left(\dAff^{\Delta^1}\right) \rightleftarrows  \int \presh\left(\dAff\right)^{\Delta^1} \,: g
\]
over $\MCcdga_k$.
They can be described on the fibers as follows. Let $A \in \MCcdga_k$.
The left adjoint $f_A$ is given on morphisms between affine schemes to the corresponding morphism of representable functors.
The right adjoint $g_A$ maps a morphism $F \to G$ to the representable simplicial presheaf
\[
\Map\left( -, F \to G\right)
\]
\end{lem}
\begin{proof}
Let us define a functor $\Cc \to \Dd$ mapping $(A,F \to G)$ to the functor 
\[
( S \to T ) \mapsto \Map(S \to T, F \to G)
\]
We can now derive this functor (replacing therefore $F \to G$ with a fibrant replacement). We get a functor
\[
g \colon \int \presh\left(\dAff\right)^{\Delta^1} \to \int \presh\left(\dAff^{\Delta^1}\right)
\]
which commutes with the projections to $\MCcdga_k$.
Let $A$ be in $\MCcdga_k$ and let $g_A$ be the induced functor 
\[
\presh(\dAff_A)^{\Delta^1} \to \presh\left( \dAff_A^{\Delta^1} \right)
\]
It naturally admits a left adjoint. Namely the left Kan extension $f_A$ to the Yoneda embedding
\[
\dAff_A^{\Delta^1} \to \presh(\dAff_A)^{\Delta^1}
\]
For any morphism $A \to B$ in $\MCcdga_k$, there is a canonical morphism
\[
f_B\left( \left(B \otimes_A - \right)_! X \right) \to 
\left(B \otimes_A -\right)_! f_A(X)
\]
which is an equivalence.
[When $X = \Spec A' \to \Spec A''$ is representable then both left and right hand sides are equivalent to $\Spec B' \to \Spec B''$ where $B' = B \otimes_A A'$ and $B'' = B \otimes_A A''$].
We complete the proof using \cite[8.2.3.11]{lurie:halg}.
\end{proof}

Let now $\int \comma{\pt}{\presh(\dAff)} \to \cdga_k$ denote the coCartesian fibration classified by the subfunctor $A \mapsto \comma{\Spec(A)}{\presh(\dAff_A)}$ of $\presh(\dAff)^{\Delta^1}$.
Let also $\int \presh(\dAff^*)$ be defined similarly to $\int \presh(\dAff^{\Delta^1})$. It is classified by a functor
\[
A \mapsto \presh\left(\comma{\Spec A}{\dAff_A}\right)
\]

\begin{prop} \label{pointed-adjunction}
The adjunction of \autoref{maps-stacks} induces a relative adjunction
\[
\int \presh\left(\dAff^\pt\right) \rightleftarrows \int \comma{\pt}{\presh(\dAff)}
\]
over $\MCcdga_k$. It moreover induces a natural transformation
\[
\mymatrix{
\cdga_k \dcell[r][{\presh(\dAff^\pt)}][{\comma{\pt}{\presh(\dAff)}}][][=>][12pt] & \inftyCatu V
}
\]
\end{prop}

\begin{proof}
We define the restriction functor
\[
\int \presh\left(\dAff^{\Delta^1}\right) \to \int \presh(\dAff^\pt)
\]
It admits a fiberwise left adjoint, namely the left Kan extension, which commutes with base change. This defines -- using \cite[8.3.2.11]{lurie:halg} -- a relative left adjoint 
\[
\int \presh(\dAff^\pt) \to \int \presh\left( \dAff^{\Delta^1} \right)
\]
Composing with the relative adjunction of \autoref{maps-stacks}, we get a relative adjunction
\[
\int \presh(\dAff^\pt) \rightleftarrows \int \presh(\dAff)^{\Delta^1}
\]
The left adjoint factors through $\int \comma{\pt}{\presh(\dAff)}$ and the composed functor
\[
\int \comma{\pt}{\presh(\dAff)} \to \int \presh(\dAff)^{\Delta^1} \to \int \presh(\dAff^\pt)
\]
is its relative right adjoint.
It follows that the functor $\int \presh(\dAff^\pt) \to \int \comma{\pt}{\presh(\dAff)}$ preserves coCartesian morphisms over $\MCcdga_k$. We get a natural transformation
\[
\mymatrix{
\MCcdga_k \dcell[r][{\presh(\dAff^\pt)}][{\comma{\pt}{\presh(\dAff)}}][][=>][12pt] & \inftyCatu V
}
\]
As both functors at hand map quasi-isomorphisms of cdga's to equivalences of categories, it factors through the localisation $\cdga_k$ of $\MCcdga_k$.
\end{proof}

\begin{prop}
Let $X$ be an algebraic derived stack.
There are functors
\[
\phi \colon \quot{\comma{X}{\presh(\dAff_k)}}{X} \to \lim_{\Spec A \to X} \comma{\Spec A}{\presh(\dAff_A)}
\]
and
\[
\theta \colon \lim_{\Spec A \to X} \comma{\Spec A}{\presh(\dAff_A)} \to \lim_{\Spec A \to X} \presh\left(\dAff_A^\pt\right)
\]
\end{prop}
\begin{proof}
The functor $\phi$ is given by the following construction:
\[
\quot{\comma{X}{\presh(\dAff_k)}}{X} 
\to \lim_{\Spec A \to X} \quot{\comma{\Spec A}{\presh(\dAff_k)}}{\Spec A}
\simeq \lim_{\Spec A \to X} \comma{\Spec A}{\presh(\dAff_A)}
\]
The second functor is constructed as follows.
From \autoref{pointed-adjunction} we get a functor
\[
\lim_{\Spec A \to X} \presh\left(\dAff_A^\pt\right) \to \lim_{\Spec A \to X} \comma{\Spec A}{\presh(\dAff_A)}
\]
It preserves colimits and both left and right hand sides are presentable. It thus admits a right adjoint $\theta$.
\end{proof}

\begin{rmq} \label{pointed-basechange}
Note that the functor $\theta$ commutes with base change. We can indeed draw the commutative diagram (where $S \to T$ is a morphism between affine derived schemes)
\[
\mymatrix{
\quot{\comma{T}{\presh(\dAff_k)}}{T} \ar[r] \ar[d] &
\presh\left(\quot{\comma{T}{\presh(\dAff_k)}}{T}\right) \ar[r] \ar[d] &
\presh\left(\quot{\comma{T}{\dAff_k}}{T}\right) \ar[d] &
\\
\quot{\comma{S}{\presh(\dAff_k)}}{S} \ar[r] &
\presh\left(\quot{\comma{S}{\presh(\dAff_k)}}{S}\right) \ar[r] &
\presh\left(\quot{\comma{S}{\dAff}}{S}\right) &
}
\]
The left hand side square commutes by definition of base change. The right hand side square also commutes as the restriction along a fully faithful functor preserves base change.
\end{rmq}

\begin{df}
Let $X$ be an algebraic derived stack.
Let us define the formal completion functor
\[
\for{(-)} \colon \quot{\comma{X}{\presh(\dAff_k)}}{X} \to 
\lim_{\Spec A \to X} \presh\left(\left(\dgExt_A\right)\op\right)
\]
as the composed functor
\begin{align*}
\quot{\comma{X}{\presh(\dAff_k)}}{X}
&\to \lim_{\Spec A \to X} \comma{\Spec A}{\presh(\dAff_A)} \\
&\to \lim_{\Spec A \to X} \presh\left(\dAff_A^\pt\right) \\
&\to \lim_{\Spec A \to X} \presh\left(\left(\dgExt_A\right)\op\right)
\end{align*}
\end{df}

\begin{rmq}
Let $u \colon S =\Spec A \to X $ be a point.
The functor $u^* \for{(-)}$ maps a pointed stack $Y$ over $X$ to the functor $\dgExt_A \to \sSets$
\[
B \mapsto \Map_{S/-/X}(\Spec B, Y)
\]
\end{rmq}

\begin{df}
Let $X$ be a derived Artin stack.
Let $\dStptArt_X$ denote the full sub-category of 
$\quot{\comma{X}{\presh(\dAff)}}{X}$ spanned by those $X \to Y \to X$ such that $Y$ is a derived Artin stack over $X$.
\end{df}

\begin{lem}
The restriction of $\for{(-)}$ to $\dStptArt_X$ has image in $\dStF_X$.
\end{lem}

\begin{proof}
We have to prove that whenever $Y$ is a pointed algebraic stack over $X$ then $\for Y$ is formal over $X$.
Because of \autoref{pointed-basechange}, it suffices to treat the case of an affine base. The result then follows from the existence of an obstruction theory for $Y \to X$.
\end{proof}

\begin{df}\label{tgtliedef}\glsadd{tgtlie}
Let $X$ be an algebraic stack locally of finite presentation. We define its tangent Lie algebra as the $X$-dg-Lie algebra 
\[
\tgtlie_X = \lie_X\left( \for{(X \times X)} \right)
\]
where the product $X \times X$ is a pointed stack over $X$ through the diagonal and the first projection.
\end{df}

\begin{proof}[of \autoref{tangent-lie}]
Let us denote $Y = X \times X$ and let $u \colon S = \Spec A \to X$ be a point.
The derived stack $X$ has a global tangent complex and the natural morphism $u^* \tgtlie_X = u^* \lie_X \for Y \to \lie_A(u^* \for Y)$ is an equivalence.
The functor $u^* \for Y$ maps $B$ in $\dgExt_A$ to
\[
\Map_{S/-/X}(\Spec B, X \times X) \simeq \Map_{S/}(\Spec B, X)
\]
We deduce using \autoref{ext-coho} that the underlying $A$-dg-module of $\lie_A(u^* \for Y)$ therefore represents the functor
\[
\app{\dgModLib_A}{\sSets}{M}{\Map_{S/}\left( \Spec(A \oplus \dual M[-1]), X \right)}
\]
Using once again that $X$ has a global tangent complex we conclude that the underlying module of $\tgtlie_X$ is indeed $\T_X [-1]$.

Let us now consider the functor 
\[
\comma{X}{\dStArtlfp} \to \dStptArt_X
\]
mapping a morphism $X \to Z$ to the stack $X \times Z$ pointed by the graph map $X \to X \times Z$ and endowed with the projection morphism to $X$.
Composing this functor with $\for{(-)}$ and $\lie_X$ we finally get the wanted functor
\[
\comma{X}{\dStArtlfp} \to \comma{\tgtlie_X}{\dgLie_X}
\]
Observing that $X \times Z$ is equivalent to $X \times_Z (Z \times Z)$ and because $Z$ has a global tangent complex, we deduce
\[
\lie_X\left( \for{(X \times Z)} \right) \simeq u^* \tgtlie_Z
\]
Let us finally note that the underlying map of quasi-coherent sheaves is the tangent map shifted by $-1$.
\end{proof}

\subsection{Derived categories of formal stacks}

The goal of this subsection is to prove the following

\begin{thm}\label{derived-global}
Let $X$ be an algebraic stack locally of finite presentation.
There is a colimit preserving monoidal functor 
\[
\lierep_X \colon \Qcoh(X) \to \dgRep_X(\tgtlie_X)
\]
where $\dgRep_X(\tgtlie_X)$ is the $(\infty,1)$-category of representations of $\tgtlie_X$.
Moreover, the functor $\lierep_X$ is a retract of the forgetful functor.
\end{thm}

We will prove this theorem at the end of the subsection. For now, let us state and prove a few intermediate results.
Let $A$ be any $\MCcdga_k$ and $L \in \MCdgLie_A$.
The category $\MCdgRep_A(L)$ of representations of $L$ is endowed with a combinatorial model structure for which equivalences are exactly the $L$-equivariant quasi-isomorphisms and for which the fibrations are those maps sent onto fibrations by the forgetful functor to $\MCdgMod_A$.
\begin{df}
Let us denote by $\dgRep_A(L)$ the underlying $(\infty,1)$-category of the model category $\MCdgRep_A(L)$.\glsadd{dgrep}
\end{df}
\begin{lem}\label{derivedcat-local}
Let $L$ be an $A$-dg-Lie algebra. There is a Quillen adjunction
\[
f_L^A \colon \MCdgMod_{\coho_A L} \rightleftarrows \MCdgRep_A(L) \noloc g_L^A
\]
Given by 
\begin{align*}
&f_L^A \colon V \mapsto \Envel_A\left(A[\eta] \otimes_A L\right) \otimes_{\coho_A L} V \\
&g_L^A \colon M \mapsto \Homint_{L} \left(\Envel_A\left(A[\eta] \otimes_A L\right), M \right)
\end{align*}
where $A[\eta] \otimes_A L$ is as in \autoref{section-liecoho} and $\Homint_L$ denotes the mapping complex of dg-representations of $L$.
\end{lem}
\begin{rmq}
The image $g_L^A(M)$ is a model for the cohomology $\RHomint_L(A,M)$ of $L$ with values in $M$.
\end{rmq}
\begin{proof}
The fact that $f_L^A$ and $g_L^A$ are adjoint functors is immediate. The functor $f_L^A$ preserves quasi-isomorphisms (see the proof of \autoref{homol-inftyfunctor}) and fibrations. This is therefore a Quillen adjunction.
\end{proof}

\begin{rmq}\label{monoidallierep}
The category $\MCdgRep_A(L)$ is endowed with a symmetric tensor product. If $M$ and $N$ are two dg-representations of $L$, then  $M \otimes_A N$ is endowed with the diagonal action of $L$. The category $\MCdgMod_{\coho_A L}$ is also symmetric monoidal.
Moreover, for any pair of $L$-dg-representations $M$ and $N$, there is a natural morphism\todo{monoidal lie rep}
\[
g_L^A(M) \otimes_{\coho_A L} g_L(N) \to g_L^A\left(M \otimes_A N\right)
\]
This makes $g_L^A$ into a weak monoidal functor.
In particular, the functor $g_L^A$ defines a functor $\MCdgLie_{L} \to \MCdgLie_{\coho_A(L)}$, also denoted $g_L^A$.
\end{rmq}

\begin{prop}\label{derivedcat-ff-local}
Let $L$ be a good dg-Lie algebra over $A$. The induced functor 
\[
f_L^A \colon \dgMod_{\coho_A L} \to \dgRep_A(L)
\]
of $(\infty,1)$-categories is fully faithful.
\end{prop}

\begin{proof}
In this proof, we will write $f$ instead of $f_L^A$.
Let $B$ denote $\coho_A L$.
We first prove that the restriction $f_{|\Perf(B)}$ is fully faithful.
Let $V$ and $W$ be two $B$-dg-modules.
There is a map $\Map(V,W) \to \Map(fV,fW)$.
Fixing $V$ (resp. $W$), the set of $W$'s (resp. $V$'s) such that this map is an equivalence is stable under extensions, shifts and retracts.
It is therefore sufficient to prove that the map $\Map(B,B) \to \Map(fB,fB)$ is an equivalence, which follows from the definition (if we look at the dg-modules of morphisms, then both domain and codomain are quasi-isomorphic to $B = \coho_A L$).

To prove that $f \colon \dgMod_B \to \dgRep_A(L)$ is fully faithful, we only need to prove that $f$ preserves perfect objects.
It suffices to prove that $f B \simeq A$ is perfect in $\dgRep_A(L)$.
The (non commutative) $A$-dg-algebra $\Envel_A(A[\eta] \otimes_A L)$ is a finite cellular object (because $L$ is good) and is endowed with a morphism to $A$.
The forgetful functor $\dgMod_A \to \dgMod_{\Envel_A(L)}^{\operatorname{left}}$ therefore preserves perfect objects (see \cite{toen:ttt}).
\end{proof}

Let us consider the category $\int \MCdgLie\op$ such that:
\begin{itemize}
\item An object is a pair $(A,L)$ with $A \in \MCcdga_k$ and with $L \in \MCdgLie_A$ and
\item A morphism $(A,L) \to (B,L')$ is a map $A \to B$ together with a map $L' \to L \otimes_A^\Lcot B$ in $\MCdgLie_B$.
\end{itemize}
It is endowed with a functor $\int \MCdgLie\op \to \MCcdga_k$. Localising along quasi-isomorphisms, we obtain a coCartesian fibration of $(\infty,1)$-categories
\[
\int \dgLie\op \to \MCcdga_k
\]
classified by the functor $A \mapsto \dgLie_A\op$.

Let $\ptfin$ denote the category of pointed finite sets -- see \autoref{ptfin}. For $n \in \N$, we will denote by $\langle n \rangle$ the finite set $\{\pt, 1, \dots, n\}$ pointed at $\pt$.
Let $\int \MCdgRep^{\otimes}$ be the following category.
\begin{itemize}
\item An object is a family $(A,L,\el{M}{m})$ with $A \in \MCcdga_k$, with $L \in \MCdgLie_A$ and with $M_i \in \MCdgRep_A(L)$.
\item A morphism $(A,L,\el{M}{m}) \to (B,L',\el{N}{n})$ is the datum of a map $(A,L) \to (B,L') \in \int \MCdgLie\op$, of a map $t \colon \langle m \rangle \to \langle n \rangle$ of pointed finite sets and for every $1 \leq j \leq n$ of a morphism $\bigotimes_{i \in t^{-1}(j)} M_i \otimes_A B \to N_j$ of $L'$-modules.
\end{itemize}
It comes with a projection functor $\int \MCdgRep^{\otimes} \to \int \MCdgLie\op \times \ptfin$.
We will say that a morphism in $\int \MCdgRep^{\otimes}$ is a quasi-isomorphism if the underlying maps of pointed finite sets and dg-Lie algebras are identities and if the maps of cdga's and of dg-representations it contains are quasi-isomorphisms.
Let us denote by $\int \dgRep^{\otimes}$ the localisation of $\int \MCdgRep^{\otimes}$ along quasi-isomorphisms.
This defines a coCartesian fibration $p \colon \int \dgRep^{\otimes} \to \int \MCdgLie\op \times \ptfin$ (using once again \cite[2.4.29]{lurie:dagx}).

Let now $\int \MCdgMod_{\coho(-)}^{\otimes}$ be the following category
\begin{itemize}
\item An object is a family $(A,L,\el{V}{m})$ with $A \in \MCcdga_k$, with $L \in \MCdgLie_A$ and with $V_i \in \MCdgMod_{\coho_A L}$.
\item A morphism $(A,L,\el{V}{m}) \to (B,L',\el{W}{n})$ is the datum of a map $(A,L) \to (B,L') \in \int \MCdgLie\op$, of a map of pointed finite sets $t \colon \langle m \rangle \to \langle n \rangle$ and for every $1 \leq j \leq n$ of a morphism of $\coho_B L'$-dg-modules $\bigotimes_{i \in t^{-1}(j)} V_i \otimes_{\coho_A L} \coho_B L' \to W_j$.
\end{itemize}
Localising along quasi-isomorphisms of modules, we get a coCartesian fibration of $(\infty,1)$-categories $q \colon \int \dgMod_{\coho(-)}^{\otimes} \to \int \MCdgLie\op \times \ptfin$.

\begin{lem}
The above coCartesian fibrations $p$ and $q$ define functors
\begin{align*}
\dgRep,\, \dgMod_{\coho(-)} \colon \int \dgLie\op \to \monoidalinftyCatu V
\end{align*}
\end{lem}

\begin{proof}
For any object $(A,L) \in \int \MCdgLie\op$, the pulled back coCartesian fibration
\[
\int \dgRep^{\otimes} \times_{\int \MCdgLie\op \times \ptfin} \{(A,L)\} \times \ptfin \to \ptfin
\]
defines a symmetric monoidal structure on the $(\infty,1)$-category $\dgRep_A(L)$ -- see \autoref{monoidalcats}.
The coCartesian fibration $p$ is therefore classified by a functor
\[
\int \MCdgLie\op \to \monoidalinftyCatu V
\]
Moreover, this functor maps quasi-isomorphisms of dg-Lie algebras to equivalences. Hence it factors through a functor
\[
\dgRep \colon \int \dgLie\op \to \monoidalinftyCatu V
\]
The case of $\dgMod_{\coho(-)}$ is isomorphic.
\end{proof}

We will now focus on building a natural transformation between those two functors.
Let us build a functor $g \colon \int \MCdgRep^{\otimes} \to \int \MCdgMod_{\coho(-)}^{\otimes}$
\begin{itemize}
\item The image of an object $(A,L,\el{M}{m})$ is the family $(A,L,\el{V}{m})$ where $V_i$ is the $\coho_A L$-dg-module 
\[
g_L^A(M_i) = \Homint_L\left(\Envel_A\left(A[\eta] \otimes_A L\right),M_i\right)
\]
\item The image of an arrow $\bigotimes_{i \in t^{-1}(j)} M_i \otimes_A B \to N_j$ is the composition 
\begin{align*}
\bigotimes g_L^A(M_i) \otimes_{\coho_A L} \coho_B L' \to &
g_L^A\left(\bigotimes M_i\right) \otimes_{\coho_A L} \coho_B L' \\
\to & g_{L'}^B\left(\bigotimes M_i \otimes_A B\right)
\\ \to & g_{L'}^B(N)
\end{align*}
where the second map sends a tensor $\lambda \otimes \mu$ to $(\lambda \otimes \id).\mu$ with
\[
\lambda \otimes \id \colon \Envel_B\left(B[\eta] \otimes_B L'\right) \to \Envel_B\left(B[\eta] \otimes_A L \right) = \Envel_A\left(A[\eta] \otimes_A L\right) \otimes_A B \to \left(\bigotimes M_i\right) \otimes_A B
\]
\end{itemize}
The functor $g$ induces a functor of $(\infty,1)$-categories
\[
g \colon \int \dgRep^{\otimes} \to \int \dgMod_{\coho(-)}^{\otimes}
\]
which commutes with the coCartesian fibrations to $\int \MCdgLie\op \times \ptfin$.

\begin{prop}\label{derivedcat-adjunction}
The functor $g$ admits a left adjoint $f$ relative to $\int \MCdgLie \op \times \ptfin$.
There is therefore a commutative diagram of $(\infty,1)$-categories 
\[
\mymatrix{
\int \dgMod_{\coho(-)}^{\otimes} \ar[rr]^f \ar[rd]_-p && \int \dgRep^{\otimes} \ar[dl]^-q \\ & \int \MCdgLie\op \times \ptfin &
}
\]
where $f$ preserves coCartesian morphisms.
It follows that $f$ is classified by a (monoidal) natural transformation
\[
\mymatrix{
\int \dgLie\op \dcell[r][{\dgMod_{\coho(-)}}][{\dgRep}][f][=>][12pt] & \monoidalinftyCatu V
}
\]
\end{prop}

\begin{proof}
Whenever we fix $(A,L,\langle m \rangle)$ in $\int \MCdgLie \op \times \ptfin$, the functor $g$ restricted to the fibre categories admits a left adjoint (see \autoref{derivedcat-local}).
Moreover when $(A,L) \to (B,L')$ is a morphism in $\int \MCdgLie \op$, the following squares of monoidal functors commutes up to a canonical equivalence induced by the adjunctions
\[
\mymatrix{
\dgMod_{\coho_A L} \ar[r]^-{f_L^A} \ar[d]_{- \otimes_{\coho_A L} \coho_B(L \otimes_A B)} & \dgRep_A(L) \ar[d]^{- \otimes_{\Envel_A L} \Envel_B(L \otimes_A B)} \\
\dgMod_{\coho_B(L \otimes_A B)} \ar[r]^-{f^B_{L \otimes_A B}} \ar[d]_{- \otimes_{\coho_B(L \otimes_A B)} \coho_B(L')} & \dgRep_B(L \otimes_A B) \ar[d]^{\oubli} \\
\dgMod_{\coho_B L'} \ar[r]^-{f^B_{L'}} & \dgRep_B(L')
}
\]
For any family $(\el{V}{m})$ of $\coho_A L$-dg-modules, the canonical morphism
\[
\left(\bigotimes f_L^A(V_i) \right) \otimes_{\coho_A L} \coho_B L' \to f_{L'}^B \left( \left( \bigotimes V_i \right) \otimes_A B \right)
\]
is hence an equivalence.
This proves that $g$ satisfies the requirements of \cite[8.3.2.11]{lurie:halg}, admits a relative left adjoint $f$ which preserves coCartesian morphisms.
\end{proof}

Let us denote by $\int \MCdgMod^{\otimes}$ the category
\begin{itemize}
\item an object is a family $(A,\el{M}{m})$ where $A \in \MCcdga_k$ and $M_i \in \MCdgMod_A$
\item a morphism $(A,\el{M}{m}) \to (B,\el{N}{n})$ is the datum of a morphism $A \to B$, of a map $t \colon \langle m \rangle \to \langle n \rangle$ of pointed finite sets and for any $1 \leq j \leq n$ of morphism of $A$-dg-modules
\[
\otimes_{i \in t^{-1}(j)} M_i \to N_j
\]
\end{itemize}
There is a natural projection $\int \MCdgMod^{\otimes} \to \MCcdga_k \times \ptfin$.
We have three functors
\[
\mymatrix{
\int \MCdgRep^{\otimes} \ar[rr]^-\pi \ar[rd]_g && \int \MCdgMod^{\otimes} \times_{\MCcdga_k} \int \MCdgLie\op  \ar[dl]^\rho
\\ & \int \MCdgMod_{\coho(-)}^{\otimes} &
}
\]
compatible with the projections to $\int \MCdgLie\op \times \ptfin$. The functor $\pi$ is defined by forgetting the Lie action, while $\rho$ maps an $A$-dg-module $M$ and an $A$-dg-Lie algebra $L$ to the $\coho_A L$-dg-module $M$, where $\coho_A L$ acts through the augmentation map $\coho_A L \to A$.
The above triangle does not commute, but we have a natural transformation $g \to \rho \pi$, defined on a triple $(A,L,V)$ by
\[
\mymatrix{
g(A,L,V) = \Homint_L\left(\Envel_A\left(A[\eta] \otimes_A L\right), V\right) \ar[r]^-{\ev_\unit} & V = \rho \pi(A,L,V)
}
\]
We check that this map indeed commutes with the $\coho_A L$-action. Localising all that 	along quasi-isomorphisms, we get a tetrahedron 
\[
\shorthandoff{:;!?}
\xy <6mm,0cm>:
(1,0)*+{\int \dgMod^{\otimes} \times_{\MCcdga_k} \int \MCdgLie\op}="0",
(5,-2)*+{\int \dgRep^{\otimes}}="1",
(-3,-2)*+{\int \dgMod_{\coho(-)}^{\otimes}}="2",
(3,-5)*+{\int \MCdgLie\op \times \ptfin}="3",
\ar "1";"0" _(0.4)\pi
\ar "0";"2" _\rho
\ar "1";"2" ^g
\ar "0";"3" |!{"1";"2"}\hole ^(0.6)r
\ar "1";"3" ^q
\ar "2";"3" _p
\endxy
\]
where $p$, $q$ and $r$ are coCartesian fibrations -- see \cite[2.4.29]{lurie:dagx} -- where the upper face is filled with the natural transformation $g \to \rho \pi$ and where the other faces are commutative.

\begin{lem}\label{cocartesianretract}
The functor $\rho$ admits a relative left adjoint $\tau$ and the functor $\pi$ preserves coCartesian maps.
Moreover, the natural transformation $\tau \to \pi f$ -- induced by $g \to \rho \pi$ and by the adjunctions -- is an equivalence.
\end{lem}

\begin{rmq}
It follows from the above lemma the existence of natural transformations
\[
\mymatrix{
\int \dgLie\op \ar@/^20pt/[rr] _{}="up" ^{\dgMod_{\coho(-)}}
\ar[rr] |{\dgRep} ^*=<3pt>{}="midup" _*=<3pt>{}="middown"
\ar@/_20pt/[rr] ^{}="down" _{\dgMod}
&&
\monoidalinftyCatu V 
\ar @{=>} "up";"midup" ^-f
\ar @{=>} "middown";"down" ^-\pi
} 
\]
This lemma also describes the composite $\pi f$ as the base change functor along the augmentation maps $\coho_A L \to A$.
\end{rmq}

\begin{proof}
Let us first prove that $\rho$ admits a relative left adjoint $\tau$.
For any pair $(A,L) \in \int \MCdgLie\op$, the forgetful functor $\dgMod_A \to \dgMod_{\coho_A L}$ admits a left adjoint, namely the base change functor along the augmentation map $\coho_A L \to A$.
This left adjoint is monoidal and commutes with base change. It therefore fulfil the assumptions of \cite[8.3.2.11]{lurie:halg}.
The induced natural transformation $\tau \to \pi f$ maps a triple $(A,L,V) \in \int \dgMod_{\coho(-)}$ to the canonical map
\[
\tau^A_L(V) = V \otimes_{\coho_A L} A \to V \otimes_{\coho_A L} \Envel_A\left( L \otimes_A A[\eta] \right) = \pi^A_L f^A_L(V)
\]
which is an equivalence of $A$-dg-modules.
\end{proof}

Let us consider the functor of $(\infty,1)$-categories
\[
\dAff_k^{\Delta^2} \to \left(\monoidalinftyCatu V\right)\op
\]
mapping a sequence $X \to Y \to Z$ of derived affine schemes to the monoidal $(\infty,1)$-category $\Qcoh(Y)$.
We form the fibre product
\[
\mymatrix{
\Cc \cart \ar[r] \ar[d] & \dAff_k^{\Delta^2} \ar[d]^p \\
\dAff_k \ar[r]_-{\id_-} & \dAff_k^{\Delta^1}
}
\]
where $p$ is induced by the inclusion $(0 \to 2) \to (0 \to 1 \to 2)$. Finaly, we define $\Dd$ as the full subcategory of $\Cc$ spanned by those triangles $\Spec A \to \Spec B \to \Spec A$ where $B \in \dgExt_A$.
We get a functor
\[
F \colon \Dd \to \left(\monoidalinftyCatu V\right)\op
\]
Note that the functor $\Dd \to \dAff_k$ is a Cartesian fibration classified by the functor $A \mapsto \dgExt_A\op$.

Let us denote by $\oint \dStF \to \dAff_k$ the Cartesian fibration classified by the functor $\Spec A \mapsto \dStF_A = \sifted(\dgExt_A\op)$.
The Yoneda natural transformation $\dgExt\op \to \dStF$ defines a functor $\Dd \to \oint \dStF$.
We define
\[
\Lqcoh \colon \oint \dStF \to \left(\monoidalinftyCatu V\right)\op
\]
as the left Kan extension of $F$ along $\Dd \to \oint \dStF$.

Let now $X$ be a derived stack. The category $\dStF_X$ defined in \autoref{dstfoverstack} is equivalent to the category of Cartesian sections $\phi$ as below -- see \cite[3.3.3.2]{lurie:htt}
\[
\mymatrix{
& \oint \dStF \ar[d] \\ \quot{\dAff_k}{X} \ar[ru]^\phi \ar[r] & \dAff_k
}
\]

\begin{df}\label{lqcohdstf}
Let $X$ be a derived stack. We define the functor of derived category of formal stacks over $X$:
\[
\Lqcoh^X \colon \mymatrix{
\dStF_X \simeq \Fct_{\dAff_k}^{\mathrm{Cart}}\left( \quot{\dAff_k}{X}, \oint \dStF \right) \ar[r]^-{\Lqcoh}
& \Fct\left( \quot{\dAff_k}{X}, \left(\monoidalinftyCatu V\right)\op \right) \ar[r]^-{\colim} & \left(\monoidalinftyCatu V\right)\op
}
\]
If $Y \in \dStF_X$ then $\Lqcoh^X(Y)$\glsadd{lqcoh} is called the derived category of the formal stack $Y$ over $X$. We can describe it more intuitively as the limit of symmetric monoidal $(\infty,1)$-categories
\[
\Lqcoh^X(Y) = \lim_{\Spec A \to X} ~ \underset{\Spec B \to Y_A}{\lim_{B \in \dgExt_A}} \dgMod_B
\]
where $Y_A \in \dStF_A$ is the pullback of $Y$ along the morphism $\Spec A \to X$.
\end{df}

The same way, the opposite category of dg-Lie algebras over $X$ is equivalent to that of coCartesian section
\[
\dgLie_X\op \simeq \Fct^{\mathrm{coC}}_{\cdga_k}\left(\left(\quot{\dAff_k}{X}\right)\op, \int \dgLie\op \right)
\]
We can thus define
\begin{df}
Let $X$ be a derived stack. We define the functor of Lie representations over $X$ to be the composite functor $\dgRep_X \colon \dgLie_X\op \to \monoidalinftyCatu V$
\[
\mymatrix{
\dgLie_X\op \ar[r] & \Fct\left(\left(\quot{\dAff_k}{X}\right)\op, \int \dgLie\op \right) \ar[r]^-{\dgRep} & \Fct\left(\left(\quot{\dAff_k}{X}\right)\op, \monoidalinftyCatu V \right) \ar[r]^-\lim & \monoidalinftyCatu V
}
\]
In particular for any $L \in \dgLie_X$, this defines a symmetric monoidal $(\infty,1)$-category 
\[
\dgRep_X(L) = \lim_{\Spec A \to X} \dgRep_A(L_A)
\]
where $L_A \in \dgLie_A$ is the dg-Lie algebra over $A$ obtained by pulling back $L$.
\end{df}

\begin{prop}\label{ffformallierep}
Let $X$ be a derived stack.
There is a natural transformation
\[
\mymatrix{
\dgLie_X\op \dcell[r][{\Lqcoh^X(\formal_X(-))}][{\dgRep_X}][\Psi] & \monoidalinftyCatu V
}
\]
Moreover, for any $L \in \dgLie_X$, the induced monoidal functor $\Lqcoh^X(\formal_X L) \to \dgRep_X(L)$ is fully faithful and preserves colimits.
\end{prop}

To prove this proposition, we will need the following

\begin{lem}\label{lqcohformaltodgrep}
The natural transformation $\formal \colon \dgLie \to \dStF$ together with the functor
\[
\Lqcoh \colon \left(\oint \dStF\right)\op \to \monoidalinftyCatu V
\]
define a functor $\phi \colon \int \dgLie\op \to \monoidalinftyCatu V$.
There is a pointwise fully faithful and colimit preserving natural transformation $\phi \to \dgRep$.
\end{lem}
\begin{proof}
Let $\Ee$ denote the full subcategory of $\int \dgLie\op$ such that the induced coCartesian fibration $\Ee \to \cdga_k$ is classified by the subfunctor
\[
A \mapsto \left(\dgLieLib_A\right)\op \subset \dgLie_A\op
\]
The functor $\phi$ is by construction the right Kan extension of its restriction $\psi$ to $\Ee$.
Moreover, the restriction $\psi$ is by definition equivalent to the composite functor
\[
\mymatrix{
\Ee \ar[r] & \int \dgLie\op \ar[rr]^-{\dgMod_{\coho(-)}} && \monoidalinftyCatu V
}
\]
Using the natural transformation $\dgMod_{\coho(-)} \to \dgRep$ from \autoref{derivedcat-adjunction}, we get
\[
\alpha \colon \psi \to \dgRep_{|\Ee} \in \Fct\left(\Ee, \monoidalinftyCatu V\right)
\]
We will prove the following sufficient conditions.
\begin{enumerate}
\item The functor $\dgRep \colon \int \dgLie\op \to \monoidalinftyCatu V$ is the right Kan extension of its restriction $\dgRep_{|\Ee}$\label{nat-kanext}
\item The natural transformation $\alpha$ is pointwise fully faithful and preserves finite colimits.\label{alphaff}
\end{enumerate}
Condition \ref{alphaff} simply follows from \autoref{derivedcat-ff-local}.
To prove condition \ref{nat-kanext}, it suffices to see that when $A$ is fixed, the functor
\[
\dgRep_A \colon \dgLie_A\op \to \monoidalinftyCatu V
\]
commutes with sifted limits. This follows from \cite[2.4.32]{lurie:dagx}.
\end{proof}

\begin{proof}[of \autoref{ffformallierep}]
Let $X$ be a derived stack and $L$ be a dg-Lie algebra over $X$. Recall that we can see $L$ as a functor
\[
L \colon \left(\quot{\dAff_k}{X}\right)\op \to \int \dgLie\op
\]
By definition, we have $\Lqcoh^X(\formal L) = \lim \phi \circ L$ and $\dgRep_X(L) = \lim \dgRep \circ L$.
We deduce the result from \autoref{lqcohformaltodgrep}.
\end{proof}

We can now prove the promised \autoref{derived-global}.
\begin{proof}[of \autoref{derived-global}]
Let us first build a functor
\[
\nu_X \colon \Qcoh(X) \to \Lqcoh^X\left(\for{(X \times X)}\right)
\]
Let us denote by $\Cc$ the $(\infty,1)$-category of diagrams
\[
\mymatrix{
\Spec A \ar[r] & \Spec B \ar[d]^\alpha \ar[r] & \Spec A \\ & X &
}
\]
where $A \in \cdga_k$ and $B \in \dgExt_A$. There is a natural functor $\Cc\op \to \monoidalinftyCatu V$ mapping a diagram as above to the monoidal $(\infty,1)$-category $\dgMod_B$.
Unwinding the definitions, we contemplate an equivalence of monoidal categories
\[
\Lqcoh^X\left(\for{(X \times X)}\right) \simeq \lim_{\Cc} \dgMod_B
\]
The maps $\alpha$ as above induce (obviously compatible) pullback functors $\alpha^* \colon \Qcoh(X) \to \dgMod_B$. This construction defines the announced monoidal functor
\[
\nu_X \colon \Qcoh(X) \to \Lqcoh^X\left( \for{(X \times X)} \right)
\]
Going back to \autoref{tgtliedef} and \autoref{dstfoverstack}, we have a canonical morphism
\[
\theta \colon \formal_X(\tgtlie_X) = \formal_X \lie_X\left( \for{(X \times X)} \right) \to \for{(X \times X)}
\]
Now using the functor from \autoref{ffformallierep}, we get a composite functor
\[
\lierep_X \colon 
\mymatrix{
\Qcoh(X) \ar[r]^-{\nu_X} & \Lqcoh^X\left(\for{(X \times X)}\right) \ar[r]^-{\theta^*} & \Lqcoh^X(\formal_X(\tgtlie_X)) \ar[r]^-\Psi & \dgRep_X(\tgtlie_X)
}
\]
As every one of those functors is both monoidal and colimit preserving, so is $\lierep_X$.
We still have to prove that $\lierep_X$ is a retract of the forgetful functor $\Theta_X \colon \dgRep_X(\tgtlie_X) \to \Qcoh(X)$.
We consider the composite functor $\Theta_X \lierep_X$
\[
\mymatrix{
\Qcoh(X) \ar[r]^-{\nu_X} & \Lqcoh^X\left(\for{(X \times X)}\right) \ar[r]^-{\theta^*} & \Lqcoh^X(\formal_X(\tgtlie_X)) \ar[r]^-\Psi & \dgRep_X(\tgtlie_X) \ar[r]^-{\Theta_X} & \Qcoh(X)
}
\]
It follows from \autoref{cocartesianretract} that the composite functor $\Theta_X \Psi$ is equivalent to the pullback
\[
\Lqcoh^X(\formal_X (\tgtlie_X)) \to \Lqcoh^X(X) \simeq \Qcoh(X)
\]
along the canonical morphism $X \to \formal_X (\tgtlie_X)$ of formal stacks over $X$.
It follows that $\Theta_X \lierep_X$ is equivalent to the composite functor
\[
\mymatrix{
\Qcoh(X) \ar[r]^-{\nu_X} & \Lqcoh^X\left(\for{(X \times X)}\right) \ar[r]^-{\alpha^*} & \Qcoh(X)
}
\]
where $\alpha$ is the morphism $X \to \for{(X \times X)}$. Unwinding the definition of $\nu_X$, we see that this composite functor is equivalent to the identity functor of $\Qcoh(X)$.
\end{proof}

\subsection{Atiyah class, modules and tangent maps}

\newcommand{\Perfdst}{\underline \Perf}
\begin{df}
Let $\Perfdst$ denote the derived stack of perfect complexes. It is defined as the stack mapping a cdga $A$ to the maximal $\infty$-groupoid in the $(\infty,1)$-category $\Perf(A)$.
For any derived stack $X$, we set $\Perfdst(X)$ to be the maximal groupoid in $\Perf(X)$. It is equivalent to space of morphisms from $X$ to $\Perfdst$ in $\dSt_k$.
\end{df}

\begin{df}\label{atiyah-def}
Let $X$ be a derived Artin stack locally of finite presentation. Any perfect module $E$ over $X$ is classified by a map $\phi_E \colon X \to \Perfdst$. Following \cite{toen:derivedK3}, we define the Atiyah class of $E$ as the tangent morphism of $\phi_E$
\[
\atiyah_E \colon \T_X \to \phi_E^* \T_{\Perfdst}
\]
\end{df}

\begin{rmq}
We will provide an equivalence $\phi_E^* \T_{\Perfdst} \simeq \End(E)[1]$ in the proof of \autoref{derivedcat-atiyah}.
The Atiyah class of $E$ should be thought as the composition
\[
\atiyah_E \colon \T_X[-1] \to \phi_E^* \T_{\Perfdst}[-1] \simeq \End(E)
\]
We will, at the end of this section, compare this definition of the Atiyah class with the usual one -- see \autoref{atiyah-compare}.
\end{rmq}

\begin{prop}\label{derivedcat-atiyah}
Let $X$ be an algebraic stack locally of finite presentation.
When $E$ is a perfect module over $X$, then the $\tgtlie_X$-action on $E$ given by the \autoref{derived-global} is induced by the Atiyah class of $E$.
\end{prop}
\begin{lem}\label{perfequiv}
Let $A \in \cdga_k$ and $L \in \dgLieLib_A$. The functor
\[
f_L^A \colon \dgMod_{\coho_A(L)} \to \dgRep_A(L)
\]
defined in \autoref{derivedcat-local} induces an equivalence 
\[
\Perf(\coho_A L) \to^\sim \dgRep_A(L) \timesunder[\dgMod_A] \Perf(A)
\]
\end{lem}
\begin{proof}
We proved in \autoref{derivedcat-ff-local} the functor $f_L$ to be fully faithful.
Let us denote by $\Cc$ the image category
\[
\Cc = f_L^A\left(\Perf(\coho_A(L))\right)
\]
From Lurie's work (\cite{lurie:dagx}) over a field and from \autoref{derivedcat-adjunction} we deduce that $\Cc$ contains any representation of $L$ whose underlying $A$-module is (equivalent to) $A^n$.
Moreover the category $\Cc$ is stable under pushouts and fibre products. The forgetful functor $\Cc \to \dgMod_A$ also preserves pushouts and fibre products. The category $\Cc$ therefore contains any representation whose underlying module is perfect. Reciprocally, every representation in $\Cc$ has a perfect underlying complex.
\end{proof}
\begin{proof}[of the \autoref{derivedcat-atiyah}]
The sheaf $E$ corresponds to a morphism $\phi_E \colon X \to \Perfdst$. Its Atiyah class is the tangent morphism
\[
\atiyah_E \colon \T_X[-1] \to \phi_E^* \T_{\Perfdst}[-1]
\]
In our setting, we get a Lie tangent map
\[
\atiyah_E \colon \tgtlie_X \to \phi_E^* \tgtlie_{\Perfdst}
\]
We observe here that $\phi_E^* \tgtlie_{\Perfdst}$ represents the presheaf on $\dgLie_X$
\[
L \mapsto \mathrm{Gpd}\left(\Lqcoh^X(\formal_X(L)) \timesunder[\Qcoh(X)] \{E\}\right)
\]
while $\gl(E)$ -- the dg-Lie algebra of endomorphisms of $E$ -- represents the functor
\[
L \mapsto \mathrm{Gpd}\left(\dgRep_X(L) \timesunder[\Qcoh(X)] \{E\}\right)
\]
where $\mathrm{Gpd}$ associates to any $(\infty,1)$-category its maximal groupoid.
We get from \autoref{ffformallierep} a morphism $\phi_E^* \tgtlie_{\Perfdst} \to \gl(E)$ of dg-Lie algebras over $X$.
Restricting to an affine derived scheme $s \colon \Spec A \to X$, we get that $s^* \phi_E^* \tgtlie_{\Perfdst}$ and $s^* \gl(E) \simeq \gl(s^* E)$ respectively represent the functors $(\dgLieLib_A)\op \to \sSets$
\[
L^0 \mapsto \Perfdst(\coho_A L^0) \timesunder[\Perfdst(A)] \{E\} \text{~~~and~~~} 
L^0 \mapsto \mathrm{Gpd}\left(\dgRep_A(L^0) \timesunder[\dgMod_A] \{ E \}\right)
\]
The natural transformation induced between those functors is the one of \autoref{perfequiv} and is thus an equivalence.
We therefore have
\[
\atiyah_E \colon \tgtlie_X \to \gl(E)
\]
and hence an action of $\tgtlie_X$ on $E$.
This construction corresponds to the one of \autoref{derived-global} through the equivalence
$\Perfdst(X) \simeq \Map(X,\Perfdst)$.
\end{proof}

We will now focus on comparing our \autoref{atiyah-def} of the Atiyah class with a more usual one. Let $X$ be a smooth variety.
Let us denote by $X^{(2)}$ the infinitesimal neighbourhood of $X$ in $X \times X$ through the diagonal embedding.
We will also denote by $i$ the diagonal embedding $X \to X^{(2)}$ and by $p$ and $q$ the two projections $X^{(2)} \to X$.
We have an exact sequence
\begin{equation}
i_* \Lcot_X \to \Oo_{X^{(2)}} \to i_* \Oo_X \label{exactseq-atiyah}
\end{equation}
classified by a morphism $\alpha \colon i_* \Oo_X \to i_* \Lcot_X[1]$. The Atiyah class of a quasi-coherent sheaf $E$ is usually obtained from this extension class by considering the induced map -- see for instance \cite{kuznetsovmarkushevich:sympandatiyah}
\begin{equation}
E \simeq p_*(i_* \Oo_X \otimes q^* E) \to p_*(i_* \Lcot_X[1] \otimes q^*E) \simeq \Lcot_X[1] \otimes E \label{defatiyahold}
\end{equation}
From the map $\alpha$, we get a morphism $i^* i_* \Oo_X \to \Lcot_X[1]$. Dualising we get
\[
\beta \colon \T_X[-1] \to \Homint_{\Oo_X}(i^* i_* \Oo_X, \Oo_X) \simeq p_* i_* \Homint_{\Oo_X}(i^* i_* \Oo_X, \Oo_X) \simeq p_* \Homint_{\Oo_{X^{(2)}}}(i_* \Oo_X, i_* \Oo_X)
\]
The right hand side naturally acts on the functor $i^* \simeq p_*(- \otimes_{\Oo_{X^{(2)}}} i_* \Oo_X)$ and hence on $i^* q^* \simeq \id$. This action, together with the map $\beta$, associates to any perfect module $E$ a morphism $\T_X[-1] \otimes E \to E$.
It corresponds to a map $E \to E \otimes \Lcot_X[1]$ which is equivalent to the Atiyah class in the sense of (\ref{defatiyahold}).

\begin{prop}\label{atiyah-compare}
Let $X$ be a smooth algebraic variety and let $E$ be a perfect complex on $X$. The Atiyah class of $E$ in the sense of \autoref{atiyah-def} and the construction (\ref{defatiyahold}) are equivalent to one another.
\end{prop}

\begin{proof}
We first observe the two following facts
\begin{itemize} 
\item The derived category $\Lqcoh^X(\for{(X^{(2)})})$ is equivalent to $\Qcoh(X^{(2)})$.
\item The tangent Lie algebra $\lie_X(\for{(X^{(2)})})$ is the free Lie algebra generated by $\T_X[-1]$.
\end{itemize}
We moreover have a commutative diagram, where $i \colon X^{(2)} \to X \times X$ is the inclusion.
\[
\mymatrix{
\Qcoh(X) \ar@/^15pt/[rr]^{\lierep_X} \ar[r] \ar[dr]_{q^*} & \Lqcoh^X\left(\for{(X \times X)}\right) \ar[d]^{u^*} \ar[r] & \dgRep_X(\tgtlie_X) \ar[d]\\
& \Qcoh(X^{(2)}) \ar[r] \ar@/_15pt/[rr]_{i^*} & \dgRep_X(\libre_X(\T_X[-1])) \ar[r] & \Qcoh(X)
}
\]
where $u$ is the natural morphism $X^{(2)} \to X \times X$.

From what precedes, the Atiyah class arises from an action on $i^*$, and we can thus focus on the composite
\[
\mymatrix{
\Qcoh(X^{(2)}) \ar[r] & \dgRep_X(\libre_X(\T_X[-1])) \ar[r] & \Qcoh(X)
}
\]
which can be studied locally. Let thus $a \colon \Spec A \to X$ be a morphism.
Let us denote by $L$ the $A$-dg-Lie algebra $\libre_A(\T_{X,a}[-1])$. Pulling back on $A$ the functors above, we get
\[
\mymatrix{
\dgMod_{\coho_A L} \ar[r]^{f^A_L} & \dgRep_A(L) \ar[r] & \dgMod_A
}
\]
where $f_L^A$ is given by the action of $L$ on $\Envel_A(L \otimes_A A[\eta])$ through the natural inclusion.
On the other hand, the universal Atiyah class $\alpha$ defined above can be computed as follows
\[
\mymatrix{
\coho_A L \ar[r] \ar[d]_\simeq & \Homint_A(\Envel_A(L \otimes_A A[\eta]),A) \ar[d]^\simeq \\
A \oplus \Lcot_{X,a} \ar[r] \ar[d] & A \oplus \left(\Lcot_{X,a} \otimes_A A[\eta]\right) \ar[d] \\
0 \ar[r] & \Lcot_{X,a}[1] \cocart
}
\]
The universal Atiyah class is thus dual to the inclusion $\T_{X,a}[-1] \to \Envel_A(L \otimes_A A[\eta])$.
It follows that the action defined by the functor $f^A_L$ is indeed given by the Atiyah class.
We now conclude using \autoref{derivedcat-atiyah}.
\end{proof}

\subsection{Adjoint represention}
In this subsection, we will focus on the following statement.
\begin{prop}\label{adjointrepresentation}
Let $X$ be a derived Artin stack. The $\tgtlie_X$-module $\lierep_X(\T_X[-1])$ is equivalent to the adjoint representation of $\tgtlie_X$.
\end{prop}

The above proposition, coupled with \autoref{derivedcat-atiyah}, implies that the bracket of $\tgtlie_X$ is as expected given by the Atiyah class of the tangent complex.
To prove it, we will need a few constructions.
\excludecomment{oldschoolmerdique}
\includecomment{newschoolcool}
\includecomment{oldschoolgood}
\excludecomment{newschoolbad}
\begin{lem}\label{adjrep-adjoint}
Let $A \in \MCcdga_k$ and $L \in \MCdgLie_A$. To any $A$-dg-Lie algebra $L'$ with a morphism $\alpha \colon L \to L'$ we associate the underlying representation $\psi^A_L(L')$ of $L$ -- ie the $A$-dg-module $L'$ with the action of $L$ through the morphism $\alpha$.
\begin{oldschoolmerdique}
The functor $\psi_L^A$ is a right Quillen functor
\[
\phi^A_L \colon \MCdgRep_A(L) \rightleftarrows \comma{L}{\MCdgLie_A} \noloc \psi^A_L
\]
\end{oldschoolmerdique}
\begin{newschoolcool}
The functor $\psi_L^A$ preserves quasi-isomorphisms. It induces a functor between the localised $(\infty,1)$-categories, which admits a left adjoint $\phi_L^A$:
\[
\phi_L^A \colon \dgRep_A(L) \rightleftarrows \comma{L}{\dgLie_A} \noloc \psi^A_L
\]
\end{newschoolcool}
\end{lem}

\begin{proof}
\begin{oldschoolmerdique}
Since $\psi^A_L$ by definition preserves fibrations and quasi-isomorphisms, it suffices to provide the left adjoint $\phi^A_L$ to $\psi^A_L$.
We consider a dg-representation $M$ of $L$ and we define $\phi^A_L(M)$ to be pushout
\[
\mymatrix{
\libre_A(L \otimes M \oplus L \otimes L) \ar[r] \ar[d]^\alpha & \libre_A((L \otimes M \oplus L \otimes L)\otimes_A A[\eta]) \ar[d] \\
\libre_A(L \oplus M) \ar[r] & \phi^A_L(M) \cocart
}
\]
where the vertical map $\alpha$ is given by
\begin{align*}
&\alpha(x \otimes m) = [x,m] - x(m) \\
&\alpha(x \otimes y) = [x,y] - [x,y]_L
\end{align*}
where $x,y \in L$ and $m \in M$, $x(m)$ is the action of $x$ on $m$ and $[x,y]_L$ is the original bracket on $L$.
It is the homotopy quotient of the free dg-Lie algebra $\libre(L \oplus M)$ by the relations
\begin{align*}
&[x,m] = x(m) \\
&[x,y] = [x,y]_L
\end{align*}
\todo{Préciser}The fact that $\phi_L^A$ is left adjoint to $\psi^A_L$ is obvious.
\end{oldschoolmerdique}
\begin{newschoolcool}
The functor $\psi^A_L$ preserves small limits and both its ends are presentable $(\infty,1)$-categories. Since both $\dgRep_A(L)$ and $\dgLie_A$ are monadic over $\dgMod_A$, the functor $\psi^A_L$ is accessible for the cardinal $\omega$. The result follows from \cite[5.5.2.9]{lurie:htt}
\end{newschoolcool}
\end{proof}

\begin{oldschoolmerdique}
\begin{rmq}
Let us remark that if $L$ is a free $A$-dg-Lie algebra generated by a module $N$, then $\phi_L^A(M)$ is quasi-isomorphic to the pushout
\[
\mymatrix{
\libre_A(N \otimes M) \ar[r] \ar[d] & \libre_A((N \otimes M)\otimes_A A[\eta]) \ar[d] \\
\libre_A(N \oplus M) \ar[r] & \phi^A_L(M) \cocart
}
\]
\end{rmq}
\end{oldschoolmerdique}

\begin{newschoolbad}
\begin{lem}\label{adjointderivation}
Let $A$ be a cdga and $L$ be a dg-Lie algebra over $A$. For any $M$ in $\MCdgRep_A(L)$ there is a derivation $\delta_M \colon \coho_A (\phi^A_L(M)) \to g^A_{L}(\dual M[-1])$ -- where $g^A_L$ is defined in \autoref{derivedcat-local}.
This construction defines a natural transformation
\[
\mymatrix{
\left(\comma{\MCcdga_A}{\coho_A L}\right)\op && \comma{L}{\MCdgLie_A} \ar[ll]_-{\coho_A} \ar@{<=}[dll] \\
\left( \MCdgMod_{\coho_A L} \right)\op \ar[u]^{\coho_A L \oplus (-)} &&
\MCdgRep_A(L) \ar[u]_{\phi^A_L} \ar[ll]^-{g^A_L(\dual {(-)}[-1])}
}
\]
\end{lem}
\begin{proof}
The canonical map $L \otimes M \to \phi^A_L(M)$ induces a map
\[
\Homint_A(\Sym_A(\phi^A_L(M)[1]),A) \to \Homint_A(\Sym(L[1]) \otimes M[1],A) \simeq \Homint_A(\Sym_A(L[1]),\dual M[-1])
\]
which in turn defines a morphism of \emph{graded} modules 
\[
\delta_L \colon \coho_A(\phi^A_L(M)) \to g_L^A(\dual M[-1])
\]
A quick computation proves it commutes with the differentials. An even quicker computation shows it is a derivation.
This construction defines the required natural transformation.
\end{proof}
\end{newschoolbad}

\begin{oldschoolgood}
\begin{lem}\label{adjointderivation}
Let $A$ be a cdga and $L_0$ be a dg-Lie algebra over $A$.
There is a natural transformation
\[
\mymatrix{
\comma{L_0}{\MCdgLie_A} \ar[rrrr]^-{\coho_A} \ar[dr]_{\psi^A_{L_0}}
&& \ar@{<=}[d] && \left(\quot{\MCcdgaunbounded_A}{\coho_A L_0}\right)\op
\\ & \MCdgRep_A(L_0) \ar[rr]_-{g^A_{L_0}\left(\dual{(-)}\right)} && \left(\MCdgMod_{\coho_A L_0}\right)\op \ar[ur]_-{\hspace{4mm}\coho_A L_0 \oplus (-)[-1]}
}
\]
where $g^A_{L_0}$ was defined in \autoref{derivedcat-local}.
\end{lem}
\begin{proof}
Let $\alpha \colon L_0 \to L$ be a morphism of $A$-dg-Lie algebras.
The composite morphism
\[
\Sym_A(L_0[1]) \otimes_A L[1] \to^\alpha \Sym_A(L[1]) \otimes_A L[1] \to \Sym_A(L[1])
\]
induces a morphism
\[
\Homint_A(\Sym_A(L[1]),A) \to \Homint_A(\Sym(L_0[1]) \otimes L[1],A) \simeq \Homint_A(\Sym_A(L_0[1]), \dual L [-1])
\]
This defines a map of \emph{graded} modules $\delta_L \colon \coho_A L \to g_{L_0}^A(\dual L)[-1]$.
Let us prove that it commutes with the differentials. Recall the notations $S_{i}$ and $T_{ij}$ from \autoref{rmq-coalg}. We compute on one hand
\begin{align*}
\delta(d \xi)(\el{\eta.x}{n}[\otimes]) &(\eta.y_{n+1}) = \sum_{i\leq n+1} (-1)^{S_i+1} \xi(\eta.y_1 \otimes \dots \otimes \eta.dy_i \otimes \dots \otimes \eta.y_{n+1}) \\
& + \sum_{i < j \leq n+1} (-1)^{T_{ij}} \xi(\eta.[y_i,y_j] \otimes \eta.y_1 \otimes \dots \otimes \widehat{\eta.y_i} \otimes \dots \otimes \widehat{\eta.y_j} \otimes \dots \otimes \eta.y_{n+1})\\
& + d(\xi(\el{\eta.y}{n+1}[\otimes]))
\end{align*}
where $y_i$ denotes $\alpha x_i$, for any $i \leq n$. On the other hand, we have
\begin{align*}
(d(\delta \xi)) (\el{\eta.x}{n}[\otimes]) &= \sum_{i\leq n} (-1)^{S_i+1} \delta\xi(\eta.x_1 \otimes \dots \otimes \eta.dx_i \otimes \dots \otimes \eta.x_{n}) \\
& + \sum_{i < j \leq n} (-1)^{T_{ij}} \delta\xi(\eta.[x_i,x_j] \otimes \eta.x_1 \otimes \dots \otimes \widehat{\eta.x_i} \otimes \dots \otimes \widehat{\eta.x_j} \otimes \dots \otimes x_{n}) \\
& + \sum_{i \leq n} (-1)^{(|x_i| + 1)S_i} x_i \bullet \delta \xi(\eta.x_1 \otimes \dots \otimes \widehat{\eta.x_i} \otimes \dots \otimes \eta.x_n) \\
& + d(\delta \xi(\el{\eta.x}{n}[\otimes]))
\end{align*}
where $\bullet$ denotes the action of $L_0$ on $\dual L$.
We thus have
\begin{align*}
(d(\delta \xi)) (&\el{\eta.x}{n}[\otimes])(\eta.y_{n+1}) = \sum_{i\leq n} (-1)^{S_i+1} \xi(\eta.y_1 \otimes \dots \otimes \eta.dy_i \otimes \dots \otimes \eta.y_{n} \otimes \eta.y_{n+1}) \\
& + \sum_{i < j \leq n} (-1)^{T_{ij}} \xi(\eta.[y_i,y_j] \otimes \eta.y_1 \otimes \dots \otimes \widehat{\eta.y_i} \otimes \dots \otimes \widehat{\eta.y_j} \otimes \dots \otimes \eta.y_{n} \otimes \eta.y_{n+1}) \\
& + \sum_{i \leq n} (-1)^{(|y_i| + 1)S_i + (|y_i|-1 + S_{n+1})|y_i|} \xi(\eta.y_1 \otimes \dots \otimes \widehat{\eta.y_i} \otimes \dots \eta.y_n \otimes \eta.[y_{n+1},y_i]) \\
& + d(\delta \xi(\el{\eta.x}{n}[\otimes]))(\eta.y_{n+1})
\end{align*}
Now computing the difference $\delta(d \xi)(\el{\eta.x}{n}[\otimes]) (\eta.y_{n+1}) - (d(\delta \xi))(\el{\eta.x}{n}[\otimes])(\eta.y_{n+1})$ we get
\begin{align*}
(-1)^{S_{n+1} + 1} \xi(\eta.y_1& \otimes \dots \otimes \eta.y_n \otimes \eta.dy_{n+1}) \\ &+ d(\xi(\el{\eta.y}{n+1}[\otimes])) - d( \delta \xi(\el{\eta.x}{n}[\otimes]))(\eta.y_{n+1}) = 0
\end{align*}
It follows that $\delta_L$ is indeed a morphism of complexes $\coho_A L \to g^A_{L_0}(\dual L)[-1]$.
It is moreover $A$-linear.
One checks with great enthusiasm that it is a derivation. This construction is moreover functorial in $L$ and we get the announced natural transformation.
\todo{derivation}
\end{proof}
\end{oldschoolgood}

Let us define the category $\int \comma{\pt}{\MCdgLie}$ as follows
\begin{itemize}
\item An object is a triple $(A,L,L \to L_1)$ where $A \in \MCcdga_k$ and $L \to L_1 \in \MCdgLie_A$.
\item A morphism $(A,L,L \to L_1) \to (B,L',L' \to L_1')$ is the data of
\begin{itemize}
\item A morphism $A \to B$ in $\MCcdga_k$,
\item A commutative diagram
\[
\mymatrix{
L' \ar[r] \ar[d] & L \otimes_A B \ar[d] \\ L_1' & L_1 \otimes_A B \ar[l]
}
\]
\end{itemize}
\end{itemize}

This category comes with a coCartesian projection to $\int \comma{\pt}{\MCdgLie} \to \int \MCdgLie\op$.
The forgetful functor $\comma{L}{\MCdgLie_A} \to \MCdgRep_A(L)$ define a functor $\adjrep$ such that the following triangle commutes
\[
\mymatrix{
\int \comma{\pt}{\MCdgLie} \ar[rr]^-\adjrep \ar[dr] & & \int \MCdgRep \ar[dl] \\ & \int \MCdgLie\op
}
\]
Let us define the category $\int (\quot{\MCcdgaunbounded}{\coho(-)})\op$ as follows
\begin{itemize}
\item An object is a triple $(A,L,B)$ where $(A,L) \in \int \MCdgLie\op$ and $B$ is a cdga over $A$ with a map $B \to \coho_A L$ ;
\item A morphism $(A,L,B) \to (A',L',B')$ is a commutative diagram
\[
\mymatrix{
A \ar[d] \ar[r] & B \ar[r] \ar[d] & \coho_A L \ar[d] \\
A' \ar[r] & B' \ar[r] & \coho_{A'} L'
}
\]
where $\coho_A L \to \coho_{A'} L'$ is induced by a given morphism $L' \to L \otimes^\Lcot_A A'$.
\end{itemize}
Let us remark here that $\coho$ induces a functor $\chi \colon \int \comma{\pt}{\MCdgLie} \to \int (\quot{\MCcdgaunbounded}{\coho(-)})\op$ which commutes with the projections to $\int \MCdgLie\op$.
The construction $\MCdgRep_A(L) \to (\quot{\MCcdgaunbounded_A}{\coho_A L})\op$
\[
V \mapsto \coho_A L \oplus g^A_L(\dual V)[1]
\]
defines a functor
\[
\theta \colon \int \MCdgRep \to \int \left(\quot{\MCcdgaunbounded}{\coho(-)}\right)\op
\]
and \autoref{adjointderivation} gives a natural transformation $\theta \adjrep \to \chi$.
Localising along quasi-isomorphisms, we get a thetaedron
\[
\shorthandoff{:;!?}
\xy <6mm,0cm>:
(1,0)*+{\int \quot{*}{\dgLie}}="0",
(5,-2)*+{\int \dgRep}="1",
(-3,-2)*+{\int (\quot{\cdgaunbounded}{\coho(-)})\op}="2",
(3,-5)*+{\int \MCdgLie\op}="3",
\ar "0";"1" ^(0.6)\adjrep
\ar "0";"2" _\chi
\ar "1";"2" ^\theta
\ar "0";"3" |!{"1";"2"}\hole ^(0.6)r
\ar "1";"3" ^q
\ar "2";"3" _p
\endxy
\]
where the upper face is filled with the natural transformation $\theta \adjrep \to \chi$ and the other faces are commutative.
\begin{lem}\label{adjrep-basechange}
The functor $\adjrep$ admits a relative left adjoint $\phi$ over $\int \MCdgLie\op$. Moreover, the induced natural transformation $\theta \to \theta \adjrep \phi \to \chi \phi$ is an equivalence.
\end{lem}

\begin{proof}
The first statement is a consequence of \autoref{adjrep-adjoint} and \cite[8.3.2.11]{lurie:halg}.
To prove the second one, we fix a pair $(A,L) \in \int \MCdgLie\op$ and study the induced natural transformation
\[
\mymatrix{
\dgRep_A(L) \ar[rr] \ar[rd] && \left(\quot{\cdgaunbounded_A}{\coho_A L}\right)\op \\
& \comma{L}{\dgLie_A} \ar[ur] \ar@{<=}[u] &
}
\]
\begin{newschoolcool}
The category $\dgRep_A(L)$ is generated under colimits of the free representations $\Envel_A L \otimes N$, where $N \in \dgMod_A$.
Both the upper and the lower functors map colimits to limits. 
Since $\phi^A_L(\Envel_A L \otimes N) \simeq L \amalg \libre_A (N)$, we can restrict to proving that the induced morphism
\[
\coho_A L \oplus \dual N[-1] \to \coho_A L \oplus g^A_L(\dual{(\Envel_A L \otimes N)}[-1])
\]
is an equivalence.
We have the following morphism between exact sequences
\[
\mymatrix{
\dual N[-1] \ar[r] \ar[d]^\beta & \coho_A L \oplus \dual N[-1] \ar[d] \ar[r] & \coho_A L \ar[d]^{=} \\
g^A_L(\dual{(\Envel_A L \otimes N)}[-1]) \ar[r] & \coho_A L \oplus g^A_L(\dual{(\Envel_A L)}[-1]) \ar[r] & \coho_A L
}
\]
Since the functors $\dual{(-)}$ and $g^A_L(\dual{(\Envel_A L \otimes -)}[-1])$ from $\dgMod_A$ to $\dgMod_A\op$ are both left adjoint to same functor, the morphism $\beta$ is an equivalence.
\end{newschoolcool}
\begin{oldschoolmerdique}
Those three functors admit right adjoints, and we get
\[
\mymatrix{
\dgRep_A(L) \ar@{<-}[rr]^{\dual{(f^A_L(\Lcot_{-/A}[1] \otimes \coho_A L))}} \ar@{<-}[rd]_{\phi_A} && \left(\quot{\cdgaunbounded_A}{\coho_A L}\right)\op \\
& \comma{L}{\dgLie_A} \ar@{<-}[ur]_{\adjoint_A} \ar@{=>}[u]^-\alpha &
}
\]
Let $A \to B \to \coho_A L$ be maps of cdga's. To see that
\[
\alpha_B \colon \phi_A \adjoint_A(B) \to \dual{\left(f^A_L\left(\Lcot_{B/A}[1] \otimes_B \coho_A L\right)\right)}
\]
is an equivalence, it suffices to test on the underlying $A$-dg-modules. Those two modules are both equivalent to $\dual{\left(\Lcot_{B/A} \otimes_B A\right)}[-1]$ and the map $\alpha_B$ is an equivalence.
\end{oldschoolmerdique}
\end{proof}

Let us now consider the functor
\[
\mymatrix{
\int \dgModLib_{\coho(-)} \ar[r] & \int \dgMod_{\coho(-)} \ar[r]^-f & \int \dgRep
}
\]
\begin{rmq}\label{sqzeronicou}
Duality and \autoref{derivedcat-ff-local} make the composite functor
\[
\mymatrix{
\int \dgModLib_{\coho(-)} \ar[r] & \int \dgRep \ar[r]^-\theta & \int (\quot{\cdgaunbounded}{\coho(-)})\op
}
\]
equivalent to the functor $(A,L,M) \mapsto \coho_A L \bigoplus \dual M[-1]$.
\end{rmq}

Moreover, the composite functor
\begin{align*}
\int \dgModLib_{\coho(-)} \times_{\int \MCdgLie\op} \int (\MCdgLieLib)\op &\to \int \dgRep \times_{\int \MCdgLie\op} \int (\MCdgLieLib)\op\\
&\to \int \comma{\pt}{\dgLie} \times_{\int \MCdgLie\op} \int (\MCdgLieLib)\op
\end{align*}
has values in the full subcategory of good dg-Lie algebras $\int \comma{\pt}{\dgLieGood}$. Using \autoref{commute-coho-good}, we see that the functor
\[
\int \dgModLib_{\coho(-)} \times_{\int \MCdgLie\op} \int (\MCdgLieLib)\op \to \int (\quot{\cdgaunbounded}{\coho(-)})\op \times_{\int \MCdgLie\op} \int (\MCdgLieLib)\op
\]
preserves coCartesian morphisms.
We finally get a natural transformation
\[
\mymatrix{
\int (\dgLieLib)\op \dcell[rr][{\dgModLib_{\coho(-)}}][{(\quot{\cdgaunbounded}{\coho(-)})\op}][][=>][12pt] && \inftyCatu V
}
\]
There is also a Yoneda natural transformation ${(\quot{\cdgaunbounded}{\coho(-)})\op} \to \comma{\Spec(\coho(-))}{\dStF}$ and we get
\[
G \colon \dgModLib_{\coho(-)} \to \comma{\Spec(\coho(-))}{\dStF}
\]
Let us recall \autoref{formal-nattrans}. It defines a natural transformation
\[
\zeta \colon \dgModLib_{\coho(-)} \times \Delta^1 \to \comma{\Spec(\coho(-))}{\dStF}
\]
such that $\zeta(-,0) \simeq \formal \phi f$ and $\zeta(-,1) \simeq G \simeq h \theta f$.

We are at last ready to prove \autoref{adjointrepresentation}.

\begin{proof}[of \autoref{adjointrepresentation}]
Extending the preceding construction by sifted colimits, we get a natural transformation $\beta \colon \Lqcoh(\formal(-)) \times \Delta^1 \to \comma{\formal(-)}{\dStF}$ of functors $\int \dgLie\op \to \inftyCatu V$.
Let now $X$ denote an Artin derived stack locally of finite presentation. We get a functor
\[
\beta_X \colon \Lqcoh^X(\formal_X \tgtlie_X) \times \Delta^1 \to \comma{\formal_X(\tgtlie_X)}{\dStF_X}
\]
On the one hand, the functor $\beta_X(-,0)$ admits a right adjoint, namely the functor
\[
\mymatrix{
\comma{\formal_X(\tgtlie_X)}{\dStF_X} \ar[r]^-{\lie_X} & \comma{\tgtlie_X}{\dgLie_X} \ar[r]^-{\adjrep_X} & \dgRep_X(\tgtlie_X) \ar[r]^-{g_X} & \Lqcoh^X(\formal \tgtlie_X)
}
\]
while on the other hand, using \autoref{sqzeronicou}, the functor
\[
\Map(\beta_X(-,1), \for{(X \times X)})
\]
is represented  by $\nu_X(\T_X[-1])$ where $\nu_X$ is the functor $\Qcoh(X) \to \Lqcoh^X(\for{(X \times X)})$ defined in the proof of \autoref{derived-global}.
We therefore have a morphism
\[
\nu_X(\T_X[-1]) \to g_X \adjrep_X \lie_X(\for{(X \times X)}) \simeq g_X \adjrep_X(\tgtlie_X)
\]
and hence a morphism $\lierep_X(\T_X[-1]) = f_X \nu_X (\T_X[-1]) \to \adjrep_X(\tgtlie_X)$.
It now suffices to test on the underlying quasi-coherent sheaves on $X$, that it is an equivalence. Both the left and right hand sides are equivalent to $\T_X[-1]$.
\end{proof}
\end{chap-tgtlie}

\chapter{Perspectives}\label{chapterperspectives}
\begin{chap-perspect}
In this last chapter, we will expose a few research directions the author would like to pursue.

\section{A symplectic structure on the formal loop space}
We approached in this thesis the following question: if $X$ is a symplectic stack, does its formal loop space $\kaploop^d(X)$ inherits a Tate symplectic structure. We showed in \autoref{L-affine-tate} that $\kaploop^d(X)$ is endowed with a "local" Tate structure.
\begin{pbnonumber}\label{L-symp}
Let $X$ be an $n$-shifted symplectic derived Artin stack. Does the formal loop space $\kaploop^d(X)$ of dimension $d$ with values in $X$ admit an $(n-d+1)$-shifted symplectic structure, as a locally Tate stack?
\end{pbnonumber}
One way of approaching this issue is by considering the nerve of the morphism $\kaploop_V^d(X) \to \kaploop_U^d(X)$. This defines a groupoid object $Z_\bullet$ in ind-pro-stacks. We showed that when $X$ is affine, the space of morphisms $Z_1$ of this groupoid is $\bubblespace^d(X)$ and admits a closed form as soon as $X$ has one.
This groupoid object is moreover expected to be compatible with the closed form so that it should define a closed form on the quotient.
One core idea of this construction is that the quotient of the fore-mentioned groupoid in the right category -- something like the category of Artin ind-pro-stacks -- should be equivalent to $\kaploop^d(X)$.

Moreover, stating and proving that the form on $\kaploop^d(X)$ is non-degenerate encounters several problems. The first one is that the tangent of $\kaploop^d(X)$ is only know to be a \emph{locally} Tate object but not necessarily globally, so that comparing the tangent with its dual, the cotangent, does not really make sense yet.

\section{A formal loop space over a variety}
As we explained in the introduction, it appeared in \cite{kapranovvasserot:loop1} that the formal loop space (in dimension $1$) can be defined over a curve. It moreover admits a factorization structure allowing us to use some local-to-global argument. This strategy also makes the core of \cite{gaitsgorylurie:weil}.
The question of defining the formal loop space over a variety $V$ -- say of dimension $d$ -- then naturally pops up.
The difficulty lies in defining properly the punctured formal neighbourhood of the diagonal embedding $V \to V \times V$. In \cite{kapranovvasserot:loop1}, the authors defined it as a locally ringed space. This approach would also make sense in the derived setting: for any cdga $A$ with morphism $u \colon \Spec A \to V$, one can consider the space $\Spec A$ with a sheaf of cdga's $\Oo_{\hat \Gamma_u \smallsetminus \Gamma_u}$ corresponding to the punctured formal neighbourhood of the graph $\Gamma_u$ of $u$. The assignement
\[
(u \colon \Spec A \to V) \mapsto \Map_{-/V}((\Spec A,\Oo_{\hat \Gamma_u \smallsetminus \Gamma_u}),X)
\]
when $X$ is a derived affine scheme over $V$, should define some sort of formal loop space over $V$.
Although this definition would only make sense for an affine scheme $X$, one could give a broader definition by enforcing étale descent.
And what of the factorization structure? The same kind of construction should allow us to build an object over the Ran space of $V$, hence defining a factorisation structure on this formal loop space over $V$.
\newcommand{\Ran}{\mathrm{Ran}}
Its Ran space can be seen as the colimit in the category of stacks 
\[
\Ran(V) = \colim_{I \in \mathrm{Fin}^{\twoheadrightarrow}} V^I
\]
where $\mathrm{Fin}^\twoheadrightarrow$ is the category of finite sets and \emph{surjective} maps.
We come to the second problem I would like to investigate.
\begin{pbnonumber}
Let $V$ be a smooth and proper variety of dimension $d$. There should be a factorization monoid $\pi \colon \kaploop(X)_{\Ran(V)} \to \Ran(V)$ in ind-pro-stacks such that, for any configuration $\{\el{v}{p}\}$ of points in $V$, we have a cartesian diagram
\[
\mymatrix{
\left(\kaploop^d(X)\right)^p \ar[r] \ar[d] \cart & \kaploop(X)_{\Ran(V)} \ar[d]^\pi \\ \pt \ar[r]_-{\el{v}{p}} & \Ran(V)
}
\]
Moreover, this factorization structure is expected to admit a flat connection -- ie the map $\pi$ is the pullback of a map $Z \to \Ran(V)_\mathrm{dR}$.
\end{pbnonumber}
When $V$ is Calabi-Yau, this factorization monoid is moreover expected to be compatible, in some sense, with the symplectic structure. A local-to-global argument should allow us to define a symplectic form on the (flat) factorization homology $H_V$ of $\kaploop(X)_{\Ran(V)}$.

This factorisation structure could also be linked to some higher dimension chiral differential operators, following the work of Kapranova and Vasserot in \cite{kapranovvasserot:loop4}. Of course, the first step in this direction would be to give a definition of what higher dimensional chiral and vertex algebras would be. 
The author hopes to follow those ideas together with Giovanni Faonte and Mikhail Kapranov.

\section{Local geometry}
Another direction the author is interested in is the development of a so-called local geometry.
In this thesis, we encountered infinite dimensional stacks with a natural ind-pro-structure. This theory of ind-pro-stacks has some good features but seemed a bit tight for some purposes. Namely, the fact that the algebra $k(\!(t)\!)$ of Laurent series over the base field $k$ does not fit in this context.
There is indeed no reasonible ind-pro-stacks representing $k(\!(t)\!)$ and remembering -- as one would want -- that it is in fact a localisation of a pro-algebra. The same problem arises of course when considering higher dimensional "loops", ie the derived scheme $U_k^d$ as defined in \autoref{chapterloops}. 
\begin{pbnonumber}
Is there a natural geometrical context in which the punctured formal neighbourhood lives and behaves?
\end{pbnonumber}
This problem could be answered with the following observation: the ring of Laurent series is an algebra in the category of ind-pro-vector spaces over $k$:
\[
k(\!(t)\!) = \colim_n \lim_p \quot{k[t]}{t^{n+p}} 
\]
where the maps $\quot{k[t]}{t^{n+p}} \to \quot{k[t]}{t^{n+(p-1)}}$ is naturally a morphism of algebras but the transition maps
\[
\quot{k[t]}{t^{n+p}} \to \quot{k[t]}{t^{(n+1)+p}}
\]
are given by the multiplication by $t$ and thus do not preserve the product.
The idea here is to build a geometry on the site of algebras in the category of ind-pro-vector spaces, or more precisely in ind-pro-perfect complexes over $k$. Those algebras come with a natural definition of a cotangent complex and an étale topology therefore arises.

Let us denote by $\dAff_\IP = (\CAlg(\IPP(k)))\op$ this opposite category of commutative algebras in ind-pro-perfect complexes.
The category $\presh(\dAff_\IP)$ admits a localisation $\dSt_\IP$ with respect to the étale topology. We will call objects in $\dSt_\IP$ local (derived) stacks.
The canonical monoidal inclusion $i \colon \dgMod_k \simeq \Indu U(\Perf(k)) \to \IPP(k)$ defines a restriction functor
\[
\presh(\dAff_\IP) \to \presh(\dAff)
\]
which descends to a restriction functor between the category of stacks for the étale topology on both sides: $i^* \colon \dSt_\IP \to \dSt$. Any local stack hence admits an underlying derived stack.
\todo{attention lax monoidal !!}There is also a lax monoidal realisation functor $\rho \colon \IPP(k) \to \dgMod_k$ which then defines a restriction functor
\[
\rho^* \colon \dSt \to \dSt_{\IP}
\]
The functor $\rho^*$ is fully faithful and defines an embedding of derived stacks into local stacks.

One of the expected feature of this geometry of so called local stacks is that it should contain everything needed to define fully and properly the punctured formal neighbourhood -- as the functor on $\dAff_\IP$ represented by $k((t))$ in dimension $1$.
More generally, let $f \colon \Spec(\quot{A}{I}) \to \Spec A$ be a closed embedding of affine schemes. Let us assume that the ideal $I$ is finitely generated by elements $(\el{a}{n})$.
The punctured formal neighbourhood associated to the map $f$ can now be defined as follows: the ind-pro-perfect vector space
\[
\colim_n \lim_p \quot{A}{(\el{a}{n}[])^{n+p}}
\]
is naturally endowed with a multiplication induced by that of $A$ and hence defines an object in our site $\dAff_\IP$. In this definition, the maps in the pro-direction are given by moding out, while the maps in the ind-direction are given by multiplying by $\el{a}{n}[]$.
The punctured formal neighbourhood of $f$ is then the local stack represented by the local affine scheme above.
This construction should generalize to a reasonible closed immersion between finitely presented schemes.

Once we have a punctered formal neighbourhood, we can define the formal loop space $\kaploop_\IP^d(X)$ in $X$ as the internal hom -- in $\dSt_\IP$ -- from $\hat \A^1 \smallsetminus \{0\}$ to $\rho^* X$. This only works because of a great feature of the monoidal product of ind-pro-vector spaces has: for any cdga $A$, we have an equivalence
\[
A(\!(t)\!) \simeq k(\!(t)\!) \otimes A
\]
Note that the restriction $\rho^*(\kaploop_\IP^d(X))$ is naturally equivalent to $\kaploop_U^d(X)$ as defined in \autoref{chapterloops}. The symplectic structure could then follow, using the same strategy as in \autoref{ipdst-form}.

This alleged local geometry would also give a well-behaved construction of the formal loop space over a variety $V$, as well as its factorisation structure.
The author would also be interested in studying the links between this local geometry and differents fields, including Beilinson's adèles, rigid analytic or adic spaces.
\end{chap-perspect}

\clearpage
\printglossaries

\clearpage
\phantomsection
\bibliographystyle{monmien-en}
\addcontentsline{toc}{chapter}{\iflanguage{francais}{Références}{References}}
\bibliography{biblio}

\end{document}